\pdfoutput=1
\documentclass[10pt,reqno]{article}
 \usepackage{tikz}
 \usepackage{graphicx}
\usepackage[backend=bibtex]{biblatex}
\usepackage[margin=1in]{geometry}
\usepackage{bbm}
\usepackage{amssymb, amsthm, mathtools}
\usepackage{amsfonts}
\usepackage{latexsym, amssymb, amsmath, amscd, amsthm, amsxtra}
\usetikzlibrary{decorations.pathmorphing}
\usepackage{mathtools}
\usepackage{enumerate}
\usepackage[all]{xy}
\usepackage{mathrsfs}
\usepackage{fancyhdr}
\usepackage{listings}
\usepackage{hyperref}
\usepackage{cleveref}
\usepackage{soul}
\usepackage{xcolor}
\usepackage{dsfont}
\usepackage{stmaryrd}
\usepackage{comment}

\allowdisplaybreaks
\DeclareMathOperator{\re}{{\mathrm{Re}}}
\DeclareMathOperator{\im}{{\mathrm{Im}}}

\hypersetup{
     colorlinks=true
}
\def\P{\mathbb{P}}

\def\11{{\mathbf{1}}}

\newcommand{\C}{\mathbb C}

\renewcommand{\bar}{\overline}
\newcommand{\wt}{\widetilde}

\DeclareMathOperator{\OO}{O}

\DeclareMathOperator{\tr}{Tr}

\newcommand{\be}{\begin{equation}}
\newcommand{\ee}{\end{equation}}

\newcommand{\nc}{\normalcolor}

\newcommand{\ba}{{\bf{a}}}

\theoremstyle{plain} 
\newtheorem{theorem}{Theorem}[section]
\newtheorem*{theorem*}{Theorem}
\newtheorem{lemma}[theorem]{Lemma}

\newtheorem*{lemma*}{Lemma}
\newtheorem{corollary}[theorem]{Corollary}
\newtheorem*{corollary*}{Corollary}

\newtheorem*{proposition*}{Proposition}

\newtheorem*{claim*}{Claim}

\newtheorem{definition}[theorem]{Definition}  
\newtheorem*{definition*}{Definition}         
\theoremstyle{remark}
\newtheorem{example}[theorem]{Example}
\newtheorem*{example*}{Example}

\newtheorem*{remark*}{Remark}
\newtheorem*{remarks*}{Remarks}

 \numberwithin{equation}{section}

\allowdisplaybreaks
\def\@setthanks{\vspace{-\baselineskip}\def\thanks##1{\@par##1\@addpunct.}\thankses}

\title{Delocalization of One-Dimensional Random Band Matrices}
\author{
    Horng-Tzer Yau\thanks{Harvard University, htyau@math.harvard.edu} \and
  Jun Yin \thanks{UCLA, jyin@math.ucla.edu}
}
\date{\today}  

\begin{document}

\maketitle

\begin{abstract}

Consider an $ N \times N$ Hermitian one-dimensional random band matrix with  band width $W > N^{1 / 2 + \frak c} $ for any $ {\frak c} > 0$. In the bulk of the spectrum and in the large $ N $ limit, we obtain the following results: (i) The semicircle law holds up to the scale $ N^{-1 + \varepsilon} $ for any $ \varepsilon > 0 $. (ii) All $ L^2 $- normalized eigenvectors are delocalized, meaning their $ L^\infty$ norms are simultaneously bounded by $ N^{-\frac{1}{2} + \varepsilon} $ with overwhelming probability, for any $ \varepsilon > 0 $. (iii) Quantum unique ergodicity holds in the sense that the local $ L^2 $ mass of eigenvectors becomes equidistributed with high probability. (iv) Universality of eigenvalue statistics holds, i.e., the local eigenvalue statistics of these band matrices are given by those of  Gaussian unitary ensembles.

\end{abstract}

\newpage

\tableofcontents

\newpage

\section{Introduction}

The localization-delocalization transition has been a central question in mathematical physics since Anderson's seminal work on the tight-binding model, which essentially describes a discrete random Schr\"odinger operator. The localization of this operator was rigorously established over four decades ago by Fr{\"o}hlich and Spencer \cite{FroSpen_1983} using a multi-scale analysis argument. Subsequently, a shorter proof was provided by Aizenman and Molchanov \cite{Aizenman1993} via the fractional moment method. Many remarkable results concerning the localization of the Anderson model have been achieved (see, e.g., \cite{FroSpen_1985,Bourgain2005,Carmona1987,DingSmart2020,Damanik2002,Germinet2013,LiZhang2019}). However, despite decades of intensive study, the existence of delocalized states remains unproven.

A prominent “toy model” for the random Schrödinger operator is the random band matrix. These matrices are characterized by the fact that the matrix elements \( H_{ij} \) become negligible when the distance between lattice points \( i \) and \( j \) exceeds a parameter \( W \), known as the \textit{band width}. Here, \( i, j \in \mathbb{Z}_N^d \) are lattice points in a \( d \)-dimensional space. Random band matrices play an important role in random matrix theory, and they serve as local models when \( W \) is small, gradually transitioning to the standard Wigner matrices as \( W \) approaches \( N \). As \( W \) varies from order one to \( N \), these matrices interpolate between local models and mean-field models, represented by Wigner matrices.

Band matrices can be real or complex, with typically no fundamental differences between them. However, complex band matrices are often easier to analyze due to simpler diagrammatic methods. Thus, we focus on complex Hermitian band matrices in this paper.

In the special case of \( d = 1 \), it was conjectured \cite{PhysRevLett.64.1851, PhysRevLett.64.5, PhysRevLett.66.986} and supported by a nonrigorous supersymmetry method \cite{PhysRevLett.67.2405} that the eigenvectors of band matrices undergo a localization-delocalization transition, accompanied by a corresponding transition in the eigenvalue distribution.
 Specifically, the conjecture suggests:

\begin{enumerate}[(i)]
    \item For \( W \gg \sqrt{N} \), the bulk eigenvectors are delocalized, and the eigenvalue statistics follow the Gaussian Unitary Ensemble (GUE).
    \item For \( W \ll \sqrt{N} \), the bulk eigenvectors are localized, and the eigenvalue statistics resemble a Poisson point process.
\end{enumerate}

There are similar statements for the edge cases, with the transition occurring at \( W = N^{5/6} \). While the bulk transition has not yet been fully established, the edge cases were solved in various models of band matrices \cite{Sod2010} by Sodin, using moment methods. Although the conjectured transition at \( W = \sqrt{N} \) in the bulk remains open, there are several partial results \cite{BaoErd2015, https://doi.org/10.48550/arxiv.2206.05545, https://doi.org/10.48550/arxiv.2206.06439, erdHos2013local, EK_band1,ErdKno2011, HeMa2018,bourgade2017universality, bourgade2019random,bourgade2020random,yang2021random,Sch2009,PelSchShaSod,Sod2010,SchMT,Sch1,Sch2,Sch3,Sch2014}, but the localization-delocalization transition in one dimension is still a fundamental problem.

In the case where the covariance of the Gaussian matrix elements follows a specific profile, supersymmetric methods can be applied \cite{BaoErd2015, SchMT, Sch1, Sch2,Shcherbina:2021wz, DisPinSpe2002} (see \cite{Efe1997, Spe} for overviews). With this method, for $d=3$, precise estimates on the density of states \cite{DisPinSpe2002} were first obtained. A transition at \( W = N^{1/2} \) was established in \cite{SchMT,Sch1} for the moments of characteristic polynomials, while a more challenging result regarding the two-point functions was proven in \cite{Shcherbina:2021wz} by Shcherbina and  Shcherbina. 

There have been partial results on delocalization in dimensions \( d > 1 \). In particular, delocalization and quantum diffusion were established in dimensions \( d \geq 7 \) \cite{yang2021delocalization, yang2022delocalization, Xu:2024aa}, using complex graphic expansion methods. Despite the complicated nature of these expansions, they introduced crucial concepts such as the sum-zero property of the self-energy, which will play an important role in our analysis.

A fundamental  quantity in all these works concerning  resolvent estimates of band matrices is the \( T \)-observable, introduced in \cite{erdos2013delocalization},  
\begin{align}\label{T}
T_{xy} \sim \sum_{a} S_{xa} |G_{ay}|^2 , \quad S_{xa} = \mathbb{E} |H_{xa}|^2.
\end{align} 
where $H$ is the band matrix. The \( T \)-observable was analyzed in details in  \cite{yang2021delocalization, yang2022delocalization, Xu:2024aa} via diagramtic method.

Recently, a sufficient condition for delocalization in terms of \( W \) was improved in \cite{DY}. Dubova and  Yang utilized a time-dependent approach (previously applied to Wigner matrices \cite{10.1214/19-ECP278, Sooster2019} by Sooster and Warzel) and a linearization of the stochastic flow of the \( T \)-observable. Recall the standard complex matrix Brownian motion:
\[
dH_{t,ij} = \sqrt{S_{ij}} dB_{t,ij}, \quad H_0 = 0,
\]
where \( B_{t,ij} \) are standard independent complex Brownian motions for all \( i \leq j = 1, \dots, N \), and \( B_{ji} = \bar{B}_{ij} \) for all \( i, j \). Following \cite{10.1214/19-ECP278, Sooster2019, DY}, we consider the Green’s function \( H_t \) with a time-dependent spectral parameter \( z_t \) such that the dynamics of
\[
G_{t}^{(E)} := (H_{t} - z_t)^{-1}
\]
are naturally renormalized up to the leading order. Under this flow, one can easily derive an equation for the \( T \)-observable, which unfortunately depends on higher-order objects.  

A nature class to consider is the generalized \( T \)-observables defined by   
\begin{align}\label{gt}
T_{x_1, \dots, x_n} = \sum_{a_1, \dots, a_n} \left( \prod_{i=1}^n S_{x_i a_i} \right) \prod_{i=1}^n G_{a_i a_{i+1}}(z_i), \quad a_{n+1} := a_1, \quad z_i \in \{z, \bar{z}\}.
\end{align} 
for all $n$. On the other hand, estimating these quantities seems to be even more daunting than  estimating \( T \). 
A nature question arises regarding the sizes of the generalized \( T \)-observables.  
 Since \( \sum_s S_{ab} = 1 \), one might expect the naive bound 
\[
T_{x_1, \dots, x_n} \sim  ( \max_{x \ne y} |G_{xy}(z)|  )^n.
\] 
 It turns out that their  true sizes are
\[
T_{x_1, \dots, x_n} \sim  ( \max_{x \ne y} |G_{xy}(z)| )^{2n-2}.
\]  
In this paper, we will work on the following \( G \)-loop observable instead of  generalized \( T \)-observables for some technical convenience, i.e., 
\[
{\cal L}_{t, \boldsymbol{\sigma}, \textbf{a}} = \big \langle \prod_{i=1}^n G_t(\sigma_i) \cdot E_{a_i} \big\rangle, \quad \text{where} \quad \left\langle A \right\rangle = \text{Tr} A.
\]
The \( G \)-loops satisfy a system of  evolution equations called the loop hierarchy  \eqref{eq:mainStoflow}. The dynamics of an \( n \)-loop depend on the \( n+1 \)-loop and a martingale term, whose quadratic variation depends on \( 2n+2 \) loops.

The absence of a closed equation for the \( n \)-loop is a common feature in many-body dynamics, similar to the well-known BBGKY hierarchy in classical mechanics, which governs \( n \)-point correlation functions. Analyzing such hierarchies often requires truncation, but estimating errors due to truncation in higher correlation functions has proven difficult. This is why no powerful rigorous analysis of the BBGKY hierarchy has been available despite its introduction over a century ago.

A main contribution of \cite{DY} was to provide a controlled truncation of the \( T \)-observable dynamics. In this paper, we instead approximate the loop hierarchy for arbitrary length by introducing the \textit{primitive hierarchy},  consisting  only of the quadratic terms on the right-hand side of the loop hierarchy. Despite the nonlinear nature of the primitive hierarchy, it turns out that there is an explicit expression, the primitive loops,  solving this hierarchy exactly.  Furthermore, we will show that the primitive loops are excellent approximations to  the \( G \)-loops. 

It is insightful to compare the current approach with previous work on high-dimensional band matrices \cite{yang2021delocalization, Xu:2024aa}. While the results in higher dimensions closely parallel those presented here, the methods differ significantly. The high-dimensional analysis relies on intricate graphical expansions derived via integration by parts, whereas our approach centers on the loop hierarchy—a stochastic dynamical structure. Despite this methodological divergence, a crucial feature identified in the high-dimensional setting—the sum-zero property—plays a major  role in our analysis of the primitive hierarchy in Section \ref{Sec:CalK}.

By analyzing the loop hierarchy, we will prove the delocalization of one dimensional band matrices for  \( W \gg  \sqrt{N} \). Additionally, we will establish the accompanying quantum diffusion,  quantum unique ergodicity and universality of local eigenvalue statistics.  We believe that the loop hierarchy method introduced in this paper can be extended to other  band matrices, including models with general variance profiles and higher dimensions,  and   random Schrödinger equations with blocked random potentials. We plan to  address these issues in  future works.

\section{The model and main results}
 \subsection{Band matrix model}\label{subsec_setting}

 In this paper, we focus on the band block matrix model, which is defined as follows. Although our methods can be extended to a larger class of band matrices, this specific model allows us to avoid many technical complications that are not central to the main results.
The model is described by a complex complex Hermitian random band matrix \( H \) whose entries  are independent complex Gaussian random variables (up to the Hermitian condition $H_{ij}=\overline{ H_{ji}}$) such that
\[
H_{ij} \sim \mathcal{CN}(0, S_{ij}), \quad  i,j \in \mathbb{Z}_N,
\]
where \( \mathbb{Z}_N \) denotes the set of integers modulo \( N \), with periodic boundary conditions.  
The matrix \( H \) has block structure, where the block size is given by an integer \( W \in \mathbb{N} \), and the number of blocks is denoted by \( L \in \mathbb{N} \). The total matrix size is \( N \times N \), with
\[
N = W \cdot L.
\]
Let \( \mathcal{I}_a \) denote the interval
\[
\mathcal{I}_a := \left[ (a-1)W + 1, \; aW \right], \quad a \in \mathbb{Z}_L.
\]
The elements of the matrix \( S \) are defined by
\[
S_{ij} = \frac{1}{3W} \sum_a \mathbf{1}(i \in \mathcal{I}_a) \mathbf{1}(j \in \mathcal{I}_a \cup \mathcal{I}_{a+1} \cup \mathcal{I}_{a-1}),
\]
where \( \mathbf{1}(\cdot) \) denotes the indicator function. Clearly, \( S = S^T \), and it satisfies
\[
\sum_j S_{ij} = 1.
\]
We will follow the convention that indices \( a, b, c, \dots \) are elements of \( \mathbb{Z}_L \), while \( i, j, k, \dots \) represent indices in \( \mathbb{Z}_N \). We can express the matrix \( S \) as a Kronecker product:
\[
S = S^{(B)} \otimes S_W,
\]
where \( S^{(B)} \) is an \( L \times L \) matrix, and \( S_W \) is a \( W \times W \) matrix. The entries of these matrices are given by
\[
S^{(B)}_{ab} = \frac{\mathbf{1}}{3}(|a-b| \le 1), \quad (S_W)_{ij} = W^{-1}, \quad a, b \in \mathbb{Z}_L, \quad i, j \in \{1, \dots, W\}.
\]

Let \( \lambda_1 \leq \lambda_2 \leq \dots \leq \lambda_N \) be the eigenvalues of \( H \).  Denote by  \( \left( \psi_k \right)_{k=1}^N \)  the corresponding normalized eigenvectors  so that
\[
H \boldsymbol{\psi}_k = \lambda_k \boldsymbol{\psi}_k,\quad\quad  k=1,2\ldots N. 
\]
It is well known that the empirical spectral measure
$\frac{1}{N} \sum_{k=1}^N \delta_{\lambda_k}$
converges almost surely to the Wigner semicircle law with density
\[
\rho_{\mathrm{sc}}(x) = \frac{1}{2\pi} \sqrt{(4 - x^2)_+}.
\]
Moreover, it is believed that the resolvent of \( H \), i.e., \( (H - z)^{-1} \), exhibits quantum diffusion, meaning that
\[
\mathbb{E} \left| (H - z)^{-1} \right|_{xy}^2 \sim \left( \frac{|m|^2}{1 - |m|^2 S} \right)_{xy},
\]
where \( m = m_{\mathrm{sc}}(z) \) is the Stieltjes transform of the semicircle density \( \rho_{\mathrm{sc}} \), defined by
\[
m(z):=m_{\mathrm{sc}}(z) := \int_{\mathbb{R}} \frac{\rho_{\mathrm{sc}}(x)}{x - z} \, dx.
\]
Using the definitions of \( S \), \( S^{(B)} \), and \( S_W \), we can express 
\[
(1 - |m|^2 \cdot S)^{-1} = I + \frac{|m|^2 S^{(B)}}{1 - |m|^2 \cdot S^{(B)}} \otimes S_W.
\]
We define the diffusion length at the block level by
\be\label{def_ellz}
 \ell(z) := \min \left( \eta^{-1/2}, L \right) + 1, \quad \eta = \mathrm{Im}(z).
\ee
In the following Lemma \ref{lem_propTH}, we show that \( S^{(B)} \) decays exponentially with scale \( \ell(z) \):
\[
\left( \frac{1}{1 - |m|^2 S^{(B)}} \right)_{ab} = O \left( (\eta \ell)^{-1} \cdot e^{-c \frac{|a - b|}{\ell}} \right),
\]
where \( c \) is a constant.

\subsection*{Stochastic Domination}

In this paper, we adopt the convention of stochastic domination introduced in \cite{Average_fluc}. This framework will be used throughout the analysis to control the behavior of random matrices and their spectral properties.

\begin{definition}[Stochastic domination and high probability event]\label{stoch_domination}
	{\rm{(i)}} Let
	\[\xi=\left(\xi^{(N)}(u):N\in\mathbb N, u\in U^{(N)}\right),\hskip 10pt \zeta=\left(\zeta^{(N)}(u):N\in\mathbb N, u\in U^{(N)}\right),\]
	be two families of non-negative random variables, where $U^{(N)}$ is a possibly $N$-dependent parameter set. We say $\xi$ is stochastically dominated by $\zeta$, uniformly in $u$, if for any fixed (small) $\tau>0$ and (large) $D>0$, 
	\[\mathbb P\bigg(\bigcup_{u\in U^{(N)}}\left\{\xi^{(N)}(u)>N^\tau\zeta^{(N)}(u)\right\}\bigg)\le N^{-D}\]
	for large enough $N\ge N_0(\tau, D)$, and we will use the notation $\xi\prec\zeta$. 
		If for some complex family $\xi$ we have $|\xi|\prec\zeta$, then we will also write $\xi \prec \zeta$ or $\xi=\OO_\prec(\zeta)$. 
	
	\vspace{5pt}
	\noindent {\rm{(ii)}} As a convention, for two deterministic non-negative quantities $\xi$ and $\zeta$, we will write $\xi\prec\zeta$ if and only if $\xi\le N^\tau \zeta$ for any constant $\tau>0$. 

 \vspace{5pt}
	\noindent {\rm{(iii)}} Let $A$ be a family of random matrices and $\zeta$ be a family of non-negative random variables. Then, we use $A=\OO_\prec(\zeta)$ to mean that $\|A\|\prec \xi$, where $\|\cdot\|$ denotes the operator norm. 
	
	\vspace{5pt}
	\noindent {\rm{(iv)}} We say an event $\Xi$ holds with high probability (w.h.p.) if for any constant $D>0$, $\mathbb P(\Xi)\ge 1- N^{-D}$ for large enough $N$. More generally, we say an event $\Omega$ holds $w.h.p.$ in $\Xi$ if for any constant $D>0$,
	$\P( \Xi\setminus \Omega)\le N^{-D}$ for large enough $N$.
\end{definition}

\subsection{Main results}

The following theorems are our main results on delocalization,  local semicircle law, quantum unique ergodicity and quantumn diffusion.

\begin{theorem}[Delocalizaiton]\label{MR:decol}
Suppose  that for some  $ \mathfrak c>0$, 
\begin{equation}\label{Main_DEL_COND}
    W \ge N^{1/2\,+\,\mathfrak c}. 
\end{equation}
For  band matrix defined in this subsection we have the following estimate. 
For any (small) constants $\kappa, \tau>0$ and (large) $D>0$, there exists $N_0$ such that for all $N \geqslant N_0$ we have
$$
\mathbb{P}\left( \max_k\left\|\boldsymbol{\psi}_k\right\|_{\infty}^2 \cdot {\bf1}\big(\lambda_k\in[-2+\kappa,2-\kappa]\big)\leqslant N^{-1+\tau}  \right) \geqslant 1-N^{-D}
$$
\end{theorem}

\begin{theorem}[Local semicircle law]\label{MR:locSC} Suppose that  the assumption of Theorem \ref{MR:decol} holds. We denote the Green's function by 
$ G(z)=(H-z)^{-1}$. 
Then for 
 any fixed small constants $\kappa, \tau>0$, large $D>0$, and 
$$z=E+i\eta,\quad\quad |E|\le  2-\kappa, \quad 1\ge \eta\ge N^{-1+\tau},
$$
there exists $N_0$ such that for all $N \geqslant N_0$  we have the  local law (with $\ell=\ell(z)$ defined in \eqref{def_ellz})
\begin{equation}\label{G_bound}
   \mathbb P\left( \max_{x,y\in \mathbb Z_N}  \left|\left(G(z)-m(z)\right)_{xy}\right|\le  \frac{W^\tau}{(W\ell\eta)^{ 1/2}}\cdot \right)\geqslant 1-N^{-D},\quad \ell=\ell(z) 
\end{equation}
and the partial   tracial  local law
\begin{equation}\label{G_bound_ave}
\mathbb{P}\left(\max_a\Big| W^{-1}\sum_{x\in {\cal I}_a}G_{xx}(z)-m(z) \Big| \le \frac{W^\tau}{W\ell\eta} \right)\geqslant 1-N^{-D},\quad \ell=\ell(z). 
\end{equation}

\end{theorem}
The partial tracial  local law implies  the standard tracial local law 
 $$
\mathbb{P}\left(\max_a\Big| N^{-1}\tr G(z)-m(z) \Big| \le \frac{W^\tau}{W\ell\eta}  \right)\geqslant 1-N^{-D},\quad \ell=\ell(z) . 
$$

\begin{theorem}[Quantum diffusion]\label{MR:QDiff}
Denote by  $E_a$ the  block identity matrix 
\begin{equation}\label{Def_matE}
(E_a)_{ij}=  \delta_{ij}\cdot W^{-1}\cdot\textbf{1}(i\in {\cal I}_a),\quad 1\le i,j\le N
\end{equation}
Under the assumptions of Theorem \ref{MR:locSC} and the notation $m=m(z)$ we have 
 \begin{align}\label{Meq:QdW1}
 &
 \mathbb P\left(\max_{a,b }   \left|\tr G E_a G^\dagger E_b- W^{-1}\left(\frac{|m|^2}{1-|m|^2 S^{(B)}} \right)_{ab}\right|\le\frac{W^\tau}{ \left(W \ell\eta\right)^{ 2}}\right)\ge 1-N^{-D},
 \\ \label{Meq:QdW2}
 &
 \mathbb P\left( \max_{a,b }   \left|\tr G E_a G  E_b- W^{-1}\left(\frac{m^2}{1-m^2 S^{(B)}} \right)_{ab}\right|\le\frac{W^\tau}{ \left(W \ell\eta\right)^{ 2}}\right)\ge 1-N^{-D},
\end{align}
and a stronger bound on the expectation value 
\begin{align}\label{Meq:QdS1}
 \max_{a,b }   \left|\mathbb E\tr G E_a G^\dagger E_b- W^{-1}\left(\frac{|m|^2}{1-|m|^2 S^{(B)}} \right)_{ab}\right|\le  \left(W \ell\eta\right)^{-3} \cdot W^\tau,   \\
 \label{Meq:QdS2}
 \max_{a,b }   \left|\mathbb E\tr G E_a G  E_b- W^{-1}\left(\frac{m^2}{1-m^2 S^{(B)}} \right)_{ab}\right|\le \left(W \ell\eta\right)^{-3} \cdot W^\tau   .
\end{align}
 
\end{theorem}

Theorems \ref{MR:decol} and \ref{MR:QUE} are simple consequences of Theorems \ref{MR:locSC} and \ref{MR:QDiff}. 
\begin{proof}[Proof of Theorem \ref{MR:decol} ]
Following  a standard delocalization argument, we have 
\begin{equation} \label{kkzxis}
|\psi_k(x)|^2\le 
\sum_{l}
\frac{\eta^2|\psi_l(x)|^2 }{(\lambda_k-\lambda_l)^2+\eta^2} 
\le \eta \im G_{xx}(\lambda_k+\eta i)
    \end{equation}
By  assumption  \eqref{Main_DEL_COND},   $\eta=N^{-1+\tau}$ with  small enough $\tau>0$ satisfying 
$ \eta\le (W/N)^2$.
With this choice of $\eta$, we have  $\ell\sim L$. Applying  \eqref{G_bound} and the fact that  $\ell\sim L$, we obtain $G-m\ll1$ and thus  $G=O(1)$ with high probability. Inserting  it back to  \eqref{kkzxis},  we obtain  
$$
|\psi_k(x)|^2\le C\eta\le N^{-1+\tau}
$$
with high probability.  This  completes the proof of Theorem \ref{MR:decol}.  
\end{proof}

 Quantum diffusion immediately implies the following probabilistic 
generalized quantum unique ergodicity   in high probability.  This notion is an extension of the probabilistic 
generalized quantum unique ergodicity first proved by \cite{vector_flow} in the setting of Wigner matrices. Notice that a deterministic version of quantum unique ergodicity  was a central concept in the work by \cite{RudSar1994}.

\begin{theorem}[Generalized quantum unique ergodicity]\label{MR:QUE} 
Suppose that the assumptions of Theorem \ref{MR:decol} hold and 
\begin{align}\label{tau} 
0< \tau ^*<  \frac {  \mathfrak c } 2, 
\end{align} 
where $ \mathfrak c$ was defined in \eqref{Main_DEL_COND}.
Then generalized quantum unique holds in high probability, i.e.,  for any small constant $\kappa>0$  and large $D>0$,   there exists  a large  $N_0$ such that for all $N \geqslant N_0$, 
\begin{equation}\label{Meq:QUE}
\max_{E: \,|E|<2-\kappa}\; 
\max_{ a\,\in\, \mathbb Z_L}\;\mathbb{P}\left(\max_{i, j \in {\cal J}_E} 
 \left|N (\psi_i^*\left(E_a-N^{-1}\right) \psi_j)\right|^2 \ge N^{-\tau ^*/6} \right) \le  N^{-\tau ^*/6}.
\end{equation}
Here 
$${\cal J}_E:=\left\{x: |x-E|\le N^{-1-\tau ^*} (W^2/N)^{1/3}\right\}.
$$
Furthermore, for  any  subset $A\subset \mathbb Z_L$,
\begin{equation}\label{Meq:QUE2}
\max_{E: \,|E|<2-\kappa}\; 
\max_{ A\,\subset\, \mathbb Z_L}\;\mathbb{P}\left(\max_{k\in {\cal J}_E} 
\left|\,\sum_{x\in {\cal I}_a}\sum_{a\,\in\, A}\left|\boldsymbol{\psi}_k(x)\right |^2 -\frac{|A|\cdot W}{N}\right| \geqslant \frac{|A|\cdot W}{N^{1+\tau ^*/6}}  \right) \le N^{-\tau ^*/6}
\end{equation}

\end{theorem}

\begin{proof}[Proof of Theorem \ref{MR:QUE}]  

We first choose  $  z=E+\eta i, \quad \eta= N^{-1-\tau ^* }(W^2/N)^{1/3}$.
By definition of $ {\cal J}_E$, we have 
\begin{align}\label{ssfa2}
& \sum_{i,j} \textbf{1}(\lambda_i,\lambda_j\in {\cal J}_E)\cdot \left|(\psi_i^*\left(E_a-N^{-1}\right) \psi_j)\right|^2 \nonumber \\
& 
\le C \eta^4 \sum_{i, j  } \frac{\left|(\psi_i^*\left(E_a-N^{-1}\right) \psi_j)\right|^2}{\left|\lambda_i-z\right|^2\left|\lambda_j-z\right|^2}\\\nonumber
& \le 
 C \eta^2\tr \left(\operatorname{Im} G(z)\left(E_a-N^{-1}\right) \operatorname{Im} G(z)\left(E_a-N^{-1}\right)\right)
\end{align}
We can bound the last term by 
\begin{align}\label{que0} 
 &\mathbb{E}  \tr \left(\operatorname{Im} G(z)\left(E_a-N^{-1}\right) \operatorname{Im} G(z)\left(E_a-N^{-1}\right)\right)
  \\ 
=&\frac{1}{L^2}\sum_{b,b'}
\mathbb{E}  
 \tr \left(\operatorname{Im} G(z)\left(E_a-E_b\right) \operatorname{Im} G(z)\left(E_a-E_{b'}\right)\right)
\nonumber \\ \nonumber
\le C&\max_{a,b,a',b'}
\Big|\mathbb{E}\tr \left(\operatorname{Im} G(z) E_a  \operatorname{Im} G(z)E_b \right)
-\mathbb{E}\tr \left(\operatorname{Im} G(z) E_{a'}  \operatorname{Im} G(z)E_{b'}\right) \Big| 
\end{align}
By \eqref{tau},    we have $ \eta^{-1/2} \ge L$ and hence $\ell(z)=L$.
Then from \eqref{Meq:QdS1} and \eqref{Meq:QdS2} in Theorem \ref{MR:QDiff}, 
we have 
 \begin{align}
 \;& \eta^2 \mathbb{E}  \tr \left(\operatorname{Im} G(z)\left(E_a-N^{-1}\right) \operatorname{Im} G(z)\left(E_a-N^{-1}\right)\right)
\nonumber \\
\le \;& \eta^2 (W\ell\eta)^{-3}\cdot W^{\delta} + \eta^2 W^{-1+\delta}\max_{a,b,a',b'}\;\max_{ \xi=|m|^2 \; \rm or\; m^2}
\left| \left(\frac{1}{1-\xi\cdot S^{(B)}}\right)_{ab}-\left(\frac{1}{1-\xi\cdot S^{(B)}}\right)_{a'b'}\right|
\nonumber \\\nonumber
\le \;&   W^{\delta} N^{-2} \big [ N^{-1}\eta^{-1}  +  N^2 \eta^{3/2} W^{-1}  \big ] =   W^{\delta} N^{-2} \big [ N^{\tau ^*-2 \mathfrak c /3}  + N^{-3 \tau ^*/2}  \big ] 
<   N^{-2}  N^{- \tau ^* /3},  
 \end{align}
provided that  $\delta$ is small enough and \eqref{tau} is satisfied. Here  we have used  the $\ell(z)=L$, 
\begin{align}\label{me}
|m(z)|^2 \le 1-c\,\eta, 
\end{align}
and the estimate on $(1-|m|^2 S^{(B)})^{-1}$ in  Lemma \ref{lem_propTH}.  Inserting the  above estimate back to the right hand side  of \eqref{ssfa2}, we have 
\begin{align}\label{que2}
\mathbb E \; \eta^4 \sum_{i, j  } \frac{ \left|N (\psi_i^*\left(E_a-N^{-1}\right) \psi_j)\right|^2}{\left|\lambda_i-z\right|^2\left|\lambda_j-z\right|^2}
 \le N^{- \tau ^* /3}.
\end{align}
 Therefore with probability $1-O(N^{- \tau ^* /6})$, the right hand side  of \eqref{ssfa2} is bounded by $N^{- \tau ^* /6}$. Together with \eqref{ssfa2}, we obtain \eqref{Meq:QUE}.  For \eqref{Meq:QUE2},  we  only need to replace $E_a$ in \eqref{ssfa2} with $|A|^{-1}\sum_{a\in A} E_a$ and use the same argument. This completes the proof Theorem \ref{MR:QUE}

\end{proof}

 \subsection{Universality} 
Recall that the $k$-point correlation functions  of $H$ are defined  by
\[
\rho_H^{(k)}\left(\alpha_1, \alpha_2, \ldots, \alpha_k\right):=\int_{\mathbb{R}^{N-k}} \rho_H^{(N)}\left(\alpha_1, \alpha_2, \ldots, \alpha_N\right) \mathrm{d} \alpha_{k+1} \cdots \mathrm{~d} \alpha_N, 
\]
where $\rho_H^{(N)}\left(\alpha_1, \alpha_2, \ldots, \alpha_N\right)$ is the joint density of all unordered eigenvalues of $H$. 

 \begin{theorem}[Bulk universality]\label{Thm: B_Univ}Suppose that  the assumptions of Theorem \ref{MR:decol} holds. For any fixed $k\in \mathbb N$, the $k$ point correlation function of $H$ converges to that of GUE in the following sense. For any $|E| \leq 2-\kappa$ and smooth test function $\mathcal{O}$ with compact support, we have
 \begin{equation}\label{eq:universality}
\lim _{N \rightarrow \infty} \int_{\mathbb{R}^k} \mathrm{~d} \boldsymbol{\alpha} \mathcal{O}(\boldsymbol{\alpha})\left\{\left(\rho_H^{(k)}-\rho_{\mathrm{GUE}}^{(k)}\right)\left(E+\frac{\boldsymbol{\alpha}}{N}\right)\right\}=0 
 \end{equation}  
 \end{theorem}

This universality result demonstrates that the local eigenvalue statistics at a fixed energy level for band matrices match those of Wigner matrices. However, the distributions of individual eigenvalues of band matrices differ significantly from those of Wigner matrices. As suggested by Theorem 1 of \cite{Sch2015}, the fluctuations of bulk eigenvalues in band matrices are of order at least $(WN)^{-1/2}$, whereas the fluctuations for the eigenvalues of Wigner matrices are of order $ N^{-1}$ \cite{erdHos2012rigidity}.
 
 \begin{proof}[Proof of Theorem \ref{Thm: B_Univ}]
 Our proof adopts the strategy employed in Theorem 1.3 of \cite{Xu:2024aa}, which establishes the bulk universality of certain high-dimensional band matrices. In \cite{Xu:2024aa}, it is demonstrated in  section 1.2 of \cite{Xu:2024aa} that universality follows from the local law (Theorem \ref{MR:locSC} in our setting), 
 delocalization (Theorem \ref{MR:decol}) 
 and QUE estimates (Theorem \ref{MR:QUE}).

To compare the correlation functions of band matrix $H$ and the GUE random matrix $H_{GUE}$,  define  the matrix Ornstein-Uhlenbeck process $H_t$ as the solution to
\begin{equation}\label{defUnivHt}
\mathrm{d} H_t=-\frac{1}{2} H_t \mathrm{~d} t+\frac{1}{\sqrt{N}} \mathrm{~d} B_t, \quad \text { with } H_0=H.   
\end{equation}
By definition, $H_\infty=H_{GUE}$. We aim to  show that 
\begin{equation}\label{0infy}
\lim _{N \rightarrow \infty} \int_{\mathbb{R}^k} \mathrm{~d} \boldsymbol{\alpha} \mathcal{O}(\boldsymbol{\alpha})\left\{\left(\rho_{H_0}^{(k)}-\rho_{H_\infty}^{(k)}\right)\left(E+\frac{\boldsymbol{\alpha}}{N}\right)\right\}=0 
 \end{equation}

\medskip 
\noindent 
\emph{Step 1}: Using the local semicircle law (Theorem \ref{MR:locSC}) as input and applying Theorem 2.2 from \cite{LANDON20191137}, we establish the universality of the correlation functions of $H_{t_*}$ at $t_*=N^{-1+\tau_*}$ for any fixed $\tau_*>0$, explicitly,
\begin{equation}\label{1infyuniv}
\lim_{N \rightarrow \infty} \int_{\mathbb{R}^k} \mathrm{d} \boldsymbol{\alpha}\, \mathcal{O}(\boldsymbol{\alpha}) \left\{\left(\rho_{H_{t_*}}^{(k)} - \rho_{H_\infty}^{(k)}\right)\left(E+\frac{\boldsymbol{\alpha}}{N}\right)\right\}=0,\quad t_*=N^{-1+\tau_*}.
\end{equation}

\medskip 
\noindent 
\emph{Step 2}: Analogously to Proposition 4.17 in \cite{Xu:2024aa}, we assert that there exists a constant $c'>0$ such that the following holds:
Fix a sufficiently small parameter $\tau_U>0$ (temporarily denoted $\tau_U$ for clarity in our universality proof) and assume $|E| \leq 2-\kappa$. Define complex parameters $z_i=E_i+\mathrm{i}\eta_i$ for $i=1,\dots,n$, satisfying
\begin{equation}\label{qkhzle23}
|E_i - E| \leq C N^{-1},\quad N^{-1-\tau_U} \leq \eta_i \leq N^{-1+\tau_U},\quad t_U:=N^{-1+\tau_U}.
\end{equation}
Under the conditions of Theorem~\ref{Thm: B_Univ}, we claim that for any fixed $n \in \mathbb{N}$, there exists $c'>0$ such that for sufficiently small $\tau_U$,
\begin{align}\label{417}
\left|\mathbb{E} \prod_{i=1}^n \operatorname{Im} m_0(z_i) - \mathbb{E} \prod_{i=1}^n \operatorname{Im} m_{t_U}(z_i)\right| \leq N^{-c'+C_n \tau_U},\quad t_U=N^{-1+\tau_U},
\end{align}
where
\[
m_t(z_i)= \frac{1}{N}\operatorname{tr}(H_t-z_i)^{-1}, \quad t\in [0, t_U].
\]

Following the approach in \cite{Xu:2024aa}, for each fixed $n \in \mathbb{N}$, we choose $\tau_U$ to be much smaller than $c'$. By applying the standard correlation function comparison theorem (Theorem 15.3 in \cite{erdHos2017dynamical}) and Proposition 4.17 from \cite{Xu:2024aa}, we conclude that for every fixed integer $k$, there exists a sufficiently small $\tau_U>0$ such that
\begin{equation}\label{01univ}
\lim_{N \rightarrow \infty} \int_{\mathbb{R}^k} \mathrm{d}\boldsymbol{\alpha}\,\mathcal{O}(\boldsymbol{\alpha})\left\{\left(\rho_{H_0}^{(k)}-\rho_{H_{t_U}}^{(k)}\right)\left(E+\frac{\boldsymbol{\alpha}}{N}\right)\right\}=0, \quad t_U=N^{-1+\tau_U}.
\end{equation}
Combining \eqref{01univ} and \eqref{1infyuniv} (choosing $t_*=t_U$), we obtain \eqref{0infy}, which yields the universality result \eqref{eq:universality}. The proof of our key claim \eqref{417} proceeds in the following steps.

\medskip 
\noindent 
\emph{Step 3}: 
By Lemma 4.18 of \cite{Xu:2024aa} we have for any fixed $n \in \mathbb{N}$ that
\begin{align}\label{EMCTE}
\; &\;   \left|\mathbb{E} \prod_{i=1}^n \operatorname{Im} m_0\left(z_i\right)-\mathbb{E} \prod_{i=1}^n \operatorname{Im} m_{t_U}\left(z_i\right)\right|
\\\nonumber
 \leq \; &\; C\cdot t_U\cdot \max_{0\le t\le t_U}\,\mathbb{E}\left[\sum_u L_{1, t}\left(z_u\right) \prod_{i \neq u} \operatorname{Im} m_t\left(z_i\right)+\sum_{u \neq v} L_{2, t}\left(z_u, z_v\right) \prod_{i \neq u, v} \operatorname{Im} m_t\left(z_i\right)\right],
\end{align}
where  $C>0$ is an absolute constant and 
$$
\begin{aligned}
L_{1, t}(z) & :=\sum_{\mathbf{G}_1, \mathbf{G}_2 \in\left\{G_t, G_t^*\right\}}\left|\frac{1}{N} \sum_{a, b}\left(\mathbf{G}_1^2(z)\right)_{a a} S_{a b}^{\circ}\left(\mathbf{G}_2(z)\right)_{b b}\right|, \\
L_{2, t}\left(z_1, z_2\right) & :=\sum_{\mathbf{G}_1, \mathbf{G}_2 \in\left\{G_t, G_t^*\right\}}\left|\frac{1}{N^2} \sum_{a, b}\left(\mathbf{G}_1^2\left(z_1\right)\right)_{a b} S_{a b}^{\circ}\left(\mathbf{G}_2^2\left(z_2\right)\right)_{b a}\right| , 
\end{aligned}
$$
$$
S_{xy}^{\circ}=S_{xy}-N^{-1}.
$$
Similar to \cite{Xu:2024aa}, we claim that the following weaker versions of Theorems \ref{MR:locSC} and \ref{MR:QUE} hold for $H_t$ defined in \eqref{defUnivHt}.

\begin{enumerate}
    \item (Weak version of Theorem \ref{MR:locSC}): Under the assumptions of Theorem \ref{MR:decol}, for sufficiently small $\tau_U$ (depending on $\mathfrak{c}$ in \eqref{Main_DEL_COND}), and parameters satisfying
    \[
    t \in [0,t_U], \quad t_U = N^{-1+\tau_U}, \quad z = E + \mathrm{i}\eta, \quad |E| < 2 - \kappa, \quad \eta = N^{-1+2\tau_U},
    \]
    we have
    \begin{equation}\label{Weakloclaw}
        \max_{x} |(G_t)_{xx}| \prec 1, \quad G_t := (H_t - z)^{-1}.
    \end{equation}
    (Note: Here, we explicitly state and prove only the case $\eta = N^{-1+4\tau_U}$, although the general case should also hold.)

    \item (Weak version of Theorem \ref{MR:QUE}): For sufficiently small fixed $\tau_U$ (depending on $\mathfrak{c}$ in \eqref{Main_DEL_COND}), we have
    \begin{equation}\label{WeakQUE}
        \text{Theorem \ref{MR:QUE} holds for $H_t$, with $\tau^* = \mathfrak{c}/3$, and $0 \leq t \le N^{-1+\tau_U}$.}
    \end{equation}
\end{enumerate}
We remark that the results above remain valid without the extra condition on $\eta$. However, this simplified form is sufficient for our universality proof and significantly streamlines the argument.

We postpone  the proof  of \eqref{Weakloclaw} and \eqref{WeakQUE} to subsection \ref{Apu}. 
Assuming these two results, we complete the proof of of Theorem \ref{Thm: B_Univ}.

It is well-known that  for any $E$ in the bulk that 
$$
\eta\le \widetilde{\eta}\implies\eta \operatorname{Im} m\left(E +\mathrm{i} \eta \right) \leq \tilde{\eta} \operatorname{Im} m\left(E +\mathrm{i} \tilde{\eta}\right).
$$ 
By choosing $\widetilde{\eta}=N^{-1+2\tau_U}$ and \eqref{Weakloclaw}, we obtain that for $z_i$'s in \eqref{qkhzle23}, 
$$
\max_{0\le t\le t_U}\im m_t (z_i)\prec N^{2\tau_U}\cdot (N\im z_i)^{-1}\prec N^{3\tau_U}.$$
Inserting  it back to \eqref{EMCTE}, we obtain that
\begin{align}\label{EMCTE2}
 \left|\mathbb{E} \prod_{i=1}^n \operatorname{Im} m_0\left(z_i\right)-\mathbb{E} \prod_{i=1}^n \operatorname{Im} m_{t_U}\left(z_i\right)\right|
\;\prec\;  N^{-1+C_n\tau_U}\cdot  \max_{0\le t\le t_U}\max_{u\ne v}\,\mathbb{E}\left[  L_{1, t}\left(z_u\right)  +  L_{2, t}\left(z_u, z_v\right)\right].
\end{align}
For bounding $L_1$ and $L_2$, in the following we will prove that, for $\mathbf{G}_1, \mathbf{G}_2 \in\left\{G_t, G_t^*\right\}$,  and $z_i$'s in \eqref{qkhzle23}
\begin{align}\label{jaklsdufowe}
  & \max_i\max_{x,y}\; \mathbb E\left| \sum_{ x}\left(\mathbf{G}_1^2(z_i)\right)_{x x} S_{x y}^{\circ}\left(\mathbf{G}_2(z_i)\right)_{yy}\right|\le N^{1-\mathfrak c/18+C\tau_U} \\\label{uywy7723r3rf}
 & \max_{i\ne j}\max_{x,y}\;\mathbb E\left|  \sum_{  x}\left(\mathbf{G}_1^2\left(z_i\right)\right)_{xy} S_{x y}^{\circ}\left(\mathbf{G}_2^2\left(z_j\right)\right)_{yx}\right| \prec N^{2-\mathfrak c/18+C\tau_U}.
\end{align}
By Lemma 4.20 of \cite{Xu:2024aa} (or use the eigen-decomposition
of $G$) 
\begin{equation}\label{yw982823}
     \left|\sum_{x}\left(\mathbf{G}_1^2(z)\right)_{xx} S_{x y}^{\circ}\left(\mathbf{G}_2(z)\right)_{yy}\right|\prec N^{-2}\sum_{\alpha,\beta} |p_\alpha(z)|^2 \cdot |p_\beta(z)|\cdot \left| M_{y,\alpha} \right| ,
\end{equation}
where $$p_\alpha(z):=\left(\lambda_\alpha-z\right)^{-1}$$
and 
$$
     M_{y,\alpha}:= N \sum_{x}|u_\alpha(x)|^2 S^{\circ}_{xy}= N \sum_{a: |a-a_0|\le 1}\frac13
     \left\langle u_\alpha \left(E_a-N^{-1} I\right) \,u^*_\alpha\right\rangle, \quad y\in {\cal I}_{a_0}.
$$
With the  local law estimate \eqref{Weakloclaw} at the scale $N^{-1+2\tau_U}$, and the  eigenvector $L_\infty$ bound in  \eqref{kkzxis}, one can obtain that 
\begin{equation}\label{80whg}
       \max_{\alpha,y} |M_{y,\alpha}|\prec  N^{C\tau_U},\quad  \sum_{\alpha } |p_\alpha(z)|^2 \prec N^{2+C\tau_U},\quad  \sum_{ \beta}    |p_\beta(z)|\prec N^{1+C\tau_U}
\end{equation} 
On the other hand,  $M_{y,\alpha}$ has a sharper bound  with  high probability  $1-o(1)$). 
To see this, define  the bad event $\cal B$ 
$$
{\cal B} = \left\{\exists \alpha \mbox { such that } |\lambda_\alpha-E|\le N^{-1+\mathfrak c/6} \quad \mbox{and} \quad \left| M_{x,\alpha}\right|\ge N^{-\mathfrak c/18}\right\}.
$$
Using the QUE estimate of $H_t$ in \eqref{WeakQUE} and \eqref{Meq:QUE} with $\tau^*=\mathfrak c/3$, we have 
$$
\mathbb P\left({\cal B} \right)=O\left(N^{-\mathfrak c/18}\right). 
$$
Therefore, by  \eqref{80whg}, 
$$
\mathbb E\; \left( {\bf 1}_{{\cal B} }\cdot \sum_{\alpha,\beta} |p_\alpha(z)|^2 \cdot |p_\beta(z)|\cdot \left| M_{x,\alpha} \right| \right)
=O( N^{3-\mathfrak c/18+C\tau_U} ).
$$
On the good  event ${\cal B}^c $, we split the sum $\alpha$ into 
$$
\sum_{\alpha }= \sum^{\le} _{\alpha }+\sum^{>} _{\alpha }\;,\quad \quad \sum^{\le} _{\alpha }:=\sum_{\alpha: |\lambda_\alpha-E|\le N^{-1+\mathfrak c/6}} .
$$
Then
\begin{align}
{\bf 1}_{{\cal B}^c }\cdot \sum_{\alpha } ^{\le}|p_\alpha(z)|^2  \cdot \left| M_{x,\alpha} \right|& \; \prec \;
{\bf 1}_{{\cal B}^c }\cdot N^{-\mathfrak c/18}\cdot  \sum_{\alpha } ^{\le} |p_\alpha(z)|^2  \le N^{2-\mathfrak c/18+C\tau_U} 
\\\nonumber
{\bf 1}_{{\cal B}^c }\cdot \sum_{\alpha } ^{>}|p_\alpha(z)|^2  \cdot \left| M_{x,\alpha} \right|  &\; \prec \;
{\bf 1}_{{\cal B}^c }\cdot N^{C\tau_U}\cdot  \sum_{\alpha } ^{>} |p_\alpha(z)|^2  \le N^{2-\mathfrak c/6+C\tau_U} .
 \end{align}
 Combining  them with the last term in \eqref{80whg},  we obtain 
$$
\mathbb E \left({\bf 1}_{{\cal B}^c }\cdot \sum_{\alpha,\beta} |p_\alpha(z)|^2 \cdot |p_\beta(z)|\cdot \left| M_{x,\alpha} \right|\right) 
\prec  N^{3-\mathfrak c/18+C\tau_U} , 
$$
which implies  \eqref{jaklsdufowe}.

To prove  \eqref{uywy7723r3rf}, we use  Lemma 4.20 of \cite{Xu:2024aa}  to have 
$$
  \left|  \sum_{  x}\left(\mathbf{G}_1^2\left(z_i\right)\right)_{xy} S_{x y}^{\circ}\left(\mathbf{G}_2^2\left(z_j\right)\right)_{yx}\right|
\prec N^{-2} \sum_{\alpha,\beta} |p_\alpha(z_i)|^2|p_\beta(z_j)|^2 \left|M_{y,\alpha,\beta}\right|,
$$
where
 $$
 M_{y,\alpha,\beta}:= N \sum_{x}u_\alpha(x) \overline{u_{\beta}(x)}\cdot S^0_{xy}= \sum_{a: |a-a_0|\le 1}\frac N 3
     \left\langle u_\alpha \left(E_a-N^{-1} I\right) \,u^*_\beta\right\rangle, \quad y\in {\cal I}_{a_0}.
$$
Similar to \eqref{80whg}, we have
$$
    \left|M_{y,\alpha,\beta}\right|\prec N^{C\tau_U}.
$$
Define the bad event  
$$
\widetilde{\cal B} := \left\{\exists \alpha_1,\alpha_2 \mbox { such that } \max_i|\lambda_{\alpha_i}-E|\le N^{-1+\mathfrak c/6} \quad \mbox {and} \quad \left| M_{y,\alpha_1,\alpha_2}\right|\ge N^{ -\mathfrak c/18}\right\}.
$$
Again using  \eqref{WeakQUE} and \eqref{Meq:QUE} with  $\tau^*=\mathfrak c/3$,  we have 
$\mathbb P (\widetilde{{\cal B}}  )=O\left(N^{-\mathfrak c/18}\right)$ and 
$$
\mathbb E\; \left(
{\bf 1}_{\widetilde{{\cal B}} }\cdot \sum_{\alpha,\beta} |p_\alpha(z_i)|^2|p_\beta(z_j)|^2 \left|M_{y,\alpha,\beta} \right|
\right)
=O( N^{4-\mathfrak c/18+C\tau_U} ).
$$
On the good set, we have 
$$
{\bf 1}_{\widetilde{{\cal B}}^c }\cdot \sum_{\alpha,\beta} |p_\alpha(z_i)|^2|p_\beta(z_j)|^2 \left|M_{y,\alpha,\beta} \right|
\le N^{4-\mathfrak c/18+C\tau_U}.  
$$
We have thus proved  \eqref{uywy7723r3rf}.  
Combining it with \eqref{jaklsdufowe} and \eqref{EMCTE2} (and the definitions of $L_1$ and $L_2$), we obtain \eqref{417} (with $c'=\mathfrak c/18$). 
This completes the proof of of Theorem \ref{Thm: B_Univ}  under the assumption that 
\eqref{WeakQUE} and \eqref{Weakloclaw} hold. 

\end{proof}

\subsection{Stochastic flow and   \texorpdfstring{$G$}{G}-loops}
In this section, we  state  a fundamental estimate on ``$G$ loops" which will be the key to prove 
Theorems \ref{MR:locSC} and \ref{MR:QDiff}.  To this end,  recall 
the matrix Brownian  motion   defined by 
\begin{align}\label{def_mainHT}
dH_{t,ij}=\sqrt{S_{ij}}dB_{t,ij}, \quad H_{0}=0, 
\end{align}
where $B_{t,ij}$ are independent standard complex Brownian motions for all $i \le j=1,\ldots,N$ and  $B_{ji} = \bar  B_{ij}$ for all $i, j$. 
We will consider the resolvent with a time dependent spectral parameter $z_t$ given by the following definition. 

\begin{definition}[The $z_t$ flow]\label{def_flow}
For fixed $E \in \mathbb R$,  
denote by $
m^{(E)}:=\lim_{\epsilon\to 0+}m_{sc}(E+i \epsilon)$.
Define the linear flow  $z_t$  $(0\le t\le 1)$ by 
$$
  z^{(E)}_t=E+(1-t) m^{(E)}  ,\quad 0\le t\le 1.   
$$
The imaginary part of $ z^{(E)}_t$ is given by 
\begin{align}\label{eta}
\eta_t = \im  z^{(E)}_t =  (1-t)  \im m^{(E)}.
\end{align}
Denote the resolvent of $H_{t}$ at  $z^{(E)}_t$ by 
\begin{align}
G_{t}^{(E)}:=(H_{t}-z^{(E)}_t)^{-1}.
\end{align}
\end{definition}
 
By Ito's formula, $G_{t} := G_{t}^{(E)}$ satisfies the SDE
\begin{align}\nonumber
dG_{t}=-G_{t}dH_{t}G_{t}+G_{t}\{\mathcal{S}[G_{t}]-m^{(E)}\}G_{t}
dt,
\end{align}
where $\mathcal{S}:M_{N}(\C)\to M_{N}(\C)$ is the linear operator defined by
\begin{align*}
\mathcal{S}[X]_{ij}:=\delta_{ij}\sum_{k=1}^{N}S_{ik}X_{kk}.
\end{align*}
Notice that $G_{t}$ depends on $E$ and  we will use $G^{(E)}_{t}$  to emphasize the $E$ dependence. 
 For any spectral parameter $z$, we are interested in the resolvent  $G(z)=(H-z)^{-1} $. This function can be related to $G_{t}^{(E)}$ by the following lemma. 
\begin{lemma}\label{zztE}
    For any $z\in \mathbb C$ with $0 <  \im z\le 1$ and $|\re z|\le  2-\kappa$ for some $\kappa>0$, there exists  $(E, t)$ satisfying $0 <   t<1$ and  $|E|\le 2- c \kappa$ for some $c > 0$  such that
\begin{equation}\label{eq:zztE}
     z=t^{-1/2}\cdot z^{(E)}_t.
\end{equation}
Denote by $m^{(E)} $ the solution of $m (m+  E) =- 1$ . Then $E$ and $t$ are explicitly given by  
\[
 E =  - 2 \frac {  \re  m_{sc} (z)} {  |m_{sc} (z) |}, \quad t  = m_{sc} (z)^2/(m^{(E)})^2.
\]
For $z$, $E$ and $t$ satisfy \eqref{eq:zztE}, we have 
\begin{equation}\label{mtEmz}
  m_{sc}(z)=t^{1/2}\cdot m^{(E)} 
\end{equation}
and 
   \begin{equation}\label{GtEGz}
    G(z) \sim t^{ 1/2}\cdot G_t^{(E)} 
   \end{equation}
in the sense that they have the same distribution. 
Furthermore, there  exists $c_\kappa>0$ (depending only on $\kappa$) such that 
  \begin{equation}\label{eq:zztE2}
       t\ge c_\kappa \quad {\rm and}\quad  c_\kappa\im z\le \im z_t\le c_\kappa^{-1}\im z
\end{equation} 
\end{lemma}
\begin{proof} For \eqref{eq:zztE}, 
we wish to solve 
\[
z=  t^{-1/2} \big [ (E+m^{(E)})- t  m^{(E)} \big ] =  \frac { E+m^{(E)}} { \sqrt t} - \sqrt  t  m^{(E)} = -  \frac { 1} {  \sqrt t m^{(E)} } -  \sqrt  t  m^{(E)}
\] 
This equation can be solved if  we can find $t$ and $m$ so that  $\sqrt  t  m^{(E)}= m_{sc} (z)$.
The last  equation is solved by $\sqrt t  = |m_{sc} (z)|$ with    $E$  solving   $m^{(E)}  = m_{sc} (z) / |m_{sc} (z) |$. Explicitly,  $E =  - 2 \frac {  \re  m_{sc} (z)} {  |m_{sc} (z) |}$. The other properties can be checked directly.  
\end{proof}

For the rest of the paper, we will consider only  $G_t^{(E)}$.  We now define the $G$ loops. 
\begin{definition}[$G$ - Loop]\label{Def:G_loop}
 For fixed $E$ and $\sigma\in \{+,-\}$,  define
 \begin{equation}\nonumber
  G^{(E)}_{t }(\alpha):=  G^{(E)}_{t,\alpha} =
   \begin{cases}
       (H_t-z_t)^{-1}, \quad \alpha=+\\
        (H_t-\bar z_t)^{-1}, \quad \alpha=-
   \end{cases}
 \end{equation}
By definition, $G^{(E)}_{t,+} =\left(G^{(E)}_{t,-}\right)^\dagger$. 
For simplicity of notations,  we sometimes write 
$G_t=G^{(E)}_{t,+}$  which matches the notation in \eqref{GtEGz}.  Recall $E_a$  defined in \eqref{Def_matE}.  
For 
$$
\boldsymbol{\sigma}=(\sigma_1, \sigma_2, \cdots \sigma_n),\quad 
\boldsymbol{a}=(a_1, a_2, \cdots a_n)
,\quad \sigma_i \in \{+,-\}, \quad
a_i\in \mathbb Z_L, \quad 1\le i\le n
$$
we define the $n$-$G$ loop by 
\begin{equation}\label{Eq:defGLoop}
    {\cal L}_{t, \boldsymbol{\sigma}, \textbf{a}}=\big \langle \prod_{i=1}^n G_t(\sigma_i)\cdot E_{a_i}\big \rangle,\quad 
\text{where} \quad \left\langle A\right\rangle=\tr A.
\end{equation}
We also denote by 
\begin{equation}\label{def_mtzk}
m (\sigma ):=m^{(E)}(\sigma ) = \begin{cases}
     m^{(E)}, & \sigma  =+ \\
     \overline{m^{(E)}} , & \sigma = -
\end{cases},
\end{equation}
and
\begin{equation}\label{Eq:defwtG}
 \widetilde{G}_t(\sigma_k) := G_t(\sigma_k) - m (\sigma_k) .
\end{equation}
\end{definition}

With these notations, we can express the quantity 
$\tr G (z) E_a G^\dagger(z)  E_b$ in Theorem \ref{MR:QDiff} by 
(we drop the  superscript $E$)
$$
\tr G(z) E_a G^\dagger(z)  E_b=t \cdot \tr (G_{t,+})\cdot E_a\cdot (G_{t,-})\cdot  E_b=t\cdot {\cal L}_{t, (+,-),(a,b)}.
$$
In order to derive the loop hierarchy, Lemma \ref{lem:SE_basic} , we need  the following notations.

\begin{definition}[Loop, Cut, and Glue]\label{Def:oper_loop}
Recall the $G$ loops ${\cal L}_{t, \boldsymbol{\sigma}, \textbf{a}}$ defined in Eq.~\eqref{Eq:defGLoop} of Definition~\ref{Def:G_loop}. Here we define some basic operators: cutting and gluing for these $G$ loops.  Assume that
\begin{equation}\label{taahCal}
   {\cal L}_{t, \boldsymbol{\sigma}, \textbf{a}} = \left\langle \prod_{i=1}^n G_t(\sigma_i) E_{a_i} \right\rangle, \quad \sigma_i \in \{+, -\}, \quad 1 \le i \le n 
\end{equation}
 
 \noindent 
 \emph{1.} 
 For $1 \le k \le n$, we define the first cut and glue operator ${\cal G}^{(a)}_{k}$ such that $
    {\cal G}^{(a)}_{k} \circ {\cal L}_{t, \boldsymbol{\sigma}, \textbf{a}}$  is the $G$ loop obtained by replacing $G_t(\sigma_k)$ by $G_{t}(\sigma_k) E_a G_{t}(\sigma_k)$.
  If we consider ${\cal L}_{t, \boldsymbol{\sigma}, \textbf{a}}$ as a loop, then the operator ${\cal G}^{(a)}_{k}$ cuts the $k$-th $G$ edge $G_t(\sigma_k)$ and glues the two new ends with $E_a$. 
We also view  ${\cal G}_k^{(a)}$  as an operator on the indices $\boldsymbol{\sigma},\textbf{a}$ so that 
    $$
     {\cal L}_{t, \;  {\cal G}^{(a)}_{k} (\boldsymbol{\sigma}, \textbf{a})}
     := {\cal G}^{(a)}_{k} \circ {\cal L}_{t, \boldsymbol{\sigma}, \textbf{a}}.
    $$
    
    The new loop is ${\cal G}^{(a)}_{k} \circ {\cal L}_{t, \boldsymbol{\sigma}, \textbf{a}}$  one unit longer than the original ${\cal L}_{t, \boldsymbol{\sigma}, \textbf{a}}$. For example in figure \ref{fig:op_gka}, for $n=4$:
    $$
    {\cal G}^{(a)}_{2} \circ {\cal L}_{t, \boldsymbol{\sigma}, \textbf{a}} = {\cal L}_{t, \; \boldsymbol{\sigma}', \textbf{a}'}, 
    \quad\quad 
    {\cal G}^{(a)}_{2} \circ ( \boldsymbol{\sigma}, \textbf{a}) = ( \boldsymbol{\sigma}', \textbf{a}')
    $$
    and 
    $$ \boldsymbol{\sigma}=(\sigma_1,\sigma_2,\sigma_3,\sigma_4),\quad 
   \textbf{a}= (a_1,a_2,a_3,a_4),\quad \boldsymbol{\sigma}' = (\sigma_1, \sigma_2, \sigma_2, \sigma_3, \sigma_4), \quad \textbf{a}' = (a_1, a, a_2, a_3, a_4)
    $$

   \begin{figure}[ht]
        \centering
        \begin{tikzpicture}
 
    \coordinate (a1) at (0, 0);
    \coordinate (a2) at (2, 0);
    \coordinate (a3) at (2, -2);
    \coordinate (a4) at (0, -2);

    \draw[thick, blue] (a1) -- (a2) -- (a3) -- (a4) -- cycle;
    \fill[red] (a1) circle (2pt) node[above left] {$a_1$};
    \fill[red] (a2) circle (2pt) node[above right] {$a_2$};
    \fill[red] (a3) circle (2pt) node[below right] {$a_3$};
    \fill[red] (a4) circle (2pt) node[below left] {$a_4$};
    
    \node[above, blue] at (1, 0.1) {$G_2$};
    \node[right, blue] at (2.1, -1) {$G_3$};
    \node[below, blue] at (1, -2.1) {$G_4$};
    \node[left, blue] at (-0.1, -1) {$G_1$};

    \draw[->, thick] (3, -1) -- (5, -1);

    \coordinate (b1) at (7, 0);
    \coordinate (b2) at (9, 0);
    \coordinate (b3) at (9, -2);
    \coordinate (b4) at (7, -2);
    \coordinate (a) at (8, 1);  

    \draw[thick, blue] (b1) -- (a) -- (b2) -- (b3) -- (b4) -- cycle;
    \fill[red] (b1) circle (2pt) node[above left] {$a_1$};
    \fill[red] (b2) circle (2pt) node[above right] {$a_2$};
    \fill[red] (b3) circle (2pt) node[below right] {$a_3$};
    \fill[red] (b4) circle (2pt) node[below left] {$a_4$};
    \fill[red] (a) circle (2pt) node[above] {$a$};

    \node[above left, blue] at (7.5, 0.5) {$G_2$};
    \node[above right, blue] at (8.5, 0.5) {$G_2$};
    \node[right, blue] at (9.1, -1) {$G_3$};
    \node[below, blue] at (8, -2.1) {$G_4$};
    \node[left, blue] at (6.9, -1) {$G_1$};  

   \end{tikzpicture}     
   \caption{Illustration of operator 
${\cal G}^{(a)}_{k}$}
        \label{fig:op_gka}
 \end{figure}

 \noindent 
 \emph{2.}  For $1 \le k < l \le n$, we define the cut and glue operator ${\cal G}^{(a), L}_{k, l}$ as follows: ${\cal G}^{(a), L}_{k, l} \circ {\cal L}_{t, \boldsymbol{\sigma}, \textbf{a}}$ (where $L$ stands for "left") is the $G$ loop obtained by cutting the $k$-th and $l$-th $G$ edges $G_t(\sigma_k)$ and $G_t(\sigma_l)$ (creating four end points and two ``chains"), then gluing the two new ends of the chain that contains $E_{a_n}$ and inserting a new $E_a$ at the gluing point. The length of the new loop will be $k + n - l + 1$. For example in figure \ref{fig:op_gLeft}, for $n=5$:
    $$
    {\cal G}^{(a), L}_{3, 5} \circ {\cal L}_{t, \boldsymbol{\sigma}, \textbf{a}} = {\cal L}_{t, \; \boldsymbol{\sigma}', \textbf{a}'}, \quad 
     {\cal G}^{(a), L}_{3, 5}\circ\left(\boldsymbol{\sigma}, \textbf{a}\right)=\left(\boldsymbol{\sigma}', \textbf{a}'\right)
    $$
and 
$$ \boldsymbol{\sigma}=(\sigma_1,\sigma_2,\sigma_3,\sigma_4,\sigma_5),\quad 
   \textbf{a}= (a_1,a_2,a_3,a_4,a_5),\quad \boldsymbol{\sigma}' = (\sigma_1, \sigma_2, \sigma_3, \sigma_5), \quad \textbf{a}' = (a_1, a_2, a, a_5)
$$
  Notice that $a_n, a_1$ are always in ${\cal G}^{(a), L}_{k, l}$ for any $k, l$.

    \begin{figure}[ht]
        \centering
         \scalebox{0.8}{
       \begin{tikzpicture}
        \coordinate (a1) at (0, 2);
        \coordinate (a2) at (2, 1);
        \coordinate (a3) at (1.5, -1);
        \coordinate (a4) at (-1.5, -1);
        \coordinate (a5) at (-2, 1);

        \draw[thick, blue] (a1) -- (a2) -- (a3) -- (a4) -- (a5) -- cycle;
        \fill[red] (a1) circle (2pt) node[above] {$a_1$};
        \fill[red] (a2) circle (2pt) node[right] {$a_2$};
        \fill[red] (a3) circle (2pt) node[below right] {$a_3$};
        \fill[red] (a4) circle (2pt) node[below left] {$a_4$};
        \fill[red] (a5) circle (2pt) node[left] {$a_5$};

        \node[above right, blue] at (1, 1.5) {$G_2$};
        \node[right, blue] at (2, 0) {$G_3$};
        \node[below, blue] at (0, -1.2) {$G_4$};
        \node[below left, blue] at (-1.8, 0) {$G_5$};
        \node[above left, blue] at (-1, 1.5) {$G_1$};

        \draw[->, thick] (3.5, 0) -- (5.5, 0);
        \node[above] at (4, 0.2) {$ {\cal G}^{(a),L}_{3,5}$};

        \coordinate (b1) at (9, 2);     
        \coordinate (b2) at (11, 0);     
        \coordinate (a) at (9, -2);     
        \coordinate (b5) at (7, 0);     

        \draw[thick, blue] (b1) -- (b2) -- (a) -- (b5) -- cycle;
        
        \fill[red] (b1) circle (2pt) node[above] {$a_1$};
        \fill[red] (b2) circle (2pt) node[right] {$a_2$};
        \fill[red] (a) circle (2pt) node[below] {$a$};
        \fill[red] (b5) circle (2pt) node[left] {$a_5$};

        \node[above right, blue] at (10, 1) {$G_2$};
        \node[right, blue] at (10, -1) {$G_3$};
        \node[below left, blue] at (8, -1) {$G_5$};
        \node[above left, blue] at (8, 1) {$G_1$};
    \end{tikzpicture}
    }
        \caption{Illustration of operator ${\cal G}^{(a), L}_{k, l}$}
        \label{fig:op_gLeft}
    \end{figure}

  \noindent 
 \emph{3.}  For $1 \le k < l \le n$, similarly, we define the cut and glue operator ${\cal G}^{(a), R}_{k, l}$ (where $R$ stands for "right") as ${\cal G}^{(a), L}_{k, l}$. The difference is that this time we glue the two new ends of the chain that does not contain $E_{a_n}$. The length of the new loop will be $l - k + 1$. For example in figure \ref{fig:op_gRight} for $n=5$:
      $$
    {\cal G}^{(a), R}_{3, 5} \circ {\cal L}_{t, \boldsymbol{\sigma}, \textbf{a}} = {\cal L}_{t, \; \boldsymbol{\sigma}', \textbf{a}'}, \quad 
     {\cal G}^{(a), R}_{3, 5}\circ\left(\boldsymbol{\sigma}, \textbf{a}\right)=\left(\boldsymbol{\sigma}', \textbf{a}'\right)
    $$
   and  $$
  \boldsymbol{\sigma}=(\sigma_1,\sigma_2,\sigma_3,\sigma_4,\sigma_5),\quad 
   \textbf{a}= (a_1,a_2,a_3,a_4,a_5),\quad \quad   \boldsymbol{\sigma}' = (\sigma_3, \sigma_4, \sigma_5), \quad \textbf{a}' = (a_3, a_4, a)
    $$
    \begin{figure}[ht]
        \centering
       \scalebox{0.8}{  
    \begin{tikzpicture}
        \coordinate (a1) at (0, 2);
        \coordinate (a2) at (2, 1);
        \coordinate (a3) at (1.5, -1);
        \coordinate (a4) at (-1.5, -1);
        \coordinate (a5) at (-2, 1);

        \draw[thick, blue] (a1) -- (a2) -- (a3) -- (a4) -- (a5) -- cycle;
        \fill[red] (a1) circle (2pt) node[above] {$a_1$};
        \fill[red] (a2) circle (2pt) node[right] {$a_2$};
        \fill[red] (a3) circle (2pt) node[below right] {$a_3$};
        \fill[red] (a4) circle (2pt) node[below left] {$a_4$};
        \fill[red] (a5) circle (2pt) node[left] {$a_5$};

        \node[above right, blue] at (1, 1.5) {$G_2$};
        \node[right, blue] at (2, 0) {$G_3$};
        \node[below, blue] at (0, -1.2) {$G_4$};
        \node[below left, blue] at (-1.8, 0) {$G_5$};
        \node[above left, blue] at (-1, 1.5) {$G_1$};

        \draw[->, thick] (3.5, 0) -- (5.5, 0);
        \node[above] at (4, 0.2) {${\cal G}^{(a),R}_{3,5}$};

        \coordinate (a) at (8, 2);      
        \coordinate (a3) at (9.5, -1);  
        \coordinate (a4) at (6.5, -1);  

        \draw[thick, blue] (a) -- (a3) -- (a4) -- cycle;

        \fill[red] (a) circle (2pt) node[above] {$a$};
        \fill[red] (a3) circle (2pt) node[below right] {$a_3$};
        \fill[red] (a4) circle (2pt) node[below left] {$a_4$};

        \node[above right, blue] at (8.75, 0.5) {$G_3$};
        \node[below, blue] at (8, -1.2) {$G_4$};
        \node[above left, blue] at (7.25, 0.5) {$G_5$};

    \end{tikzpicture}
    }  
        \label{fig:op_gRight}
         \caption{Illustration of operator ${\cal G}^{(a), R}_{k, l}$}
    \end{figure}
Notice that the loop containing  the index $a_n$ is the left loop, another one is right loop.  
\end{definition}

Denote by  
$\partial_{(i,j)}:=\partial_{H_t(i,j)},\quad 1\le i\le j\le N$. By  It\^o's formula, we have the following lemma. 
We will use the convention that $\textbf{a} = (a_1 \ldots a_n)$  is a vector while $a$ will be used  an index independent of $\textbf{a} $. We will use this convention throughout the paper.

\begin{lemma}[The loop hierarchy] \label{lem:SE_basic}
The $G$-loops satisfy  the loop hierarchy 
\begin{align}\label{eq:mainStoflow}
    d\mathcal{L}_{t, \boldsymbol{\sigma}, \textbf{a}} =& 
    \mathcal{E}^{(M)}_{t, \boldsymbol{\sigma}, \textbf{a}} 
    +
    \mathcal{E}^{(\widetilde{G})}_{t, \boldsymbol{\sigma}, \textbf{a}} 
    +
   W\cdot \sum_{1 \le k < l \le n} \sum_{a, b} \left( \mathcal{G}^{(a), L}_{k, l} \circ \mathcal{L}_{t, \boldsymbol{\sigma}, \textbf{a}} \right) S^{(B)}_{ab} \left( \mathcal{G}^{(b), R}_{k, l} \circ \mathcal{L}_{t, \boldsymbol{\sigma}, \textbf{a}} \right) dt, 
\end{align}
 where the martingale term and the $\widetilde{G}$ terms are defined by 
 \begin{align} \label{def_Edif}
\mathcal{E}^{(M)}_{t, \boldsymbol{\sigma}, \textbf{a}} :  = & \sum_{\alpha=(i,j)} 
  \left( \partial_{  \alpha} \; {\cal L}_{t, \boldsymbol{\sigma}, \textbf{a}}  \right)
 \cdot \left(S _\alpha\right)^{1/2} \cdot
  \mathrm{d} \left({\cal B}_t\right)_{\alpha} \\\label{def_EwtG}
  \mathcal{E}^{(\widetilde{G})}_{t, \boldsymbol{\sigma}, \textbf{a}}: = &   {W}\cdot \sum_{1 \le k \le n} \sum_{a, b} \; 
 \left \langle \widetilde{G}_t(\sigma_k) E_a \right\rangle
  \cdot S^{(B)}_{ab} \cdot
  \left( {\cal G}^{(b)}_{k} \circ {\cal L}_{t, \boldsymbol{\sigma}, \textbf{a}} \right) dt ,\quad \quad \widetilde{G}_t = G_t - m
\end{align}
Notice that the factor $W$ comes from that $E_a$ has an $W^{-1}$ factor.

  \end{lemma}

In order to solve this hierarchy, we introduce the primitive loops, denoted by $\cal K$. 

\begin{definition}[The primitive equation]\label{Def_Ktza}
For 
$$
\boldsymbol{\sigma}=(\sigma_1, \sigma_2, \cdots \sigma_n),\quad 
\boldsymbol{a}=(a_1, a_2, \cdots a_n)
,\quad \sigma_i \in \{+,-\}, \quad
a_i\in \mathbb Z_L, \quad 1\le i\le n
$$
and  $m(\sigma)$ defined in \eqref{def_mtzk}, we define  ${\cal K}_{t, \boldsymbol{\sigma}, \textbf{a}}$ to be the unique solution to the equation 
 \begin{align}\label{pro_dyncalK}
       \frac{d}{dt}\,{\cal K}_{t, \boldsymbol{\sigma}, \textbf{a}} 
       =  
       {W}\cdot  \sum_{1 \le k < l \le n} \sum_{a, b} \; \left( {\cal G}^{(a), L}_{k, l} \circ {\cal K}_{t, \boldsymbol{\sigma}, \textbf{a}} \right) S^{(B)}_{ab} \left( {\cal G}^{(b), R}_{k, l} \circ {\cal K}_{t, \boldsymbol{\sigma}, \textbf{a}} \right)  
    \end{align}
with the initial value 
$$
    {\cal K}_{0, \boldsymbol{\sigma}, \textbf{a}} =  {W^{-n+1}} \cdot \prod_{k=1}^n m (\sigma_k) \cdot \textbf{1}(a_1 = a_2 = \cdots = a_n).
$$
Here we define the ${\cal G}$ operator acting on ${\cal K}$ in the same way as it acts on ${\cal L}$,  i.e., ,  
  \be\label{calGonIND}
     {\cal G}^{(a),\,L }_{k,l}  \circ {\cal K}_{t, \boldsymbol{\sigma}, \textbf{a}} = {\cal K}_{t, \; {\cal G}^{(a),\,L }_{k,l}  (\boldsymbol{\sigma}, \textbf{a})}  \quad \forall t, k, \boldsymbol{\sigma}, \textbf{a},   
       \ee
and similarly for  ${\cal G}^{(b),\, R }_{k,l}$. 
For the special case $n=1$, we define 
$
{\cal K}_{t,\,+,\, a}=m,\quad {\cal K}_{t,\,-,\, a}=\overline m$ for $0\le t\le 1$. 
\end{definition}

Notice that  ${\cal L}_{t, \boldsymbol{\sigma}, \textbf{a}}$ is an $n$ loop and $\partial_{  \alpha} \; {\cal L}$ is a product of  $n+1$  resolvent. Hence  \eqref{eq:mainStoflow} is not an equation, but  a hierarchy. However, the  lengths of ${\cal G}^{(a), L}_{k, l} \circ {\cal K}_{t, \boldsymbol{\sigma}, \textbf{a}}$ and ${\cal G}^{(a), R}_{k, l} \circ {\cal K}_{t, \boldsymbol{\sigma}, \textbf{a}}$ are no greater than the  length of ${\cal K}_{t, \boldsymbol{\sigma}, \textbf{a}}$. Therefore, 
the system of equations for $\cal K$ can be solved inductively. 
 In the next section, we will provide an explicit  solution to the primitive equation.  

\subsection{Propagator \texorpdfstring{$\Theta_\xi^{(B)}$}{Theta} and examples of \texorpdfstring{${\cal K}$}{ K}}
 
 \begin{definition}[Propagator $\Theta_\xi^{(B)}$]\label{def_Theta}
Define the propagator  $\Theta_\xi^{(B)}$ by 
\begin{equation}\label{def_Thxi}
    \Theta^{(B)}_\xi := \frac{1}{1 - \xi \cdot S^{(B)}}, \quad \xi \in \mathbb{C}, \quad |\xi| < 1
\end{equation}
where the superscript $B$  indicates the block level matrix. Clearly, 
\begin{equation}\label{deri_Thxi}
    \partial_\xi \Theta^{(B)}_\xi = \Theta^{(B)}_\xi \cdot S^{(B)} \cdot \Theta^{(B)}_\xi.
\end{equation}
By definition, $\Theta_\xi^{(B)}$ is an \(L \times L\) matrices at the block level. We will 
omit the subscript \((B)\) for the rest of this paper. We remind the readers that $\Theta$ 
was often  used to denote the \(N \times N\) matrix 
 \((1 - \xi S)^{-1}\) in the literature.   In this paper,  $\Theta_\xi= \Theta_\xi^{(B)}$.
The following  three special  cases are often used in this paper: 
\[
\Theta^{(B)}_{t m^2}, \quad \Theta^{(B)}_{t \bar{m}^2}, \quad \text{and} \quad \Theta^{(B)}_{t |m|^2} = \Theta^{(B)}_{t}.
\]
where we have used  \(|m| = |m^{(E)}|=1\) in our setting.

\end{definition}

The following properties of $ \Theta^{(B)}$ can be easily verified. 
\begin{lemma}\label{lem_propTH}
Suppose that  $\xi\in \mathbb C$, $\im \,\xi>0$ and $|\xi|\le 1$.  Define
$$
 \hat{\ell}(\xi) := \min\left((|1-\xi|)^{-1/2},\; L\right),\quad \xi\in \mathbb C 
$$
Then  $\Theta^{(B)}_\xi$ has the following properties:
\begin{enumerate}
    \item Symmetric: $ (\Theta^{(B)}_{\xi})_{xy}= (\Theta^{(B)}_{\xi})_{yx}$.
    \item Translation invariant: $
    (\Theta^{(B)}_{\xi})_{xy}= (\Theta^{(B)}_{\xi})_{x+1, y+1}$.
    \item Commutativity:  
    $$
    \forall \;  \xi, \; \xi' \quad \quad  [S^{(B)}, \Theta^{(B)}_{\xi}]=[\Theta^{(B)}_{\xi\,'}, \Theta^{(B)}_{\xi}]=0
    $$
     \item Exponential decay at length scale $\hat \ell_\xi $:
     \begin{equation}\label{prop:ThfadC}     (\Theta^{(B)}_{\xi})_{xy}\le \frac{C\cdot e^{-c\cdot|x-y|\big/ \hat{\ell}(\xi)}}{|1-\xi|\cdot \hat \ell (\xi)}   \cdot  \,,\quad   \hat{\ell}(\xi) := \min\left((|1-\xi|)^{-1/2},\; L\right)  
     \end{equation}
     
     \item The following  random walk representation of $\Theta_\xi^{(B)}$ converges for any $ 0 \le \xi < 1$.
     \begin{align}\nonumber
     \Theta_\xi^{(B)}=\sum_{k=0}^{\infty}\;  \xi^k \cdot \left(S^{(B)}\right)^k
     \end{align}

     \item Derivative bounds.   Use above random walk representation,    we have 
     \begin{equation}\label{prop:BD1}     
    \left| (\Theta^{(B)}_{\xi})_{x, y}-(\Theta^{(B)}_{\xi})_{x,{y+1}}\right|\prec \frac{1}{\hat{ \ell}(\xi)\cdot |1-\xi|^{1/2} }   
     \end{equation}
 \begin{equation}\label{prop:BD2}     
     \left|2 (\Theta^{(B)}_{\xi})_{x, y}-(\Theta^{(B)}_{\xi})_{x,{y+1}}-(\Theta^{(B)}_{\xi})_{x,{y-1}}\right|\prec  \frac{1}{\|x-y\|+1 } . 
     \end{equation}
\end{enumerate}
\end{lemma}

We can  use $\Theta^{(B)}_\xi$ to solve  the primitive equation  $\mathcal{K}_t$ in the special cases $n = 2$ or $ 3$.

\begin{example} 
For  $n = 2$, the  primitive  equation  \eqref{pro_dyncalK} can be written as 
\begin{align}
    \frac{d}{dt}\, \mathcal{K}_{t, \boldsymbol{\sigma}, (a_1, a_2)} = W \cdot \sum_{a, b} \mathcal{K}_{t, \boldsymbol{\sigma}, (a_1, a)} \cdot S^{(B)}_{ab} \cdot \mathcal{K}_{t, \boldsymbol{\sigma}, (b, a_2)}
\end{align}
Denote 
\begin{equation}\label{eq_defm_i}
    m_i = m(\sigma_i).
\end{equation}
Using the propagator  $\Theta^{(B)}$ defined in equation~\eqref{def_Thxi} and the property \eqref{deri_Thxi}, one can easily verify that
\begin{equation}\label{Kn2sol}
    \mathcal{K}_{t, \boldsymbol{\sigma}, \textbf{a}} = W^{-1} m_1 m_2 \left(\Theta^{(B)}_{t \cdot m_1 m_2}\right)_{a_1 a_2}, \quad \textbf{a} = (a_1, a_2)
\end{equation}
Explicitly, we have 
\begin{equation}\label{Kn2sol2}
{\cal K}_{t,\boldsymbol{\sigma},(a,b)}
=   W^{-1} \cdot \begin{cases}
     |m |^2\left[\left(1- {  t }|m |^2 S^{(B)}\right)^{-1}\right]_{ab} &, \quad \boldsymbol{\sigma}=(+,-)\\
  m^2\left[\left(1- {  t } m^2 S^{(B)}\right)^{-1}\right]_{ab}  &, \quad \boldsymbol{\sigma}=(+,+)    
\end{cases}
\end{equation}

\end{example}

\begin{example}
For  $n=3$, let 
$\boldsymbol{\sigma}=(\sigma_1, \sigma_2,\sigma_3), \,
\textbf{a}=(a_1, a_2,a_3)$.
Then  \eqref{pro_dyncalK} becomes 
\begin{align}    \frac{d}{dt}\, \mathcal{K}_{t, \boldsymbol{\sigma}, \textbf{a}} = & W \cdot \sum_{b_1, c_1} \mathcal{K}_{t, (\sigma_1,\sigma_2), (a_1, b_1)} \cdot S^{(B)}_{b_1c_1} \cdot \mathcal{K}_{t, \boldsymbol{\sigma}, (c_1, a_2, a_3)}\nonumber \\
    + & W \cdot \sum_{b_2, c_2} \mathcal{K}_{t, (\sigma_2,\sigma_3), (a_2, b_2)} \cdot S^{(B)}_{b_2c_2} \cdot \mathcal{K}_{t, \boldsymbol{\sigma}, (a_1, c_2, a_3)}\nonumber \\
    + & W \cdot \sum_{b_3, c_3} \mathcal{K}_{t, (\sigma_3,\sigma_1), (a_3, b_3)} \cdot S^{(B)}_{b_3c_3} \cdot \mathcal{K}_{t, \boldsymbol{\sigma}, (a_1, a_2, c_3)} \nonumber
\end{align}
Using \eqref{Kn2sol}, we can rewrite it as 
\begin{align} 
    \frac{d}{dt}\, \mathcal{K}_{t, \boldsymbol{\sigma}, \textbf{a}} = &  \sum_{  c_1} \left( m_1m_2\Theta^{(B)}_{tm_1m_2} \cdot S^{(B)}\right)_{a_1c_1} \cdot \mathcal{K}_{t, \boldsymbol{\sigma}, (c_1, a_2, a_3)}\nonumber \\
    + & \sum_{  c_2} \left(m_2m_3 \Theta^{(B)}_{tm_2m_3} \cdot S^{(B)}\right)_{a_2c_2}  \cdot \mathcal{K}_{t, \boldsymbol{\sigma}, (a_1, c_2, a_3)}\nonumber \\
   + & \sum_{  c_3} \left( m_3m_1\Theta^{(B)}_{tm_3m_1} \cdot S^{(B)}\right)_{a_3c_3}  \cdot \mathcal{K}_{t, \boldsymbol{\sigma}, (a_1, a_2, c_3)} \nonumber
\end{align}
Recall 
$
m_i = m  (\sigma_i)
$ defined in \eqref{eq_defm_i}. Using the propagator  $\Theta^{(B)}$  \eqref{def_Thxi} and  \eqref{deri_Thxi}, one can easily verify that 
$$
{\cal K}_{t, \boldsymbol{\sigma}, \textbf{a}} = \sum_b \left(\Theta^{(B)}_{t \cdot m_1 m_2}\right)_{a_1 b} \left(\Theta^{(B)}_{t \cdot m_2 m_3}\right)_{a_2 b} \left(\Theta^{(B)}_{t \cdot m_3 m_1}\right)_{a_3 b} \cdot W^{-2}\cdot m_1 m_2 m_3, \quad  \textbf{a}= (a_1, a_2,a_3).
$$
Alternatively, it can be expressed as
$$
{\cal K}_{t, \boldsymbol{\sigma}, \textbf{a}} = \sum_{b_1 b_2 b_3} \left(\Theta^{(B)}_{t \cdot m_1 m_2}\right)_{a_1 b_1} \left(\Theta^{(B)}_{t \cdot m_2 m_3}\right)_{a_2 b_2} \left(\Theta^{(B)}_{t \cdot m_3 m_1}\right)_{a_3 b_3} \cdot {\cal K}_{0, \boldsymbol{\sigma}, \textbf{b}}, \quad \textbf{b} = (b_1, b_2, b_3).
$$
\end{example}

In the next section, using the primitive equation and the tree representation of $\mathcal{K}$, we will establish the following estimate \eqref{eq:bcal_k_2} for $\mathcal{K}$. 
A key ingredient in this proof is the  sum-zero property,  \eqref{eq:bcal_k_pi}, which serves as a critical input.

\begin{lemma}[An upper bound on ${\cal K}$]\label{ML:Kbound}
Recall $\hat \ell $ defined in \eqref{prop:ThfadC}. 
The  primitive loop ${\cal K}_{t, \boldsymbol{\sigma}, \textbf{a}}$ is bounded by
\begin{equation}\label{eq:bcal_k}
  {\cal K}_{t, \boldsymbol{\sigma}, \textbf{a}}  \prec \left(W \ell_t \cdot \eta_t \right)^{-n+1},\quad  \ell_t :  =\hat\ell(t)= \min\left((|1-t|)^{-1/2},\; L\right)       
\end{equation}
\end{lemma}

\bigskip

Finally, we note that for any fixed \(|E| < 2 - \kappa\), it is straightforward to verify that 
 \[
\im z_t = \im z_t^{(E)} = \im m^{(E)} \cdot (1 - t) \sim (1 - t).
\]
 Therefore, the \(\ell(z)\) defined in \eqref{def_ellz} and \(\ell_t\) defined in \eqref{eq:bcal_k} are of the same order, as follows:
 \[
\ell(z_t) \sim \ell_t = \hat{\ell}(t).
\]
Furthermore,  for \(z\), \(t\), and \(E\) satisfying the conditions \eqref{eq:zztE} and \eqref{eq:zztE2}, we have 
 \[
\im z\sim \im z^{(E)}_t,\quad \ell(z) \sim \ell(z^{(E)}_t) \sim \ell_t.
\]
Notice that \(\ell_t\) is just a simpler notation for $\hat{\ell}(t)$. 

\subsection{Estimates on loops}

Our main estimates for the  $G$-loop  are given in the following lemmas (Recall that $\prec$ and $\ell_t$ are defined in Definition \ref{stoch_domination} and \eqref{eq:bcal_k} respectively). 

\begin{lemma}[loop estimates]\label{ML:GLoop}
For any small constants $\kappa, \tau>0$ and  
$E \in[-2+\kappa, 2-\kappa], \, 0\le t\le 1-N^{-1+\tau}$,  
we have 
\begin{equation}\label{Eq:L-KGt}
 \max_{\boldsymbol{\sigma}, \textbf{a}}\left|{\cal L}_{t, \boldsymbol{\sigma}, \textbf{a}}-{\cal K}_{t, \boldsymbol{\sigma}, \textbf{a}}\right|\prec (W\ell_t\eta_t)^{-n},\quad \eta_t:=\im z_t,  
\end{equation}
 \begin{equation}\label{Eq:L-KGt2}
\max_{\boldsymbol{\sigma}, \textbf{a}}\left|{\cal L}_{t, \boldsymbol{\sigma}, \textbf{a}} \right|\prec (W\ell_t\eta_t)^{-n+1}. 
\end{equation}
\end{lemma}

\begin{lemma}[$2$-loop estimate]\label{ML:GLoop_expec}
With the notations and assumptions of the previous lemma, 
the expectation of  $2$-$G$-loop is bounded by  
\begin{equation}\label{Eq:Gtlp_exp}
 \max_{\boldsymbol{\sigma}, \textbf{a}}\left|\mathbb E{\cal L}_{t, \boldsymbol{\sigma}, \textbf{a}}-{\cal K}_{t, \boldsymbol{\sigma}, \textbf{a}}\right|\prec (W\ell_t\eta_t)^{-3}
 ,\quad \eta_t:=\im z_t
\end{equation}
for any choice of  $\boldsymbol{\sigma}=\{+,-\}^2$.
In case $\boldsymbol{\sigma}=(+,-)$ and $ \textbf{a}=(a_1, a_2)$, 
we have
\begin{equation}\label{Eq:Gdecay}
 \left| {\cal L}_{t, \boldsymbol{\sigma}, \textbf{a}}-{\cal K}_{t, \boldsymbol{\sigma}, \textbf{a}}\right|\prec (W\ell_t\eta_t)^{-2}\exp \left(-\left|\frac{ a_1-a_2 }{\ell_t}\right|^{1/2}\right)+W^{-D}. 
\end{equation}
\end{lemma}

\begin{lemma}[Local law for $G_t$]\label{ML:GtLocal}
   With the notations and assumptions of the previous lemma, we have 
\begin{equation}\label{Gt_bound}
 \|G^{(E)}_{t,+}-m^{(E)}\|_{\max}\prec (W\ell_t\eta_t)^{-1/2}.
   \end{equation} 
\end{lemma}

\medskip 

\begin{proof}[Proof of Theorems \ref{MR:locSC} and \ref{MR:QDiff}]
For each $z$ in Theorems \ref{MR:locSC} and \ref{MR:QDiff}, Lemma \ref{zztE} shows that there exist $E$ and $t$ satisfying \eqref{eq:zztE} and \eqref{eq:zztE2}, along with the following conditions:
$$
(1-t) \sim \im z_t \sim \im z \geq N^{-1+\tau}, \quad \eta_t \sim \eta, \quad \ell_t \sim \ell(z).
$$
Combining \eqref{mtEmz} and \eqref{GtEGz}, we obtain
\begin{equation}\label{Gmt1/2Gm}
G(z) - m_{sc}(z) = t^{1/2} \left(G_t^{(E)} - m^{(E)}\right).
\end{equation}
Thus, the estimate on $G - m$ in \eqref{G_bound} of Theorem \ref{MR:locSC} follows from the estimate on $G_t - m$ in \eqref{Gt_bound} of Lemma \ref{ML:GtLocal}.

Similarly, \eqref{G_bound_ave} follows from \eqref{Eq:L-KGt} for the $1$-$G$-loop, with the definition of $\cal K$ in Definition \ref{Def_Ktza}:
$$
{\cal K}_{t,+,a} = m^{(E)}, \quad 0\le t\le 1. 
$$
Analogously to \eqref{Gmt1/2Gm}, for the $2$-$G$ terms we have:
\begin{equation}\label{Gmt1/2Gm2}
\tr G (z) E_aG (z) E_b = t \cdot {\cal L}_{t, (+,+),(a,b)}, \quad \tr G(z) E_aG^\dagger(z)  E_b = t \cdot {\cal L}_{t, (+,-),(a,b)}.
\end{equation}
Using \eqref{Kn2sol} for rank-$2$ $\cal K$, we derive:
$$
t \cdot {\cal K}_{t, (+,+),(a,b)} = W^{-1} m^2_{sc}(z) \cdot \frac{1}{1 - m^2_{sc}(z) \cdot S^{(B)}}, \quad t \cdot {\cal K}_{t, (+,-),(a,b)} = W^{-1} |m_{sc}(z)|^2 \cdot \frac{1}{1 - |m_{sc}(z)|^2 \cdot S^{(B)}}.
$$
Therefore, \eqref{Meq:QdW1} and \eqref{Meq:QdW2} in Theorem \ref{MR:QDiff} follow from \eqref{Eq:L-KGt} in the case $n=2$, while \eqref{Meq:QdS1} and \eqref{Meq:QdS2} follow from \eqref{Eq:Gtlp_exp}. 
This completes the proof of Theorems \ref{MR:locSC} and \ref{MR:QDiff}.
 \end{proof}

We remark that there is a  subtle difference between the $2$-loop and the  $T$-observable. 
By definition, the $2$-loop is given by 
$$
{\cal L}_{t, (+,-),(a,b)}= \sum_{i, j}   (G_{t,+})_{ij}  E_a(j) (G_{t,-})_{ji}   E_b(i)=  \sum_{i, j}   |(G_{t,+})_{ij}|^2  E_a(j) E_b(i).
$$
Comparing ${\cal L}_{t, (+,-),(a,b)}$ with the  $T$ observable \eqref{T}, we find that there are two averaging over indices in the  loop observable, but  only one averaging in   $T$ .  Here we neglect the  unimportant  difference between $S$ and $E_b$ operators.
While an extra averaging might seem to be insignificant, we remind the reader that in the special case of Wigner matrices,  
\[
N^{-1} \sum_{x y} |G_{x,y}|^2 = (\im z )^{-1}N^{-1}  \sum_x \im G_{xx}. 
\]
The averaging  in the $x$ index is critical for the local law of  Wigner matrices asserting that   the fluctuation of $N^{-1} \tr G$ is one order smaller than that of $G_{xx}$.  For similar reasons, our results for $2$-loop will not hold for the $T$ observable \eqref{T}.

\bigskip

\subsection{Strategy of the proofs of main lemmas}

We now outline the proofs of Lemmas \ref{ML:GLoop}, \ref{ML:GLoop_expec} and \ref{ML:GtLocal}.
By Definitions \ref{Def:G_loop} and \ref{Def_Ktza},   
$$
G_{0}(+)=m\cdot I_{N\times N}, \quad {\cal L}_{0, \boldsymbol{\sigma},\textbf{a}}= {\cal K}_{0, \boldsymbol{\sigma},\textbf{a}},\quad \forall \; \boldsymbol{\sigma},\textbf{a}, 
$$
where $I_{N\times N}$ is the identity matrix. It is easy to check that  
\begin{align}\label{ini)bigasya}
    \hbox{    Lemmas \ref{ML:GLoop}, \ref{ML:GLoop_expec} and \ref{ML:GtLocal} hold at  $t=0$ with no error.}
\end{align}
For $t>0$, we will prove the following theorem.

\begin{theorem}\label{lem:main_ind} Assume for some fixed $E$: $|E|\le 2-\kappa$ and $s\in [0,1]$ that Lemmas \ref{ML:GLoop}, \ref{ML:GLoop_expec}  and \ref{ML:GtLocal} hold at time $s$, namely, 
for $1 \le n\in \mathbb N$ and large $D>0$, 
\begin{align}\label{Eq:L-KGt+IND}
 \max_{\boldsymbol{\sigma}, \textbf{a}}\left|{\cal L}_{s, \boldsymbol{\sigma}, \textbf{a}}-{\cal K}_{s, \boldsymbol{\sigma}, \textbf{a}}\right|&\prec (W\ell_s\eta_s)^{-n}  
\\\label{Eq:Gdecay+IND}
 \left| {\cal L}_{s, \boldsymbol{\sigma}, \textbf{a}}-{\cal K}_{s, \boldsymbol{\sigma}, \textbf{a}}\right|&\prec (W\ell_s\eta_s)^{-2}\exp \left(- \left|\frac{ a_1-a_2 }{\ell_s}\right|^{1/2}\right)+W^{-D} , \quad \boldsymbol{\sigma}=(+,-)
\\\label{Gt_bound+IND}
 \|G^{(E)}_{s ,+}-m^{(E)}\|_{\max} & \prec (W\ell_s\eta_s)^{-1/2},
 \\\label{Eq:Gtlp_exp+IND}
 \max_{\boldsymbol{\sigma}, \textbf{a}}\left|\mathbb E{\cal L}_{s, \boldsymbol{\sigma}, \textbf{a}}-{\cal K}_{s, \boldsymbol{\sigma}, \textbf{a}}\right|&\prec (W\ell_s\eta_s)^{-3},\quad  \forall\; \boldsymbol{\sigma}\in \{+,-\}^2.
\end{align}
 Then for any $t>s$ satisfying 
\begin{equation}\label{con_st_ind}
    (W\ell_t\eta_t)^{-1}\le \left(\frac{1-t}{1-s}\right)^{30}
\end{equation}
we have that 
\eqref{Eq:L-KGt+IND}, \eqref{Eq:Gtlp_exp+IND}, \eqref{Eq:Gdecay+IND} and \eqref{Gt_bound+IND} hold with $s$  replaced by $t$.   
\end{theorem}

\begin{proof}[Proof of Lemmas \ref{ML:GLoop}, \ref{ML:GLoop_expec},  and \ref{ML:GtLocal}]
For any fixed $\tau$ and  $t \le 1-N^{-1+\tau}$,  choose $\tau'>0$ and $n_0 \in \mathbb N$ such that 
$$
(W\ell_t\eta_t)^{-1}\le  W^{- 30 \tau'} , \quad (1-t) = W^{-n_0 \tau'}
$$
Let 
\[
1-s_{k} = W^{-k \tau'}, \quad k \le n_0, \quad s_{n_0} = t
\]
 
Since  $W\ell_s\eta_s$ is decreasing in  $s\in [0,1]$, we have 
   $$
(W\ell_{s_{k+1}}\eta_{s_{k+1}})^{-1}\le (W\ell_{t}\eta_{t})^{-1}\le W^{- 30\tau'}
\le \left(\frac{1-s_{k}}{1-s_{k+1}}\right)^{30}
$$
for all $k$ such that  $k+1\le n_0$. 
We can now apply  Theorem  \ref{lem:main_ind} from $s_k$ to $s_{k+1}$ for $k=0$ until $k=n_0-1$ so that the conclusions of  Theorem  \ref{lem:main_ind} hold for  $s_{n_0}=t$.  We have thus proved Lemmas \ref{ML:GLoop}, \ref{ML:GLoop_expec}, and \ref{ML:GtLocal}.  Notice that for any $\tau$ fixed, $n_0$ is a finite number depending on $\tau$.  Thus we only have finite iterations. This is important because every time we apply  Theorem  \ref{lem:main_ind} our inequalities deteriate by a factor 
$N^\varepsilon$. At the end, we will have a factor $N^{n_0 \varepsilon}$ at the time $t$. Since $\varepsilon$ is arbitrary small, this factor is still harmless for any $\tau$ fixed. 
\end{proof}

Theorem \ref{lem:main_ind} will be proved in six steps, with their detailed proofs provided in Section \ref{Sec:Stoflo}. Throughout these steps, we assume that the conditions of Theorem \ref{lem:main_ind} are satisfied. In addition, each step builds on the conclusions established in the preceding steps.

\medskip 
\noindent 
\textbf{Step 1}   (A priori loop bounds): The $n$-loop is bounded by 
 \begin{equation}\label{lRB1}
   {\cal L}_{u,\boldsymbol{\sigma}, \textbf{a}}\prec (\ell_u/\ell_s)^{(n-1)}\cdot 
   (W\ell_u\eta_u)^{-n+1},\quad  s\le u\le t.
\end{equation}
Furthermore,  the  weak local law holds in the sense 
\begin{equation}\label{Gtmwc}
    \|G_u-m\|_{\max}\prec  (W\ell_u\eta_u)^{-1/4},\quad s\le u\le t .
\end{equation} 
Here the exponent  is $1/4$ instead of $1/2$  in \eqref{Gt_bound+IND}.

 \bigskip
\noindent 
\textbf{Step 2}  (A priori $2$-loop decay): 
 The following local law holds for $u\in [s,t]$, namely, 
 \begin{equation}\label{Gt_bound_flow}
     \|G_u-m\|_{\max}\prec  (W\ell_u\eta_u)^{-1/2},\quad s\le u\le t.
\end{equation} 
Hence \eqref{Gt_bound+IND} in Theorem \ref{lem:main_ind} holds.  
In addition,  with 
 $\boldsymbol{\sigma}=(+,-)$, we have for any   $s\le u\le t$ and $D>0$ that 
 \begin{equation}\label{Eq:Gdecay_w}
\left| {\cal L}_{u, \boldsymbol{\sigma}, \textbf{a}}-{\cal K}_{u, \boldsymbol{\sigma}, \textbf{a}}\right| \prec \left(\eta_s/\eta_u\right)^4\cdot (W\ell_u\eta_u)^{-2}\exp \left(- \left|\frac{ a_1-a_2 }{\ell_u}\right|^{1/2}\right)+W^{-D} . \quad 
\end{equation}
   \bigskip
 \noindent 
\textbf{Step 3}   (Sharp  loop  bounds): 
 The following sharp estimate on $n$-$G$-loop holds: 
\begin{equation}\label{Eq:LGxb}
\max_{\boldsymbol{\sigma}, \textbf{a}}\left| {\cal L}_{u, \boldsymbol{\sigma}, \textbf{a}} \right|
\prec 
  (W\ell_u\eta_u)^{-n+1} ,\quad \quad s\le u\le t, \quad\forall\, n\in \mathbb N.
\end{equation}

 \bigskip
 \noindent 
\textbf{Step 4}  (A sharp  $\cal L - \cal K$  bound): 
The following sharp estimate on $ {\cal L}_t-{\cal K}_t$ of length $n$ holds: 
\begin{equation}\label{Eq:L-KGt-flow}
 \max_{\boldsymbol{\sigma}, \textbf{a}}\left|{\cal L}_{u, \boldsymbol{\sigma}, \textbf{a}}-{\cal K}_{u, \boldsymbol{\sigma}, \textbf{a}}\right|\prec (W\ell_u\eta_u)^{-n},\quad \quad s\le u\le t, \quad \forall\, n\in \mathbb N
\end{equation}
This implies \eqref{Eq:L-KGt+IND} in Theorem  \ref{lem:main_ind}.

 \bigskip
 \noindent 
\textbf{Step 5}  (A sharp decay bound): For  $ \boldsymbol{\sigma}=(+,-)$,
 \begin{equation}\label{Eq:Gdecay_flow}
\left| {\cal L}_{u, \boldsymbol{\sigma}, \textbf{a}}-{\cal K}_{u, \boldsymbol{\sigma}, \textbf{a}}\right| \prec   (W\ell_u\eta_u)^{-2}\exp \left(- \left|\frac{ a_1-a_2 }{\ell_u}\right|^{1/2}\right)+W^{-D}. 
\end{equation}
The last bound implies \eqref{Eq:Gdecay+IND} in Theorem  \ref{lem:main_ind}.

\bigskip
\noindent 
\textbf{Step 6} (A sharp  $\mathbb E \cal L - \cal K$  bound): 
The following estimate on $2$-$G$-loop with $\boldsymbol{\sigma}=\{+,-\}^2$ holds: 
\begin{equation}\label{Eq:Gtlp_exp_flow}
 \max_{\boldsymbol{\sigma}, \textbf{a}}\left|\mathbb E{\cal L}_{u, \boldsymbol{\sigma}, \textbf{a}}-{\cal K}_{t, \boldsymbol{\sigma}, \textbf{a}}\right|\prec (W\ell_t\eta_t)^{-3},\quad  
 s \le u \le t, \quad\boldsymbol{\sigma}=\{+,-\}^2.
\end{equation}
We will use Steps 1-5 to  prove  that \eqref{Eq:L-KGt+IND}, \eqref{Eq:Gdecay+IND}, and \eqref{Gt_bound+IND} of Theorem \ref{lem:main_ind} hold with $s$ replaced by $t$. Here \eqref{Eq:Gtlp_exp+IND} will not be needed for Steps 1-5, i.e.,  Theorem \ref{lem:main_ind}  holds if \eqref{Eq:Gtlp_exp+IND} was removed from both the assumption and statement.

\subsection{Sum zero properties} 

Recall the loop hierarchy \eqref{eq:mainStoflow} of the $n$-loop  is of the form 
\begin{align}
    d\mathcal{L}_{t, \boldsymbol{\sigma}, \textbf{a}} =&
    \mathcal{E}^{(M)}_{t, \boldsymbol{\sigma}, \textbf{a}} 
    +
    \mathcal{E}^{(\widetilde{G})}_{t, \boldsymbol{\sigma}, \textbf{a}} 
    + \mbox{ quadratic terms} . 
\end{align}
The first two terms are linear in the loops (assuming $\tilde G$ is given) and will be shown to be error terms. The primitive hierarchy drop these two error terms but keep the quadratic terms.  The  term $  \mathcal{E}^{(\widetilde{G})}_{t, \boldsymbol{\sigma}, \textbf{a}} $  involves $n+1$-loop and the quadratic variation of  $ \mathcal{E}^{(M)}_{t, \boldsymbol{\sigma}, \textbf{a}} $ depends on $2n+2$ loops.  The quadratic term, however,  involves only loops up to length $n$.  Therefore, we can solve the primitive equation stating from $n=1, 2 \ldots$. This procedure clearly cannot be applied to the loop hierarchy. 

It turns out  that both $\mathcal K$ and $\mathcal L$ have similar singularities as $t \to 1$ in the form 
\[
\mathcal{L}  \sim 
\left( 1-t \right)^{-C_n} \sim {\mathcal K },    \quad \im z \sim 1-t.
\]
This singularity  at $t \to 1$ is  difficult to control.  It  is a common phenomenon for  quadratic differential equations which  typically  are  unstable under perturbation. Since perturbations of  quadratic differential equations are governed by a linear one, we consider a toy  equation
\[
\partial_t f = 2f - c \cdot t + 1, \quad f(0) = a.
\]
This equation can be solved  explicitly 
\[
f(t)= e^{2 t} \big [a  + \frac 1 2    - \frac c 4    \big ] + \frac  { 2 ct + c-2}  4.
\]
If $a  + \frac 1 2    -  \frac c 4   =0$  then \( f(t) = O(t) \).  The subtle condition 
\[
a  + \frac 1 2    -  \frac c 4   =0
\]
changes the exponential growth of $f$ to a linear growth!  Without explicit solutions, it is not easy to prove the sub-exponential bound of the last toy equation. In our setting,   \(\mathcal{K}\)-loops can be solved by an explicit  \emph{tree representation  formula} (Lemma~\ref{Lemma_TRofK}) and  the previous subtle condition will be implemented by a \emph{sum-zero property} of the \(\mathcal{K}\)-loops.   We will show that both \(\mathcal{L}\) and  \(\mathcal{K}\)-loops satisfy   Ward's identity  (Lemma~\ref{lem_WI_K}). From these  Ward's identities, 
we will prove a {sum-zero property} for the \(\mathcal{K}\)-loops.

\subsection{Notations}
Here we summarize  global notations used in this paper. 
\begin{itemize}
    \item $W$ is band width, $L$ is the number of blocks, $N=W\times L$
    \item ${\cal I}_a$ is the $a$-th block, $[i]$ is the block where index $i$ is. $E_a$ is the following matrix only supported on ${\cal I}_a$. 
    $$
    i\in {\cal I}_{[i]}, \quad (E_a)_{ij}=\delta_{ij}{\bf 1}(i\in {\cal I}_a)$$
    \item $S\in \mathbb R^{N\times N}$, $S^{(B)}\in \mathbb R^{L\times L}$, $S_W\in \mathbb R^{W\times W}$ are all related to the variances of the matrix entries, and 
    $$
S= S^{(B)}\otimes S_{W}
$$
\item In the proof involving the stochastic flow, we typically omit the superscript $(E)$ for simplicity. As a result, the notations $z_t$, $G_t$, and $m$ are defined as follows:
\[
z_t = z_t^{(E)}, \quad G_t := (\sqrt{t} H - z_t)^{-1}, \quad m = \lim_{\varepsilon \searrow 0} m_{sc}(E + i \varepsilon).
\]
Additionally, we use $G_t$ to denote $(H_t - z_t)^{-1}$, where $H_t$ has the same distribution as $\sqrt{t} H$. The context will make it clear which interpretation of $G_t$ is being applied in the proof.
\item Follow the $z_t$, the $\eta_t=\im z_t$ and $\ell_t$ is defined in \eqref{eq:bcal_k} as 
$$ \ell_t:=\hat \ell(t)=\min(|1-t|^{-1/2}, L)$$
Here $\ell_t\sim \ell(z_t)\sim \ell(z)$ is  the decay  length. 

    \item $\cal L$ is for $G$-loop, $\cal K$ is  for the deterministic partner of $G$-loop,  and $\cal C$ is for $G$-chain. 

    \item The propagator $\Theta_\xi = \Theta^{(B)}_{\xi}$ is defined in Definition \ref{def_Theta}. It is used to explicitly define the solution of $\mathcal{K}$ in Definition \ref{def33}.

    \item The $\Xi^{({\cal L})}$, $\Xi^{({\cal L-K})}$, $\Xi^{({\cal C})}$ are ratios between these quantities and their heuristic size.  
    
    \begin{align}
\Xi^{({\cal L})}_{t, m}&:= \max_{\boldsymbol{\sigma}, \textbf{a}} 
\left|{\cal L}_{t, \boldsymbol{\sigma}, \textbf{a}} \right|\cdot \left(W\ell_t \eta_t\right)^{m-1} \cdot {\bf 1}(\sigma \in \{+,-\}^m),
\nonumber \\
\Xi^{({\cal L-K})}_{t, m}&:= \max_{\boldsymbol{\sigma}, \textbf{a}}
\left|{ ({\cal L-K})}_{t, \boldsymbol{\sigma}, \textbf{a}} \right|\cdot \left(W\ell_t \eta_t\right)^{m }\cdot {\bf 1}(\sigma \in \{+,-\}^m) ,
\nonumber \\
\Xi^{({{\cal C},\;diag })}_{t, m} &:= \max_{\boldsymbol{\sigma}, \textbf{a}} \max_i \left|\left({\cal C}^{(m)}_{t, \boldsymbol{\sigma}, \textbf{a}}\right)_{ii}\right| \cdot \left(W\ell_t \eta_t\right)^{m-1} \cdot {\bf 1}(\sigma \in \{+,-\}^m), \nonumber \\
\Xi^{({{\cal C},\;off})}_{t, m} &:= \max_{\boldsymbol{\sigma}, \textbf{a}} \max_{i \neq j}
\left|\left({\cal C}^{(m)}_{t, \boldsymbol{\sigma}, \textbf{a}}\right)_{ij}\right| \cdot \left(W\ell_t \eta_t\right)^{m-1/2} \cdot {\bf 1}(\sigma \in \{+,-\}^m).
\end{align}
\item The operators $\varTheta_{t,\boldsymbol{\sigma}}$ and ${\cal U}_{s, t,\boldsymbol{\sigma}}$ are defined in Def. \ref{DefTHUST} as the linear operators for the integrated loop hierarchy for $\cal L-\cal K$ in Lemma \ref{Sol_CalL}. 
\item  The \(\cal G\) operators, such as \({\cal G}^{(a)}_k\), \({\cal G}^{(a),\,L}_{k,l}\), and \({\cal G}^{(a),\,R}_{k,l}\), are defined in Definition \ref{Def:oper_loop}. These operators represent the cutting and gluing of Loop operators within the loop hierarchy.

\item The \(\Gamma_{t,\boldsymbol{\sigma}, \textbf{a}}\) represents the tree graph used in the tree representation of \(\cal K\). 
\item  The sets \({\cal F}(\Gamma)\) and \({\cal F}_{\text{long}}(\Gamma)\) correspond to the non-neighboring internal edges and the neighboring long internal edges of $\Gamma$, respectively.

\item The \({\cal K}^{(\pi)}\) is defined in Definition \ref{def_Kpi} as the sum of certain tree graphs. 
\item The \(\Sigma^{(\pi)}\) is further defined as the self-energy of \({\cal K}^{(\pi)}\) in Definition \ref{def_Kpi}. Certain specific \(\Sigma^{(\pi)}\) exhibit the sum-zero property, as demonstrated in Lemma \ref{lem:SZ}.
 
\item  The \(\cal E\) terms represent the non-leading terms that arise in the (integrated) loop hierarchy \eqref{eq:mainStoflow} and \eqref{int_K-LcalE}. Specifically, \({\cal E}^{(M)}\) and \({\cal E}^{(\widetilde{G})}\) are defined in \eqref{def_Edif} and \eqref{def_EwtG}, respectively. The term \(\mathcal{E}^{((\mathcal{L}-\mathcal{K}) \times (\mathcal{L}-\mathcal{K}))}\) is defined in \eqref{def_ELKLK}. Additionally, \(\left( \mathcal{E} \otimes \mathcal{E} \right)\) and \(\left( \mathcal{E} \otimes \mathcal{E} \right)^{(k)}\) are defined in Definition \ref{def:CALE}.

\item  The \({\cal T}^{(\cal L-\cal K)}_{t}\) is defined as the tail function of \(\cal L - \cal K\) in \eqref{def_WTu}. The \({\cal T}_{t,D}\) represents the deterministic rough tail function, as defined in \eqref{def_WTuD}. 
\item The terms \({\cal J}_{u,D}\) and \(\cal J^*\) are introduced in \eqref{def_Ju} and \eqref{shoellJJ}, respectively, to describe the ratio between \({\cal T}^{(\cal L-\cal K)}_{t}\) and \({\cal T}_{t,D}\).

\item The scale $\ell_t^*$ is define as $\ell_t^*=(\log W)^{3/2}\ell_t$. In this scale $\Theta_t$ is exponentially small, while ${\cal T}_t$ is not. 

\item The operators \(\cal P\) and \({\cal Q}_t\), along with the function \(\vartheta\), are used to define and construct a sum-zero tensor. Their definitions can be found in Definition \ref{Def:QtPt}.

\end{itemize}

\section{Definition and properties of \texorpdfstring{$\cal K$}{K}}\label{Sec:CalK}

The primitive loop ${\cal K}$ has an exact formula in terms of summation over tree graphs which we now present. 

\subsection{Tree representation of \texorpdfstring{${\cal K}_{t, \boldsymbol{\sigma}, \textbf{a}}$}{K}}

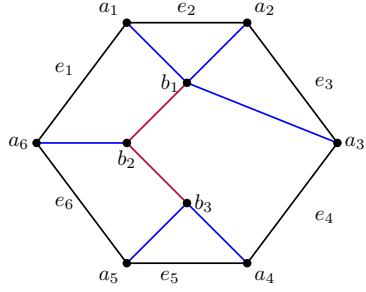
\begin{definition}[Canonical partition of polygon]
   \begin{figure}[ht]     \centering
      \scalebox{0.8}{
   \begin{tikzpicture} 
        \coordinate (a1) at (-1, 2);   
        \coordinate (a2) at (1, 2);   
        \coordinate (a3) at (2.5, 0); 
        \coordinate (a4) at (1, -2);  
        \coordinate (a5) at (-1, -2); 
        \coordinate (a6) at (-2.5, 0);
        \coordinate (b1) at (0, 1);   
        \coordinate (b2) at (-1, 0); 
        \coordinate (b3) at (0, -1);  

        \draw[thick, black] (a1) -- (a2) -- (a3) -- (a4) -- (a5) -- (a6) -- cycle;

        \draw[thick, blue] (a1) -- (b1) -- (a2); 
        \draw[thick, blue] (b1) -- (a3) ; 
        \draw[thick, blue] (a6) -- (b2)  ; 
        \draw[thick, blue] (a5) -- (b3) -- (a4); 

        \draw[thick, purple] (b1) -- (b2);
        \draw[thick, purple] (b2) -- (b3);

        \node[above left] at (a1) {$a_1$};
        \node[above right] at (a2) {$a_2$};
        \node[right] at (a3) {$a_3$};
        \node[below right] at (a4) {$a_4$};
        \node[below left] at (a5) {$a_5$};
        \node[left] at (a6) {$a_6$};

        \node[left] at (b1) {$b_1$};
        \node[below] at (b2) {$b_2$};
        \node[ right] at (b3) {$b_3$};
    
        \node[above] at (0, 2) {$e_2$};         
        \node[right] at (2, 1) {$e_3$};      
        \node[below right] at (2, -1) {$e_4$}; 
        \node[below left] at (0, -2) {$e_5$};  
        \node[left] at (-1.75 , -1) {$e_6$};     
        \node[above left] at (-1.75, 1) {$e_1$}; 

        \fill[black] (a1) circle (2pt);
        \fill[black] (a2) circle (2pt);
        \fill[black] (a3) circle (2pt);
        \fill[black] (a4) circle (2pt);
        \fill[black] (a5) circle (2pt);
        \fill[black] (a6) circle (2pt);
        \fill[black] (b1) circle (2pt);
        \fill[black] (b2) circle (2pt);
        \fill[black] (b3) circle (2pt);
    \end{tikzpicture}
    }
    $${\cal V}(\Gamma)=\{ a_1,a_2,a_3,a_4,a_5,a_6,b_1,b_2,b_3 \}$$
    $${\cal E}(\Gamma)=\{\{a_1,b_1\}, \{a_2,b_1\}, 
    \{a_3,b_1\}, \{a_4,b_3\},\{a_5,b_3\},\{a_6,b_2\},\{b_1,b_2\}, \{b_2,b_3\}, \}$$
    $$
    {\cal F}(\Gamma)=\{\{1,4\},\{4,6\}\}$$
      \caption{Illustration of canonical partition of the polygons}
      \label{fig_SPPol}
\end{figure}

Let ${{\cal P} }_{\textbf{a}}$ be an oriented polygon with vertices $\textbf{a}=(a_1,a_2,$ $\ldots ,a_n)$ such that $a_k$ and $a_{k+1}$ are next to each other. 
We will use periodic  convention so that $a_0=a_n$. The edge $\overline{a_{k-1}a_k}$ is called the $k$-th edge of ${{\cal P} }_{\textbf{a}}$. By definition, 
${{\cal P} }_{(a,b,c)}\ne {{\cal P} }_{(b,c,a)}$, 
since the  2nd edge of  ${{\cal P} }_{(a,b,c)}$  is $\overline{ab}$, while the the 2nd edge $of {{\cal P} }_{( b,c,a)}$  is $\overline{bc}$. 

A partition of  ${{\cal P} }_{\textbf{a}}$ is called canonical  
if and only if 
\begin{itemize}
\item Each  sub-region in the partition is also a polygon. 
     \item There is one to one correspondence between the edges of the polygon and the  sub-regions. Each edge $e_k:=\overline {a_{k-1}a_{k}}$ belongs to  exactly one sub-region, and each  sub-region contains exactly one edge  $e_k$.   We denote the subregion containing $e_i$ by $R_i$.  
   \item Each vertex $a_i$  belongs to exactly two regions, i.e.,  $R_i$ and $R_{i+1}$ (with $R_1=R_{n+1}$). 
  
\end{itemize}
Note that following a canonical partition, the $n$-polygon (i.e., the black edges in Fig. \ref{fig_SPPol}) can be compressed into a zero-area loop along the interior boundaries (i.e., the blue and purple edges in Fig. \ref{fig_SPPol}).

 We define the equivalent class of the canonical  partition as follows: for partitions $P$ and $\widetilde{P}$, 
 $$P\sim \widetilde{P}\;\iff
\quad \left(\forall\; 1\le  i,j\le n, \quad \quad \quad \; R_{i}\cap R_j=\emptyset \iff
  \widetilde{R}_{i}\cap \widetilde{R}_j=\emptyset \right)$$
i.e., the sub-regions have the same neighbors. 
We denote $SP({\cal P}_{\textbf{a}})$ the collection of equivalent classes of the canonical partition of polygon ${\cal P}_{\textbf{a}}$: 
$$
SP({\cal P}_{\textbf{a}}):= \{[P]: \; \hbox{$[P]$ is equivalent class  of the canonical partition of polygon ${\cal P}_{\textbf{a}}$}\}$$
For each class of canonical partition of $n$ polygon, we assign {a tree structure by removing the edges of the  polygon. }  We denote $TSP({\cal P}_{\textbf{a}})$  the collection of  trees for the classes of the canonical partitions of the polygon ${\cal P}_{\textbf{a}}$:  
$$
TSP({\cal P}_{\textbf{a}}):=\left\{\Gamma: \Gamma\sim [P]\in SP({\cal P}_{\textbf{a}})\right\}$$

\end{definition}

{We divide edges of a tree $\Gamma$ into two classes: 1. boundary edges consisting of any edge with a vertex in the polygon. 2. internal edges consisting of the rest.  Finally, polygon edges are those edges in the original polygon. }

\begin{lemma}[Classification of canonical partitions]\label{lem_CSP}
Let $\Gamma_\textbf{a}\in TSP({\cal P}_\textbf{a})$ be the  tree for a canonical partition   polygon ${\cal P}_{a}$. 
Denote by ${\cal F}(\Gamma_{\textbf{a}})$    as the collection of the pairs of subregions that are non-adjacent but sharing an internal edge in $\Gamma_a$, i.e., 
\begin{equation}\label{set_csp}
{\cal F}(\Gamma_{\textbf{a}}):=
\Big\{\{i,j\}: |R_i\,\cap\,R_j|\in {\cal E} (\Gamma_{\textbf{a}}),    \quad |i-j|>1 \mod n \Big\}.  
\end{equation}
Here  $\mod n$ implies that $|1- n |= 1$. 
 E.g., for the partition in Fig. \ref{fig_SPPol}, we have ${\cal F}=\{\{1,4\},\{6,4\} \}$. 

Then for  $\Gamma'_\textbf{a}$, $\Gamma_\textbf{a} \in TSP({\cal P}_\textbf{a})$, we have 
$$
\Gamma'_\textbf{a}=\Gamma _\textbf{a}\iff {\cal F}(\Gamma_{\textbf{a}})={\cal F}(\Gamma'_{\textbf{a}})
$$
Futhermore,  we says $\{i,j\}$ and  $\{k,l\}$ (with $i<j$, $k<l$) are crossing pairs if they satisfy 
 $$
 i<k<j<l \quad \text{or} \;  k<i<l<j
 $$
 where the ordering is on $\mathbb Z$ instead of $\mathbb Z_n$. 
 Then we have the following properties 
 \begin{enumerate}
     \item  ${\cal F}(\Gamma_{\textbf{a}})$ contains  {\textit no} crossing pairs.
     \item If ${\cal F}^*$ is a subset of $\left\{\{i, j\}:   | i-j|>1, \mod  \, n\right\}$ and there is {\textit no} crossing pairs in ${\cal F}^*$, then there exists a 
     canonical partition $\Gamma_{\textbf{a}}$ such that  ${\cal E}^*={\cal F}(\Gamma_{\textbf{a}})$.
 \end{enumerate}
\end{lemma}

\begin{proof}[Proof of lemma \ref{lem_CSP}] We prove this lemma by  mathematical induction. First if \eqref{set_csp} is empty, then the tree graph inside must be a star as in Figure \ref{fig_star}. 
\begin{figure}[ht]
    \centering
      \scalebox{0.5}{
    \begin{tikzpicture}
        \coordinate (A1) at (0, 3);      
        \coordinate (A2) at (2.6, 1.5);  
        \coordinate (A3) at (2.6, -1.5); 
        \coordinate (A4) at (0, -3);     
        \coordinate (A5) at (-2.6, -1.5);
        \coordinate (A6) at (-2.6, 1.5); 
        \coordinate (Center) at (0, 0);  

        \draw[thick, black] (A1) -- (A2) -- (A3) -- (A4) -- (A5) -- (A6) -- cycle;

        \draw[thick, blue] (A1) -- (Center);
        \draw[thick, blue] (A2) -- (Center);
        \draw[thick, blue] (A3) -- (Center);
        \draw[thick, blue] (A4) -- (Center);
        \draw[thick, blue] (A5) -- (Center);
        \draw[thick, blue] (A6) -- (Center);

        \fill[black] (A1) circle (2pt);
        \fill[black] (A2) circle (2pt);
        \fill[black] (A3) circle (2pt);
        \fill[black] (A4) circle (2pt);
        \fill[black] (A5) circle (2pt);
        \fill[black] (A6) circle (2pt);
        \fill[black] (Center) circle (2pt);

    \end{tikzpicture}
    }
    \caption{Star graph}
    \label{fig_star}
\end{figure}
From now on, we assume that the set in \eqref{set_csp} is nonempty. Assume  for example that $(1,3)\in {\cal F}$. By assumption,  there exists an internal edge connecting with $R_1$ and $R_3$. Then the partition can be reduced to two partitions in the smaller polygons:
$$
 {\cal P}_{(a_1,a_2,a)} ,\quad \quad  {\cal P}_{(a_3,a_4,a_5,a_6,b)}. 
$$
where $a$ and $b$ are two additional vertices to form two polygons. The rationale for this construction is that once $(1,3)\in {\cal F}$ is given, the original polygon
will be divided into two regions which will not ``communicate". The vertices $a$ and. $b$ are added to get back to polygon language. 
Based on this observation, one can easily finish the induction proof.  
$$
 \scalebox{0.8}{\begin{minipage}{0.55\textwidth}
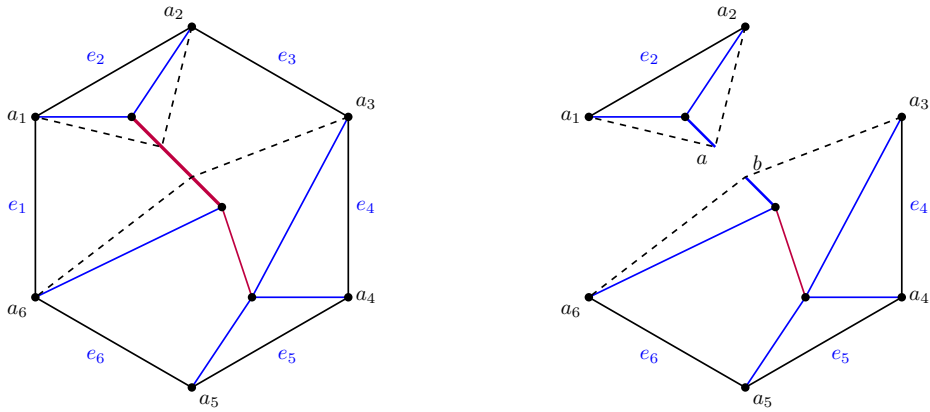
    \begin{tikzpicture}
        \coordinate (A1) at (0, 3);      
        \coordinate (A2) at (2.6, 1.5);  
        \coordinate (A3) at (2.6, -1.5); 
        \coordinate (A4) at (0, -3);     
        \coordinate (A5) at (-2.6, -1.5);
        \coordinate (A6) at (-2.6, 1.5); 

        \coordinate (B1) at (-1, 1.5);          
        \coordinate (B2) at (-0.5, 1.0);        
        \coordinate (B3) at (0, 0.5);       
        \coordinate (B4) at (0.5,  0);      
        \coordinate (B5) at (1, -1.5);       

        \draw[thick, black] (A1) -- (A2) node[midway, above right, blue] {$e_3$};
        \draw[thick, black] (A2) -- (A3) node[midway, right, blue] {$e_4$};
        \draw[thick, black] (A3) -- (A4) node[midway, below right, blue] {$e_5$};
        \draw[thick, black] (A4) -- (A5) node[midway, below left, blue] {$e_6$};
        \draw[thick, black] (A5) -- (A6) node[midway, left, blue] {$e_1$};
        \draw[thick, black] (A6) -- (A1) node[midway, above left, blue] {$e_2$};

       \draw[thick, blue] (A1) -- (B1) -- (A6);
        \draw[ultra thick, purple] (B1) -- (B4); 
        \draw[thick, blue] (B4) -- (A5);
        \draw[thick, purple] (B4) -- (B5) ;
        \draw[thick, blue] (B5) -- (A2);
        \draw[thick, blue] (A3) -- (B5) -- (A4);

        \draw[thick, black, dashed] (A1) -- (B2)--(A6);
        \draw[thick, black, dashed] (A2) -- (B3)--(A5);

        \fill[black] (A1) circle (2pt);
        \fill[black] (A2) circle (2pt);
        \fill[black] (A3) circle (2pt);
        \fill[black] (A4) circle (2pt);
        \fill[black] (A5) circle (2pt);
        \fill[black] (A6) circle (2pt);
        \fill[black] (B1) circle (2pt);
        \fill[black] (B4) circle (2pt);
        \fill[black] (B5) circle (2pt);
        
        \node[above left] at (A1) {$a_2$};
        \node[above right] at (A2) {$a_3$};
        \node[right] at (A3) {$a_4$};
        \node[below right] at (A4) {$a_5$};
        \node[below left] at (A5) {$a_6$};
        \node[left] at (A6) {$a_1$};
    \end{tikzpicture}
\end{minipage}
\hfill
\begin{minipage}{0.45\textwidth}
    \begin{tikzpicture}
        \coordinate (A1) at (0, 3);      
        \coordinate (A2) at (2.6, 1.5);  
        \coordinate (A3) at (2.6, -1.5); 
        \coordinate (A4) at (0, -3);     
        \coordinate (A5) at (-2.6, -1.5);
        \coordinate (A6) at (-2.6, 1.5); 

        \coordinate (B1) at (-1, 1.5);          
        \coordinate (B2) at (-0.5, 1.0);        
        \coordinate (B3) at (0, 0.5);       
        \coordinate (B4) at (0.5,  0);      
        \coordinate (B5) at (1, -1.5);       

        \draw[thick, black] (A2) -- (A3) node[midway, right, blue] {$e_4$};
        \draw[thick, black] (A3) -- (A4) node[midway, below right, blue] {$e_5$};
        \draw[thick, black] (A4) -- (A5) node[midway, below left, blue] {$e_6$};
        \draw[thick, black] (A6) -- (A1) node[midway, above left, blue] {$e_2$};

        \draw[thick, blue] (A1) -- (B1) -- (A6);
        \draw[ very thick, blue] (B1) -- (B2); 
        \draw[ very thick, blue] (B3) -- (B4); 
        \draw[thick, blue] (B4) -- (A5);
        \draw[thick, purple] (B4) -- (B5) ;
        \draw[thick, blue] (B5) -- (A2);
        \draw[thick, blue] (A3) -- (B5) -- (A4);

        \draw[thick, black, dashed] (A1) -- (B2)--(A6);
        \draw[thick, black, dashed] (A2) -- (B3)--(A5);

        \fill[black] (A1) circle (2pt);
        \fill[black] (A2) circle (2pt);
        \fill[black] (A3) circle (2pt);
        \fill[black] (A4) circle (2pt);
        \fill[black] (A5) circle (2pt);
        \fill[black] (A6) circle (2pt);
        \fill[black] (B1) circle (2pt);
        \fill[black] (B4) circle (2pt);
        \fill[black] (B5) circle (2pt);

        \node[above left] at (A1) {$a_2$};
        \node[above right] at (A2) {$a_3$};
        \node[right] at (A3) {$a_4$};
        \node[below right] at (A4) {$a_5$};
        \node[below left] at (A5) {$a_6$};
        \node[left] at (A6) {$a_1$};
 \node[below left] at (B2) {$a $}; 
 \node[above right] at (B3) {$b $};
        
    \end{tikzpicture}
\end{minipage}
}
$$

\end{proof}
\bigskip

\begin{definition}[Representation of ${\cal K}_{t, \boldsymbol{\sigma}, \textbf{a}}$]\label{def33} 
Assume that  $\Gamma_{\textbf{a}}\in TSP({\cal P}_{\textbf{a}})$. 
Associated with each edge $e_i$ (or the corresponding region) there is charge $\sigma_i$ and we denote by  $\boldsymbol{\sigma}\in \{+,-\}^n$ the collection of all charges. 
Let $b_1\cdots b_m$ be internal vertices of the $\Gamma_{\textbf{a}}$. Given $\boldsymbol{\sigma}$, 
$ \textbf{a}=(a_1, a_2,\cdots, a_n)$,  $\textbf{b}=(b_1\cdots b_m)$  (here we slightly abuse the notations for the vertices and their  values) and 
$t\in[0,1]$, define 
\begin{align}\label{defGabst}
\Gamma_{\textbf{a}}^{(\textbf{b})}\left( t,\boldsymbol{\sigma}\right)=& 
\prod_{e\in {\cal E} (\Gamma)} \left(f_{ t}(e)\right)_{e_i,\,e_f},\quad \quad m_i=m(\sigma_i)\in \{m, \bar m \}
\end{align}
here $e_i$ and $e_f$ are the vertices of the edge $e$,  and $f(e)$ is a matrix depending on the edge $e$ defined as follows:  
\begin{enumerate}
    \item If $e=\{a_i, b_j\}$, 
    then $e$ is the boundary between $R_i$ and $R_{i+1}$, and $f_t(e)$ is defined as: 
    $$ 
    f_t(e)= \Theta^{(B)}_{t m_i m_{i+1}},\quad\quad  e=\{a_i, b_j\}
    $$
     \item If $e=\{b_i, b_j\}$, and $e$ is the boundary between   $R_k$ and $R_{l}$, then 
    $$ 
    f_t(e)= \Theta^{(B)}_{t m_k m_{l}}-1,\quad \quad e=\{b_i, b_j\},\quad R_k\cap R_l =e. 
    $$
\end{enumerate}
A compact definition is
\begin{equation}\label{def_feRR}
    f(e) =
 \Theta^{(B)}_{t m_k m_{l}}-{\bf 1}\Big(|k  - l|\ne 1,\mod n\Big),\quad  \quad R_k\cap R_l =e 
\end{equation}
For $\Gamma_{\textbf{a}}\in TSP({\cal P}_{\textbf{a}})$, define  
\begin{equation}\label{def_Gatsiga}
   \Gamma_{t, \boldsymbol{\sigma}, \textbf{a}}
 :=
 \Gamma_{\textbf{a}}(t, \boldsymbol{\sigma})
 : = \sum_{\textbf{b}\,\in \,\mathbb Z_L^m} \Gamma^{(\textbf{b})}_{\textbf{a}}(t, \boldsymbol{\sigma}) .
\end{equation}
 
\end{definition}

For example $\Gamma_{\textbf{a}}^{(\textbf{b})}$ in Fig. \ref{fig_SPPol} is given by 
\begin{align}   \nonumber
\Gamma_{\textbf{a}}^{(\textbf{b})}
=& (\Theta^{(B)}_{t, m_1 m_2})_{a_1, b_1} \cdot (\Theta^{(B)}_{t, m_2 m_3})_{a_2, b_1} \cdot (\Theta^{(B)}_{t, m_3 m_4})_{a_3, b_1} \cdot (\Theta^{(B)}_{t, m_4 m_5})_{a_4, b_3} \nonumber \\  \nonumber
& 
\times (\Theta^{(B)}_{t, m_5 m_6})_{a_5, b_3} \cdot (\Theta^{(B)}_{t, m_6 m_1})_{a_6, b_2}\cdot 
\left(\Theta^{(B)}_{t, m_1 m_4} -1\right)_{b_1, b_2}
\cdot
\left(\Theta^{(B)}_{t, m_4 m_6} -1\right)_{b_2, b_3} 
 \end{align}
The key result in this subsection is the following  representation of ${\cal K}_{t, \boldsymbol{\sigma}, \textbf{a}}$. 

\begin{lemma}[Tree Representation of $\cal K$]\label{Lemma_TRofK}
   For $n \ge 2$, we have
   \begin{equation}\label{eq_trcalK}
    {\cal K}_{t, \boldsymbol{\sigma}, \textbf{a}} =m_{\boldsymbol{\sigma}}\cdot W^{-n+1}\sum_{\Gamma_{\textbf{a}}\, \in\,  {TSP}({\cal P}_{\textbf{a}})} \Gamma_{ \textbf{a}}(t, \boldsymbol{\sigma}),\quad \quad \quad m_{\boldsymbol{\sigma}}:=\prod_{i=1}^n m(\sigma_i);
   \end{equation}  
\end{lemma}

As an example, we give the tree graph representation of $\cal K$ for $n = 4$.
\begin{figure}[ht]
    \centering
    \scalebox{0.8}{
    \begin{tikzpicture}
        \coordinate (A1) at (0, 2);
        \coordinate (A2) at (2, 0);
        \coordinate (A3) at (0, -2);
        \coordinate (A4) at (-2, 0);
        \coordinate (B1) at (0, 0);

        \draw[thick, black] (A1) -- (A2) -- (A3) -- (A4) -- cycle;

        \draw[thick, blue] (A1) -- (B1);
        \draw[thick, blue] (A2) -- (B1);
        \draw[thick, blue] (A3) -- (B1);
        \draw[thick, blue] (A4) -- (B1);

        \node[above] at (A1) {$a_1$};
        \node[right] at (A2) {$a_2$};
        \node[below] at (A3) {$a_3$};
        \node[left] at (A4) {$a_4$};
        \node[below right] at (B1) {$b_1$};

        \begin{scope}[shift={(7, 0)}]
            \coordinate (A1) at (0, 2);
            \coordinate (A2) at (2, 0);
            \coordinate (A3) at (0, -2);
            \coordinate (A4) at (-2, 0);
            \coordinate (B1) at (0, 0.75);
            \coordinate (B2) at (0, -0.75);

            \draw[thick, black] (A1) -- (A2) -- (A3) -- (A4) -- cycle;

            \draw[thick, blue] (A1) -- (B1);
            \draw[thick, blue] (A2) -- (B1);
            \draw[thick, blue] (A3) -- (B2);
            \draw[thick, blue] (A4) -- (B2);

            \draw[very thick, purple] (B1) -- (B2);

            \node[above] at (A1) {$a_1$};
            \node[right] at (A2) {$a_2$};
            \node[below] at (A3) {$a_3$};
            \node[left] at (A4) {$a_4$};
            \node[below right] at (B1) {$b_1$};
            \node[below left] at (B2) {$b_3$};
        \end{scope}

         \begin{scope}[shift={(14, 0)}]
            \coordinate (A1) at (0, 2);
            \coordinate (A2) at (2, 0);
            \coordinate (A3) at (0, -2);
            \coordinate (A4) at (-2, 0);
            \coordinate (B1) at (0, 0.75);
            \coordinate (B2) at (0, -0.75);

            \draw[thick, black] (A1) -- (A2) -- (A3) -- (A4) -- cycle;

            \draw[thick, blue] (A1) -- (B1);
            \draw[thick, blue] (A4) -- (B1);
            \draw[thick, blue] (A3) -- (B2);
            \draw[thick, blue] (A2) -- (B2);

            \draw[very thick, purple] (B1) -- (B2);

            \node[above] at (A1) {$a_1$};
            \node[right] at (A2) {$a_2$};
            \node[below] at (A3) {$a_3$};
            \node[left] at (A4) {$a_4$};
            \node[below right] at (B1) {$b_1$};
            \node[below left] at (B2) {$b_2$};
        \end{scope}

    \end{tikzpicture}
    }
    \caption{Graphs for $n=4$}
    \label{fig_n=4}
\end{figure}
There are three graphs for the case $n = 4$, as in Figure \ref{fig_n=4}.  The blue edges are boundary edges and equal to $\Theta^{(B)}$. The purple edges are internal edges equal to  $\Theta^{(B)} - 1$.  The r.h.s. of \eqref{eq_trcalK} equals to 

\begin{align}\nonumber
\sum_{\Gamma_{\textbf{a}}\, \in\,  {TSP}({\cal P}_{\textbf{a}})} \Gamma_{ \textbf{a}}(t, \boldsymbol{\sigma})= & \sum_{b_1, b_2, b_3, b_4}  \left(\prod_{i=1}^4 \left(\Theta^{(B)}_{t m_{i-1} m_i}\right)_{a_i b_i}\right) \times
\\\nonumber
& \times \left( \delta_{b_1 b_2 b_3 b_4}
+ \delta_{b_1 b_2} \delta_{b_3 b_4} \left(\Theta^{(B)}_{t m_1 m_3} - 1\right)_{b_1 b_3}
+ \delta_{b_1 b_4} \delta_{b_2 b_3} \left(\Theta^{(B)}_{t m_2 m_4} - 1\right)_{b_1 b_2}\right)
\end{align}

\begin{corollary}[Pure loop $\cal K$]\label{lem_pureloop} In the special case 
$\sigma=(+,+\cdots +)$, we have 
\begin{align}\label{res_pureKes}\left|{\cal K}_{t, \boldsymbol{\sigma}, \textbf{a}}\right|\le C_n \exp\left(-c_n\max_{ij}\|a_i-a_j\|\right),\quad \|a_i-a_j\|=(a_i-a_j) \mod L
\end{align} 
\end{corollary}
\begin{proof}[Proof of Corollary \ref{lem_pureloop}]
By  assumption that  $\sigma_i=+$ for all $i$,   $f_t(e)$ in \eqref{def_feRR} is either  $\Theta^{(B)}_{tm^2}$ or $\Theta^{(B)}_{tm^2}-1$. Applying  \eqref{prop:ThfadC} with $\xi=tm^2$, we obtain that $\|\Theta^{(B)}_{tm^2}\|_{\max}=O(1)$ and $\Theta^{(B)}_{tm^2}$decays exponentially. Thus 
$$  \left|\Gamma^{(\textbf{b})}_{\textbf{a}}(t, \boldsymbol{\sigma})\right|\le C_n \exp\left(-c_n\max_{ij}\|b_i-b_j\|-c_n \max_{ij}\|a_i-b_j\|\right).
 $$
Together with the \eqref{def_Gatsiga} and \eqref{eq_trcalK}, we obtain the desired result \eqref{res_pureKes}. 
\end{proof}
\begin{proof}[Proof of Lemma \ref{Lemma_TRofK}]
 In this proof, we temporally denote 
   $$
   \widetilde{ {\cal K}}_{t, \boldsymbol{\sigma}, \textbf{a}}:= m_{\boldsymbol{\sigma}}\cdot W^{-n+1}    \sum_{\Gamma_{\textbf{a}}\, \in\,  {TSP}({\cal P}_{\textbf{a}})}  \Gamma_{t, \boldsymbol{\sigma}, \textbf{a}},\quad \quad \quad \quad m_{\boldsymbol{\sigma}}:=\prod_{i=1}^n m(\sigma_i);
   $$
   We will show that 
  \begin{equation}\label{jlajus}
     \widetilde{ {\cal K}}_{t, \boldsymbol{\sigma}, \textbf{a}}
   =
    { {\cal K}}_{t, \boldsymbol{\sigma}, \textbf{a}},\quad  \text{for} \; \quad  t=0   \end{equation}
and   $\widetilde{\cal K}$ satisfies the dynamics equation for $ \cal K$ in \eqref{pro_dyncalK}, i.e., 
  \begin{align}\label{pjlacalK}
       \frac{d}{dt}\, \widetilde{ {\cal K}}_{t, \boldsymbol{\sigma}, \textbf{a}} 
       =  
       {W}\cdot  \sum_{1 \le k < l \le n} \sum_{a, b} \; \left( {\cal G}^{(a), L}_{k, l} \circ  \widetilde{ {\cal K}}_{t, \boldsymbol{\sigma}, \textbf{a}} \right) S^{(B)}_{ab} \left( {\cal G}^{(b), R}_{k, l} \circ  \widetilde{ {\cal K}}_{t, \boldsymbol{\sigma}, \textbf{a}} \right)  
    \end{align}
    here
    $$
    {\cal G}^{(a), L}_{k, l} \circ  \widetilde{ {\cal K}}_{t, \boldsymbol{\sigma}, \textbf{a}}  \,=\,\widetilde{ {\cal K}}_{t, \; {\cal G}^{(a), L}_{k, l}\left( \boldsymbol{\sigma}, \; \textbf{a}\right) } ,\quad \quad 
    \quad 
    {\cal G}^{(b), R}_{k, l} \circ  \widetilde{ {\cal K}}_{t, \boldsymbol{\sigma}, \textbf{a}} 
    \,=\,\widetilde{ {\cal K}}_{t, \; {\cal G}^{(b), R}_{k, l}\left( \boldsymbol{\sigma}, \; \textbf{a}\right) }.
    $$
    For  $t = 0$, we have  by definition that 
    $$
    t = 0 \implies \Theta^{(B)}_{t\, \xi} = I 
    $$
    If the tree graph $\Gamma$ has an internal edge,  then $\Gamma_{t, \boldsymbol{\sigma}, \textbf{a}} = 0$ when $t = 0$. Hence for $t=0$, the  only non-trivial graph is the star-shaped graph (figure. \ref{fig_star}), which has only one internal vertex.
Hence  \eqref{jlajus} can be explicitly verified, i.e., 
$$
 \widetilde{{\cal K}}_{0, \boldsymbol{\sigma}, \textbf{a}} 
  =   m_{\boldsymbol{\sigma}} W^{ -n+1}\cdot\delta_{a_1, a_2, \dots, a_n}  
  = 
   W^{ -n+1}\Gamma^{(\text{star})}_{0, \boldsymbol{\sigma}, \textbf{a}} 
     = {\cal K}_{0, \boldsymbol{\sigma}, \textbf{a}} .
$$

\medskip
Now we prove \eqref{pjlacalK}. By definition,
$$
      \frac{d}{dt}\, \widetilde{ {\cal K}}_{t, \boldsymbol{\sigma}, \textbf{a}} =
     m_{\boldsymbol{\sigma}} W^{-n+1}\cdot \sum_{\Gamma_{\textbf{a}}\, \in\,  {TSP}({\cal P}_{\textbf{a}})} 
       \frac{d}{dt}\, \Gamma_{t, \boldsymbol{\sigma}
       , \textbf{a}}
       =   m_{\boldsymbol{\sigma}} W^{-n+1}\cdot \sum_{\Gamma_{\textbf{a}}\, \in\,  {TSP}({\cal P}_{\textbf{a}})}
       \sum_{\textbf{b}}
       \frac{d}{dt} \, \Gamma^{(\textbf{b})}_{t, \boldsymbol{\sigma}
       , \textbf{a}}
$$
 Recall $\Gamma^{(\textbf{b})}_{t, \boldsymbol{\sigma}, \textbf{a}}$ defined in \eqref{defGabst}. 
 Due to $ d m_i/dt = 0$, the derivative $d/dt$  acts only on the $f_t(e)$.   By definition, 
$$
\frac{d}{dt} \Theta^{(B)}_{t \xi} = \Theta^{(B)}_{t \xi} \cdot \xi S^{(B)} \cdot \Theta^{(B)}_{t \xi}
$$
Therefore, 
$$
\frac{d}{dt} \Theta^{(B)}_{t m_i m_j} = m_i m_j \left( \Theta^{(B)}_{t m_i m_j} \cdot S^{(B)} \cdot \Theta^{(B)}_{t m_i m_j}\right)
$$
In other words, the derivative of a blue edge $\Theta^{(B)}$ or a purple edge $\Theta^{(B)} - 1$ equals $m_i m_j$ times two blue edges with $S^{(B)}$ in the middle.

\begin{figure}[ht]
    \centering
    \begin{tikzpicture}[baseline=(current bounding box.center)]
        \node[blue] at (-2, 1) {\(\frac{d}{dt}\)};
        \draw[thick, blue] (-1.5, 1) -- (1.5, 1); 

        \node[purple] at (-2, 0) {\(\frac{d}{dt}\)};
        \draw[thick, purple] (-1.5, 0) -- (1.5, 0); 

        \node at (2, 0.5) {\(=\)};

        \draw[thick, blue] (2.5, 0.5) -- (4, 0.5);
        \node at (4.5, 0.5) {\(S^{(B)}\)};
        \draw[thick, blue] (5 , 0.5) -- (6.5, 0.5);

        
        \node at (7.5,0.5) {\( \cdot\; m_i m_j \)};
    \end{tikzpicture}
    \caption{Derivatives of edges}
\end{figure}

 On the other hand, we know that for fixed $\Gamma_\textbf{a}\in TSP({\cal P}_\textbf{a})$, $\boldsymbol{\sigma}$, $\textbf{a}$, $i$, $j$, there is at most one edge $e\in {\cal E}(\Gamma_\textbf{a})$ such that $e=R_i \cap R_j$. Then we can write the derivative of $\Gamma_{t, \boldsymbol{\sigma}, \textbf{a}}$  as follows
\begin{equation}\label{9aksufuo44}
  \frac{d}{dt} \, \Gamma^{(\textbf{b})}_{t, \,\boldsymbol{\sigma}
       ,\, \textbf{a}} = \sum_{1\le i < j\le n }\; \; \sum_{e\,\in \,{\cal E}(\Gamma_\textbf{a})} \textbf{1} \left(e=R_i\cap R_j\right) 
       \cdot \Gamma^{(\textbf{b})}_{t, \,\boldsymbol{\sigma}
       ,\, \textbf{a}}
       \cdot \frac{m_i \,m_j \,\left(\Theta^{(B)}_{t m_i m_j} \cdot S^{(B)} \cdot \Theta^{(B)}_{t m_i m_j}\right)_{e_i,e_f}}{\left(f(e)\right)_{e_i,e_f}} 
\end{equation}
where $f_t(e)$ is defined in \eqref{def_feRR}.
Suppose  that there  exists $e\,\in \,{\cal E}(\Gamma_\textbf{a})$ such that  $e=R_i\cap R_j$.  Then 
$$
\Gamma^{(\textbf{b})}_{t, \,\boldsymbol{\sigma}
       ,\, \textbf{a}}
       \cdot \frac{  \,\left(\Theta^{(B)}_{t m_i m_j} \cdot S^{(B)} \cdot \Theta^{(B)}_{t m_i m_j}\right)_{e_i,e_f}}{\left(f(e)\right)_{e_i,e_f}} ,\quad e=R_i\cap R_j 
       $$
    is equal to removing the edge $e$ in $\Gamma_\textbf{a}$ and adding two edges $\{e_i, a\}$, $\{e_f, b\}$, and a $S^{(B)}_{ab}$ in the middle. Here is an example with $i=1$ and $j=3$ in Figure \ref{fig:hexagons}. Note: this statement also holds for the case $|i-j|=1$. For example, if $i=1$, $j=2$, the triangle in the r.h.s. of Figure \ref{fig:hexagons} will become a $n=2$ polygon ${\cal P}_{(\sigma_1,\sigma_2), (a_1,a)}$ with a blue edge $(a_1,a)$ inside. 
   \begin{figure}[h!]
   \hspace{0.5in}
 \scalebox{0.8}{\begin{minipage}{0.55\textwidth}
   \begin{tikzpicture}
        \coordinate (A1) at (0, 3);      
        \coordinate (A2) at (2.6, 1.5);  
        \coordinate (A3) at (2.6, -1.5); 
        \coordinate (A4) at (0, -3);     
        \coordinate (A5) at (-2.6, -1.5);
        \coordinate (A6) at (-2.6, 1.5); 

        \coordinate (B1) at (-1, 1.5);          
        \coordinate (B2) at (-0.5, 1.0);        
        \coordinate (B3) at (0, 0.5);       
        \coordinate (B4) at (0.5,  0);      
        \coordinate (B5) at (1, -1.5);       

        \draw[thick, black] (A1) -- (A2) node[midway, above right, blue] {$\sigma_3$};
        \draw[thick, black] (A2) -- (A3) node[midway, right, blue] {$\sigma_4$};
        \draw[thick, black] (A3) -- (A4) node[midway, below right, blue] {$\sigma_5$};
        \draw[thick, black] (A4) -- (A5) node[midway, below left, blue] {$\sigma_6$};
        \draw[thick, black] (A5) -- (A6) node[midway, left, blue] {$\sigma_1$};
        \draw[thick, black] (A6) -- (A1) node[midway, above left, blue] {$\sigma_2$};

       \draw[thick, blue] (A1) -- (B1) -- (A6);
        \draw[ultra thick, purple] (B1) -- (B4); 
        \draw[thick, blue] (B4) -- (A5);
        \draw[thick, purple] (B4) -- (B5) ;
        \draw[thick, blue] (B5) -- (A2);
        \draw[thick, blue] (A3) -- (B5) -- (A4);


        \fill[black] (A1) circle (2pt);
        \fill[black] (A2) circle (2pt);
        \fill[black] (A3) circle (2pt);
        \fill[black] (A4) circle (2pt);
        \fill[black] (A5) circle (2pt);
        \fill[black] (A6) circle (2pt);
        \fill[black] (B1) circle (2pt);
        \fill[black] (B4) circle (2pt);
        \fill[black] (B5) circle (2pt);
        
        \node[above left] at (A1) {$a_2$};
        \node[above right] at (A2) {$a_3$};
        \node[right] at (A3) {$a_4$};
        \node[below right] at (A4) {$a_5$};
        \node[below left] at (A5) {$a_6$};
        \node[left] at (A6) {$a_1$};
        
    \end{tikzpicture}
\end{minipage}
\hfill
\begin{minipage}{0.45\textwidth}
    \begin{tikzpicture}
     
        \coordinate (A1) at (0, 3);      
        \coordinate (A2) at (2.6, 1.5);  
        \coordinate (A3) at (2.6, -1.5); 
        \coordinate (A4) at (0, -3);     
        \coordinate (A5) at (-2.6, -1.5);
        \coordinate (A6) at (-2.6, 1.5); 

        \coordinate (B1) at (-1, 1.5);          
        \coordinate (B2) at (-0.5, 1.0);        
        \coordinate (B3) at (0, 0.5);       
        \coordinate (B4) at (0.5,  0);      
        \coordinate (B5) at (1, -1.5);       

        \draw[thick, black] (A2) -- (A3) node[midway, right, blue] {$\sigma_4$};
        \draw[thick, black] (A3) -- (A4) node[midway, below right, blue] {$\sigma_5$};
        \draw[thick, black] (A4) -- (A5) node[midway, below left, blue] {$\sigma_6$};
        \draw[thick, black] (A6) -- (A1) node[midway, above left, blue] {$\sigma_2$};

        \draw[thick, blue] (A1) -- (B1) -- (A6);
        \draw[ very thick, blue] (B1) -- (B2); 
        \draw[ very thick, blue] (B3) -- (B4); 
        \draw[thick, blue] (B4) -- (A5);
        \draw[thick, purple] (B4) -- (B5) ;
        \draw[thick, blue] (B5) -- (A2);
        \draw[thick, blue] (A3) -- (B5) -- (A4);

        \draw[thick, black ] (A1) -- (B2) 
        node[midway, below right, blue]{$\sigma_3$};
         \draw[thick, black ] (B2)--(A6)
         node[midway, below left, blue]{$\sigma_1$};
        \draw[thick, black ] (A2) -- (B3) 
        node[midway, above  , blue]{$\sigma_3$};
         \draw[thick, black ] (B3)--(A5)
         node[midway, above left , blue]{$\sigma_1$};

        \fill[black] (A1) circle (2pt);
        \fill[black] (A2) circle (2pt);
        \fill[black] (A3) circle (2pt);
        \fill[black] (A4) circle (2pt);
        \fill[black] (A5) circle (2pt);
        \fill[black] (A6) circle (2pt);
        \fill[black] (B1) circle (2pt);
        \fill[black] (B4) circle (2pt);
        \fill[black] (B5) circle (2pt);

        \node[above left] at (A1) {$a_2$};
        \node[above right] at (A2) {$a_3$};
        \node[right] at (A3) {$a_4$};
        \node[below right] at (A4) {$a_5$};
        \node[below left] at (A5) {$a_6$};
        \node[left] at (A6) {$a_1$};
 \node[below left] at (B2) {$a $}; 
 \node[above right] at (B3) {$b $};
 
 \end{tikzpicture}

   \end{minipage}
 
}
  \caption{Derivative of tree}
    \label{fig:hexagons}
\end{figure}

The new edges created by the partition will inherit the original charges. We denote them by 
\begin{equation}
\label{zzzrelation}
(\boldsymbol{\sigma}', \textbf{a}')= {\cal G}^{(a), L}_{i, j}(\boldsymbol{\sigma}, \textbf{a}) \quad\quad 
(\boldsymbol{\sigma}'', \textbf{a}'')={\cal G}^{(b), R}_{i, j}(\boldsymbol{\sigma}, \textbf{a}).
\end{equation}
Therefore, there exist 
$\Gamma'_{\textbf{a}'}\in TSP({\cal P}_{\textbf{a}'})$ and $ \Gamma''_{\textbf{a}''}\in TSP({\cal P}_{\textbf{a}''})$  such that (as in Figure \ref{fig:hexagons})
$$
  m_{\boldsymbol{\sigma}}\cdot \sum_{\textbf{b}}\Gamma^{(\textbf{b})}_{t, \,\boldsymbol{\sigma}
       ,\, \textbf{a}}
       \cdot \frac{m_i \,m_j \,\left(\Theta^{(B)}_{t m_i m_j} \cdot S^{(B)} \cdot \Theta^{(B)}_{t m_i m_j}\right)_{e_i,e_f}}{\left(f(e)\right)_{e_i,e_f}}     = \sum_{a, b} 
        m_{\boldsymbol{\sigma}'}\cdot \Gamma'_{t, \boldsymbol{\sigma}', \textbf{a}'} \cdot S^{(B)}_{ab} \cdot 
         m_{\boldsymbol{\sigma}''}\Gamma''_{t, \boldsymbol{\sigma}'', \textbf{a}''}.
$$
Thus,
$$
 m_{\boldsymbol{\sigma}}\cdot \frac{d}{dt} \sum_{\Gamma} \Gamma_{t, \boldsymbol{\sigma}, \textbf{a}} = \sum_{i < j} \sum_{a, b} \left(m_{\boldsymbol{\sigma}'}\sum_{\Gamma'} \Gamma'_{t, \boldsymbol{\sigma}', \textbf{a}'}\right) \cdot S^{(B)}_{ab} \cdot \left(m_{\boldsymbol{\sigma}''}\sum_{\Gamma''} \Gamma''_{t, \boldsymbol{\sigma}'', \textbf{a}''}\right)
$$
where $i$, $j$, $a$, $b$, $\boldsymbol{\sigma}$, $\textbf{a}$, $\boldsymbol{\sigma}'$, $\textbf{a}'$, $\boldsymbol{\sigma}''$, and $\textbf{a}''$ satisfy relation \eqref{zzzrelation}. Clearly it is equivalent to \eqref{pjlacalK} and this completes the proof of Lemma  \ref{Lemma_TRofK}.  

 \end{proof}

\subsection{Ward's identity on \texorpdfstring{$\cal K$}{K}}
Recall  Ward's identity on resolvents stating that 
$$
G \cdot G^\dagger =G(z)\cdot G^\dagger(z)=G(z)\cdot G(\bar z)=\frac{G(z)-G(\bar z)}{z-\bar z}= \frac{G-G^\dagger}{2\eta}, \quad \eta=\im z.
$$
In our setting, $G_t(+)=(H_t-z_t)^{-1}=\left(G_t(-)\right)^\dagger $ and thus
\begin{equation}
    \label{WI_2G}
G_t(+)\cdot G_t(-)= \frac{G_t(+)-G_t(-)}{2\eta_t}=\frac{G_t(+)-G_t(-)}{ 2 (1-t)\im m},\quad \eta_t=\im z_t.
\end{equation}
 From Ward's identity, we  have the following identity for $2-G$ loop: 
$$
\sum_a {\cal L}_{t, (+,-), (a,b)}=\frac{1}{2W\eta_t}\left({\cal L}_{t, (+), (b)}-{\cal L}_{t, (-), (b)}\right).
$$
We extend this identity  to all $n-G$ loops ${\cal L}$  in the following lemma. The main purpose of this subsection is show that the same identity holds for  $\cal K$
loop as well. 

\begin{lemma}\label{lem_WI_K} For $n-G$ loop ${\cal L}_{t, \boldsymbol{\sigma}, \textbf{a}}$ with 
$\sigma_1=+$ and $\sigma_n=-$, 
we have
\begin{equation}\label{WI_calL}
  \sum_{a_n}{\cal L}_{t, \boldsymbol{\sigma}, \textbf{a}}=
\frac{1}{2W\eta_t}\left({\cal L}_{t, \; \boldsymbol{\sigma}+,  \; \textbf{a}/a_n}- {\cal L}_{t, \; \boldsymbol{\sigma}-,  \; \textbf{a}/a_n}\right)  
\end{equation}
and
\begin{equation}\label{WI_calK}
  \sum_{a_n}{\cal K}_{t, \boldsymbol{\sigma}, \textbf{a}}=
\frac{1}{2W\eta_t}\left({\cal K}_{t, \; \boldsymbol{\sigma}+,  \; \textbf{a}/a_n}
- {\cal K}_{t, \; \boldsymbol{\sigma}-,  \; \textbf{a}/a_n}\right)  
\end{equation}
where 
\begin{itemize}
\item  $\boldsymbol{\sigma}\pm$ is obtained by removing $\sigma_n$ from $\boldsymbol{\sigma}$ and replacing $\sigma_1$ with $\pm$, i.e., 
    $$
\boldsymbol{\sigma}\pm=(\pm, \sigma_2, \sigma_3,\cdots \sigma_{ n-1}).$$
Notice that  the length of $\boldsymbol{\sigma}\pm$ is $n-1$. 
\item  $\textbf{a}/a_n$ is obtained by removing $a_n$ from $\textbf{a}$:
 $$
\textbf{a}/a_n=(a_1, a_2,a_3,\cdots, a_{n-1})
$$
\end{itemize}
\end{lemma}

\begin{corollary}\label{lem_sumAinK}
    Under the assumption of Lemma \ref{lem_WI_K}, we have 
\begin{align}\label{sumallAinK}
      \sum_{a_2}\cdots\sum_{a_{n-1}}\sum_{a_n}{\cal K}_{t, \boldsymbol{\sigma}, \textbf{a}}=O(W\eta_t)^{-n+1}
\end{align}
 \end{corollary}

\begin{proof}[Proof of Corollary \ref{lem_sumAinK}]
By definition of $\cal K$, we know that ${\cal K}$ is  translation invariant. Then the left side  of equation \eqref{sumallAinK} is independent of $a_1$. 
Hence it is equivalent to 
$$
 \frac1L\sum_{\textbf{a}\in\mathbb Z_L^n}{\cal K}_{t, \boldsymbol{\sigma}, \textbf{a}}=O(W\eta_t)^{-n+1}.
$$
 Applying \eqref{WI_calK} repeatedly, we can reduce it to shorter pure loops to get  
\begin{align}\label{sfauuu}
 \Big | L^{-1}\cdot \sum_{\textbf{a}\in\mathbb Z_L^n}{\cal K}_{t, \boldsymbol{\sigma}, \textbf{a}}\Big |
  \le C_n \sum_{1\le m\le n}(W\eta_t)^{-n+m}\cdot L^{-1}\cdot
  \max_{\boldsymbol{\sigma}'\in \{+\}^m \cup \{-\}^m}\left| \sum_{\textbf{a}'\in\mathbb Z_L^m}{\cal K}_{t, \boldsymbol{\sigma}', \textbf{a}'}\right|,\quad 
\end{align}
Here $\boldsymbol{\sigma}'\in \{+\}^m \cup \{-\}^m$ means that $\boldsymbol{\sigma}'=(+,+,\cdots,+)\; \hbox{or }(-,-,\cdots,-)$.
By the estimate of pure loop $\cal K$  in   Lemma \ref{lem_pureloop}, we have 
$$
\sum_{\textbf{a}'\in\mathbb Z_L^m}{\cal K}_{t, \boldsymbol{\sigma}', \textbf{a}'}=O(1).
$$
Together with \eqref{sfauuu}, this  completes the proof of the corollary. 

\end{proof}

\bigskip

 \begin{proof}[Proof of Lemma \ref{lem_WI_K}]
 The equation \eqref{WI_calL} for  $\cal L$ follows from \eqref{WI_2G} directly.  For the \eqref{WI_calK}, we consider first that  $n=2$. 
By the explicit formula for $\cal K$ in this case, we have 
 \begin{align}\label{hsgw2}
 \sum_{a_2} {\cal K}_{t, (+,-), (a_1,a_2)}
 =\frac1W \sum_{a_2}
\left(\frac{|m^2|}{1-t|m|^2S^{(B)}}\right)_{a_1a_2}
=\frac{1}{W(1-t)}    
 \end{align}
 On the other hand, by definition, we have 
 $$
 \frac{1}{2W\eta_t}\left({\cal K}_{t, (+ ), ( a_1)}-{\cal K}_{t, (- ), ( a_1)}\right)=\frac{m-\bar m}{2W\eta_t}=\frac{1}{W(1-t)}, \quad m:= m^{(E)}
 $$
 Here we used $|m|=1$ and $\eta_t=(1-t)\im m$ .  Combining  these two identities, we have proved   \eqref{WI_calK} for $n=2$. 
 Similarly, the case $n=3$ follows from a direct calculation and  the following identities 
 $$
 \left(\frac{|m^2|}{1-t|m|^2S^{(B)}}\right)-\left(\frac{m^2}{1-tm^2S^{(B)}}\right)=
 \left(1-m^2\right)\cdot\left(\frac{1}{1-t|m|^2S^{(B)}}\right)\cdot \left(\frac{1}{1-tm^2S^{(B)}}\right).
 $$
 and
 $$
 1-m^2=m(\bar m-m)
 $$
 
 For $n\ge 4$,   we will use the primitive equation instead of the tree representation. For simplicity, we temporally denote
 $$
{\cal K}^{(*)}_{t, \boldsymbol{\sigma}, \textbf{a}}
=\sum_{a_1}{\cal K} _{t, \boldsymbol{\sigma}, \textbf{a}},\quad
{\cal K}^{(**)}_{t, \boldsymbol{\sigma}, \textbf{a}}
=\frac1{2W\eta_t}\left({\cal K}^{(+)}_{t, \boldsymbol{\sigma}, \textbf{a}}-
{\cal K}^{(-)}_{t, \boldsymbol{\sigma}, \textbf{a}}\right),
$$
\begin{equation}\label{def_calKpm}
 {\cal K}^{(+)}_{t, \boldsymbol{\sigma}, \textbf{a}}: = {\cal K}_{t, \; \boldsymbol{\sigma}+,  \; \textbf{a}/a_n}
\quad\quad 
{\cal K}^{(-)}_{t, \boldsymbol{\sigma}, \textbf{a}}: = {\cal K}_{t, \; \boldsymbol{\sigma}-,  \; \textbf{a}/a_n}
   \end{equation}
Our goal is  to prove ${\cal K}^{(*)}={\cal K}^{(**)}$.  With
 $$
 m_n=\bar m, \quad m_1= m, \quad |m|=1, \quad \eta_t\big|_{t=0}=\im m. 
 $$
By Definition  \ref{Def_Ktza},    
 $$
  \sum_{a_n}{\cal K}_{0, \boldsymbol{\sigma}, \textbf{a}}=W^{-n+1}\cdot \left(\prod_{i=1}^n m_i \right){\bf 1}(a_1=a_3=\cdots =a_{n-1}) 
  =
  \frac{1}{2W\eta_0}\left(
  {\cal K}^{(+)}_{0,\; \boldsymbol{\sigma},\; \textbf{a}}
  - 
  {\cal K}^{(-)}_{0,\; \boldsymbol{\sigma},\; \textbf{a}}
  \right) .
 $$
This implies that 
 $${\cal K}^{(*)}_{0, \boldsymbol{\sigma}, \textbf{a}} ={\cal K}^{(**)}_{0, \boldsymbol{\sigma}, \textbf{a}}.$$
 In the remainder of this subsection,  under the inductive assumption that ${\cal K}^{(*)} ={\cal K} ^{(**)}$ holds for ${\cal K}^{(*)}$ of lengths strictly less than $n$,  we will  prove  the following identity: 
  \begin{align}\label{finK*-K}
     \frac d{dt}\left({\cal K}^{(*)}-{\cal K}^{(**)}\right)_{t, \boldsymbol{\sigma}, \textbf{a}} &\;=\;\sum_{ k=1}^{n-1} \sum_{a\;b}
        \left[\left({\cal K}^{(*)}-{\cal K}^{(**)}\right)_{t, \boldsymbol{\sigma}, \textbf{a}}\Bigg|(a_k\to a)\right]
        \cdot S^{(B)}_{ab}\cdot  {\cal K}_{t, ( \sigma _k,  \sigma _{k+1}), (a_k,b)}
        \\\nonumber
       &\;+\; \frac1{1-t}\cdot \left({\cal K}^{(*)}-{\cal K}^{(**)}\right)_{t, \boldsymbol{\sigma}, \textbf{a}} 
 \end{align}
 where $(a_k\to a)$ means replacing $a_k$ in $\textbf{a}$ by $a$.  Together with ${\cal K}^{(*)}-{\cal K}^{(**)}=0$ at $t=0$, this linear differential equation  has only the trivial solution, i.e., 
$$
{\cal K}^{(*)}_{t, \boldsymbol{\sigma}, \textbf{a}} ={\cal K}^{(**)}_{t, \boldsymbol{\sigma}, \textbf{a}}, \quad 0\le t\le 1.
$$
 
 For the l.h.s. of \eqref{finK*-K}, by the primitive equation for  $\cal K$, we have 
  \begin{align} \label{gsasawwe}
     \frac{d}{dt} {\cal K}^{(*)}_{t, \boldsymbol{\sigma}, \textbf{a}}=\sum_{a_n} \frac{d}{dt}\,{\cal K}_{t, \boldsymbol{\sigma}, \textbf{a}} 
       =  
       {W}\sum_{a_n}  \sum_{1 \le k < l \le n}  
       \sum_{a\;b} 
        {\cal J}^{L}_{k,l,a}\cdot{\cal J}^{R}_{k,l,b}\cdot S^{(B)}_{ab},  
    \end{align}
    where 
    $$
    {\cal J}^L_{k,l,a}:= \left( {\cal G}^{(a), L}_{k, l} \circ {\cal K}_{t, \boldsymbol{\sigma}, \textbf{a}} \right), \quad
    {\cal J}^R_{k,l,b}:=\left( {\cal G}^{(b), R}_{k, l} \circ {\cal K}_{t, \boldsymbol{\sigma}, \textbf{a}} \right).
    $$
By definition, the index $a_n$ appears  in above $  {\cal J}^{L} $.    On the other hand,  $\frac{d}{dt} \eta_t^{-1}=(1-t)^{-1}\eta_t^{-1}$ and thus 
  \begin{align} \label{gsasawII}
     \frac{d}{dt} {\cal K}^{(**)}_{t, \boldsymbol{\sigma}, \textbf{a}}=\frac1{1-t}\cdot {\cal K}^{(**)}_{t, \boldsymbol{\sigma}, \textbf{a}}+ 
       \frac{W}{2W\eta_t}  \sum_{1\le k<l\le n-1} 
       \sum_{a\;b} 
       \left(
       {\cal J}^{L,(+)}_{k,l,a}\cdot{\cal J}^{R,(+)}_{k,l,b}
       -
      {\cal J}^{L,(-)}_{k,l,a}\cdot{\cal J}^{R,(-)}_{k,l,b}
       \right)\cdot S^{(B)}_{ab},
       \quad   
    \end{align}
 where (recall $\cal K^{(\pm)}$ defined in \eqref{def_calKpm})
 $$
    {\cal J}^{L,(\pm)}_{k,l,a}:= 
    \left( {\cal G}^{(a), L}_{k, l} \circ {\cal K}^{(\pm)}_{t, \boldsymbol{\sigma}, \textbf{a}} \right), \quad
    {\cal J}^{R,(\pm)}_{k,l,b}:=\left( {\cal G}^{(b), R}_{k, l} \circ  {\cal K}^{(\pm)}_{t, \boldsymbol{\sigma}, \textbf{a}} \right).
$$
We now demonstrate that the right-hand sides of \eqref{gsasawwe} and \eqref{gsasawII} are identical, up to terms involving \(({\cal K}^* - {\cal K}^{**})\). Specifically, the special case where \(k = 1\) and \(l = n\) in \eqref{gsasawwe} contributes the term \(\frac{1}{1-t} \cdot {\cal K}^{(*)}_{t, \boldsymbol{\sigma}, \textbf{a}}\). In most other cases, we find that 
$
{\cal J}^{R,(+)}_{k,l,b} = {\cal J}^{R,(-)}_{k,l,b} = {\cal J}^{R}_{k,l,b},
$
while 
$
{\cal J}^{L}_{k,l,a}
$  
can be expressed using 
$
{\cal J}^{L,(\pm)}_{k,l,a}
$
by inductive assumption. For the remaining few cases (e.g., \(k = 1\), \(l = 2\)), direct cancellations occur between the right-hand sides of \eqref{gsasawwe} and \eqref{gsasawII}. The following proof provides a detailed, albeit tedious,  argument. 

The following  identities  can be easily verified from their definitions. These identities rely on the fact that ${\cal K}^{(+)}$ and ${\cal K}^{(-)}$ differ only in the first component of $\sigma$  being $+$ or  $-$. Similarly, ${\cal K}^{(\pm)}$ and ${\cal K}$ differ slightly in their definitions. 

\begin{itemize}
    \item  For $k,l\in[[2, n-1]]$,  
 \begin{equation}\label{kdjuwsz}
  {\cal J}^{R,(+)}_{k,l,b}=  {\cal J}^{R,(-)}_{k,l,b}={\cal J}^{R }_{k,l,b} ,\quad k,l\in[[2, n-1]]   
 \end{equation}

\item For $k,l\in[[2, n-1]]$ and $ l-k  \ge 2$, the length of  $ {\cal J}^{L }_{k,l,a} $ is no longer than  $ n-1$. By induction, we have 
\begin{equation}\label{a';slw}
    \sum_{a_n}{\cal J}^{L}_{k,l,a}-\frac{1}{2W\eta_t}\left({\cal J}^{L,(+)}_{k,l,a} -{\cal J}^{L,(-)}_{k,l,a} \right)=0 ,\quad  k,l\in[[2, n-1]], \;  l-k \ge 2
\end{equation}

\item For $k,l\in[[2, n-1]]$ and $l=k+1$,
 $$
{\cal J}^L_{k,l,a} ={\cal K}_{t, \boldsymbol{\sigma}, \textbf{a}}\Big|(a_k\to a)
 ,\quad  {\cal J}^R_{k,l,b}={\cal K}_{t, (\sigma_k, \sigma_{k+1}), (a_k, b)},\quad k,l\in[[2, n-1]],\quad l=k+1
$$
 
Similarly, 
$$
{\cal J}^{L,(\pm)}_{k,l,a} ={\cal K}^{(\pm)}_{t, \boldsymbol{\sigma}, \textbf{a}}\Big|(a_k\to a).
$$
Therefore, under the same conditions on $k, l$, 
\begin{equation}\label{uaiw03}
    \sum_{a_n}{\cal J}^{L}_{k,l,a}-\frac{1}{2W\eta_t}\left({\cal J}^{L,(+)}_{k,l,a} -{\cal J}^{L,(-)}_{k,l,a} \right)=
    \left({\cal K}^{(*)}-{\cal K}^{(**)}\right)_{t, \boldsymbol{\sigma}, \textbf{a}}\Bigg|(a_k\to a). 
\end{equation}

\end{itemize}

Combining the identities \eqref{uaiw03}, \eqref{kdjuwsz} and \eqref{a';slw}, we can bound the following parts in \eqref{gsasawII} and \eqref{gsasawwe} with $\cal K^{(*)}-\cal K^{(**)}$ 
\begin{align}\label{WI-pr-Mid}
    & \sum_{  2\le k<l\le n-1} \sum_{a\;b}
    \left( {W}\sum_{a_n} 
        {\cal J}^{L}_{k,l,a}\cdot{\cal J}^{R}_{k,l,b}
        -\frac{W}{2W\eta_t} 
        \left(
        {\cal J}^{L,(+)}_{k,l,a}\cdot{\cal J}^{R,(+)}_{k,l,b}
        -
        {\cal J}^{L,(-)}_{k,l,a}\cdot{\cal J}^{R,(-)}_{k,l,b}
        \right)
        \right)\cdot S^{(B)}_{ab} 
        \\\nonumber
      =  & \sum_{ k=2 }^{n-2} \sum_{a\;b}
        \left[\left({\cal K}^{(*)}-{\cal K}^{(**)}\right)_{t, \boldsymbol{\sigma}, \textbf{a}}\Bigg|(a_k\to a)\right]
        \cdot S^{(B)}_{ab}\cdot  {\cal K}_{t, ( \sigma _k,  \sigma _{k+1}), (a_k,b)}
\end{align}

\bigskip

Next, we estimate the cases that $k=1$ or $l=n$.

\begin{itemize}
\item For $k=1$ and $l=n$, we have 
$$
{\cal J}^L_{k,l,a}  =
 {\cal K}_{t, \,(+,\,-),\, (a,\,a_n)},\quad 
 {\cal J}^R_{k,l,b}={\cal K}_{t, \boldsymbol{\sigma}, \textbf{a}}\Bigg|(a_n\to b)
 $$
 and thus 
 \begin{align}\label{nssfku}
     {W}\sum_{a_n} 
        {\cal J}^{L}_{k,l,a}\cdot{\cal J}^{R}_{k,l,b}\cdot S_{ab}^{(B)}=\frac{1}{1-t}\cdot {\cal K}^{(*)}_{t,\boldsymbol{\sigma},\textbf{a}}.
 \end{align}

\item 
\begin{equation}\label{consjahz}
    k=1,\; l=m,  \quad \hbox{and}\quad   k=m, \;  l=n
\end{equation} 
for some $3\le m\le n-2 $. By induction, 
\begin{align}\label{awslw22}
    &\sum_{a_n}{\cal J}^{L}_{k,l,a}-\frac{1}{2W\eta_t}\left({\cal J}^{L,(+)}_{1,\,m,\,a} -{\cal J}^{L,(-)}_{1,\,m,\,a} \right)=0,\quad  k=1, \; l=m
    \\\nonumber
    &\sum_{a_n}{\cal J}^{L}_{k,l,a}-\frac{1}{2W\eta_t}\left({\cal J}^{R,(+)}_{1,\,m,\,a} -{\cal J}^{R,(-)}_{1,\,m,\,a} \right)=0,\quad  k=m, \; l=n
\end{align} 
and 
\begin{align}\label{awslw33}
    &{\cal J}^R_{k,l,b}={\cal J}^{R,(+)}_{1,m,b},\quad  k=1, \; l=m
    \\\nonumber
    &{\cal J}^R_{k,l,b}={\cal J}^{L,(-)}_{1,m,b},\quad  k=m, \; l=n
\end{align}
Hence  for calculating the  $\sum_{ab}\sum_{a_n} 
        {\cal J}^{L}_{k,l,a}\cdot{\cal J}^{R}_{k,l,b}S_{ab}$, one will see the following terms for $(k,l)=(1,m)$ and $(k,l)=(m,n)$, 
  \begin{align}\nonumber
   \frac{1}{2W\eta_t}   \sum_{ab}  {\cal J}^{L,(-)}_{1,\,m,\,a}\cdot {\cal J}^{R,(+)}_{1,\,m,\,b} S_{ab}\quad  & for \quad (k,l)=(1,m)\\\nonumber
\quad  \frac{-1}{2W\eta_t} \sum_{ab}  {\cal J}^{R,(+)}_{1,\,m,\,a}\cdot {\cal J}^{R,(-)}_{1,\,m,\,b} S_{ab}\quad &  for \quad (k,l)=( m,n)
  \end{align}
   They  cancel each other,  therefore 
 \begin{align}\label{uisiip}
    \sum_{k,l}^*  \sum_{a\;b}
    \left( {W}\sum_{a_n} 
        {\cal J}^{L}_{k,l,a}\cdot{\cal J}^{R}_{k,l,b}
        -\frac{W}{2W\eta_t} 
        \left(
        {\cal J}^{L,(+)}_{k,l,a}\cdot{\cal J}^{R,(+)}_{k,l,b}
        -
        {\cal J}^{L,(-)}_{k,l,a}\cdot{\cal J}^{R,(-)}_{k,l,b}
        \right)
        \right)\cdot S^{(B)}_{ab} =0; 
 \end{align}
 here  $\sum_{kl}^*$ denote   summing  over  $k,l$ satisfying  \eqref{consjahz}.  

 \item Since we assume that  $n\ge 4$, there are only four cases left, i.e., 
 \begin{equation}\label{keayew}
  (k,l)=(1,2), \; (1,n-1), \;(2,n), \;(n-1, n).
 \end{equation}
 
By definition,  
\begin{align} 
 (k,l)=(1,2) \quad &\quad   {\cal J}^L_{k,l,a}  ={\cal K}_{t, \boldsymbol{\sigma}, \textbf{a}}\Bigg|(a_1\to a) ,  
 &{\cal J}^R_{k,l,b}=  {\cal J}^{R,(+)}_{1,\,2,\,b}\quad 
 \\\nonumber
 (k,l)=(1,n-1)\quad &\quad {\cal J}^L_{k,l,a}  =
 {\cal K}_{t, \,(+,\;\sigma_{n-1},\,-),\, (a,\,a_{n-1},\,a_n)},  
 &{\cal J}^R_{k,l,b}=  {\cal J}^{R,(+)}_{1,\,{n-1},\,b}
  \\\nonumber
 (k,l)=(2,n)  \quad &\quad  {\cal J}^L_{k,l,a}  =
 {\cal K}_{t, \,(+,\;\sigma_{2},\,-),\, (a_1,\,a,\,a_n)} ,  
 &{\cal J}^R_{k,l,b}=  {\cal J}^{L,(-)}_{1,\,2,\,b}\quad 
 \\\nonumber
 (k,l)=(n-1,n)\quad &\quad {\cal J}^L_{k,l,a}  ={\cal K}_{t, \boldsymbol{\sigma}, \textbf{a}}\Bigg|(a_{n-1}\to a), 
 &{\cal J}^R_{k,l,b}=  {\cal J}^{R,(+)}_{1,\,{n-1},\,b}
\end{align}
 Summing  up $a_n$ and multiplying  $W$,  we obtain that
 \begin{align} 
 (k,l)=(1,2) \quad &\quad  W\sum_{a_n} {\cal J}^L_{k,l,a}  ={\cal K}^{(*)}_{t, \boldsymbol{\sigma}, \textbf{a}}\Bigg|(a_1\to a) ,  
   \\\nonumber
 (k,l)=(1,n-1)\quad &\quad W \sum_{a_n}{\cal J}^L_{k,l,a}  =
 \frac{W}{2W\eta_t} 
 \left({\cal K}_{t, \,(+,\;\sigma_{n-1}),\, (a,\,a_{n-1} )} 
 - {\cal K}_{t, \,( \sigma_{n-1},-),\, ( a_{n-1}, \, a )}  \right)
 \\\nonumber
 & \quad \quad \quad\quad \quad \quad\;=  \frac{W}{2W\eta_t} 
 \left({\cal J}^{L, (+)}_{1,n-1,a} 
 - {\cal J}^{L, (-)}_{1,n-1,a} \right)
   \\\nonumber
 (k,l)=(2,n)  \quad &\quad  W\sum_{a_n} {\cal J}^L_{k,l,a}  =
\frac{W}{2W\eta_t} 
 \left({\cal K}_{t, \,(+,\;\sigma_{2}),\, (a,\,a_{2} )} 
 - {\cal K}_{t, \,( \sigma_{2},-),\, ( a_{2}, \, a )}  \right)
 \\\nonumber
 & \quad \quad \quad\quad \quad \quad\;=  \frac{W}{2W\eta_t} 
 \left({\cal J}^{R, (+)}_{1,\,2,\,a} 
 - {\cal J}^{R, (-)}_{1,\,2,\,a} \right)  \\\nonumber
 (k,l)=(n-1,n)\quad &\quad  W\sum_{a_n}{\cal J}^L_{k,l,a}  ={\cal K}^{(*)}_{t, \boldsymbol{\sigma}, \textbf{a}}\Bigg|(a_{n-1}\to a), 
 \end{align}
  On the other hand, we have 
  \begin{align}
      {\cal K}^{(**)}_{t, \boldsymbol{\sigma}, \textbf{a}}\Bigg|(a_1\to a) 
      & = \frac{W}{2W\eta_t}
      \left({\cal J}^{L, (+)}_{1,\,2,\,a}-{\cal J}^{L, (-)}_{1,\,2,\,a}\right)
      \\\nonumber
       {\cal K}^{(**)}_{t, \boldsymbol{\sigma}, \textbf{a}}\Bigg|(a_{n-1}\to a) 
      & = \frac{W}{2W\eta_t}
      \left({\cal J}^{R, (+)}_{1,\,n-1,\,a}-{\cal J}^{R, (-)}_{1,\,n-1,\,a}\right)
  \end{align}
  Therefore, with $\sum_{kl}^{**}$ denoting  summing  $k,l$ in \eqref{keayew}, we have  
  \begin{align}\label{uisiip222}
    &\sum_{k,l}^{**}  \sum_{a\;b}
    \left( {W}\sum_{a_n} 
        {\cal J}^{L}_{k,l,a}\cdot{\cal J}^{R}_{k,l,b}
        -\frac{W}{2W\eta_t} 
        \left(
        {\cal J}^{L,(+)}_{k,l,a}\cdot{\cal J}^{R,(+)}_{k,l,b}
        -
        {\cal J}^{L,(-)}_{k,l,a}\cdot{\cal J}^{R,(-)}_{k,l,b}
        \right)
        \right)\cdot S^{(B)}_{ab} 
        \\\nonumber
       &=\sum_{k}\left({\bf 1}_{k=1}+{\bf 1}_{k=n-1}\right) \sum_{a\;b}
        \left[\left({\cal K}^{(*)}-{\cal K}^{(**)}\right)_{t, \boldsymbol{\sigma}, \textbf{a}}\Bigg|(a_k\to a)\right]
        \cdot S^{(B)}_{ab}\cdot  {\cal K}_{t, ( \sigma _k,  \sigma _{k+1})_{a_k,b}}
 \end{align}
 
\end{itemize}
At last,  combining  the identities \eqref{uisiip222}, \eqref{uisiip}, \eqref{nssfku}, \eqref{WI-pr-Mid}, \eqref{gsasawwe} and \eqref{gsasawII}, we obtain  the desired result  \eqref{finK*-K} and prove the Lemma \ref{lem_WI_K} by induction. 
\end{proof}

\bigskip

\subsection{Sum zero property of \texorpdfstring{$\cal K$}{K}}
Recall the tree representation of  $\cal K$ in  Lemma \ref{Lemma_TRofK}. 
The key quantity $\Gamma_{ \textbf{a}}(t, \boldsymbol{\sigma})$ in this representation contains only three types of edges, namely,  
$$
S^{(B)},\quad \Theta^{(B)}_{tm^2}, \quad \Theta^{(B)}_{t|m|^2}
$$
which  commute one another. By explicit computations, we have 
 \begin{align}
     & \left\|\Theta^{(B)}_{t\,m^2}\right\|_{\max} \prec 1, \quad \left\|\Theta^{(B)}_{t\,|m|^2}\right\|_{\max}\prec \frac{1}{\ell_t\eta_t}; \\
   & \left\|\Theta^{(B)}_{tm^2}\right\|_{1} \prec 1, \quad \left\|\Theta^{(B)}_{t|m|^2}\right\|_1\prec \frac{1}{\eta_t}
 \end{align}
Due to the big difference in ranges  between $\Theta^{(B)}_{tm^2}$ and $\Theta^{(B)}_{t|m|^2}$, we  call them short and long edges respectively:
$$
\Theta^{(B)}_{\,t\xi} \; \hbox{or} \; \left(\Theta^{(B)}_{\,t\xi}-1\right)=\begin{cases}
   \;  \hbox{short $\Theta$ edge}, \quad &\xi=m^2, \; \overline m^2
     \\
     \;\hbox{long $\Theta$ edge}, \quad &\xi=|m|^2 
\end{cases}
$$

\begin{definition} 
Fix two sequences \(\boldsymbol{\sigma} = (\sigma_1, \sigma_2, \cdots, \sigma_n)\) and \(\textbf{a} = (a_1, a_2, \cdots, a_n)\). For a partition \(\Gamma_{\textbf{a}} \in TSP(\mathcal{P}_{\textbf{a}})\), recall that  \({\cal F}(\Gamma_{\textbf{a}})\) \eqref{set_csp} is  the collection of pairs of sub-regions that are non-adjacent but share an internal edge in \(\Gamma_{\textbf{a}}\). There is a one-to-one correspondence between the elements of \({\cal F}(\Gamma_{\textbf{a}})\) and the internal edges in \({\cal E}(\Gamma_{\textbf{a}})\).

Given \(\boldsymbol{\sigma}\), we define \({\cal F}_{\text{long}}(\Gamma_{\textbf{a}}, \boldsymbol{\sigma})\) to be  the subset of \({\cal F}(\Gamma_{\textbf{a}})\) corresponding to long internal edges, i.e., 
\[
{\cal F}_{\text{long}}(\Gamma_{\textbf{a}}, \boldsymbol{\sigma}) := \Big\{\{i, j\} \in {\cal F}(\Gamma_{\textbf{a}}) : \{\sigma_i, \sigma_j\} = \{+, -\}\Big\}.
\]
In other words, for \(\{i, j\} \in {\cal F}_{\text{long}}(\Gamma_{\textbf{a}}, \boldsymbol{\sigma})\), there exists an internal (i.e., purple) edge in \(\Gamma_{\textbf{a}}\) separating \(R_i\) and \(R_j\) with \(\sigma_i \neq \sigma_j\).

By definition, 
\[
{\cal F}_{\text{long}}(\Gamma_{\textbf{a}}, \boldsymbol{\sigma}) \subset {\cal F}(\Gamma_{\textbf{a}}) \subset \mathbb{Z}_n^{\text{off}} := \big\{\{i, j\} : 1 \leq i < j \leq n, \quad |i - j| \neq 1 \mod n\big\}.
\]
For a subset \(\pi\) satisfying \(\pi \subset \mathbb{Z}_n^{\text{off}}\), we denote by 
\[
TSP(\mathcal{P}_{\textbf{a}}, \boldsymbol{\sigma}, \pi) := \big\{\Gamma_{\textbf{a}} \in TSP(\mathcal{P}_{\textbf{a}}) : {\cal F}_{\text{long}}(\Gamma_{\textbf{a}}, \boldsymbol{\sigma}) = \pi\big\}, \quad \pi \subset \mathbb{Z}_n^{\text{off}},
\]
the subset of \(TSP(\mathcal{P}_{\textbf{a}})\) with \(\pi\) as the collection of long internal  edges (which may be empty, i.e., \(\pi = \emptyset\)).
\end{definition}

\noindent 
{\bf Example}: In the case  $\pi=\emptyset$,    $\Gamma_\textbf{a}$ contains  no internal long edges. We will ignore all  short edges
and use a big dot representing  some tree structure of consisting entirely of short edges. 
In previous band papers \cite{yang2021delocalization} and \cite{yang2022delocalization}, we  called this dot a  molecule.

\begin{figure}[ht]
    \centering
    \begin{tikzpicture}[scale=1 ]
\begin{scope}[shift={(-3, 0)}]
        \coordinate (D1) at (0, 1);         

        \coordinate (A1) at (-1.5, 1);       
        \coordinate (A2) at (-0.5, 2);        
        \coordinate (A3) at (.5, 2);        
        \coordinate (A4) at (1.5, 1);        
        \coordinate (A5) at (0, -.5);

        \draw[- , thick, blue] (D1) -- (A1) node[ above ] {$a_1$};
        \draw[- , thick, blue] (D1) -- (A2) node[  above] {$a_2$};
        \draw[- , thick, blue] (D1) -- (A3) node[ above] {$a_3$};
        \draw[- , thick, blue] (D1) -- (A4) node[ above] {$a_4$};
         \draw[- , thick, blue] (D1) -- (A5) node[ left] {$a_5$};

        \fill[thick, purple] (D1) circle (0.3) node {$d_1$};
           \end{scope}

       \begin{scope}[shift={( 3, 0)}]
        \coordinate (D1) at (0, 1);         

        \coordinate (A1) at (-1.5, 1);       
        \coordinate (A2) at (-0.5, 2);        
        \coordinate (A3) at (.5, 2);        
        \coordinate (A4) at (1.5, 1);        
        \coordinate (A5) at (0, -.5);

        \draw[- , thick, blue] (D1) -- (A1) node[ left ] {$a_1$};
        \draw[- , thick, blue] (D1) -- (A2) node[  above] {$a_2$};
        \draw[- , thick, blue] (D1) -- (A3) node[ above] {$a_3$};
        \draw[- , thick, blue] (D1) -- (A4) node[ right] {$a_4$};
         \draw[- , thick, blue] (D1) -- (A5) node[ below] {$a_5$};

         \draw[thick, black] (A1) -- (A2) node[midway, above left, black] {$\sigma_2$};
        \draw[thick, black] (A2) -- (A3) node[midway, below,black] {$\sigma_3$};
        \draw[thick, black] (A3) -- (A4) node[midway, above right, black] {$\sigma_4$};
        \draw[thick, black] (A4) -- (A5) node[midway, right, black] {$\sigma_5$};
        \draw[thick, black] (A5) -- (A1) node[midway, left, black] {$\sigma_1$};

        \fill[thick, purple] (D1) circle (0.3) node {$d_1$};
           \end{scope}
    \end{tikzpicture}
    \caption{Single molecule partition}
    \label{fig:circle_arrows_0}
\end{figure}

\noindent 
{\bf Example}: In the following example in Figure \ref{fig:circle_arrows}, we have 
\begin{equation}\label{exn10}
  \textbf{a}=(a_1,a_2\cdots a_{10}), \quad \pi=\left\{ \{1,5\}, \{5,7\},\{8,1\}\right\}  
\end{equation} 
Then  $\Gamma_{\textbf{a}}$ has the following structure. The big dots are connected via internal long edges. Locally, each sub-tree containing a red dot ${\cal M}$ and the edges connecting with this $\cal M$ matches a tree structure of a single-molecule partition. More precisely, for these 4 local sub-trees, we have  
\begin{align}
 \Gamma^{(top)}\in \;&\;TSP ({\cal P}_{(a_1,a_2,a_3,a_4, c_1)}, (\sigma_1, \sigma_2,\sigma_3,\sigma_4,\sigma_5) , \pi=\emptyset )
 \\\nonumber 
 \Gamma^{(left)}\in \;&\;TSP ({\cal P}_{(a_9,a_{10},c_2,a_8)}, (\sigma_9, \sigma_{10},\sigma_1,\sigma_8) , \pi=\emptyset )
 \\\nonumber  
 \Gamma^{(middle)}\in \;&\;TSP ({\cal P}_{(a_7,c_2,c_1,c_3)}, (\sigma_7, \sigma_8,\sigma_1,\sigma_5) , \pi=\emptyset )
   \\\nonumber   
\Gamma^{(right)}\in \;&\;TSP ({\cal P}_{(a_5,a_6,c_3)}, (\sigma_5, \sigma_6,\sigma_7) , \pi=\emptyset )
\end{align}
Here $c_1$, $c_2$ and $c_3$ are not in the initial tree $\Gamma_{\textbf{a}}$. We only use them to represent the local structure. 
\begin{figure}[ht]
    \centering
    \begin{tikzpicture}[scale=1 ]
\begin{scope}[shift={(-4, 0)}]

        \coordinate (D1) at (0, 1);         
        \coordinate (D2) at (-2, -1);    
        \coordinate (D3) at (0, -1);       
        \coordinate (D4) at (2, -1);     

        \coordinate (A1) at (-1.5, 1);       
        \coordinate (A2) at (-0.5, 2);        
        \coordinate (A3) at (.5, 2);        
        \coordinate (A4) at (1.5, 1);        

        \coordinate (A8) at (-2 , -2);  
        \coordinate (A9) at (-3, -1);    
        \coordinate (A10) at (-2, 0);   

        \coordinate (A7) at (0, -2);       

        \coordinate (A5) at (2.5, 0);   
        \coordinate (A6) at (2.5, -2);     

        \draw[- , thick, blue] (D1) -- (A1) node[ above ] {$a_1$};
        \draw[- , thick, blue] (D1) -- (A2) node[  above] {$a_2$};
        \draw[- , thick, blue] (D1) -- (A3) node[ above] {$a_3$};
        \draw[- , thick, blue] (D1) -- (A4) node[ above] {$a_4$};

        \draw[- , thick, blue] (D2) -- (A8) node[below] {$a_8$};
        \draw[- , thick, blue] (D2) -- (A9) node[ above] {$a_9$};
        \draw[- , thick, blue] (D2) -- (A10) node[ above] {$a_{10}$};

        \draw[- , thick, blue] (D3) -- (A7) node[below] {$a_7$};

        \draw[- , thick, blue] (D4) -- (A5) node[ above] {$a_5$};
        \draw[- , thick, blue] (D4) -- (A6) node[below] {$a_6$};

        \draw[- , ultra thick, purple] (D1) -- (D3);
        \draw[- , ultra thick, purple] (D2) -- (D3);
        \draw[- , ultra thick, purple] (D3) -- (D4);

        \fill[thick, purple] (D1) circle (0.3) node {$d_1$};
        \fill[thick, purple] (D2) circle (0.3) node {$d_2$};
        \fill[thick, purple] (D3) circle (0.3) node {$d_3$};
        \fill[thick, purple] (D4) circle (0.3) node {$d_4$};
        \end{scope}

        \begin{scope}[shift={(4, 0)}]
        \coordinate (D1) at (0, 1);         
        \coordinate (D2) at (-2, -1);    
        \coordinate (D3) at (0, -1);       
        \coordinate (D4) at (2, -1);     

        \coordinate (C1) at (0, 0);         
        \coordinate (C2) at (-1, -1);    
        \coordinate (C3) at (1, -1);       

        \coordinate (A1) at (-1.5, 1);       
        \coordinate (A2) at (-0.5, 2);        
        \coordinate (A3) at (.5, 2);        
        \coordinate (A4) at (1.5, 1);        

        \coordinate (A8) at (-2 , -2);  
        \coordinate (A9) at (-3, -1);    
        \coordinate (A10) at (-2, 0);   

        \coordinate (A7) at (0, -2);       

        \coordinate (A5) at (2.5, 0);   
        \coordinate (A6) at (2.5, -2);     

        \draw[- , thick, black] (C1)--(A1) -- (A2)--(A3) -- (A4)--(C1) node[ left ] {$c_1$};
        
        \draw[- , thick, black] (C3) -- (A5) -- (A6)--(C3) node[ above] {$c_3$} ;
        
        \draw[- , thick, black]  (C2)--(C1) -- (C3)--(A7) -- (C2)node[ above] {$c_2$} ;
        
        \draw[- , thick, black] (A9) -- (A10)--(C2) -- (A8)--(A9);

        \draw[- , thick, blue] (D1) -- (A1) node[ above ] {$a_1$};
        \draw[- , thick, blue] (D1) -- (A2) node[  above] {$a_2$};
        \draw[- , thick, blue] (D1) -- (A3) node[ above] {$a_3$};
        \draw[- , thick, blue] (D1) -- (A4) node[ above] {$a_4$};

        \draw[- , thick, blue] (D2) -- (A8) node[below] {$a_8$};
        \draw[- , thick, blue] (D2) -- (A9) node[ above] {$a_9$};
        \draw[- , thick, blue] (D2) -- (A10) node[ above] {$a_{10}$};

        \draw[- , thick, blue] (D3) -- (A7) node[below] {$a_7$};

        \draw[- , thick, blue] (D4) -- (A5) node[ above] {$a_5$};
        \draw[- , thick, blue] (D4) -- (A6) node[below] {$a_6$};

        \draw[- , ultra thick, purple] (D1) -- (D3);
        \draw[- , ultra thick, purple] (D2) -- (D3);
        \draw[- , ultra thick, purple] (D3) -- (D4);

        \fill[thick, purple] (D1) circle (0.3) node {$d_1$};
        \fill[thick, purple] (D2) circle (0.3) node {$d_2$};
        \fill[thick, purple] (D3) circle (0.3) node {$d_3$};
        \fill[thick, purple] (D4) circle (0.3) node {$d_4$};
        \end{scope}
    \end{tikzpicture}
    \caption{Multiple molecule Tree}
    \label{fig:circle_arrows}
\end{figure}

\begin{definition}[Definition of $\cal K^{(\pi)}$ and $\Sigma^{(\pi)}$]\label{def_Kpi}
Given a subset 
 \begin{equation}\label{defpiij}
  \pi \subset \mathbb Z_n^{off} 
 \end{equation}
 define 
 \begin{equation} 
     {\cal K}^{(\pi)}_{t, \boldsymbol{\sigma}, \textbf{a}} = \sum_{\Gamma_{\textbf{a}}\, \in\,  {TSP}({\cal P}_{\textbf{a}},\; \boldsymbol{\sigma},\; \pi )} \Gamma_{ \textbf{a}}(t, \boldsymbol{\sigma}).
   \end{equation} 
   where $\Gamma_{ \textbf{a}}(t, \boldsymbol{\sigma})$ was defined in Definition \ref{def33}.
Clearly, 
\begin{equation}\label{KKpi}
  {\cal K}_{t, \boldsymbol{\sigma}, \textbf{a}} =W^{-n+1}\cdot m_{\boldsymbol{\sigma}}\cdot \sum_{\pi}{\cal K}_{t, \boldsymbol{\sigma}, \textbf{a}} ^{(\pi)},\quad m_{\boldsymbol{\sigma}}=\prod_i m_i. 
\end{equation}
Notice that there is a factor $m_{\boldsymbol{\sigma}}W^{-n+1}$ in the last equation due to our convention that $\cal K^{(\pi)}$ is independent of $W$. 
 Next, we define the self energy $\Sigma^{(\pi)}$  of ${\cal K}^{(\pi)}$.   To this end, we relabel vertices by  introducing 
  $d_i$ as the ending vertex of the boundary edge starting from  $a_i$.  When two edges ending at $d_i, d_j$ join, we identify them by adding a delta function.  We then label all other internal vertices by $s_1$, $s_2, \ldots$.  
   Define the self energy by removing from $ {\cal K}^{(\pi)}_{t, \boldsymbol{\sigma}, \textbf{a}}$ the boundary edges and then summing all $s$ indices.  
Clearly, we have 
\begin{equation} 
\Sigma^{(\pi)}(t, \boldsymbol{\sigma}, \textbf{d}): 
\quad  
{\cal K}^{(\pi)}_{t, \boldsymbol{\sigma}, \textbf{a}}
=
\sum_{\textbf{d}}\left(\Sigma^{(\pi)}(t, \boldsymbol{\sigma}, \textbf{d})\right)
\cdot 
\prod_{i=1}^n \left(\Theta^{(B)}_{tm_im_{i+1}}\right)_{a_i, d_i}   
\end{equation} 
\end{definition}

\noindent 
{\bf Example}: For Figure \ref{fig_n=4} with $n=4$ and 
$\boldsymbol{\sigma}=(+,-,+,-) $, 
$$
\Sigma^{(\emptyset)}(t, \boldsymbol{\sigma}, \textbf{d}) = \delta_{d_1 d_2 d_3 d_4}
+ \delta_{d_1 d_2} \delta_{d_3 d_4} \left(\Theta^{(B)}_{t m^2} - 1\right)_{d_1 d_3}
+ \delta_{d_1 d_4} \delta_{d_2 d_3} \left(\Theta^{(B)}_{t \bar m^2} - 1\right)_{d_1 d_2}.
$$
On the other hand, if  $\boldsymbol{\sigma}=(+,-,+,-)$  then 
$\Sigma^{(\pi)}(t, \boldsymbol{\sigma}, \textbf{d})=0$ if $\pi\ne \emptyset$.

By \eqref{prop:ThfadC} and  explicit calculations,   we have 
$\Sigma^{(\emptyset)}(t, \boldsymbol{\sigma}, \textbf{d}) =O(1) $
and  that $\Sigma^{(\emptyset)}(t, \boldsymbol{\sigma}, \textbf{d})$ is short-ranged in the sense that 
\begin{align}\label{res_SIGempdd}
  \Sigma^{(\emptyset)}(t, \boldsymbol{\sigma}, \textbf{d})\le C_n\exp\left(-c_n \max_{ij}\|d_i-d_j\|\right)   .
\end{align}
In addition,  we have the following sum zero property :
\[
\sum_{d_1, d_2, d_3} \Sigma^{(\emptyset)}(t, \boldsymbol{\sigma}, \textbf{d}) = O(1-t) = O(\eta_t).
\]  
(The name sum zero comes from the fact that the above quantity equals to $0$ when $t=1$) 
Due to the translation invariance, the last  bound  is equivalent to 
\[
L^{-1} \cdot \sum_{d_1, d_2, d_3, d_4} \Sigma^{(\emptyset)}(t, \boldsymbol{\sigma}, \textbf{d}) = O(1-t) = O(\eta_t).
\] 
It turns out that this property holds  for all alternating  $\boldsymbol{\sigma}$ with  $\pi=\emptyset$.  

\begin{lemma}[Sum zero]\label{lem:SZ}
For fixed $4 \le n\in 2\mathbb Z$ and an alternating  loop $\boldsymbol{\sigma}^{(alt)}$ with 
\begin{equation}\label{assum_SZ_SinMole}
1\le k\le n, \quad \quad {\sigma}^{(alt)}_k =\begin{cases}
    +, \quad k\in 2\mathbb Z-1
    \\\nonumber
    -, \quad k\in 2\mathbb Z 
\end{cases}\; \;,\quad t\in[0,1]  ,  
\end{equation}
the single molecule tree graphs (i.e., $\pi=\emptyset$) have the following sum zero property  
\begin{equation}\label{res_SZ_SinMole}
L^{-1} \cdot\sum_{\textbf{d}\,\in \,\mathbb Z^{n}_L } \Sigma^{(\emptyset)}(t, \boldsymbol{\sigma}^{(alt)}, \textbf{d}) = O(1-t)=O(\eta_t)
\end{equation}
\end{lemma}
This sum-zero property is the key input for the  following estimates on $\cal K^{(\pi)}$ and $\cal K$. 
\begin{lemma}[Bound on ${\cal K}$]\label{ML:Kbound+pi}
For any ${\cal K}^{(\pi)}_{t, \boldsymbol{\sigma}, \textbf{a}}$ defined in Definition \ref{def_Kpi},   we have 
\begin{equation}\label{eq:bcal_k_pi}
  {\cal K}^{(\pi)}_{t, \boldsymbol{\sigma}, \textbf{a}} =O_\prec \left( \ell_t \cdot \eta_t \right)^{-n+1}    
\end{equation}
Together with \eqref{KKpi}, 
we have  
\begin{equation}\label{eq:bcal_k_2}
  {\cal K}_{t, \boldsymbol{\sigma}, \textbf{a}} =O_\prec\left(W \ell_t \cdot \eta_t \right)^{-n+1}    .
\end{equation}
Notice that there is no $W$ factor in ${\cal K}^{(\pi)}$ due to its definition.
\end{lemma}

\subsection{Proof of Lemma  \ref{lem:SZ}}

\begin{proof}[ Proof of Lemma  \ref{lem:SZ}]
By   definitions of ${\cal K}^{(\pi)}$ and $\Sigma^{(\pi)}$, for any $n\ge 3$ we have that 
\begin{equation} 
 \sum_{\textbf{a}\,\in \,\mathbb Z^{n}_L } 
{\cal K}^{(\pi)}(t, \boldsymbol{\sigma}, \textbf{a})
= \sum_{\textbf{a}\,\in \,\mathbb Z^{n}_L }  \prod_{i=1}^n \left(\Theta^{(B)}_{tm_im_{i+1}}\right)_{a_i,d_i}\cdot\sum_{\textbf{d}\,\in \,\mathbb Z^{n}_L } \Sigma^{(\pi)}(t, \boldsymbol{\sigma}, \textbf{d})    
\end{equation}
Use $S^{(B)} \bf 1 = \bf 1$,  we have  
$$
\sum_{a_i}\left(\Theta^{(B)}_{tm_im_{i+1}}\right)_{a_i,d_i}=\frac{1}{1-tm_im_{i+1}}.
$$
Therefore, we have 
\begin{equation}\label{KSIGetat}
 \sum_{\textbf{a}\,\in \,\mathbb Z^{n}_L } 
{\cal K}^{(\pi)}(t, \boldsymbol{\sigma}, \textbf{a})
 \Bigg/\sum_{\textbf{d}\,\in \,\mathbb Z^{n}_L } \Sigma^{(\pi)}(t, \boldsymbol{\sigma}, \textbf{d})    \sim \eta_t^{-\big|\{\;i: \;\sigma_i\ne \sigma_{i+1}, \; i\in \mathbb Z_L\}\big|}  .
\end{equation}
By\eqref{KKpi} and   Corollary \ref{lem_sumAinK}, we have  
\begin{align}\label{LKsigm0}
 L^{-1}\cdot \sum_{\pi}\sum_{\textbf{a}\,\in \,\mathbb Z^{n}_L } 
{\cal K}^{(\pi)}(t, \boldsymbol{\sigma}, \textbf{a})= O(\eta_t)^{-n+1}.
\end{align}
We now  show that it holds without sum over $\pi$, i.e.,    for  any fixed $\boldsymbol{\sigma}$, $n\ge 3$ and   $\pi$   in \eqref{defpiij}, 
\begin{align}\label{LKsigmas}
 L^{-1}\cdot \sum_{\textbf{a}\,\in \,\mathbb Z^{n}_L } 
{\cal K}^{(\pi)}(t, \boldsymbol{\sigma}, \textbf{a})=O (\eta_t)^{-n+1}  
\end{align}
Assuming that  the last equation  holds in the case that
$\pi=\emptyset$ and $\boldsymbol{\sigma}=\boldsymbol{\sigma}^{(alt)}$, 
  together with \eqref{KSIGetat},   we have 
\begin{equation}
   L^{-1} \cdot \sum_{\textbf{d}\,\in \,\mathbb Z^{n}_L } \Sigma^{(\emptyset)}(t, \boldsymbol{\sigma}^{(alt)}, \textbf{d}) =O(\eta_t).
\end{equation}
This  implies the desired result \eqref{res_SZ_SinMole}. 

We now prove \eqref{LKsigmas} by  induction. If $n=3$, then $\pi$ can only be $\emptyset$. Hence \eqref{LKsigm0} implies \eqref{LKsigmas} in the case $n=3$. Next, we assume that \eqref{LKsigmas} holds with \( n \) replaced by \( m < n \). Under this assumption, we first prove that if  $\pi\ne \emptyset$ then \eqref{LKsigmas} holds, i.e., 
\begin{align} 
 \pi\ne\emptyset\implies  L^{-1}\cdot \sum_{\textbf{a}\,\in \,\mathbb Z^{n}_L } 
{\cal K}^{(\pi)}(t, \boldsymbol{\sigma}, \textbf{a})=O (\eta_t)^{-n+1}  
\end{align}
If $\pi\ne \emptyset$, we can always represent 
$ 
{\cal K}^{(\pi)}(t, \boldsymbol{\sigma}, \textbf{a})
$ with the molecule structure and the self-energy  $\Sigma^{(\emptyset)}$. For example, for $n=10$ and $\pi$ in \eqref{exn10}, the molecule structure is the one in Figure \ref{fig:circle_arrows}. Then as shown in the Figure \ref{fig:Decp}, where $d_i$ is the vertex connecting with $a_i$'s (which ere not marked  in the figure) and $c_i$ are the vertices of edges connecting molecules (i.e., big dots) :
$$
{\cal K}^{(\pi)}(t, \boldsymbol{\sigma}, \textbf{a})=
\sum_{\textbf{d}}\sum_{\textbf{c}}\left(\prod_{i=1}^{10}\left(\Theta^{(B)}_{tm_im_{i+1}}\right)_{a_id_i}\right)\cdot \prod_{k=1}^3\Big(\Theta_{{t|m|^2}}-1\Big)_{c_{2k-1}c_{2k}}
\cdot
 \prod_{k=1}^4{\Sigma}^{(\emptyset)}(t, \boldsymbol{\sigma}^{(k)}, \textbf{d}^{(k)})
$$
\begin{figure}[ht]
    \centering
    \begin{tikzpicture}[scale=1 ]
\begin{scope}[shift={(-8, 0)}]

        \coordinate (D1) at (0, 1);         
        \coordinate (D2) at (-2, -1);    
        \coordinate (D3) at (0, -1);       
        \coordinate (D4) at (2, -1);     

        \coordinate (A1) at (-1.5, 1);       
        \coordinate (A2) at (-0.5, 2);        
        \coordinate (A3) at (.5, 2);        
        \coordinate (A4) at (1.5, 1);        

        \coordinate (A8) at (-2 , -2);  
        \coordinate (A9) at (-3, -1);    
        \coordinate (A10) at (-2, 0);   

        \coordinate (A7) at (0, -2);       

        \coordinate (A5) at (2.5, 0);   
        \coordinate (A6) at (2.5, -2);     

        \draw[- , thick, blue] (D1) -- (A1) node[ above ] {$a_1$};
        \draw[- , thick, blue] (D1) -- (A2) node[  above] {$a_2$};
        \draw[- , thick, blue] (D1) -- (A3) node[ above] {$a_3$};
        \draw[- , thick, blue] (D1) -- (A4) node[ above] {$a_4$}
       ; 

        \draw[- , thick, blue] (D2) -- (A8) node[below] {$a_8$};
        \draw[- , thick, blue] (D2) -- (A9) node[ above] {$a_9$};
        \draw[- , thick, blue] (D2) -- (A10) node[ above] {$a_{10}$};

        \draw[- , thick, blue] (D3) -- (A7) node[below] {$a_7$};

        \draw[- , thick, blue] (D4) -- (A5) node[ above] {$a_5$};
        \draw[- , thick, blue] (D4) -- (A6) node[below] {$a_6$};

       \draw[-, ultra thick, purple] (D1) -- (D3)
    node[black, above, xshift=6pt, yshift=5pt] {$c_2$}
    node[black, below left, xshift=-3pt] {$c_5$}
    node[black, below right, xshift=3pt] {$c_4$};
      \draw[-, ultra thick, purple] (D3) -- (D1)
    node[black, below left, xshift=-1pt, yshift=-5pt] {$c_1$};
        \draw[- , ultra thick, purple] (D3) -- (D2)node[black, below right, xshift=2pt, yshift=-2pt] {$c_6$};
        \draw[- , ultra thick, purple] (D3) -- (D4)node[black, below left, xshift=-2pt, yshift=-2pt] {$c_3$};

        \fill[thick, purple] (D1) circle (0.3) node {$d_1$};
        \fill[thick, purple] (D2) circle (0.3) node {$d_2$};
        \fill[thick, purple] (D3) circle (0.3) node {$d_3$};
        \fill[thick, purple] (D4) circle (0.3) node {$d_4$};
        \end{scope}
      \begin{scope}[shift={(0, 0)}]

        \coordinate (D1) at (-1, 1.5);         
        \coordinate (D2) at (-2, -1);    
        \coordinate (D3) at (0, -1);       
        \coordinate (D4) at (2, -1); 
          \coordinate (D5) at (0, 1);

        \coordinate (A1) at (-2.5, 1.5);       
        \coordinate (A2) at (-1.5, 2.5);        
        \coordinate (A3) at (-.5, 2.5);        
        \coordinate (A4) at ( .5, 1.5);        

        \coordinate (A8) at (-2 , -2);  
        \coordinate (A9) at (-3, -1);    
        \coordinate (A10) at (-2, 0);   

        \coordinate (A7) at (0, -2);       

        \coordinate (A5) at (2.5, 0);   
        \coordinate (A6) at (2.5, -2);     

        \draw[- , thick, blue] (D1) -- (A1) node[ above ] {$a_1$};
        \draw[- , thick, blue] (D1) -- (A2) node[  above] {$a_2$};
        \draw[- , thick, blue] (D1) -- (A3) node[ above] {$a_3$};
        \draw[- , thick, blue] (D1) -- (A4) node[ above] {$a_4$}
       ;  \draw[- , thick, blue] (A4) -- (D1)
    node[black, below left, xshift=-1pt, yshift=-5pt] {$c_1$};

        \draw[- , thick, blue] (D2) -- (A8) node[below] {$a_8$};
        \draw[- , thick, blue] (D2) -- (A9) node[ above] {$a_9$};
        \draw[- , thick, blue] (D2) -- (A10) node[ above] {$a_{10}$};

        \draw[- , thick, blue] (D3) -- (A7) node[below] {$a_7$};

        \draw[- , thick, blue] (D4) -- (A5) node[ above] {$a_5$};
        \draw[- , thick, blue] (D4) -- (A6) node[below] {$a_6$};

        \draw[-, ultra thick, purple] (D3) -- (D5)
    node[black, below left] {$c_1$};
       \draw[-, ultra thick, purple] (D5) -- (D3)
    node[black, above, xshift=6pt, yshift=5pt] {$c_2$}
    node[black, below left, xshift=-3pt] {$c_5$}
    node[black, below right, xshift=3pt] {$c_4$};
        \draw[- , ultra thick, purple] (D3) -- (D2)node[black, below right, xshift=2pt, yshift=-2pt] {$c_6$};
        \draw[- , ultra thick, purple] (D3) -- (D4)node[black, below left, xshift=-2pt, yshift=-2pt] {$c_3$};

        \fill[thick, purple] (D1) circle (0.3) node {$d_1$};
        \fill[thick, purple] (D2) circle (0.3) node {$d_2$};
        \fill[thick, purple] (D3) circle (0.3) node {$d_3$};
        \fill[thick, purple] (D4) circle (0.3) node {$d_4$};
        \end{scope}
    \end{tikzpicture}
    \caption{Decomposition of $\cal K^{(\pi)}$}
    \label{fig:Decp}
\end{figure}
Here  ${\Sigma}^{(\emptyset)}(t, \boldsymbol{\sigma}^{(k)}, \textbf{d}^{(k)})$, $1\le k\le 4$, 
represent the four self-energies   (i.e., four big dots, top, bottom, left, right). More precisely,  
$$
\textbf{d}^{(top)}=(c_1,d_1,d_2,d_3,d_4),\quad
\textbf{d}^{(left)}=(c_6,d_8,d_9,d_{10}),\quad
\textbf{d}^{(bottom)}=(c_2,c_4,d_7,c_5),\quad
\textbf{d}^{(right)}=(c_3,d_5,d_6),\quad
$$
and   $\cal K^{(\pi)}$ is decomposed  into three parts: 
\begin{itemize}
    \item The edges connect with external vertices, i.e., (blue) boundary edges.  
    \item The edges connect two different molecules, which is always $\left(\Theta_{{t|m|^2}}-1\right)$
    \item The cores ${\Sigma}^{(\emptyset)}(t, \boldsymbol{\sigma}^{(k)}, \textbf{d}^{(k)})$ for each single molecule. 
\end{itemize}

One can easily extend it to the general cases. Given a set 
$\pi$ (which can be the empty set), we label all vertices of long internal edges for $\pi$ by $c_i$, $(1\le i\le 2|\pi|)$. Recall that the vertex connecting with the 
boundary vertex $a_i$ is denoted by  $d_i$. We now denote all internal vertices other than $d_i, c_j$ by $s_k$. 
The indices $k$ is a finite set less than $n$, but we will not specify it. 
We now explain how to construct the molecule and their tree structure. Given $ \{i, j \} \in \pi$, we draw a line in the polygon from the center of edge $i$ to 
that of $j$.  In this way, we have a partition of the polygon. The set $\pi$ for which there is $\Gamma_{\textbf{a}}$ such that $ {\cal F}_{long}(\Gamma_{\textbf{a}}, \sigma)=\pi $ satisfies that these lines representing the pairing are non-crossing. From now on, we will call  $\pi$ non-crossing pairing. 
Given a non-crossing pairing, we divide the polygon into several regions, say, $M$ regions (note $M=|\pi|+1$. We represent each region by a big dot (molecule), and there is an edge connecting two dots if and only if these two regions are neighboring.    Notice that each dot typically has many vertices connecting to it, as shown in Figure \ref{fig:clockwise-decagon}. 
\usetikzlibrary{calc}
\begin{figure}[ht]
    \centering 
  \begin{tikzpicture}[scale=0.8]

   \begin{scope} 
  
    \def\radius{2.5}
    \def\labeldist{0.3} 

    \foreach \i [evaluate=\i as \angle using {144 - (\i-1)*36}] in {1,...,10} {
        \node (A\i) at ({\angle}:\radius) {};
        \node at ({\angle}:{\radius + \labeldist}) {\(a_{\i}\)};
   
     \fill[black] ({\angle}:\radius) circle[radius=0.06];
     }

    \foreach \i [count=\j from 2] in {1,...,9} {
        \draw [thick] (A\i) -- (A\j);
    }
    \draw [thick] (A10) -- (A1); 

    \usetikzlibrary{calc}
    \coordinate (M1) at ($ (A1)!0.5!(A10) $);
    \coordinate (M2) at ($ (A4)!0.5!(A5) $);
    \coordinate (M3) at ($ (A6)!0.5!(A7) $);
    \coordinate (M4) at ($ (A7)!0.5!(A8) $);

    \coordinate (N1) at (0.4,-0.4);
    \coordinate (N2) at (-1,-1);
    \coordinate (N3) at (1.7,-0.6);
    \coordinate (N4) at (0,1.6);

    \draw[black, thick, bend left=10] (M1) to (M2);
    \draw[black, thick, bend right=60] (M2) to (M3);
    \draw[black, thick, bend right=45] (M4) to (M1);

    \fill[purple] (N1) circle[radius=0.4];
    \fill[purple] (N2) circle[radius=0.4];
    \fill[purple] (N3) circle[radius=0.4];
    \fill[purple] (N4) circle[radius=0.4];
     
   \end{scope} 

     \begin{scope}[xshift=8cm]
  
    \def\radius{2.5}
    \def\labeldist{0.3} 

    \foreach \i [evaluate=\i as \angle using {144 - (\i-1)*36}] in {1,...,10} {
        \node (A\i) at ({\angle}:\radius) {};
        \node at ({\angle}:{\radius + \labeldist}) {\(a_{\i}\)};
     \fill[black] ({\angle}:\radius) circle[radius=0.06];}

    \foreach \i [count=\j from 2] in {1,...,9} {
        \draw [thick] (A\i) -- (A\j);
    }
    \draw [thick] (A10) -- (A1); 

    \coordinate (M1) at ($ (A1)!0.5!(A10) $);
    \coordinate (M2) at ($ (A4)!0.5!(A5) $);
    \coordinate (M3) at ($ (A6)!0.5!(A7) $);
    \coordinate (M4) at ($ (A7)!0.5!(A8) $);

    \coordinate (N1) at (0.4,-0.4);
    \coordinate (N2) at (-1,-1);
    \coordinate (N3) at (1.7,-0.6);
    \coordinate (N4) at (0,1.6);

    \draw[black, thick, bend left=10] (M1) to (M2);
    \draw[black, thick, bend right=60] (M2) to (M3);
    \draw[black, thick, bend right=45] (M4) to (M1);

 \draw[red, thick ] (N1) to (N2);
 \draw[red, thick] (N4) to (N1);
 \draw[red, thick ] (N3) to (N1);

 \draw[blue, thick ] (N1) to (A7);
\draw[blue, thick ] (N2) to (A8);
\draw[blue, thick ] (N2) to (A9);
\draw[blue, thick ] (N2) to (A10);
 \draw[blue, thick ] (N3) to (A5);
  \draw[blue, thick ] (N3) to (A6);
     \draw[blue, thick ] (N4) to (A1);
     \draw[blue, thick ] (N4) to (A2);
     \draw[blue, thick ] (N4) to (A3);
     \draw[blue, thick ] (N4) to (A4);

   \fill[purple] (N1) circle[radius=0.4];
    \fill[purple] (N2) circle[radius=0.4];
    \fill[purple] (N3) circle[radius=0.4];
    \fill[purple] (N4) circle[radius=0.4];
       
   \end{scope} 
    \end{tikzpicture}
    \caption{$\pi=\left\{ \{1,5\}, \{5,7\},\{8,1\}\right\} $ }
     \label{fig:clockwise-decagon}
\end{figure}

In general, there are complicated structures inside these molecules; there are short edges and other vertices labeled by $s_k$. All vertices labeled by $s_k$
 are required to be summed. 
With this convention,  for $\pi$  with   $M$ molecules, we can write 
\begin{align}\label{KThThSigma}
{\cal K}^{(\pi)}(t, \boldsymbol{\sigma}, \textbf{a})= 
\sum_{\textbf{d}}\sum_{\textbf{c}}\left(\prod_{i=1}^{n} \left(\Theta^{(B)}_{tm_im_{i+1}}\right)_{a_id_i}\right)\cdot \prod_{k=1}^{M-1}\Big(\Theta_{{t|m|^2}}-1\Big)_{c_{2k-1}c_{2k}}
\cdot
 \prod_{k=1}^{M}{\Sigma}^{(\emptyset)}(t, \boldsymbol{\sigma}^{(k)}, \textbf{d}^{(k)})   
\end{align}
Here the vertices labelled by $s_k$ are summed and thus they no longer appear explicitly in the formula above. 
Summing over  $\textbf{a}$,  we have 
\begin{align}\label{sumkpitsa}
\sum_{\textbf{a}}{\cal K}^{(\pi)}(t, \boldsymbol{\sigma}, \textbf{a})=
\sum_{\textbf{a},\,\textbf{d},\,\textbf{c}}\left(\prod_{i=1}^{n}  \left(\Theta^{(B)}_{tm_im_{i+1}}\right)_{a_id_i}\right)\cdot \prod_{k=1}^{M-1}\Big(\Theta_{{t|m|^2}}-1\Big)_{c_{2k-1}c_{2k}}
\cdot
 \prod_{k=1}^{M}{\Sigma}^{(\emptyset)}(t, \boldsymbol{\sigma}^{(k)}, \textbf{d}^{(k)})   
\end{align}
 Given a molecule structure (or equivalently a set $\pi$ representing non-crossing pairings), there must be a molecule  containing just one $c$ vertex, i.e., the big dot for this molecule  connects to only  one  $(\Theta^{(B)}_{t|m^2|-1})  _{c_{2k-1},c_{2k}}$ edge. (For example, the top, left and right molecules in Figure \ref{fig:Decp}). 
  In a different language, this molecule represents a region with exactly  one  paring line.  
For simplicity, we assume that it is the first molecule  containing  $c_1$ and  connecting  with $a_1,a_2\cdots a_{m-1}$. 
With these notations, we have 
\begin{equation} \label{jjshsaa0}
\{1,m \}\in \pi, \quad \sigma_1\ne \sigma_{m }
\end{equation}
The expression in the formula of ${\cal K}^{(\pi)}$ related to this molecule is 
$$
\sum_{d_1,\cdots, d_{m-1}}\Sigma^{(\emptyset)}\Bigg(t, (\sigma_1, \cdots \sigma_m), (d_1,\cdots, d_{m-1}, c_1)\Bigg)\cdot  \left(\prod_{i=1}^{m-1} \left(\Theta^{(B)}_{tm_im_{i+1}}\right)_{a_id_i}\right) 
$$
(For example: the top right part in Figure \ref{fig:Decp} is for the case $m=5$.) Notice that  $d_i, a_i, 1\le i\le m-1 $ do not  appear in other molecules.  The following part is separated  from  other parts in \eqref{sumkpitsa}, i.e., 
 $$
f^*(c_1):=\sum_{a_1,\cdots, a_{m-1}}\sum_{d_1,\cdots, d_{m-1}}\Sigma^{(\emptyset)}\Bigg(t, (\sigma_1, \cdots \sigma_m), (d_1,\cdots, d_{m-1}, c_1)\Bigg)\cdot  \left(\prod_{i=1}^{m-1}\left(\Theta^{(B)}_{tm_im_{i+1}}\right)_{a_id_i}\right) 
$$
Now we can write the $\cal K^{(\pi)}$ in \eqref{sumkpitsa} in the terms of $f^*(c_1)$ as follows: 
\begin{align}\label{sumkp2itsa2}
\sum_{\textbf{a}}{\cal K}^{(\pi)}(t, \boldsymbol{\sigma}, \textbf{a})=& \sum_{c_1} f^*(c_1)\cdot  
\sum_{a_m,\cdots , a_n}
\sum_{d_m,\cdots , d_n}
\sum_{c_2,\cdots,c_{2M-2}}\\\nonumber
&\left(\prod_{i=m}^{n}\left(\Theta^{(B)}_{tm_im_{i+1}}\right)_{a_id_i}\right)\cdot \prod_{k=1}^{M-1}\Big(\Theta_{{t|m|^2}}-1\Big)_{c_{2k-1}c_{2k}}
\cdot
 \prod_{k=2}^{M}{\Sigma}^{(\emptyset)}(t, \boldsymbol{\sigma}^{(k)}, \textbf{d}^{(k)})   
\end{align}
In the Figure \ref{fig:Decp}, the upper right part represents the $f ^*(c_1)$ and the lower right part presents the 2nd line of the \eqref{sumkp2itsa2}. 

Due to the translation invariant, $ f^*(c_1)$ does not depend on $c_1$, i.e.,  
$f^*(c_1)= f^*(1)\in \mathbb C$.
Inserting it  back to \eqref{sumkp2itsa2}, we obtain that
\begin{align}\label{sumkp3itsa3}
\sum_{\textbf{a}}{\cal K}^{(\pi)}(t, \boldsymbol{\sigma}, \textbf{a})=f^*(1) &\cdot  
\sum_{a_m,\cdots , a_n}\;
\sum_{d_m,\cdots , d_n}\;
\sum_{\textbf{c}}\\\nonumber
&\left(\prod_{i=m}^{n}\left(\Theta^{(B)}_{tm_im_{i+1}}\right)_{a_id_i}\right)\cdot \prod_{k=1}^{M-1}\Big(\Theta_{{t|m|^2}}-1\Big)_{c_{2k-1}c_{2k}}
\cdot
 \prod_{k=2}^{M}{\Sigma}^{(\emptyset)}(t, \boldsymbol{\sigma}^{(k)}, \textbf{d}^{(k)})   
\end{align}
Since   $c_1$  appears only  in the internal edge $ (\Theta_{{t|m|^2}}-1 )_{c_1c_2}$, we can sum over $c_1$. By definition of 
$\Theta_{{t|m|^2}}$,  $S^{(B)} \bf 1 = \bf 1$ and $|m|=1$, we have 
$$  \sum_{a}\left(\Theta_{t|m|^2}-1\right)_{ab}= \frac t { 1-t} =t \sum_{a}\left(\Theta_{t|m|^2}\right)_{ab}
$$

Thus  we can replace $(\Theta^{(B)}_{t|m|^2}-1)_{c_1c_2}$ in \eqref{sumkp3itsa3} with $t(\Theta^{(B)}_{t|m|^2})_{c_1c_2}$.  
This replacement shows that after  summing  over $c_1$, an internal edge $c_1-c_2$ edge becomes an external edge with a factor $t$. This will be crucial later on when we split the graph.  We can now rewrite 
\begin{align}\label{sumkp4itsa4}
\sum_{\textbf{a}}{\cal K}^{(\pi)}(t, \boldsymbol{\sigma}, \textbf{a}) \;=\;&t\cdot f^*(1) \cdot  
\sum_{a_m,\cdots , a_n}\;
\sum_{d_m,\cdots , d_n}\;
\sum_{\textbf{c}}\\\nonumber
&\left(\prod_{i=m}^{n}\left(\Theta^{(B)}_{tm_im_{i+1}}\right)_{a_id_i}\right)\cdot 
\left(\Theta^{(B)}_{t|m|^2}\right)_{c_1c_2}
\prod_{k=2}^{M-1}\Big(\Theta_{{t|m|^2}}-1\Big)_{c_{2k-1}c_{2k}}
\cdot
 \prod_{k=2}^{M}{\Sigma}^{(\emptyset)}(t, \boldsymbol{\sigma}^{(k)}, \textbf{d}^{(k)})   
\end{align}
  Define $\pi'$ to be $\pi$ with the pair for edge $\{c_1, c_2\}$ removed, i.e. 
  $$
  \pi'=\pi\setminus\{\{1,m\}\}
  $$
  Denote 
$$
\boldsymbol{\sigma}'=(\sigma_1,  \,\sigma_{m}, \,\sigma_m, \,\cdots  \,\sigma_{n}) ,\quad \textbf{a}'=(c_1, \,a_{m},\,a_{m+1},\,\cdots\, a_{n})
$$
Then we have (see  Figure \ref{fig:Decp} for an example) 
\begin{align}\label{sumkp5itsa5}
\sum_{\textbf{a}}{\cal K}^{(\pi)}(t, \boldsymbol{\sigma}, \textbf{a})=t\cdot f^*(1) &\cdot  
\sum_{\textbf{a}'}{\cal K}^{(\pi')}(t, \boldsymbol{\sigma}', \textbf{a}') 
\end{align}
\nc 
Notice that   $\{c_1, c_2\}$  is  a boundary  edge in  
$\mathcal{K}^{\left(\pi^{\prime}\right)}\left(t, \boldsymbol{\sigma}^{\prime}, \mathbf{a}^{\prime}\right)$ and it needs to be of the form $\left(\Theta_{t|m|^2}^{(B)}\right)_{c_1 c_2}$.
Since $\boldsymbol{a}'\in \mathbb Z_L^{n-m+2}$, we can apply   induction  assumption  \eqref{LKsigmas} to  $\cal K^{(\pi)'}$. Thus 
\begin{align}\label{sumkp6itsa6}
\sum_{\textbf{a}}{\cal K}^{(\pi)}(t, \boldsymbol{\sigma}, \textbf{a})= {O(t\cdot f^*(1) \cdot L\cdot (\eta_t)^{-n+m-1})}
\end{align}

Multiplying  $(\Theta^{(B) }_t|m|^2)_{c_1,a}=(\Theta^{(B) }_tm_1m_m)_{c_1,a}$ to   $f^*(c_1)$ and summing up $c_1$ and $a$,  we obtain 
\begin{align}\label{sumkp7itsa7}
   & \sum_{a}\sum_{c_1}(\Theta^{(B) }_{tm_1m_m})_{c_1 ,a}\cdot f^*(c_1)
    \\\nonumber
     =& \sum_{a,\,c_1}\;\sum_{a_1,\cdots, a_{m-1}}\sum_{d_1,\cdots, d_{m-1}}\Sigma^{(\emptyset)}\Bigg(t, (\sigma_1, \cdots \sigma_m), (d_1,\cdots, d_{m-1}, c_1)\Bigg)\cdot  \left(\prod_{i=1}^{m-1}\left(\Theta^{(B)}_{tm_im_{i+1}}\right)_{a_id_i}\right)(\Theta^{(B) }_{tm_1m_m})_{c_1 ,a} 
\end{align}
By definition, the right hand side  can be written in terms of  ${\cal K}^{(\emptyset)} \big(t, (\sigma_1, \cdots \sigma_m), (a_1,\cdots, a_{m-1}, a)\big)$, namely, 
\begin{align}\label{jyya}
 \sum_a\sum_{c_1}f^*(c_1)(\Theta^{(B) }_t)_{c_1,a}&=\sum_{a_1,\cdots, a_{m-1},\;a}{\cal K}^{(\emptyset)} \Bigg(t, (\sigma_1, \cdots \sigma_m), (a_1,\cdots, a_{m-1}, a)\Bigg) .
\end{align}
By  inductive assumption on \eqref{LKsigmas},  the right hand side  of \eqref{jyya} is equal to $O(L\cdot\eta_t^{-m+1})$. Thus 
\begin{align}\label{jyasdfya}
\sum_a\sum_{c_1}f^*(c_1)(\Theta^{(B) }_t)_{c_1,a} = O(L\cdot\eta_t^{-m+1})
\end{align}
On the other hand, since $f^*(c_1)=f^*(1)$, we have the identity 
\begin{align}
\sum_a\sum_{c_1}f^*(c_1)(\Theta^{(B) }_t)_{c_1,a}=L\cdot f^*(c_1)\cdot (1-t)^{-1}.
\end{align}
Hence
 $$
  f^*(c_1)=f^*(1)=O(\eta_t^{-m+2}) $$
Together with \eqref{sumkp6itsa6}, we have proved   \eqref{LKsigmas}  if  $\pi\ne \emptyset$.

For  $\pi=\emptyset$, we write ${\cal K}^{(\emptyset)}$  as
\begin{align}\label{LKsigmasES}
 \sum_{\textbf{a}\,\in \,\mathbb Z^{n}_L } 
{\cal K}^{(\emptyset)}(t, \boldsymbol{\sigma}, \textbf{a})
=
 \sum_{\pi}\sum_{\textbf{a}\,\in \,\mathbb Z^{n}_L } 
{\cal K}^{(\pi)}(t, \boldsymbol{\sigma}, \textbf{a})
-
\sum_{\pi\ne\emptyset}  \sum_{\textbf{a}\,\in \,\mathbb Z^{n}_L } 
{\cal K}^{(\pi)}(t, \boldsymbol{\sigma}, \textbf{a}) 
=
  O(L\cdot \eta_{t}^{-n+1}) 
\end{align}
Here  we have used \eqref{LKsigm0} to bound the first term on the right hand side and  \eqref{LKsigmas} for the second term. 
We  have thus proved  \eqref{LKsigmas} and Lemma \ref{lem:SZ}. 
 
\end{proof} 

\subsection{Proof of Lemma   \ref{ML:Kbound+pi}}

\begin{proof}[Proof of Lemma  \ref{ML:Kbound+pi}]
 We first focus the case $\pi=\emptyset$. By symmetry, without loss of generality, we can split it into three cases
 \begin{enumerate}
     \item Pure loop, i.e., $\sigma_k=+$ for all $1\le k\le n$ or  $\sigma_k=-$ for all $1\le k\le n$. In this case, all edges are short edges, hence we can easily obtain \eqref{eq:bcal_k_pi}. 
     \item $\sigma_1=+$, $\sigma_2=-$, and  there exists $j$ s.t. $\sigma_j=\sigma_{j+1}=+$. 
    \item $\boldsymbol{\sigma}$ is alternative as in \eqref{assum_SZ_SinMole}, i.e. $\boldsymbol{\sigma}=\boldsymbol{\sigma}^{(alt)}$. 
 \end{enumerate}
  Recall  that for $\sigma_1\ne\sigma_2$, and $\pi=\emptyset$, we have 
\begin{align*}
{\cal K}^{(\emptyset)}_{t, \boldsymbol{\sigma}, \textbf{a}}
=\;&
 \sum_{\textbf{d}}\left(\Sigma^{(\emptyset)}(t, \boldsymbol{\sigma}, \textbf{d})\right)
\cdot 
\prod_{i=1}^n \left(\Theta^{(B)}_{tm_im_{i+1}}\right)_{a_i, d_i}
\\\nonumber
=\;&\sum_{d_1}\left(\Theta^{(B)}_{t|m|^2}\right)_{a_1, d_1}
\cdot \sum_{d_2,\,\cdots,\, d_n}\left(\Sigma^{(\emptyset)}(t, \boldsymbol{\sigma}, \textbf{d})\right)
\cdot 
\prod_{i=2}^n \left(\Theta^{(B)}_{tm_im_{i+1}}\right)_{a_i, d_i}  
\end{align*}
We are going to prove the following statement, which is  slightly stronger than \eqref{eq:bcal_k_pi}:
\begin{align}\label{spwow3}
    \sum_{d_2,\,\cdots,\, d_n}\left(\Sigma^{(\emptyset)}(t, \boldsymbol{\sigma}, \textbf{d})\right)
\cdot 
\prod_{i=2}^n \left(\Theta^{(B)}_{tm_im_{i+1}}\right)_{a_i, d_i}
=O(\ell_t\eta_t)^{-n+2}\cdot \left(\min_{2\le k\le n} \|a_k-d_1\|+1\right)^{-1}
\end{align}
First, if there exists $2\le j\le n$ s.t.  $\sigma_j=\sigma_{j+1}=+$, then  \eqref{prop:ThfadC} shows that  
$$
\left(\Theta^{(B)}_{tm_jm_{j+1}}\right)_{a_j, d_j}= \left(\Theta^{(B)}_{tm^2}\right)_{a_j, d_j}=O(1)
$$
It implies that
$$
\prod_{i=2}^n \left(\Theta^{(B)}_{tm_im_{i+1}}\right)_{a_i, d_i}=O(\ell_t\eta_t)^{-n+2}
$$
On the other hand, due to the short range property of $\Sigma^{(\emptyset)}(t, \boldsymbol{\sigma}, \textbf{d})$ (in \eqref{res_SIGempdd}), we have
$$
\sum_{d_2,\,\cdots,\, d_n}\left(\Sigma^{(\pi)}(t, \boldsymbol{\sigma}, \textbf{d})\right)
\cdot 
\prod_{i=2}^n \left(\Theta^{(B)}_{tm_im_{i+1}}\right)_{a_i, d_i}\le C \prod_{i=2}^n \left\|\Theta^{(B)}_{tm_im_{i+1}}\right\|_{\max}
$$
By combining these two bounds, we derive \eqref{spwow3} for this case.

Next we prove \eqref{spwow3}  in the case that $\pi=\emptyset$ and $\boldsymbol{\sigma}=\boldsymbol{\sigma}^{(alt)}$. 
In this case, 
$$
\prod_{i=2}^n \left(\Theta^{(B)}_{tm_im_{i+1}}\right)_{a_i, d_i}
=
\prod_{i=2}^n \left(\Theta^{(B)}_{t|m|^2}\right)_{a_i, d_i} 
$$
With short range property of $\Sigma^{(\emptyset)}(t, \boldsymbol{\sigma}, \textbf{d})$ (in \eqref{res_SIGempdd}), and the estimate of $(\Theta^{(B)}_{t|m|^2})_{a_id_i}$  in \eqref{prop:ThfadC}, one can easily bound the l.h.s. of \eqref{spwow3} with 
$ 
  O (\ell_t\eta_t)^{-n+1} 
$. To obtain the missing factor for \eqref{spwow3}, we need to apply the sum zero property which  we proved in Lemma \ref{lem:SZ}, i.e., 
\begin{equation}\label{SZjadljsk}
   \sum_{d_2,\,\cdots,\,d_n}\left(\Sigma^{(\emptyset)}(t, \boldsymbol{\sigma}, \textbf{d})\right)=O(\eta_t) 
\end{equation}
 For simplicity, for fixed $d_1$ and $\textbf{a}$, we temporally denote $s$, $f$ and $g$ as follows 
$$
s_i:=d_i-d_1,\quad \quad f(a_i, s_i):=\left(\Theta^{(B)}_{t|m|^2}\right)_{a_i,(d_1+s_i)}, \quad g(s_2, \cdots,s_n):=\Sigma^{(\emptyset)}(t, \boldsymbol{\sigma}, \textbf{d})$$
Then the l.h.s. of \eqref{spwow3} can be written as 
\begin{align} \label{Kfadgd}
 \sum_{d_2,\,\cdots,\, d_n}\left(\Sigma^{(\pi)}(t, \boldsymbol{\sigma}, \textbf{d})\right)
\cdot 
\prod_{i=2}^n \left(\Theta^{(B)}_{tm_im_{i+1}}\right)_{a_i, d_i}
=\sum_{d_1}\sum_{s_2\,\cdots,\,s_n} \prod_{i=2}^nf(a_i, s_i)\cdot g(\textbf{\textbf{s}})
\end{align}
Note due the fast decay of $g(\textbf{\textbf{s}})$, we can focus on the case that $\max_i|s_i|\prec 1$. 
For each $i$, we write 
$$
f(a_i, s_i)=f_0(a_i,s_i)+f_1(a_i,s_i )+f_2(a_i,s_i)
$$
where 
\begin{align}
f_0(a_i,s_i)=&f(a_i,0)
\\\nonumber
f_1(a_i,s_i)=&\frac12\left(f(a_i,s_i)-f(a_i, -s_i)\right)  
\\\nonumber
f_2(a_i,s_i)=&f(a_i,s_i)-f(a_i,0)-\frac12\left(f(a_i,s_i)-f(a_i,-s_i)\right)  
\end{align}
By Lemma \ref{lem_propTH}
\begin{align}
f_0(a_i,s_i)=&O(\ell_t\eta_t)^{-1}, 
\\\nonumber
f_1(a_i,s_i)=&O(\ell_t^{-1}\eta_t^{-1/2}),
\\\nonumber
f_2(a_i,s_i)=&O(\|a_i-d_1\|^{-1}) .
\end{align}
Inserting 
$$\prod_if(a_i,s_i)=\prod_i \left(f_0(a_i,s_i)+f_1(a_i,s_i)+f_2(a_i,s_i)\right)$$
into \eqref{Kfadgd}, we obtain 
\begin{align} \label{Kfs2dgd}
 \sum_{d_2,\,\cdots,\, d_n}\left(\Sigma^{(\pi)}(t, \boldsymbol{\sigma}, \textbf{d})\right)
\cdot 
\prod_{i=2}^n \left(\Theta^{(B)}_{tm_im_{i+1}}\right)_{a_i, d_i}
= 
\sum_{s_2\,\cdots,\,s_n}\;  \sum_{0\le \xi_2, \,\xi_3, \cdots,  \xi_n \le 2 }\;  \prod_{i=2}^nf_{\xi_i}(a_i, s_i)\cdot g(\textbf{\textbf{s}})
\end{align}
We claim that for any fixed $\xi$'s, 
\begin{equation}\label{aauduyuadw}
\sum_{s_2\,\cdots,\,s_n}    \prod_{i=2}^nf_{\xi_i}(a_i, s_i)\cdot g(\textbf{\textbf{s}})= O(\ell_t\eta_t)^{-n+2}\cdot \left(\min_{2\le k\le n} \|a_k-d_1\|+1\right)^{-1} .
\end{equation}
We estimate the r.h.s. in the following cases: 
\begin{itemize}
    \item If one of $\xi_k=2$, without loss of generality, let  $\xi_n=2$, then 
    $$
     \prod_{i=2}^nf_{\xi_i}(a_i, s_i)=O\left(\ell_t\eta_t\right)^{-n+2}\cdot  \|a_n-d_1\|^{-1}
    $$
    It implies that in this case \eqref{aauduyuadw} holds.  
  
\item If $0\le \xi_k\le 1$ for all $2\le k\le n$, and there two of $\xi_k$ equal to $1$.  Similar to above case, 
$$
    \prod_{i=2}^nf_{\xi_i}(a_i, s_i)=O\left(\ell_t\eta_t\right)^{-n+1}\cdot \eta_t
    =O\left(\ell_t\eta_t\right)^{-n+2}\cdot \ell_t^{-1}
    $$
 If $\|a_k-d_1\|=O(\ell_t)$ for all $2 \le k \le n$, this  implies  \eqref{aauduyuadw};   if $\|a_k-d_1\|\gg \ell_t$ for some $2 \le k \le n$, then the l.h.s. of \eqref{aauduyuadw} is  exponentially small. 

    \item If $0\le \xi_k\le 1$ for all $2\le k\le n$, and only one  of $\xi_k$ equal to $1$, without loss of generality, let $\xi_n=1$.  Similar to above case, 
$$
   \prod_{i=2}^nf_{\xi_i}(a_i, s_i)g(\textbf{s})=
\cdot \prod_{i=2}^{n-1}f(a_i,0) f_{1}(a_n,s_n)  g(\textbf{s})
    $$
By the definition, we know 
$$
f_{1}(a_n,s_n) =-f_{1}(a_n,-s_n) 
$$
and due to the symmetric, we have 
$$
g(\textbf{s})=g(-\textbf{s})
$$
 Then by symmetry,  the r.h.s. of \eqref{aauduyuadw} equals to zero in this case. 
    \item At last, if  $\xi_k=0$ for all $2\le k\le n$. Then 
    $$
    \prod_{i=2}^nf_{\xi_i}(a_i, s_i)g(\textbf{s})=
  \prod_{i=2}^{n }f(a_i,0)   g(\textbf{s})
$$
Applying \eqref{SZjadljsk}, we have 
$$
 \sum_{\textbf{s}}g(\textbf{s})=O(\eta_t)
$$
 It implies that in this case \eqref{aauduyuadw} holds. 
\end{itemize}

Therefore, we have proved that for any $\xi$'s, the \eqref{aauduyuadw} holds, which complete the proof of \eqref{spwow3} for the case of $\pi=\emptyset$ and $\sigma_1=+$, $\sigma_2=-$. 

\bigskip

We  now prove  \eqref{eq:bcal_k_pi} in the case $\pi\ne \emptyset$. As above, we use the decomposition method as in Figure \ref{fig:Decp}.  By \eqref{KThThSigma} we write  \begin{align} 
{\cal K}^{(\pi)}(t, \boldsymbol{\sigma}, \textbf{a})=
\sum_{\textbf{d}}\sum_{\textbf{c}}\left(\prod_{i=1}^{n}\left(\Theta^{(B)}_{tm_im_{i+1}}\right)_{a_id_i}\right)\cdot \prod_{k=1}^{M-1}\Big(\Theta_{{t|m|^2}}-1\Big)_{c_{2k-1}c_{2k}}
\cdot
\prod_{k=1}^{M}{\Sigma}^{(\emptyset)}(t, \boldsymbol{\sigma}^{(k)}, \textbf{d}^{(k)})   \end{align}
Among the molecules in this molecule structure, there must be one molecule containing only  one $c-$vertex, i.e., the big dot of this molecule  connects with only  one  $(\Theta^{(B)}_{t|m^2|-1})  _{c_{2k-1},c_{2k}}$ edge. (For example, the top, left and right molecules in Figure \ref{fig:Decp}). For simplicity, we assume that it is the first molecule, and it contains $c_1$ and  connects with $a_1,a_2\cdots a_{m-1}$. Note that in this case 
\begin{equation} \label{jjshsaa}
\{1,m \}\in \pi, \quad \sigma_1\ne \sigma_{m }
\end{equation}
As in the r.h.s of Figure \ref{fig:Decp}, we decompose  ${\cal K}^{(\pi)}(t, \boldsymbol{\sigma}, \textbf{a})$ into a  product of two parts,  ${\cal A}$ and ${\cal B}$, such that 
\begin{align} \label{aklsdjfyuadsfyo}
{\cal K}^{(\pi)}(t, \boldsymbol{\sigma}, \textbf{a})
& =
\sum_{c_1} \left({\cal A}_{a_1,\cdots,a_{m-1},c_1}\right)\cdot \left({\cal B}_{c_1, a_m,\cdots,a_{n}}\right)
\\\nonumber
{\cal A}_{a_1,\cdots,a_{m-1},c_1}& =\sum_{d_1,\,\cdots,\,d_{m-1}} \left(\prod_{i=1}^{m-1}\left(\Theta^{(B)}_{tm_im_{i+1}}\right)_{a_id_i}\right)
\cdot
 {\Sigma}^{(\emptyset)}(t, \boldsymbol{\sigma}^{(1)}, \textbf{d}^{(1)}),
 \quad  \textbf{d}^{(1)}  =(d_1, d_2\cdots, d_{m-1}, c_1)  
\\\nonumber
 {\cal B}_{c_1, a_m,\cdots,a_{n}} &=\sum_{d_m,\cdots,d_n}\; \sum_{c_2,\cdots,c_{2M-2}}\left(\prod_{i=m}^{n}\left(\Theta^{(B)}_{tm_im_{i+1}}\right)_{a_id_i}\right)\cdot \prod_{k=1}^{M-1}\Big(\Theta_{{t|m|^2}}-1\Big)_{c_{2k-1}c_{2k}}
\cdot
 \prod_{k=2}^{M}{\Sigma}^{(\emptyset)}(t, \boldsymbol{\sigma}^{(k)}, \textbf{d}^{(k)}). 
\end{align}
Using \eqref{spwow3}, we have 
$$
 {\cal A}_{a_1,\cdots,a_{m-1},c_1}=O(\ell_t\eta_t)^{-m+2}\cdot \left(\min_{1\le k\le m} \|a_k-c_1\|+1\right)^{-1}
$$
On the other hand, with inductive assumption and $(\Theta_t-1)_{c_1c_2}=t\left(S^{(B)}\cdot \Theta_t\right)_{c_1c_2}$,   we have 
$$
{\cal B}_{c_1, a_m,\cdots,a_{n}}=O(\ell_t\eta_t)^{-n+m-1}
$$
Combining  these two bounds and inserting  them back to \eqref{aklsdjfyuadsfyo}, we obtain \eqref{eq:bcal_k_pi} for the case that $\pi\ne \emptyset$, and complete the proof of Lemma  \ref{ML:Kbound+pi}

\end{proof}

\section{\texorpdfstring{$G$}{G}-entry estimates}
  In this section, we estimate $G$ entries and 1-loops with the 2-loops.    The proof  of the following  Lemma \ref{lem_GbEXP} will be via a time-independent method. 
It relies on a standard decomposition widely employed for Wigner matrices (e.g., \cite{erdHos2012rigidity}) and band matrices (e.g., \cite{Average_fluc}).

\begin{lemma}[Resolvent entry estimate]\label{lem_GbEXP}  
Recall   $z_t$ and  $G_t$ (which depend on  $N$ ) from  Definitions \ref{def_flow} and \ref{Def:G_loop}. 
Suppose that $|E| < 2 - \kappa, 0 \leq t <  1$.
For a fixed constant $c > 0$, define the event
\begin{equation}\label{def_asGMc}
    \Omega(t, c) := \big\{\|G_t - m\|_{\max} \leq W^{-c} \big\}.
\end{equation}
Then the entries of $G_t$ can be bounded in terms of $2$-$G$-loops as follows:
\begin{equation}\label{GijGEX}
    {\bf 1}_{\Omega(t, c)} \cdot \max_{i \in {\cal I}_a} \max_{j \in {\cal I}_b} |(G_t)_{ij}|^{\,2} \prec 
    \sum_{a' = a - 1}^{a + 1} \; \sum_{b' = b - 1}^{b + 1} {\cal L}_{t, (+,-), (a', b')}
    + W^{-1} \cdot {\bf 1}(|a - b| \leq 1),
\end{equation}
\begin{equation}\label{GiiGEX}
    {\bf 1}_{\Omega(t, c)} \cdot \max_{i} \left|(G_t)_{ii} - m\right|^{\,2} \prec \max_{a, b} {\cal L}_{t, (+,-), (a, b)}.
\end{equation}
 In particular, if for some $c > 0$, 
\begin{equation}\label{asGMc} 
    \|G_t - m\|_{\max} \prec W^{-c},
\end{equation}
then \eqref{GijGEX} and \eqref{GiiGEX} hold without the indicator ${\bf 1}_{\Omega(t, c)}$. 
Furthermore, under the same assumption, the $1$-loop estimate (interpreted as an average local law) holds:
\begin{equation}\label{GavLGEX}
    \max_{a} \left| \left\langle \left(G_t - m\right) E_a \right\rangle \right| \prec \max_{a, b} {\cal L}_{t, (+,-), (a, b)}.
\end{equation} 
\end{lemma}
 We first recall the following perturbation formulas from  Lemma 4.2 of \cite{erdHos2012bulk}:

\begin{lemma} Let $H$ be a Hermitian matrix and $H^{(i)}$ denote the $N-1$ by $N-1$ submatrix of $H$ after removing the $i-t h$ rows and columns.  Define 
\begin{align}\label{def_G(i)}
 G_{kl}^{(i)}:=\left[H^{(i)}-z\right]^{-1}(k,l). 
 \end{align}
Then we have 
\begin{align}\label{GiiSCP}
    G_{ii} &\;\;= \left(H_{ii}-z-\sum_{kl}H_{ik}G^{(i)}_{kl}H_{li}\right)^{-1},
    \\\label{GijSCP2}
   G_{ij} &\;\;= G_{ii} \sum_{k}H_{ik}G^{(i )}_{kj},
\\\label{GijkSCP}
   G^{(i)}_{jk} &\;\;= G_{jk}-\frac{G_{ji}G_{ik}}{G_{ii}} .
   \end{align}
\end{lemma}

\begin{proof}[Proof of Lemma \ref{lem_GbEXP}]
For simplicity, we ignore the subscript \( t \) in our proof and write \( \Omega = \Omega(t, c) \). For \( i \neq j \), using \eqref{GijSCP2}, we have
\[
|G_{ij}| = |G_{ii}| \left| \sum_{k} H_{ik} G^{(i)}_{kj} \right|.
\]
Here \( \mathbb{E} |H_{ij}|^2 = t \cdot S_{ij} \leq S_{ij} \). By definition in \eqref{asGMc}, we have $
{\bf 1}_{\Omega} |G_{ii}| = O(1)$. 
Since \( G^{(i)} \) is independent of the \( i \)-th row of \( H \), we can apply Lemma 3.3 in \cite{erdHos2012rigidity} to get
\[
\left| \sum_{k} H_{ik} G^{(i)}_{kj} \right| \prec 
\left( \sum_{k} S_{ik} \left| G^{(i)}_{kj} \right|^2 \right)^{1/2}.
\]
Using \eqref{GijkSCP}, we further deduce that
\[
{\bf 1}_{\Omega} \cdot |G^{(i)}_{kj}| \leq |G_{kj}| + O\left(W^{-c} |G_{ij}|\right).
\]
Combining these estimates, we obtain
\begin{align}
{\bf 1}_{\Omega} \cdot |G_{ij}|^2 & \;\prec\; \sum_{k} S_{ik} \left| G_{kj} \right|^2 + O\left(W^{-c} |G_{ij}|^2\right) 
 \;\prec\; \sum_{k} S_{ik} \left| G_{kj} \right|^2.
\end{align}
Applying this bound on \( G_{kj} \) and iterating the process (for the special case \( k = j \), we can bound \( G_{kk} \prec 1 \)), we find
\begin{align}
{\bf 1}_{\Omega} \cdot |G_{ij}|^2 & \;\prec\; \sum_{kl} S_{ik} \left| G_{kl} \right|^2 S_{lj} +
W^{-1} {\bf 1}\Big(|[j] - [i]| \leq 1\Big).
\end{align}
This implies \eqref{GijGEX}.

Now we prove \eqref{GiiGEX}. Using \eqref{GiiSCP} and Lemma 3.3 in \cite{erdHos2012rigidity}, we get
\[
{\bf 1}_{\Omega} \cdot \sum_{kl} H_{ik} G^{(i)}_{kl} H_{li} =
\sum_{k} S_{ik} G^{(i)}_{kk} +
O_\prec\left(\sum_{kl} S_{ik} \left| G^{(i)}_{kl} \right|^2 S_{li} \right)^{1/2}.
\]
As above, using \eqref{GijkSCP} to remove the \( (i) \) superscript and applying \eqref{GijGEX}, we find
\[
{\bf 1}_{\Omega} \cdot \sum_{kl} H_{ik} G^{(i)}_{kl} H_{li} =
{\bf 1}_{\Omega} \cdot \sum_{k} S_{ik} G_{kk} + O_\prec\left({\bf 1}_{\Omega} \cdot {\cal E}\right),
\]
where
\[
{\cal E}^2 = {\bf 1}_{\Omega} \cdot \max_{a,b} {\cal L}_{t,(+,-),(a,b)} +
\left(\max_{a,b} {\cal L}_{t,(+,-),(a,b)}\right)^2 + W^{-1}.
\]
On the other hand, it is easy to verify that in \( \Omega(t, c) \),
\[
c W^{-1} \leq \max_{a,b} {\cal L}_{t,(+,-),(a,b)} \leq C W^{-1} + W^{-c}.
\]
Substituting back into \eqref{GiiSCP}, we get
\[
{\bf 1}_{\Omega} \cdot G_{ii} =
{\bf 1}_{\Omega} \left(-z - \sum_{k} S_{ik} G_{kk} - O_\prec \left(\max_{a,b} {\cal L}_{t,(+,-),(a,b)}\right)^{1/2} \right)^{-1}.
\]
By definition, one can easily check that \( m = -(m + z)^{-1} \). Expanding the right-hand side around \( (-z - m)^{-1} \), we have
\[
{\bf 1}_{\Omega} \cdot \left( G_{ii} - m \right) =
\sum_j \left[ (1 - m^2 S)^{-1} \right]_{ij} \cdot {\cal E}_j, \quad {\cal E}_j \prec \left( \max_{a,b} {\cal L}_{t,(+,-),(a,b)} \right)^{1/2}+W^{-c}\cdot \max_j|G_{jj}-m|.
\]
Together with the fact that \( \|(1 - m^2 S)^{-1}\|_{\max \to \max} = O(1) \), we conclude \eqref{GiiGEX}.

Next, we prove \eqref{GavLGEX}, which is a type of estimate commonly referred to as fluctuation averaging. A brief historical context is provided in Section 10.3.1 of \cite{erdHos2017dynamical}. In particular, very similar results are established in equation (3.7) of \cite{Average_fluc} and equation (4.11) of \cite{erdHos2013local}. 

Recall the definition of $H^{(i)}$ introduced above \eqref{def_G(i)}. Denote $\mathbb{E}_i[X] = \mathbb{E}[X \mid H^{(i)}]$, i.e., the conditional expectation with respect to the $i$-th row and column of $H$. Previously, we showed that 
\[
\|G_{ij} - m\|_{\max}^2 \prec \Psi^2 := \max_{\mathbf{a}} \mathcal{L}_{t, (+,-), \mathbf{a}},
\]
where $\Psi$ aligns with the notation in \cite{erdHos2013local}. It is established in equation (4.11) of \cite{erdHos2013local} that, for any $\{t_k\}_{k \in \mathbb{Z}_N}$ satisfying 
\[
0 \leq \left|t_k\right| \leq W^{-1}, \quad \sum_k \left|t_k\right| \leq 1,
\]
we have
\begin{align} \label{jasdu}
\sum_k t_k (1 - \mathbb{E}_k)(G_{kk} - m) \prec \Psi^2 = \max_{\mathbf{a}} \mathcal{L}_{t, (+,-), \mathbf{a}}.
\end{align}

On the other hand, using the identity $G - m = m(H - m)G$ and applying Gaussian integration by parts, we obtain 
\[
\mathbb{E}_i(G_{ii} - m) = \mathbb{E}_i[m(H - m)G] = \sum_k \mathbb{E}_i[m(G_{kk} - m) S_{ki} G_{ii}] = m^2 \sum_k S_{ik} \cdot \mathbb{E}_i(G_{kk} - m) + O_\prec(\Psi^2).
\]

Using \eqref{GijkSCP}, we deduce that
\[
\mathbb{E}_i(G_{kk} - m) = \mathbb{E}_i(G_{kk}^{(i)} - m) + O_\prec(\Psi^2) = (G_{kk}^{(i)} - m) + O_\prec(\Psi^2) = G_{kk} - m + O_\prec(\Psi^2).
\]

Substituting this into the earlier equation, we find
\[
\mathbb{E}_i(G_{ii} - m) = m^2 \sum_k S_{ik} \cdot (G_{kk} - m) + O_\prec(\Psi^2) 
= m^2 \sum_k S_{ik} \cdot \mathbb{E}_k(G_{kk} - m) + O_\prec(\Psi^2),
\]
where the last estimate follows from \eqref{jasdu}. Solving this equation, we conclude that
\[
\mathbb{E}_i(G_{ii} - m) = O_\prec(\Psi^2).
\]

Combining this with \eqref{jasdu}, we derive \eqref{GavLGEX} with $t_k = W^{-1}\cdot \mathbf{1}(k \in \mathcal{I}_a)$.

 \end{proof}

\section{Analysis of loop hierarchy}\label{Sec:Stoflo}

In this section, we prove Theorem \ref{lem:main_ind}. Except for Step 1, the proof primarily relies on analyzing the G-loop hierarchy.

\subsection{Proof of Theorem \ref{lem:main_ind}, step 1}

\begin{proof}

Our goal is to establish \eqref{lRB1} and \eqref{Gtmwc}. Using the assumption \eqref{con_st_ind} and the definitions \(\ell_t = \ell(z_t)\) and \(\eta_t = \im z_t\), we have 
$$
1 - u \gg N^{-1}, \quad \eta_u \gg N^{-1}, \quad u \ge t.
$$
Given that \(\partial_z (H - z)^{-1} = (H - z)^{-2}\) and the entries of \(H_t\) follow a Gaussian distribution, it follows that for any \(C > 0\), there exists a constant \(C'\) such that 
\begin{equation}\label{Gopboundu}
    \max_{u \ge N^{-1}} \max_{|u - u'| \leq N^{-C'}} \| G_u - G_{u'} \|_{\max} \leq N^{-C},
\end{equation}
holds up to events that are exponentially small (negligible). Hence, through a standard \(N^{-C}\) net argument, we can reduce the proof of \eqref{lRB1} and \eqref{Gtmwc} for all \(u\) to the case \(u = t\). This standard procedure, which we will refer to as a continuity argument, will be used repeatedly in this paper. We now focus on the proof for the case \(u = t\).

\paragraph{Case 1: \(s < t \leq 1/2\).} 
Under this assumption on \(t\),  
$$
\|G_t\|_{op} = 1/\eta_t = O(1).
$$
Combined with the fact that \(\|E_a\|_{op} \leq W^{-1}\) for all \(a \in \mathbb{Z}_L\), we have  
\begin{equation}\label{Lboundfor1/2}
    {\cal L}_{t, \boldsymbol{\sigma}, \textbf{a}} = \left\langle \prod_{i=1}^n G_t(\sigma_i) E_{a_i} \right\rangle = O(W^{-n+1}), \quad t \leq 1/2, 
\end{equation}
which implies \eqref{lRB1} for \(t \leq 1/2\). Applying \eqref{GijGEX} and \eqref{GiiGEX} from Lemma \ref{lem_GbEXP} to \(\|G_t - m\|_{\max}\), and using the bound \(\cal L\) in \eqref{Lboundfor1/2}, we find that for \(t \leq 1/2\),
$$
{\bf 1}\left(\|G_t - m\|_{\max} \leq W^{-1/10}\right) \cdot \|G_t - m\|_{\max} \prec W^{-1/2}, \quad t \leq 1/2, 
$$
where we have used the fact that \(\ell_t\) and \(\eta_t\) are of order one. On the other hand, using the assumption \eqref{Gt_bound+IND} for \(s\), we have \(\|G_s - m\| \prec W^{-1/2}\). Then, applying a standard continuity argument (with an \(N^{-C}\) net between \(s\) and \(t\)) and \eqref{Gopboundu}, we obtain that for any \(u \in [s, t]\), 
\begin{equation}\label{Gtmt12}
    \|G_u - m\|_{\max} \prec W^{-1/2}, \quad u \leq 1/2.
\end{equation}
Hence, we have proved \eqref{lRB1} and \eqref{Gtmwc} for the case \(t \leq 1/2\).

\paragraph{Case 2: \(t > s \geq 1/2\).} 
Combining the assumption \eqref{Eq:L-KGt+IND} and \eqref{eq:bcal_k} from Lemma \ref{ML:Kbound}, i.e., 
$$
{\cal K}_{s, \boldsymbol{\sigma}, \textbf{a}} \prec (W\ell_s\eta_s)^{-n+1}, 
$$
we obtain that for any fixed \(n\) and \(s\), 
\begin{equation}\label{Lboundfors}
    {\cal L}_{s, \boldsymbol{\sigma}, \textbf{a}} = O((W\ell_s\eta_s)^{-n+1}).
\end{equation}
The following lemma will be proved in Section \ref{sec:Inh_LE}.

\begin{lemma}[Continuity  estimate on loops]\label{lem_ConArg}
Suppose that \(c < t_1 \leq t_2 \leq 1\) for some constant \(c > 0\). Assume that for any fixed \(n \in \mathbb{N}\), the following bounds hold at time \(t_1\): 
\begin{equation}\label{55}
    \max_{\boldsymbol{\sigma}, \textbf{a}} {\cal L}_{t_1, \boldsymbol{\sigma}, \textbf{a}} \prec \left( W \ell_1 \eta_1 \right)^{-n+1}, \quad \eta_i = \eta_{t_i}, \quad \ell_i = \ell_{t_i}.
\end{equation}
Define \(\Omega\) as 
\begin{equation}\label{usuayzoo}
    \Omega := \left\{\|G_{t_2}\|_{\max} \leq 2\right\}.
\end{equation}
Then, for any fixed \(n \geq 1\), we have
\begin{equation}\label{res_lo_bo_eta}
    {\bf 1}_\Omega \cdot \max_{\boldsymbol{\sigma}, \textbf{a}} {\cal L}_{t_2, \boldsymbol{\sigma}, \textbf{a}} \prec \left( W \ell_1 \eta_2 \right)^{-n+1} = \left( \frac{\ell_2}{\ell_1} \right)^{n-1} \cdot \left( W \ell_2 \eta_2 \right)^{-n+1}, \quad \eta_i = \eta_{t_i}, \quad \ell_i = \ell_{t_i}.
\end{equation}
\end{lemma}

Using \eqref{Lboundfor1/2} and \eqref{Lboundfors}, we know that the estimate \eqref{55} for \({\cal L}_s\) holds. Applying Lemma \ref{lem_ConArg}, we obtain
\begin{equation}\label{jsajufua}
    {\bf 1}\left(\|G_{u}\|_{\max} \leq 2\right) \cdot \max_{\boldsymbol{\sigma}, \textbf{a}} {\cal L}_{u, \boldsymbol{\sigma}, \textbf{a}} \prec \left( \frac{\ell_u}{\ell_s} \right)^{n-1} \cdot \left( W \ell_u \eta_u \right)^{-n+1}, \quad s \leq u \leq t.
\end{equation}
Combining the special case \(n = 2\) of \eqref{jsajufua} with Lemma \ref{lem_GbEXP}, we have
$$
{\bf 1}\left(\|G_u - m\|_{\max} \leq (W \ell_u \eta_u)^{-1/6}\right) \cdot \|G_u - m\|_{\max} \prec \left( \frac{\ell_u}{\ell_s} \right)^{1/2} \cdot \left( W \ell_u \eta_u \right)^{-1/2}, \quad s \leq u \leq t.
$$
Using the assumption \eqref{con_st_ind}, we find that for any fixed \(D > 0\),
\begin{equation}
    \mathbb{P}\left((W \ell_u \eta_u)^{-1/4} \leq \|G_u - m\|_{\max} \leq (W \ell_u \eta_u)^{-1/6}\right) \leq N^{-D}, \quad s \leq u \leq t.
\end{equation}
This shows that the interval \(\big((W \ell_u \eta_u)^{-1/4}, (W \ell_u \eta_u)^{-1/6}\big)\) is a forbidden region for \(\|G_u - m\|_{\max}\) for any \(u\) between \(s\) and \(t\). Since this event holds with very high probability, the standard continuity argument implies that it holds for all time between \(s\) and \(t\) simultaneously.

On the other hand, by assumption and \eqref{Gtmt12}, we have the initial bound at time \(s\):
$$
\|G_s - m\|_{\max} \prec (W \ell_s \eta_s)^{-1/2}.
$$
Using a standard continuity argument between \(s\) and \(t\), we conclude
$$
\|G_t - m\|_{\max} \prec (W \ell_t \eta_t)^{-1/4}.
$$
Combining this with \eqref{jsajufua}, we have proved \eqref{lRB1} in this case.

\paragraph{Case 3: \(t > 1/2 > s\).} From Case 1, we know that \eqref{Lboundfor1/2} and \eqref{Gtmt12} hold for \(t = 1/2\). Note that the proof of Case 2 uses only these two conditions. Therefore, the current case follows as a consequence of Case 2 with \(s = 1/2\).

\end{proof}

\subsection{Dynamics of  \texorpdfstring{$\cal L-\cal K$}{L-K}}

 Recall the loop hierarchy from Lemma \ref{lem:SE_basic}. Using the notations
$$
\mathcal{L}_{t,\;  \mathcal{G}^{(a), \,L}_{k, \,l}\left(\boldsymbol{\sigma},\, \textbf{a}\right)} := \left( \mathcal{G}^{(a), L}_{k, l} \circ \mathcal{L}_{t, \boldsymbol{\sigma}, \textbf{a}} \right), \quad  
\mathcal{L}_{t,\;  \mathcal{G}^{(b), \,R}_{k, \,l}\left(\boldsymbol{\sigma},\, \textbf{a}\right)} := \left( \mathcal{G}^{(b), R}_{k, l} \circ \mathcal{L}_{t, \boldsymbol{\sigma}, \textbf{a}} \right),
$$
the hierarchy takes the compact form:
\begin{align}
    d\mathcal{L}_{t, \boldsymbol{\sigma}, \textbf{a}} =& 
    \mathcal{E}^{(M)}_{t, \boldsymbol{\sigma}, \textbf{a}} +
    \mathcal{E}^{(\widetilde{G})}_{t, \boldsymbol{\sigma}, \textbf{a}} +
    W\cdot\sum_{1 \leq k < l \leq n} \sum_{a, b} \mathcal{L}_{t,\;  \mathcal{G}^{(a), \,L}_{k, \,l}\left(\boldsymbol{\sigma},\, \textbf{a}\right)} \cdot  S^{(B)}_{ab} \cdot \mathcal{L}_{t,\;  \mathcal{G}^{(b), \,R}_{k, \,l}\left(\boldsymbol{\sigma},\, \textbf{a}\right)}  \, dt.
\end{align}

The primitive hierarchy governing the evolution of \(\mathcal{K}\) is given by:
\begin{align}
    d\mathcal{K}_{t, \boldsymbol{\sigma}, \textbf{a}} = 
    W\cdot\mathcal{K}_{t,\;  \mathcal{G}^{(a), \,L}_{k, \,l}\left(\boldsymbol{\sigma},\, \textbf{a}\right)} \cdot S^{(B)}_{ab} \cdot \mathcal{K}_{t,\;  \mathcal{G}^{(b), \,R}_{k, \,l}\left(\boldsymbol{\sigma},\, \textbf{a}\right)} \, dt.
\end{align}
Combining the two equations, we obtain:
\begin{align}\label{eq_L-K-1}
    d(\mathcal{L} - \mathcal{K})_{t, \boldsymbol{\sigma}, \textbf{a}} 
    = &\; 
   W\cdot \sum_{1 \leq k < l \leq n} \sum_{a, b} 
   \Big( 
   (\mathcal{L} - \mathcal{K})_{t,\;  \mathcal{G}^{(a), \,L}_{k, \,l}\left(\boldsymbol{\sigma},\, \textbf{a}\right)} 
   \cdot S^{(B)}_{ab}
   \cdot 
   \mathcal{K}_{t,\;  \mathcal{G}^{(b), \,R}_{k, \,l}\left(\boldsymbol{\sigma},\, \textbf{a}\right)} 
     +  
     \left(\mathcal{K} \iff \mathcal{L-K} \right)
   \Big) \, dt \nonumber \\ 
    & + \mathcal{E}^{((\mathcal{L} - \mathcal{K}) \times (\mathcal{L} - \mathcal{K}))}_{t, \boldsymbol{\sigma}, \textbf{a}} 
    + \mathcal{E}^{(M)}_{t, \boldsymbol{\sigma}, \textbf{a}} 
    + \mathcal{E}^{(\widetilde{G})}_{t, \boldsymbol{\sigma}, \textbf{a}}, 
\end{align}
where
\begin{equation}\label{def_ELKLK}
    \mathcal{E}^{((\mathcal{L} - \mathcal{K}) \times (\mathcal{L} - \mathcal{K}))}_{t, \boldsymbol{\sigma}, \textbf{a}} :=
    W\cdot\sum_{1 \leq k < l \leq n} \sum_{a, b} (\mathcal{L} - \mathcal{K})_{t,\;  \mathcal{G}^{(a), \,L}_{k, \,l}\left(\boldsymbol{\sigma},\, \textbf{a}\right)} \cdot  S^{(B)}_{ab} \cdot (\mathcal{L} - \mathcal{K})_{t,\;  \mathcal{G}^{(b), \,R}_{k, \,l}\left(\boldsymbol{\sigma},\, \textbf{a}\right)} \, dt.
\end{equation}
Here \( \mathcal{K} \iff \mathcal{L-K} \) represents the terms obtained by swapping \(\mathcal{K}\) and \(\mathcal{L-K}\). The terms on the first line of the right-hand side correspond to \((\mathcal{L} - \mathcal{K})\) "loops" connected with a \(\mathcal{K}\) "loop" via \(S^{(B)}_{ab}\). These terms can be rearranged by the rank of $\mathcal{K}$, i.e., the length of the corresponding $G$-loop. It
 allows us to rewrite the first line of \eqref{eq_L-K-1} as:
\begin{align}\label{DefKsimLK}
\eqref{eq_L-K-1} = \sum_{l_{\cal K}=2}^n &\; W\cdot \sum_{1 \leq k < l \leq n} \sum_{a, b}  
   \Big( 
   (\mathcal{L} - \mathcal{K})_{t,\;  \mathcal{G}^{(a), \,L}_{k, \,l}\left(\boldsymbol{\sigma},\, \textbf{a}\right)} 
   \cdot S^{(B)}_{ab}
   \cdot 
   \mathcal{K}_{t,\;  \mathcal{G}^{(b), \,R}_{k, \,l}\left(\boldsymbol{\sigma},\, \textbf{a}\right)} \cdot \textbf{1}(\text{length of $\mathcal{K}$ loop equals } l_\mathcal{K})   \nonumber \\
   & + \left(\mathcal{K} \iff \mathcal{L-K} \right) \Big) 
    := \sum_{l_{\cal K}=2}^n\Big[\mathcal{K} \sim (\mathcal{L} - \mathcal{K})\Big]^{l_\mathcal{K}}_{t, \boldsymbol{\sigma}, \textbf{a}}.
\end{align}
Separating the special case $l_\mathcal{K} =2$, we have 
 the  $\cal L-K$ hierarchy 
\begin{align}\label{eq_L-Keee}
    d(\mathcal{L} - \mathcal{K})_{t, \boldsymbol{\sigma}, \textbf{a}} = &
    \Big[\mathcal{K} \sim (\mathcal{L} - \mathcal{K})\Big]^{l_\mathcal{K} = 2}_{t, \boldsymbol{\sigma}, \textbf{a}} + \sum_{l_\mathcal{K} > 2} \Big[\mathcal{K} \sim (\mathcal{L} - \mathcal{K})\Big]^{l_\mathcal{K}}_{t, \boldsymbol{\sigma}, \textbf{a}} + \mathcal{E}^{((\mathcal{L} - \mathcal{K}) \times (\mathcal{L} - \mathcal{K}))}_{t, \boldsymbol{\sigma}, \textbf{a}} +
    \mathcal{E}^{(M)}_{t, \boldsymbol{\sigma}, \textbf{a}} +
    \mathcal{E}^{(\widetilde{G})}_{t, \boldsymbol{\sigma}, \textbf{a}}.
\end{align} 
Clearly, we can view $\Big[\mathcal{K} \sim (\mathcal{L} - \mathcal{K})\Big]^{l_\mathcal{K} = 2}_{t, \boldsymbol{\sigma}, \textbf{a}}$ as a linear transform of  the  tensor $(\mathcal{L} - \mathcal{K})_{t, \boldsymbol{\sigma}, \textbf{a}}$.

\begin{definition}[Definition of ${\varTheta}_{\,t, \,\boldsymbol{\sigma}}$ and ${\cal U}_{\,s,\, t, \,\boldsymbol{\sigma}}$]\label{DefTHUST}
Define the linear operator ${\varTheta}_{t, \boldsymbol{\sigma}}$ on a tensor ${\cal A}: \mathbb Z^n_L\to \mathbb C$ by 
\begin{align}
    \left({{\varTheta}}_{t, \boldsymbol{\sigma}} \circ \mathcal{A}\right)_{\textbf{a}} & = \sum_{i=1}^n \sum_{b_i} \left(\frac{m_i m_{i+1}  }{1 - t m_i m_{i+1} S^{(B)}}\right)_{a_i b_i} \cdot \mathcal{A}_{\textbf{a}^{(i)}},\quad m_i=m(\sigma_i) \nonumber \\
    \quad \textbf{a}^{(i)}&  = (a_1, \ldots, a_{i-1}, b_i, a_{i+1}, \ldots, a_n).
\end{align}  
We also define the evolution kernel 
    \begin{align}\label{def_Ustz}
        \left(\mathcal{U}_{s, t, \boldsymbol{\sigma}} \circ \mathcal{A}\right)_{\textbf{a}} = \sum_{b_1, \ldots, b_n} \prod_{i=1}^n \left(\frac{1 - s \cdot m_i m_{i+1} S^{(B)}}{1 - t \cdot m_i m_{i+1} S^{(B)}}\right)_{a_i b_i} \cdot \mathcal{A}_{\textbf{b}}, \quad \textbf{b} = (b_1, \ldots, b_n), \quad m_i=m(\sigma_i)
    \end{align}
\end{definition}
For any $\xi \in \mathbb{C}$, we have 
 \begin{align}\label{def_Ustz_2}
\left(\frac{1 - s \xi S^{(B)}}{1 - t \xi S^{(B)}}\right) = I - (s - t) \xi \cdot \Theta^{(B)}_{t \xi} \cdot S^{(B)}.
\end{align}
Using these notations and the rank-2 $\mathcal{K}$ tensor from \eqref{Kn2sol}, we derive the following identity for the first term on the right-hand side of \eqref{eq_L-Keee}.

\begin{align}
    \left({\varTheta}_{t, \boldsymbol{\sigma}} \circ (\mathcal{L} - \mathcal{K})\right)_{\textbf{a}} = \Big[\mathcal{K} \sim (\mathcal{L} - \mathcal{K})\Big]^{l_\mathcal{K} = 2}_{t, \boldsymbol{\sigma}, \textbf{a}}.
\end{align}
 The following lemma is just a form of  Duhamel formula. 

\begin{lemma}[Integrated loop hierarchy] 
    \label{Sol_CalL}
Let $\mathcal{A}_t$ be a tensor satisfying the stochastic equation:
$$
    d \mathcal{A}_t = {\varTheta}_{t, \boldsymbol{\sigma}} \circ \mathcal{A}_t \; dt + \mathcal{D}_t \; dt + \mathcal{C}_t \cdot d{\cal B}_t.
$$
Then for $s<t$ the $\mathcal{A}_t$ satisfies
 $$
    \mathcal{A}_t = \mathcal{U}_{s, t, \boldsymbol{\sigma}} \circ \mathcal{A}_s +
    \int_{s}^t \mathcal{U}_{u, t, \boldsymbol{\sigma}} \circ \mathcal{D}_u \; du +
    \int_{s}^t \mathcal{U}_{u, t, \boldsymbol{\sigma}} \circ 
    \left(\mathcal{C}_u \cdot d{\cal B}_u\right). 
$$
Inserting this solution into \eqref{eq_L-Keee}, we have the following integrated loop hierarchy  
\begin{align}\label{int_K-LcalE}
    (\mathcal{L} - \mathcal{K})_{t, \boldsymbol{\sigma}, \textbf{a}} \;=\; \; &
    \left(\mathcal{U}_{s, t, \boldsymbol{\sigma}} \circ (\mathcal{L} - \mathcal{K})_{s, \boldsymbol{\sigma}}\right)_{\textbf{a}} \nonumber \\
    &+ \sum_{l_\mathcal{K} > 2} \int_{s}^t \left(\mathcal{U}_{u, t, \boldsymbol{\sigma}} \circ \Big[\mathcal{K} \sim (\mathcal{L} - \mathcal{K})\Big]^{l_\mathcal{K}}_{u, \boldsymbol{\sigma}}\right)_{\textbf{a}} du \nonumber \\
    &+ \int_{s}^t \left(\mathcal{U}_{u, t, \boldsymbol{\sigma}} \circ \mathcal{E}^{((\mathcal{L} - \mathcal{K}) \times (\mathcal{L} - \mathcal{K}))}_{u, \boldsymbol{\sigma}}\right)_{\textbf{a}} du  \nonumber \\
    &+ \int_{s}^t \left(\mathcal{U}_{u, t, \boldsymbol{\sigma}} \circ \mathcal{E}^{(\widetilde{G})}_{u, \boldsymbol{\sigma}}\right)_{\textbf{a}} du + \int_{s}^t \left(\mathcal{U}_{u, t, \boldsymbol{\sigma}} \circ \mathcal{E}^{(M)}_{u, \boldsymbol{\sigma}}\right)_{\textbf{a}}  
\end{align}
Furthermore, let $T$ be a stopping time with respect to the matrix Brownian motion $H_t$ and denote  $\tau:= T\wedge t$. 
Then we have the stopped integrated loop hierarchy  
\begin{align}\label{int_K-L_ST}
   (\mathcal{L} - \mathcal{K})_{\tau , \boldsymbol{\sigma}, \textbf{a}} \;=\; \; &
    \left(\mathcal{U}_{s, \tau, \boldsymbol{\sigma}} \circ (\mathcal{L} - \mathcal{K})_{s, \boldsymbol{\sigma}}\right)_{\textbf{a}}  \nonumber \\
    &+ \sum_{l_\mathcal{K} > 2} \int_{s}^\tau \left(\mathcal{U}_{u, \tau, \boldsymbol{\sigma}} \circ \Big[\mathcal{K} \sim (\mathcal{L} - \mathcal{K})\Big]^{l_\mathcal{K}}_{u, \boldsymbol{\sigma}}\right)_{\textbf{a}} du  \nonumber \\
    &+ \int_{s}^\tau \left(\mathcal{U}_{u, \tau, \boldsymbol{\sigma}} \circ \mathcal{E}^{((\mathcal{L} - \mathcal{K}) \times (\mathcal{L} - \mathcal{K}))}_{u, \boldsymbol{\sigma}}\right)_{\textbf{a}} du  \nonumber \\
    &+ \int_{s}^\tau \left(\mathcal{U}_{u, \tau, \boldsymbol{\sigma}} \circ \mathcal{E}^{(\widetilde{G})}_{u, \boldsymbol{\sigma}}\right)_{\textbf{a}} du + 
    {\cal U}_{t,\tau}\circ \int_{s}^\tau \left(\mathcal{U}_{u, t, \boldsymbol{\sigma}} \circ \mathcal{E}^{(M)}_{u, \boldsymbol{\sigma}}\right)_{\textbf{a}}
    .
\end{align}
  
 
\end{lemma}

The equation \eqref{int_K-LcalE} will serve as our fundamental equation for estimating $\mathcal{L} - \mathcal{K}$. To analyze this equation, we first introduce the following notation, which is necessary to compute the quadratic variation of the martingale term.

\begin{definition}[Definition of $\mathcal{E}\otimes  \mathcal{E}$]\label{def:CALE} 
Denote 
  \begin{align}\label{defEOTE}
 \left( \mathcal{E}\otimes  \mathcal{E} \right)_{t, \,\boldsymbol{\sigma}, \,\textbf{a}, \,\textbf{a}'} \; := \; &
 \sum_{k=1}^n\left( \mathcal{E}\otimes  \mathcal{E} \right)^{(k)}_{t, \,\boldsymbol{\sigma}, \,\textbf{a}, \,\textbf{a}'}
 \nonumber \\
 \left( \mathcal{E}\otimes  \mathcal{E} \right)^{(k)}_{t, \,\boldsymbol{\sigma}, \,\textbf{a}, \,\textbf{a}'} :
\; = \; & W \sum_{b,b'} S^{(B)}_{b,b'}{\cal L}_{t, \boldsymbol{\sigma}^{(k ) },\textbf{a}^{(k )}},\quad \boldsymbol{\sigma}^{(k ) }\in \{+,-\}^{2n+2}, 
\end{align}
where the loop ${\cal L}_{t, \boldsymbol{\sigma}^{(k ) },\textbf{a}^{(k )}}$ is obtained by cutting the $k$-th edge of ${\cal L}_{t,\boldsymbol{\sigma},\textbf{a}}$ and then attach itself (with indices $\textbf{a}$) with its complex conjugate loop (with indices $\textbf{a}'$)  into a bigger loop, with the new indices $b$ and $b'$. Hence, there are $n$ indices between $b$ and $b'$ so that 
\begin{align}\label{def_diffakn_k}
   & \textbf{a}^{(k )}=( a_k,a_{k+1},\cdots a_n, a_1,\cdots a_{k-1}, b', a'_{k-1}\cdots a_1', a_n'\cdots a'_{k},b)
    \nonumber \\
  &  \boldsymbol{\sigma}^{(k )}=( \, \sigma_k, \, \sigma_{k+1},\cdots \sigma_n,  \, \sigma_1,\cdots \sigma_{k}, \,  \overline\sigma_{k},  \cdots \overline\sigma_1 , \,  \overline\sigma_n \cdots \overline\sigma_{k })
\end{align}
The symbol $\otimes$ in the notation $\mathcal{E} \otimes \mathcal{E}$ was used  to emphasize the symmetric structure (as illustrated in Figure \eqref{fig:8shapedgraph}); it  does not denote a tensor product.  

Example: For a $3\text{-}G$ loop $\mathcal{L}_{\boldsymbol{\sigma}, \mathbf{a}}$, Figure \ref{fig:8shapedgraph} represents one loop that appears in $\mathcal{E} \otimes \mathcal{E}$. Although $S^{(B)}$ is not part of the loop, it does appear within $\mathcal{E} \otimes \mathcal{E}$.

\begin{figure}[ht]
\centering
\begin{minipage}{0.45\textwidth} 
    \centering
    \begin{tikzpicture}[scale=1.0]

    \node[draw, circle, inner sep=1pt, fill=black] (A1) at (-2, 2) {};
    \node at (-2.3, 2) {$a_1$}; 

    \node[draw, circle, inner sep=1pt, fill=black] (A2) at (0, 3) {};
    \node at (0 , 3.3) {$a_2$}; 

    \node[draw, circle, inner sep=1pt, fill=black] (A3) at (2, 2) {};
    \node at (2.3, 2 ) {$a_3$}; 

    \node[draw, circle, inner sep=1pt, fill=black] (B) at (-0.5, 1) {};
    \node at (-0.9, 1) {$b$}; 

    \node[draw, circle, inner sep=1pt, fill=black] (B1) at (0.5, 1) {};
    \node at (0.9, 1) {$b'$}; 

    \node[draw, circle, inner sep=1pt, fill=black] (A1b) at (-2, 0) {};
    \node at (-2.3, -0) {$a'_1$}; 

    \node[draw, circle, inner sep=1pt, fill=black] (A2b) at (0, -1) {};
    \node at (0, -1.3) {$a'_2$}; 

    \node[draw, circle, inner sep=1pt, fill=black] (A3b) at (2, 0) {};
    \node at (2.3, -0) {$a'_3$}; 

    \draw (A1) -- (A2) node[midway, above left] {\small $G_2$}; 
    \draw (A2) -- (A3) node[midway, above right] {\small $G_3$}; 
    \draw (A1) -- (B) node[midway, below left, xshift=-2pt, yshift= 5pt] {\small $G_1$}; 
    \draw (A3) -- (B1) node[midway, below right, xshift=2pt, yshift= 5pt] {\small $G_1$}; 

    \draw (B) -- (A1b) node[midway, above left,  xshift=-2pt, yshift=-5pt] {\small $\overline{G_1}$}; 
    \draw (A1b) -- (A2b) node[midway, below left, yshift=-2pt] {\small $\overline{G_2}$}; 
    \draw (A2b) -- (A3b) node[midway, below right, yshift=-2pt] {\small $\overline{G_3}$}; 
    \draw (B1) -- (A3b) node[midway, above right, xshift=2pt, yshift=-5pt] {\small $\overline{G_1}$}; 

    \draw[decorate, decoration={snake, amplitude=0.4mm, segment length=2mm}, thick] (B) -- (B1)
        node[midway, above] {$S^{(B)}$};

    \end{tikzpicture}
   
\end{minipage}%
\hfill
\begin{minipage}{0.45\textwidth} 
    \raggedright
    $$k=1 $$\\ 
    $$\textbf{a}^{(1 )}=( a_1, a_2, a_3, b', a_3', a_2', a_1',b);$$
    \\  
     $$\boldsymbol{\sigma}^{(1 )}=(\;\sigma_1, \sigma_2, \sigma_3,\sigma_1, \overline{\sigma_1},\overline{\sigma_3}, \overline{\sigma_2},\overline{\sigma_1}, ).$$
\end{minipage}
 \caption{Example of the $8-G$ loop in $\mathcal{E}\otimes  \mathcal{E}$: for $\sigma\in \{+,-\}^3$}
    \label{fig:8shapedgraph}
\end{figure}

\end{definition}

\begin{lemma}[The martingale term]\label{lem:DIfREP}
For any stopping time $T$ with respect to  $H_t$ and  $\tau:= t\wedge T$, we have  
 \begin{equation}\label{alu9_STime}
   \mathbb{E} \left [    \int_{s}^\tau\left(\mathcal{U}_{u, t, \boldsymbol{\sigma}} \circ \mathcal{E}^{(M)}_{u, \boldsymbol{\sigma}}\right)_{\textbf{a}} \right ]^{2p} 
   \le C_{n,p} \; 
   \mathbb{E} \left(\int_{s}^\tau 
   \left(\left(
   \mathcal{U}_{u, t,  \boldsymbol{\sigma}}
   \otimes 
   \mathcal{U}_{u, t,  \overline{\boldsymbol{\sigma}}}
   \right) \;\circ  \;
   \left( \mathcal{E} \otimes  \mathcal{E} \right)
   _{u, \,\boldsymbol{\sigma}  }
   \right)_{{\textbf a}, {\textbf a}}du 
   \right)^{p}
  \end{equation} 
  where  $\overline{\boldsymbol{\sigma}}$ is the conjugate sign vector of ${\boldsymbol{\sigma}}$. More precisely, as in \eqref{def_Ustz}, 
  $$
  \left[\left(
   \mathcal{U}_{u, t,  \boldsymbol{\sigma}}
   \otimes 
   \mathcal{U}_{u, t,  \overline{\boldsymbol{\sigma}}}
   \right)\circ \cal A\right]_{\textbf{a},\textbf{a}'}=
   \sum_{\textbf{b},\;\textbf{b}'} \; \prod_{i=1}^n \left(\frac{1 - u \cdot m_i m_{i+1} S^{(B)}}{1 - t \cdot m_i m_{i+1} S^{(B)}}\right)_{a_i b_i}
   \;\cdot \;\prod_{i=1}^n \left(\frac{1 - u \cdot \overline{m_i m_{i+1}} S^{(B)}}{1 - t \cdot \overline{m_i m_{i+1}} S^{(B)}}\right)_{a'_i b'_i} \cdot \mathcal{A}_{\textbf{b},\textbf{b}'}.
  $$
\end{lemma}
\begin{proof}[Proof of Lemma \ref{lem:DIfREP}]

We will prove the case $\tau =t$; the general case is identical. 
By Definition \ref{def_Edif}, $
 \mathcal{E}^{(M)}_{t, \boldsymbol{\sigma}, \textbf{a}}=\sum_{\alpha=(i,j)} \mathcal{E}^{(M)}_{t, \boldsymbol{\sigma}, \textbf{a}}(\alpha)\cdot d{\cal B}_\alpha(t) $ 
 and 
 $$
 \left(\mathcal{U}_{u, t, \boldsymbol{\sigma}}\circ \mathcal{E}^{(M)}_{u, \boldsymbol{\sigma}}\right)_{\textbf{a}}=
 \sum_{\alpha=(i,j)}
 \left(
 \mathcal{U}_{u, t, \boldsymbol{\sigma}} \circ\mathcal{E}^{(M)}_{u, \boldsymbol{\sigma}}
 (\alpha)
 \right)_{\textbf{a}}\cdot d{\cal B}_\alpha(t)
 $$
Here with $\quad \alpha=(i,j)$ we have 
$$
\mathcal{E}^{(M)}_{t, \boldsymbol{\sigma},\textbf{a}}
 (\alpha)
 =(S_{ij})^{1/2}\cdot \partial_{(H_t)_{xy}} {\cal L}_{t,\boldsymbol{\sigma},\textbf{a}}.
$$
Using  the chain rule and the structure of $\cal L$, we can write 
 $$
\mathcal{E}^{(M)}_{t, \boldsymbol{\sigma},\textbf{a}}
 (\alpha)=\sum_{k=1}^{n}\;\mathcal{E}^{(M)}_{t, \boldsymbol{\sigma},\textbf{a}}
 (\alpha, k), \quad \quad \mathcal{E}^{(M)}_{t, \boldsymbol{\sigma},\textbf{a}}
 (\alpha, k):=(S_{ij})^{1/2}\cdot  {\cal L}_{t,\boldsymbol{\sigma},\textbf{a}}\Big|\left(G_k\to \partial_{(H_t)_{xy}}G_k\right)$$
 where $ (\alpha, k)$ denotes  the part that  the derivative acts on the $k$-th $G$ edge in the $\cal L$.  By definition  \eqref{defEOTE},  we have 
 $$
 \sum_\alpha   \mathcal{E}^{(M)}_{t, \boldsymbol{\sigma},\textbf{a}}
 (\alpha, k)\cdot \overline {\mathcal{E}^{(M)}_{t, \boldsymbol{\sigma},\textbf{a}'}(\alpha, k)}=
 \left(\mathcal{E}\otimes  \mathcal{E}\right)^{(k)}_{t, \,\boldsymbol{\sigma}, \,\textbf{a},\,\textbf{a}'}
 $$
 Since  $\cal U$ is a deterministic linear operator,   the  quadratic  variation of the martingale  term in \eqref{int_K-LcalE}
 can be bounded by 
 \begin{align} \label{UEUEdif}
  \left[\int \left(\mathcal{U}_{u, t, \boldsymbol{\sigma}} \circ \mathcal{E}^{(M)}_{u, \boldsymbol{\sigma}}\right)_{\textbf{a}} \right ]_t
\;=\; & \int_s^t
\sum_{\alpha}\left|  \mathcal{U}_{u, t, \boldsymbol{\sigma}} \circ\mathcal{E}^{(M)}_{u, \boldsymbol{\sigma}}(\alpha )
 \right|^2 du 
 = \int_s^t
\sum_{\alpha}\left|\sum_{k=1}^n \mathcal{U}_{u, t, \boldsymbol{\sigma}} \circ\mathcal{E}^{(M)}_{u, \boldsymbol{\sigma}}(\alpha, k)
 \right|^2 du
 \nonumber \\
\; \;\le C_n\; &\int_s^t
\sum_{\alpha}\sum_{k=1}^n\left| \mathcal{U}_{u, t, \boldsymbol{\sigma}} \circ\mathcal{E}^{(M)}_{u, \boldsymbol{\sigma}}(\alpha, k)
 \right|^2 du
\nonumber \\
\;=\;C_n \; &\int_{s}^t
\left(\left(
   \mathcal{U}_{u, t,  \boldsymbol{\sigma}}
   \otimes 
   \mathcal{U}_{u, t,  \overline{\boldsymbol{\sigma}}}
   \right) \;\circ  \;
   \left( \mathcal{E} \otimes  \mathcal{E} \right)
   _{u, \,\boldsymbol{\sigma}  }
   \right)_{{\textbf a}, {\textbf a}} du.
\end{align}
Here we have used the Schwarz inequality in expanding the square. 
Our desired result,  \eqref{alu9_STime}, follows from the BDG inequality.

\end{proof}

\subsection{Proof of Theorem \ref{lem:main_ind}, step 2}
In this section,  we focus on the $(+,-)$ $2$-$G$-loop, i.e. $\boldsymbol{\sigma}=(+,-)$. The subscript $\boldsymbol{\sigma}$ will be dropped  in this subsection.  We will prove \eqref{Eq:Gdecay_w} first, and  \eqref{Gt_bound_flow} will be proved at the end of this  section. 
Define the tail functions ${\cal T}^{(\cal L-K)}_u (\ell)$ and ${\cal T}_{u,D} (\ell)$ 
\begin{equation}\label{def_WTu}
      {\cal T}^{(\cal L-K)}_u (\ell):= \max_{a,\,b\; : \; |a-b|\ge \ell} \left| ({\cal L-K})_{u,  \boldsymbol{\sigma},  ( a, b)}\right|,\quad  \boldsymbol{\sigma}=(+,-),
\end{equation}
\begin{equation}\label{def_WTuD}
{\cal T} _{u, D} (\ell) :=  (W\ell_u\eta_u)^{-2}
\exp \left(- \left( \ell /\ell_u \right)_+^{1/2} \right)+W^{-D}.
\end{equation} 
Both  ${\cal T}^{(\cal L-K)}_u$ and ${\cal T} _{u, D}$ are non-decreasing funtions:
 $$
 0\le \ell_1\le \ell_2,\quad \quad 
 {\cal T}^{(\cal L-K)}_u (\ell_1)\ge {\cal T}^{(\cal L-K)}_u (\ell_2),\quad {\cal T} _{u,D} (\ell_1)\ge{\cal T} _{u,D}(\ell_2)
 $$
 Denote  the ratio between them by  
\begin{equation}\label{def_Ju}
     {\cal J} _{u,D} (\ell):=   \left({\cal T}^{(\cal L-K)}_u (\ell) \Big/{\cal T}_{u,D} (\ell)\right)+1
\end{equation} 
We aim to bound $ {\cal J} _{u,D} (\ell)$ for any $u\in [s,t]$ and large $D>0$, i.e., 
 \begin{align}\label{shoellJJ}
    {\cal J}^* _{u,D} :=\max_{\ell}   {\cal J} _{u,D} (\ell)\le (\eta_s/\eta_u)^4 
  \end{align}
Define a  scale parameter 
$$
  \ell^{*}_t := (\log W)^{3/2}\cdot \ell_t .
$$
In this scale $\ell_t^*$,  $\Theta_t$ is exponentially small, whereas ${\cal T}_t$ is not. 
More precisely, we have the following lemma.  
\begin{lemma}\label{lem_dec_calE_0}
For any fixed large $D$ and small $\delta>0$,  
 \begin{align}
 \label{Kell*}
   |b -a|\ge \delta \cdot \ell_t^* &\implies 
   \left({\Theta}_t\right)_{ab } \le W^{-D} ,\quad    \left({\Theta}_s^{-1}{\Theta}_t\right)_{ab } \le W^{-D}
   \\
  |b -a|\ge \delta \cdot \ell_t^* 
 & \implies  {\cal L}_{t, (-,+),(a, b)} \prec {\cal J}^*_{t,D}\cdot{\cal T}_{t,D}(|a-b|) . \label{auiwii}
\end{align}
For any constant $C>0$, \begin{equation}\label{Tell*}
       { {\cal T}_{u,D} \left(\ell-C\cdot \ell_u^*\right)}\;\prec \;  {\cal T}_{u,D} \left(\ell \right).
\end{equation} 

\end{lemma}
\begin{proof} We will only prove \eqref{auiwii}. By definition, 

\[
{\cal L}_{t, (-,+),(a, b)} = \left({\cal L} - \cal K\right)_{t, (-,+),(a, b)} + {\cal K}_{t, (-,+),(a, b)}.
\]
Using the $\Theta$ representation of the 2-\(\cal K\) in \eqref{Kn2sol} and the decay property of \(\Theta\) from \eqref{prop:ThfadC}, we have  for any \(D > 0\) that 
\[
{\cal K}_{t, (-,+),(a, b)} \prec N^{-D} \leq {\cal T}_{t, D}(|a-b|).
\]
Together with the definition \eqref{shoellJJ}  of \(\cal J^*\), \eqref{def_Ju}, and \eqref{def_WTu}, we have proved  \eqref{auiwii}.
We remark that \eqref{auiwii} is significant in that ${\cal T}_{t,D} $ 
is of $\left(W \eta_u \ell_u\right)^{-1}$  smaller than the typical size of $2-{\cal L}$ loop. The reason we gain an extra small factor is due to the assumption  $|b -a|\ge \delta \cdot \ell_t^*$.
\end{proof} 

Our proof of \eqref{Eq:Gdecay_w}   relies on the loop hierarchy \eqref{int_K-L_ST} in the special case $n=2$. 
We begin the analysis of the hierarchy by  bounding  terms in \eqref{int_K-L_ST}.  

\begin{lemma}\label{lem_dec_calE}  Suppose  the assumptions of Theorem  \ref{lem:main_ind} and the conclusions  of Step 1, i.e., \eqref{lRB1} and \eqref{Gtmwc}, 
hold. Assume that 
\begin{align} 
s\le u\le t, \quad  D\ge 10, \; \textbf{a}=(a_1,a_2),\; \textbf{a}'=(a'_1,a'_2),\; \boldsymbol{\sigma}=(+,-), \quad 
\max_i |a_i-a_i'|\le \ell_t ^*. \label{57}
\end{align}
 We  have 
\begin{align}\label{res_deccalE_lk}
       {\cal E}^{( (L-K)\times(L-K) )}
       _{u, \boldsymbol{\sigma},\textbf{a}} \Big/\; 
  {\cal T}_{t,D}(|a_1-a_2|) \quad 
 \;  &\; \prec \; 
   \left(\eta_u\right)^{-1} 
   \cdot  
   \left(W\eta_u\ell_u\right)^{-1}
    \cdot 
 \left({\cal J}^{*}_{u,D} \right)^{2}
 \\
      \mathcal{E}^{( \widetilde{G} )}_{u, \boldsymbol{\sigma},\textbf{a}} 
    \;\Big/\; 
 {\cal T}_{t,D}(|a_1-a_2|)
 \;  &\; \prec \; 
 \left(\eta_u\right)^{-1} 
  \cdot \left(\ell_u/\ell_s\right)^2\cdot{\bf 1}(|a_1-a_2|\le \ell_u^*)
    \nonumber  \\
       \;  &\; + \;  \left(\eta_u\right)^{-1} 
  \cdot
       \left(W \eta_u \ell_u\right)^{-1/3}  \cdot 
       \left({\cal J}^{*}_{u,D} \right)^{3}
 \label{res_deccalE_wG}\\
\label{res_deccalE_dif}
    \left(
   \mathcal{E}\otimes  \mathcal{E} \right)
   _{u, \,\boldsymbol{\sigma},  \,{\textbf a} ,\, {\textbf a}'
     }
     \;\Big/\; 
 \left( {\cal T}_{t,D}(|a_1-a_2|)\right)^2
 \;  &\; \prec \; 
  \left(\eta_u\right)^{-1} 
  \cdot \left(\ell_u/\ell_s\right)^5\cdot{\bf 1}(|a_1-a_2|\le 4\ell_t^*)
     \nonumber \\
       \;  &\; + \;  \left(\eta_t\right)^{-1} 
  \cdot
       \left(W \eta_u \ell_u\right)^{-1/2}  \cdot 
    \left({\cal J}^{*}_{u,D} \right)^{3}
\end{align} 
  \end{lemma}
By assumption  \eqref{con_st_ind},  $s$ and $t$ are near each other.  So the exact exponents  on the right side  of \eqref{res_deccalE_lk}-\eqref{res_deccalE_dif} are not important for our purpose.  We only  need the errors are of the form 
$$
  (\eta_s/\eta_t)^\alpha\cdot \left(W \eta_u \ell_u\right)^{-\beta}\cdot \left({\cal J}^{*}_{u,D} \right) ^{\gamma}
  $$
for some positive constants $\alpha$, $\beta$ and $\gamma$. 

We now provide a power counting to guess the sizes of terms in the previous lemma. Since 
$ {\cal E}^{( (L-K)\times(L-K) )}$ is a higher order term, we will ignore it in the following heuristic. 
Denote by $A_u \sim  W\ell_u\eta_u$. By definition, $\mathcal{E}^{( \widetilde{G} )}_u$ is a product of an $1$-$\cal L -\cal K$ loop and 
a $3$-$\cal L$ loop. We know that  $1$-$\cal L -\cal K$ loop is of order $A_u^{-1}$ and $3$-$\cal L$ loop is of order $A_u^{-2}$. In addition, 
the  summation index in $\mathcal{E}^{( \widetilde{G} )}$ yields a factor $\ell_u$ if we assume the correct decay property. Since there is an additional $W$ factor in $\mathcal{E}^{( \widetilde{G} )}_u$, 
\begin{align}
\mathcal{E}^{( \widetilde{G} )}_u \sim A_u^{-3} W \ell_u = \eta_u^{-1} A_u^{-2}.
\end{align}
The factor $A_u^{-2}$ is exactly the prefactor  in the definition of $ {\cal T}_u$ in \eqref{res_deccalE_wG}.  For $ \mathcal{E}\otimes  \mathcal{E}_u $, it is a $6$-$\cal L$ loop of order $A_u^{-5}$. Hence 
\begin{align}
(\mathcal{E}\otimes  \mathcal{E})_u \sim A_u^{-5} W \ell_u = \eta_u^{-1} A_u^{-4}, 
\end{align}
which explains the order in \eqref{res_deccalE_dif}.  Notice that in both \eqref{res_deccalE_wG} and  \eqref{res_deccalE_dif}, 
we used $ {\cal T}_t$ on the left sides of the equations while both $\mathcal{E}^{( \widetilde{G} )}_u$ and $(\mathcal{E}\otimes  \mathcal{E})_u$ are at the time $u$. 

So far we only used the loop bounds which are consequences of Step 1. 
It remains to understand the last terms in \eqref{res_deccalE_wG} and  \eqref{res_deccalE_dif}. Assuming  Lemma \ref{lem_dec_calE}, we now  prove  \eqref{Eq:Gdecay_w}. We will use extensively the kernel estimates concerning  the operator $\cal U$ in Section \ref{ks}.

\begin{proof}[Proof of \eqref{Eq:Gdecay_w}]
By  assumption \eqref{Eq:Gdecay+IND} on $({ \cal L-\cal K})_s$, the operator norm bound on $\cal U$ in Lemma \ref{lem:sum_Ndecay},  and the  tail estimate  \eqref{neiwuj}, 
we can  bound the first term on the right side  of \eqref{int_K-L_ST} by 
\begin{equation} \label{res_deccalE_0}
\left(\mathcal{U}_{s, t, \boldsymbol{\sigma}} \circ (\mathcal{L} - \mathcal{K})_{s, \boldsymbol{\sigma}}\right)_{\textbf{a}}\;\Big/\; 
  {\cal T}_{t, D}(|a_1-a_2|) \quad 
  \prec \;   (\ell_t/\ell_s)^2\cdot {\bf 1}(|a_1-a_2|\le \ell_t^*) + 1,
\end{equation}
where the last term of order one comes from applying \eqref{neiwuj}. Notice that the expansion factor $\left(\eta_u / \eta_t\right)^2$  from applying Lemma \ref{lem:sum_Ndecay} has become 
$ (\ell_t/\ell_s)^2$ due to the prefactors in $ {\cal T}_{s, D}$ and  ${\cal T}_{t, D}$.
For  any $\textbf{a}$ fixed and any function $f$,  we decompose $f  = f_1  + f_2   $ where 
$f_1 (\textbf{b}) = f (\textbf{b}) {\bf 1} ( \| \textbf{b}-\textbf{a}\|\le \ell_t^*) $. 
From the decay of ${\cal U}_{u,t}$,  $\Big ( {\cal U}_{u,t, \boldsymbol{\sigma}} \circ f_2 \Big)_\textbf{a}$ is exponentially small.  
Hence we only have to bound $  {\cal U}_{u,t, \boldsymbol{\sigma}} \circ f_1$, for which we apply  Lemma \ref{lem:sum_Ndecay}. Therefore,  
we can bound the second  term of \eqref{int_K-L_ST} by 
$$
 \Big({\cal U}_{u,t, \boldsymbol{\sigma}}\circ  {\cal E}^{( (L-K)\times(L-K) )}_{u, \boldsymbol{\sigma}}\Big)_\textbf{a}
 \prec (\eta_u/\eta_t)^2 \max_{\| \textbf{b}-\textbf{a}\|\le \ell_t^*}{\cal E}^{( (L-K)\times(L-K) )}_{u, \boldsymbol{\sigma}, \textbf{b}}+W^{-D}, 
$$
$$
 \|\textbf{b}-\textbf{a}\|=\max_i |b_i-a_i|\le \ell_t^*.
 $$
Under the last  condition, \eqref{Tell*} implies that 
$$
 {\cal T}_{t, D}(|b_1-b_2|)\prec  {\cal T}_{t, D}(|a_1-a_2|).
$$
Using  ${\cal E}^{( (L-K)\times(L-K) )}$ estimate  \eqref{res_deccalE_lk} and $(s,t)$-condition  \eqref{con_st_ind}, we have 
\begin{align} 
\label{res_deccalE_1}
     \Big({\cal U}_{u,t, \boldsymbol{\sigma}}\circ  {\cal E}^{( (L-K)\times(L-K) )}_{u, \boldsymbol{\sigma}}\Big)_\textbf{a}\;\Big/\; 
 {\cal T}_{t, D}(|a_1-a_2|) \quad 
  & \prec \; 
  \frac{1}{\eta_u}\left(\eta_u / \eta_t\right)^2 
   \cdot  
   \left(W\eta_u\ell_u\right)^{-1 }
    \cdot 
\left({\cal J}^*_{u,D}\right)^2
\end{align} 
Similarly,   the estimate \eqref{res_deccalE_wG} on ${\cal E}^{(\widetilde{G} )}$ implies   
\begin{align}  \label{res_deccalE_2}
     \Big(\mathcal{U}_{u,t, \boldsymbol{\sigma}} \circ  \mathcal{E}^{( \widetilde{G} )}_{u, \boldsymbol{\sigma}}\Big)_{\mathbf{a}} 
    \;\Big/\; 
{\cal T}_{t, D}(|a_1-a_2|)
 &\prec \;  {\bf1}(|a_1-a_2|\le 3\ell_t^*)\cdot  \frac{1} {\eta_u} \cdot \left(\eta_u / \eta_t\right)^2\cdot \left(\ell_u / \ell_s\right)^2 
 \nonumber \\
  &+\;    \frac{1}{\eta_u} 
  \cdot\left(\eta_u / \eta_t\right)^2 \cdot
       \left(W \eta_u \ell_u\right)^{-1/3}  \cdot \left({\cal J}^{*}_{u,D} \right)^{3 }, 
       \end{align} 
and  the  estimate  \eqref{res_deccalE_dif} on $\cal E\otimes\cal E$  implies  
\begin{align} \label{res_deccalE_3}
  \left(\left(
   \mathcal{U}_{u, t,  \boldsymbol{\sigma}}
   \otimes 
   \mathcal{U}_{u, t,  \overline{\boldsymbol{\sigma}}}
   \right) \;\circ  \;
   \left( \mathcal{E} \otimes  \mathcal{E} \right)
   _{u, \,\boldsymbol{\sigma}  }
   \right)_{{\textbf a}, {\textbf a}}
     \;\Big/\; 
 \left({\cal T}_{t, D}(|a_1-a_2|)\right)^2
 &\prec \;  {\bf1}(|a_1-a_2|\le 6\ell_t^*)\cdot  \frac{1} {\eta_u} \cdot \left(\eta_u / \eta_t\right)^4\cdot \left(\ell_u / \ell_s\right)^5 
 \nonumber \\
 &+ \;  \frac{1}{\eta_u} 
  \cdot
       \left(W \eta_u \ell_u\right)^{-1/3}  \cdot \left({\cal J}^{*}_{u,D} \right)^{3 }  
\end{align}
In the last inequality, we have absorbed the  expansion factor $\left(\eta_u / \eta_t\right)^4$ by the change of the exponent in $   \left(W \eta_u \ell_u\right)$ from $-1/2$ to $-1/3$.

We now insert  these bounds   into the $2$-$G$-loop equation \eqref{int_K-L_ST} and 
bound  the martingale term by \eqref{alu9_STime}.
Denote by $T$ the stopping time 
\begin{align}
T:=\min \{u: {\cal J}^*_{u,D} \ge \left(\eta_s / \eta_t\right)^4\}
\end{align}
and set $\tau = T \wedge t$. 
Clearly, the stopped versions of Lemma \ref{lem_dec_calE} and the previous bounds in this proof are valid by similar arguments. 
The quadratic variation of the stopped  martingale term is then bounded by 
\begin{align}\label{51}
& \int_s^\tau  \left( \left(\mathcal{U}_{u,t,  \boldsymbol{\sigma}}
   \otimes 
   \mathcal{U}_{u, t,  \overline{\boldsymbol{\sigma}}}
   \right) \;\circ  \;
   \left( \mathcal{E} \otimes  \mathcal{E}  \right)
   _{u, \,\boldsymbol{\sigma},\, \overline{\boldsymbol{\sigma}} }
   \right)_{{\textbf a}, {\textbf a}}  d u
 \Big/{\cal T}_{t, D}(|a_1-a_2|)^2  \nonumber \\
 \;\prec \; &   \int_s^\tau   d u \Big \{ 
 {\bf1}(|a_1-a_2|
\le 6\ell_t^*)\cdot  \frac{1} {\eta_u} \cdot \left(\eta_u / \eta_t\right)^4\cdot \left(\ell_u / \ell_s\right)^5 
+ \;  \frac{1}{\eta_u} 
  \cdot
       \left(W \eta_u \ell_u\right)^{-1/3}  \cdot \left({\cal J}^{*}_{u,D} \right)^{3 }  \Big \}   \nonumber \\
\;\le \; &    \big [ (\eta_s/\eta_t)^{4}  \cdot {\bf 1}(|a_1-a_2|\le 6\ell_t^*)+1 \big ].
 \end{align}
Applying  \eqref{alu9_STime}, we have 
 $$
 \int_s^\tau\left(\mathcal{U}_{u, t, \sigma} \circ \mathcal{E}_{u, \sigma}^{(M)}\right)_{\textbf{a}}\;\Big/{\cal T}_{t, D}(|a_1-a_2|)
 \prec \big [ (\eta_s/\eta_t)^{2}  \cdot {\bf 1}(|a_1-a_2|\le 6\ell_t^*)+1 \big ].
 $$
 This statement holds for any $t\ge s$. Furthermore, due to the monotonicity of $\ell$ and $\eta$, for any $s\le t'\le t$ we have
\begin{align}\label{bsluwous1}
  \int_s^{T\wedge t'}\left(\mathcal{U}_{u,t', \sigma} \circ \mathcal{E}_{u, \sigma}^{(M)}\right)_{\textbf{a}}\;\Big/{\cal T}_{t', D}(|a_1-a_2|)
   & \; \prec \; \big [ (\eta_s/\eta_{t'})^{2}  \cdot {\bf 1}(|a_1-a_2|\le 6\ell_{t'}^*)+1 \big ]
 \\\nonumber
  & \; \prec \;  \big [ (\eta_s/\eta_{t })^{2}  \cdot {\bf 1}(|a_1-a_2|\le 6\ell_{t }^*)+1 \big ]
 \end{align}
 We extend this bound to a $N^{-C}$ net between $s$ and $t$, i.e., $t^{(k)}=s+(t-s)\cdot k\cdot N^{-C}$   for a fixed large $C$.  
 \begin{align}\label{bsluwous2}
  & \max_k\int_s^{T\wedge t^{(k)}} \left(\mathcal{U}_{u,\, t^{(k)},\, \sigma} \circ \mathcal{E}_{u, \sigma}^{(M)}\right)_{\textbf{a}}\;\Big/{\cal T}_{t^{(k)}, \,D}(|a_1-a_2|) \prec  \big [ (\eta_s/\eta_{t })^{2}  \cdot {\bf 1}(|a_1-a_2|\le 6\ell_{t }^*)+1 \big ]
 \end{align}
For any $k$, we can rewrite  the last term in   \eqref{int_K-L_ST} as follows: 
$$
 {\cal U}_{t,\tau}\circ \int_{s}^\tau \left(\mathcal{U}_{u, t, \boldsymbol{\sigma}} \circ \mathcal{E}^{(M)}_{u, \boldsymbol{\sigma}}\right)_{\textbf{a}}\;=\;{\cal U}_{t^{(k)},\tau }\int_s^{\tau } \left(\mathcal{U}_{u,\, t^{(k)},\, \sigma} \circ \mathcal{E}_{u, \sigma}^{(M)}\right)_{\textbf{a}} 
$$
Choose  a $k$ such that  $t^{(k-1)}\le \tau < t^{(k)}$. Then $T\wedge t^{(k)}= \tau$. 
Since $C$ can be arbitrary fixed large,  $ {\cal U}_{t^{(k)},\tau }$ is  nearly an identity operator. Then
$$
{\cal U}_{t^{(k)},\tau }\int_s^{\tau } \left(\mathcal{U}_{u,\, t^{(k)},\, \sigma} \circ \mathcal{E}_{u, \sigma}^{(M)}\right)_{\textbf{a}} =\int_s^{\tau } \left(\mathcal{U}_{u,\, t^{(k)},\, \sigma} \circ \mathcal{E}_{u, \sigma}^{(M)}\right)_{\textbf{a}} +O(N^{-C/2})
$$
and
$$
{\cal T}_{t^{(k)}, \,D}(|a_1-a_2|)\Big/{\cal T}_{\tau, \,D}(|a_1-a_2|)=1+O(N^{-C/2})
$$
Combining  these two bounds, we have 
$$
 {\cal U}_{t,\tau}\circ \int_{s}^\tau \left(\mathcal{U}_{u, t, \boldsymbol{\sigma}} \circ \mathcal{E}^{(M)}_{u, \boldsymbol{\sigma}}\right)_{\textbf{a}} \; \Big/\; {\cal T}_{\tau, \,D}(|a_1-a_2|) \prec  \big [ (\eta_s/\eta_{t })^{2}  \cdot {\bf 1}(|a_1-a_2|\le 6\ell_{t }^*)+1 \big ]
 $$
 Combining this  bound with  \eqref{res_deccalE_0}, \eqref{res_deccalE_1},   \eqref{res_deccalE_2}  and  \eqref{int_K-L_ST}, we have 
$$
\left({\cal L}-{\cal K}\right)_{\tau, \textbf{a}}\Big/{\cal T}_{t, D}(|a_1-a_2|) \prec   \big [ (\eta_s/\eta_t)^{2}  \cdot {\bf 1}(|a_1-a_2|\le 6\ell_t^*)+1 \big ], 
$$
where we have used the initial condition ${\cal J}^*_{s,D} \prec 1$ from \eqref{Eq:L-KGt+IND} and \eqref{Eq:Gdecay+IND}.

This implies that 
\begin{align}\label{52}
 {\cal J}^*_{\tau,D}\prec     (\eta_s/\eta_t)^{2}  
\end{align}
Hence  ${\mathbb P} (T\le t ) $ is negligible and we have completed  the proof of \eqref{Eq:Gdecay_w}.
 Notice that we have also proved 
\begin{align}\label{53}
\left({\cal L}-{\cal K}\right)_{t, \textbf{a}}\Big/{\cal T}_{t, D}(|a_1-a_2|) \prec   \big [ (\eta_s/\eta_t)^{2}  \cdot {\bf 1}(|a_1-a_2|\le 6\ell_t^*)+1 \big ].
\end{align}

\end{proof}

 \begin{proof}[Proof of Lemma \ref{lem_dec_calE}]
\emph{Proof of \eqref{res_deccalE_lk}.}  
We first  note the monotonicity properties 
$$
u\le t\implies \ell_u\le \ell_t,   \quad {\cal T}_{u,D}\le {\cal T}_{t,D} .$$
By definition, 
\begin{align} {\cal E}^{( ({\cal L-K})\times({\cal L-K}) )}_{u,  \boldsymbol{\sigma},  \textbf{a} }
= &W\sum_{b_1,b_2}({\cal L-K})_{u,
\boldsymbol{\sigma}, 
( a_1, b_1)
}
S^{(B)}_{b_1,b_2}
({\cal L-K})_{u,  \boldsymbol{\sigma},  ( b_2, a_2)}
\end{align}
By definition of $S^{(B)}$,  we have $|b_1-b_2|\le 1$.  
Using 
$$
\int_{0 }^a\exp\left(-\sqrt{(a-x)}-\sqrt x+ \sqrt a\right)dx\le C\approx 6.12
,\quad \text{ and } \quad 
\int_{0 }^\infty\exp\left({-\sqrt x}\right)dx=2
$$
we have 
\begin{align}\nonumber
{\cal E}^{( ({\cal L-K})\times({\cal L-K}) )}_{u,  \boldsymbol{\sigma},  \textbf{a} }
\; \prec \;& W ({\cal J}_{u,D}^*) ^2\cdot 
\sum_{x }  {\cal T}_{u, \,D}(|a_1-x|) \cdot {\cal T}_{u, \,D}(|a_2-x|)   
\\\label{mmxiaoxi}
\; \prec \;&
({\cal J}_{u,D}^*) ^2 \cdot 
 \eta_u^{-1}\cdot (W\ell_u\eta_u)^{-1}\cdot  {\cal T}_{u,D }(|a_1-a_2|) .
\end{align} 
We  have thus proved   \eqref{res_deccalE_lk}.

\medskip
\noindent 
\emph{Proof of   \eqref{res_deccalE_wG}}.    By definition, we  write 
\begin{align}
 \max_{\textbf{a}}{\cal E}^{(\widetilde G)}_{u,  \boldsymbol{\sigma},  \textbf{a} }
 \prec W \sum_{b_1,b_2} \langle \widetilde G_u E_{b_1}\rangle\cdot S^{(B)}_{b_1,b_2}\cdot  {\cal L}_{u, (-,+,+),(a_1, b_2, a_2)}+c.c.,\quad \widetilde{G}=G-m
 \end{align}
By \eqref{GavLGEX} and \eqref{Gtmwc}, we have 
$$
\langle \widetilde G_u E_{b_1}\rangle\prec \Xi^{(\cal L)}_{u,2}\cdot (W\ell_u\eta_u)^{-1}\prec  (W\ell_u\eta_u)^{-1}+ {\cal J}^*_{u,D}\cdot (W\ell_u\eta_u)^{-2}.
$$
Therefore 
 \begin{align}\label{jiizziy}
\max_{\textbf{a}}{\cal E}^{(\widetilde G)}_{u,  \boldsymbol{\sigma},  \textbf{a} } \prec \frac{1}{ \ell_u\eta_u} \cdot \sum_{b }   \left|{\cal L}_{u, (-,+,+),(a_1, b , a_2)}\right|(1+{\cal J}^*_{u,D}(W\ell_u\eta_u)^{-1})
\end{align}
Consider first the case  $|a_1-a_2|\le \ell_u^*$.  Denote 
$$
\ell_u^{**}:= \ell_u\cdot (\log W)^3=\ell^*_u\cdot (\log W)^{3/2}
$$
Clearly,  ${\cal T}_{u,D}(\ell_u^{**})\prec W^{-D}$. 
For $|b-a_1|\le  \ell_u^{**}$, 
the loop bound \eqref{lRB1}  proved in the Step 1 implies that 
\begin{align}\label{alkkj}
 \sum_{b }{\bf 1}\big(|b-a_1|\le  \ell_u^{** }\big)\cdot  \left|{\cal L}_{u, (-,+,+),(a_1, b , a_2)}\right|\prec (\ell_u/\ell_s)^2\cdot (W\ell_u\eta_u)^{-2} \cdot \ell_u   
\end{align}
where the summation over $b$ provides a factor $ \ell_u^{** } \prec\ell_u$. 
If $|b-a_1|\ge \ell_u^{**}$, we can easily bound 
 $$
 \left|{\cal L}_{u, (-,+,+),(a_1, b , a_2)}\right|\prec \max_{x_1\in {\cal I}_{a_1},y\in {\cal I}_{b}}|G_{x_1,y}|.
$$
Using  \eqref{GijGEX}, we have that 
$$
\max_{x_1\in {\cal I}_{a_1},y\in {\cal I}_{b}}|G_{x_1,y}|\prec \sum_{a_1', b'}\left({\cal L}_{u,(+,-),{(a_1',b')}}\right)^{1/2}{\bf1}\left(|a'_1-a_1|\le 1, |b-b'|\le 1\right).
$$
Since $|b-a_1|\ge \ell_u^{** }$, by  \eqref{auiwii} and ${\cal T}_{u,D}(|a_1-b|)\prec W^{-D}$, we obtain that  
$$
{\cal L}_{u,(+,-),{(a_1',b')}} \prec {\cal J}^*_{u,D}\cdot W^{-D}.
$$
Therefore, the contribution from $|b-a_1|\ge  \ell_u^{** }$ part is negligible in the sense 
\begin{align}\label{alkkj2}
 \sum_{b }{\bf 1}\big(|b-a_1|\ge  \ell_u^{** }\big)\cdot  \left|{\cal L}_{u, (-,+,+),(a_1, b , a_2)}\right|\prec  {\cal J}^*_{u,D}\cdot W^{-10}. 
\end{align}
It  is easy to check that the contribution of the last term does not affect the argument given below.  Furthermore,  the contributions from both  \eqref{alkkj} and  \eqref{jiizziy} are bounded by 
\begin{align}\label{82jjdopasj}
 |a_1-a_2|\le \ell_u^*\implies  \mathcal{E}_{u, \boldsymbol{\sigma}, \mathbf{a}}^{(\widetilde{G})} \Big/{\cal T}_{u,D}(|a_1-a_2|)\prec 
 (\eta_u)^{-1} \cdot \left(\ell_u / \ell_s\right)^2\cdot   \left(1+  (\mathcal{J}_{u, D}^*)^2 \left(W \ell_u \eta_u\right)^{-1 }\right).
 \end{align}
This implies that   \eqref{res_deccalE_wG} holds for $|a_1-a_2|\le \ell_u^*$.

For   $|a_1-a_2|\ge \ell_u^*$,  we split it into  two cases  
 $$
 (1): \min_i|a_i-b|\le \ell_u^*/2,\quad \quad 
 (2): \min_i|a_i-b|\ge \ell_u^*/2.
 $$
In the first case,  we assume without loss of generality that $|a_1-b|\le \ell_u^*/2$. Applying  the Schwarz inequality to  the two $G$ edges  in ${\cal L}_{u, (-,+,+),(a_1, b , a_2)}$ connecting the block $a_2$,  we have 
$$
G_{x_1x_2}G_{x_2y}G_{yx_1}\prec |G_{x_1x_2}|^2|G_{yx_1}|+
|G_{ x_2 y}|^2|G_{yx_1}|.
$$
Therefore, we have 
\begin{align}\label{kskjw}
  {\cal L}_{u, (-,+,+),(a_1, b , a_2)}
  \;\prec\; &\;  W^{-1} \Big ( {\cal L}_{u, (-,+),(a_1, a_2)} \ \max_{x_1\in {\cal I}_{a_1}}\sum_{y\in {\cal I}_{b}}|G_{x_1y}|+ {\cal L}_{u, (-,+),(a_2, b)}\ \max_{y\in {\cal I}_{b}}\sum_{x_1\in {\cal I}_{a_1}}|G_{x_1y}| \Big )  .
\end{align}
Using \eqref{GijGEX},  the $\cal K$ bound and the trivial fact $\sqrt { 1+c } \le 1+c $,  we can bound $|G_{x_1x_2}|$ by 
$$
{\bf 1}(x_1\ne x_2)|G_{x_1x_2}|\prec (W\ell_u\eta_u)^{-1/2}(1+{\cal J}^*_{u,D}(W\ell_u\eta_u)^{-1}).
$$
It implies that for any $a $ and $b$, 
\begin{align}\label{jaysw2002}
  W^{-1}  \max_{y\in{\cal I}_b}\sum_{x\in{\cal I}_a}|G_{xy}|\prec (W\ell_u\eta_u)^{-1/2}(1+{\cal J}^*_{u,D}(W\ell_u\eta_u)^{-1}).
\end{align}
Inserting  this bound  into \eqref{kskjw}, we have 
\begin{align}
  {\cal L}_{u, (-,+,+),(a_1, b , a_2)}
     \;\prec\; &\;\left( {\cal L}_{u, (-,+),(a_1, a_2)}   + {\cal L}_{u, (-,+),(a_2, b)}  \right)  \cdot (W\ell_u\eta_u)^{-1/2}(1+{\cal J}^*_{u,D}(W\ell_u\eta_u)^{-1})
\end{align}
 Since $|a_2-b|\ge \ell_u^*/2$ and $|a_1-a_2|\ge  \ell_u^*$,  we can estimate  $ {\cal L}_{u, (-,+),(a_1, a_2)}   + {\cal L}_{u, (-,+),(a_2, b)}$ 
with \eqref{auiwii} to have  
$$
 {\cal L}_{u, (-,+,+),(a_1, b , a_2)} \;\prec\; {\cal J}^*_{u,D}\cdot \left(  {\cal T}_{u,D}(|a_1-a_2|)+{\cal T}_{u,D}(|b-a_2|) \right)  \cdot (W\ell_u\eta_u)^{-1/2}(1+{\cal J}^*_{u,D}(W\ell_u\eta_u)^{-1}).
 $$
Furthermore since $|b-a_2|\ge |a_1-a_2|-|a_1-b|\ge |a_1-a_2|-\ell_u^*/2$, then
$${\cal T}_{u,D}(|b-a_2|) \prec  {\cal T}_{u,D}(|a_1-a_2|).
$$
Combining these bounds,  we obtain,   for $|a_1-a_2|\ge \ell_u^*$, that  
\begin{align}
    \label{lwjufw}
    \sum_{b } ^{\text{Case } 1} \left|{\cal L}_{u, (-,+,+),(a_1, b , a_2)}\right|\prec \left({\cal J}^*_{u,D}\right)^2\cdot{\cal T}_{u,D}(|a_1-a_2|) \cdot \ell_u(W\ell_u\eta_u)^{-1/2} .
\end{align}
For case (2),  we bound  ${\cal L}_{u, (-,+,+),(a_1, b , a_2)}$ as follows
\begin{align}\label{LGGGXXX}
{\cal L}_{u, (-,+,+),(a_1, b , a_2)}
\prec \max_{x_1,x_2, y} \left|G_{x_1y}\right| \left|G_{yx_2}\right|\left|G_{x_1x_2}\right|{\bf 1}(x_1\in {\cal I}_{a_1}){\bf 1}(x_2\in {\cal I}_{a_2}){\bf 1}(y\in {\cal I}_{b}).
\end{align}
 In this case, since $a_1$, $a_2$ and $b$ are all different,  $x_1$, $x_2$ and $y$ must be different.

 Next we use   \eqref{GijGEX} to bound a single $G$ with the $2$-$G$-loops.   Because $a_1$, $a_2$ and $b$ are away from each other by $\ell_u^*/2$, we have, by  \eqref{auiwii}, 
 \begin{align}\label{GGTLJ}
 \left|G_{yx_2}\right|\cdot \left|G_{x_1x_2}\right| \cdot
\left|G_{x_1y}\right|
\;\prec \; & {\cal J}_*^{3/2} \Big( {\cal T}_{u,D}(|a_1-a_2|) {\cal T}_{u,D}(|a_1-b|){\cal T}_{u,D}(|a_2-b|)\Big)^{1/2}
 \end{align}
Then we can bound 
\begin{align}
    \sum_{b}^{ \text {Case } 2 }{\cal L}_{u, (-,+,+),(a_1, b , a_2)} 
 \;\prec \;  &\;{\cal J}_*^{3/2}\Big( {\cal T}_{u,D}(|a_1-a_2|)\Big)^{1/2} \sum_b \Big({\cal T}_{u,D}(|a_1-b|){\cal T}_{u,D}(|a_2-b|)\Big)^{1/2}
 \nonumber \\
  \;\prec \; &\;{\cal J}_*^{3/2}\Big( {\cal T}_{u,D}(|a_1-a_2|)\Big) \cdot  \ell_u\cdot  (W\ell_u\eta_u)^{-1}
\end{align}
where we have used
$$
\int_{0 }^a\exp\left(-\sqrt{(a-x)}/2-\sqrt x/2+ \sqrt a/2\right)dx\le C .
$$
Combining  this estimate with \eqref{lwjufw}, we obtain that if $|a_1-a_2|\ge \ell_u^*$ then
\begin{align}
    \sum_{b} {\cal L}_{u, (-,+,+),(a_1, b , a_2)} 
  \;\prec \; &\;{\cal J}_*^{2}\Big( {\cal T}_{u,D}(|a_1-a_2|)\Big) \cdot  \ell_u\cdot  (W\ell_u\eta_u)^{-1/2}
\end{align}
Inserting  it into \eqref{jiizziy} and using  \eqref{82jjdopasj}, we have completed the proof of  \eqref{res_deccalE_wG}.

 \medskip
\emph{Proof of \eqref{res_deccalE_dif}. } By definition \eqref{defEOTE},  $(\mathcal{E} \otimes \mathcal{E})$ is the sum of  $(\mathcal{E} \otimes \mathcal{E})^{(k)}$, and the latter ones can be written in terms of the following loops $
{\cal L}^{(k )}:= {\cal L}_{u, \boldsymbol{\sigma}^{(k ) },\textbf{a}^{(k )}} $. Here
$ \boldsymbol{\sigma}^{(1 )}=( +,-,+,-,+,-)$, $\textbf{a}^{(1 )}=( a_1,a_2, b', a_2',a_1',b)$
and $\boldsymbol{\sigma}^{( 2)}= $ $(-,+,-,+,-,+)$, $\textbf{a}^{( 2)}=( a_2,a_1, b', a_1',a_2',b)$. Hence  by symmetry, we only need to prove \eqref{res_deccalE_dif} for $(\mathcal{E} \otimes \mathcal{E})^{(1)}$, i.e.,  $k=1$ case. Notice that $|b-b'| = 1$ in this case and 
we can treat $b=b'$ for all practical purpose in the following proof. 
 
\begin{figure}[ht]
\centering
\begin{tikzpicture}[scale=1.2, every node/.style={font=\small}]

\node (a1) at (0, 2) {$\bullet$};
\node[above left] at (a1) {$a_1$};

\node (a2) at (4, 2) {$\bullet$};
\node[above right] at (a2) {$a_2$};

\node (b) at (1.5, 1) {$\bullet$};
\node[below] at (b) {$b$};

\node (b') at (2.5, 1) {$\bullet$};
\node[below] at (b') {$b'$};

\node (a1') at (0, 0) {$\bullet$};
\node[below left] at (a1') {$a_1'$};

\node (a2') at (4, 0) {$\bullet$};
\node[below right] at (a2') {$a_2'$};

\draw[thick] (a1) -- (b) node[midway, left] {$+$};
 \draw[thick] (b') -- (a2) node[midway, right] {$+$};
\draw[thick] (b') -- (a2') node[midway, right] {$-$};
 \draw[thick] (b) -- (a1') node[midway, left] {$-$};
\draw[thick] (a1) -- (a2) node[midway, above] {$-$};
\draw[thick] (a1') -- (a2') node[midway, below] {$+$};
\draw[decorate, decoration={snake, amplitude=0.4mm, segment length=2mm}, thick] (b) -- (b')
        node[midway, above] {$S^{(B)}$};

\node (a1r) at (6, 2) {$\bullet$};
\node[above left] at (a1r) {$a_2$};

\node (a2r) at (10, 2) {$\bullet$};
\node[above right] at (a2r) {$a_1$};

\node (br) at (7.5, 1) {$\bullet$};
\node[below] at (br) {$b$};

\node (b'r) at (8.5, 1) {$\bullet$};
\node[below] at (b'r) {$b'$};

\node (a1r') at (6, 0) {$\bullet$};
\node[below left] at (a1r') {$a_2'$};

\node (a2r') at (10, 0) {$\bullet$};
\node[below right] at (a2r') {$a_1'$};

\draw[thick] (a1r) -- (br) node[midway, left] {$-$};
 \draw[thick] (b'r) -- (a2r) node[midway, right] {$-$};
\draw[thick] (b'r) -- (a2r') node[midway, right] {$+$};
 \draw[thick] (br) -- (a1r') node[midway, left] {$+$};
\draw[thick] (a1r) -- (a2r) node[midway, above] {$+$};
\draw[thick] (a1r') -- (a2r') node[midway, below] {$-$};
\draw[decorate, decoration={snake, amplitude=0.4mm, segment length=2mm}, thick] (br) -- (b'r)
        node[midway, above] {$S^{(B)}$};
        
\end{tikzpicture}
\caption{\hspace{.1in} Left one: $k=1$ \hspace{1.4in}Right one: $k=2$}
\label{fig:graph-6Gfor2G}
\end{figure}

\noindent 
{\bf Case 1: $|a_1-a_2|\le 4\ell_t^*$ } 
We split the sum $\sum_{b,b'}$ into two parts 
$$
|b-a_1|\le \ell_u^{** }:=(\log W)^3 \ell_u,\quad |b-a_1|\ge \ell_u^{** } 
$$
Using ${\cal T}_{u,D}(\ell_u^{** }) $ is very small,  one can easily bound 
$$
\sum_{b,b'} {\bf 1}\left(|b-a_1|\ge \ell_u^{** }\right){\cal L}^{(1)}\prec {\cal J}^*_{u,D}\cdot W^{-3},
$$
by arguments similar to those used in \eqref{alkkj2}.
For $|b-a_1|\le  \ell_u^{** } $, we use 
the loop bound \eqref{lRB1} proved in Step 1 to have 
\begin{align}\label{alkkj3}
\sum_{b,b'}{\bf 1}\big(|b-a_1|\le  \ell_u^{** }\big) {\cal L}^{(1)}  \prec (\ell_u/\ell_s)^5\cdot (W\ell_u\eta_u)^{-5} \cdot \ell_u   .
\end{align}
Combining  these two bounds, we obtain that \eqref{res_deccalE_dif} for $|a_1-a_2|\le 4\ell_t^*$. Notice that the application of the loop bound of length $6$ yields a very strong bound $(W\ell_u\eta_u)^{-5}$ which is not easy to see without the loop estimate. 

 \medskip

 \noindent 
{\bf Case 2: $|a_1-a_2|\ge 4\ell_t^*$ .} 
Recall  the assumption  $|a'_i-a_i|\le \ell_t^*$ \eqref{57} and the fact  that we can treat $b=b'$  in the following proof.  
We  split the sum over $b$ (and $b'$) into two parts 
$$
(1): |b -a_1|\le |b-a_2|,\quad (2): |b-a_1|\ge |b-a_2|.
$$
By symmetry, we only consider  the first case.   Similar to \eqref{LGGGXXX} (see also figure \ref{fig:graph-6Gfor2G}), we  bound ${\cal L}^{(1)}$ by the product of four $G$'s and  $G^\dagger E_b G$ as follows: 
\begin{align} 
{\cal L}^{(1)}\le W^{-2}   \max_{ x_2\in {\cal I}_{a _2}}\max_{ y'\in {\cal I}_{b'}}\max_{x_2'\in {\cal I}_{a'_2}} \sum_{x_1\in {\cal I}_{a_1}}\sum_{x'_1\in {\cal I}_{a'_1}}
\left|G_{x_1x_2}\right|
\left|G_{x_2y'}\right|
\left|G_{y'x_2'}\right|
\left|G_{ x_2'x_1'}\right|
\left|\left(G^\dagger E_b G\right)_{ x_1x_1'}\right|.
\end{align}
Using Schwarz's inequality (on $\sum_{x_1,x'_1}$) and the definition $\cal L$, we have
\begin{align}\label{L6G6X}
{\cal L}^{(1)}\le \max^*_{ x_1, x_2, y', x_2, x_2'}  
\left|G_{x_1x_2}\right|
\left|G_{x_2y'}\right|
\left|G_{y'x_2'}\right|
\left|G_{ x_2'x_1'}\right|
\cdot \left(\max_{ \textbf{a}}\max_{\boldsymbol{\sigma}\in \{+,-\}^4 }{\cal L} _{\boldsymbol{\sigma}, a}\right)^{1/2}, 
\end{align}
where the $\max $ is over  the condition 
$$
 {\bf 1}(x_1\in {\cal I}_{a_1}){\bf 1}(x_2\in {\cal I}_{a_2}){\bf 1}(y'\in {\cal I}_{b'}){\bf 1}(x'_1\in {\cal I}_{a'_1}){\bf 1}(x'_2\in {\cal I}_{a'_2}).
$$
We claim that 
\begin{align} \label{GGTLJ4G}
|G_{y'x_2}|\cdot |G_{y'x_2'}| 
\cdot 
|G_{x_1x_2}|\cdot |G_{x_1'x_2'}|
\prec 
\left({\cal J}^*_{u,D}\right)^2\cdot 
{\cal T} _{t,D}(|a_1-a_2|)\cdot {\cal T} _{t,D}( |b-a_2|)  .  
\end{align}
To prove this bound, we split it into two cases:  
\begin{align}
(1a): \;&\; |a_1-b|\le \ell_u^*,\quad \text{or}\quad |a_1'-b|\le \ell_u^*\nonumber \\
(1b):   \;&\;  |a_1-b|\ge \ell_u^*,\quad and\quad |a_1'-b|\ge \ell_u^*.
\end{align} 
Since $|a_1-a_2|\ge 4\ell_t^*$, we have,  similar to \eqref{GGTLJ}, that 
\begin{align} 
 |G_{x_1x_2}|^2\le  & {\cal J}^*_{u,D}\cdot {\cal T}_{u,D}(|a_1-a_2|), \\
 |G_{x_1'x_2'}|^2\le & {\cal J}^*_{u,D}\cdot {\cal T}_{u,D}\left( |a_1'-a_2'|\right)  \prec {\cal J}^*_{u,D}\cdot {\cal T}_{t,D}\left( |a_1-a_2|\right), 
\end{align} 
where  we have used ${\cal T}_u\le {\cal T}_t$, the assumption $|a_i'-a_i|\le \ell_t^*$ and \eqref{Tell*} in the   second inequality.  Thus we have 
$$
|G_{x_1x_2}|\cdot |G_{x_1'x_2'}|\prec   {\cal J}^*_{u,D}\cdot {\cal T}_{t,D}\left( |a_1-a_2|\right).
$$ 
For edges connecting with $b'$, we have 
$$
|G_{y'x_2}|\cdot |G_{y'x_2'}| \prec 
{\cal J}^*_{u,D}\cdot  {\cal T}^{1/2}_{u,D}\left( |b-a_2| \right)\cdot   {\cal T}^{1/2}_{u,D}\left(  |b-a_2'| \right),
$$
where we have used that  both  $|b-a_2|$ and $|b-a_2'|$ are larger than $\ell_t^*\ge \ell_u^*$.
Using  ${\cal T}_u\le {\cal T}_t$ and $|a_2-a_2'|\le \ell_t^*$, we have 
$$
|G_{y'x_2}|\cdot |G_{y'x_2'}|\le {\cal J}^*_{u,D}\cdot  {\cal T} _{t,D}\left(  |b-a_2| \right).
$$
Combining  these  bounds, we have proved \eqref{GGTLJ4G}. 

Next, applying  \eqref{lRB1} on $\cal L$ term in \eqref{L6G6X}, together with \eqref{GGTLJ4G} we obtain that 
$$
{\cal L}^{(1)}\prec (\ell_u/\ell_s)^{3/2}\left(W\ell_u\eta_u\right)^{-3/2}\cdot \left({\cal J}^*_{u,D}\right)^2\cdot 
{\cal T} _{t,D}(|a_1-a_2|)\cdot {\cal T} _{t,D}( |b-a_2|)  . 
$$
Consider the case (1a) so that   $|a_2-b|\ge |a_1-a_2|-\ell_t^*$. Therefore, ${\cal T} _{t,D}\left(  |b-a_2| \right) \prec 
{\cal T} _{t,D}\left(  |a_1-a_2| \right)$. 
Combining  these bounds  with \eqref{GGTLJ4G}, we have 
\begin{align}
  \sum_{b,b'}^{(1a)} {\cal L}^{(1)}
 \prec &(\ell_u/\ell_s)^{3/2} \cdot \ell_u \cdot\left(W \ell_u \eta_u\right)^{-3/2}
\cdot\left({\cal J}^*_{u,D}\right)^2  \cdot \left({\cal T}_{t,D}(|a_1-a_2|\right) ^2.
\end{align}
For (1b), we bound  $(G^\dagger E_bG)_{x_1x_1'}$ by 
$$
(G^\dagger E_bG)_{x_1x_1'}\le \max_{y\in {\cal I}_b}|G_{x_1y}||G_{x'_1y}|.
$$
Since  $ |a_1-b|,\; |a_1'-b| \ge \ell_u^*$,    we have, similar to \eqref{GGTLJ}, 
$$
(G^\dagger E_bG)_{x_1x_1'}
\prec {\cal J}^*_{u,D} \cdot  {\cal T}_{t,D}(|b-a_1|) .
$$
Together  with \eqref{GGTLJ4G}, we have  
\begin{align}
  \sum_{b,b'}^{(1b)} {\cal L}^{(1)}
 \prec\; & \left( {\cal J}^*_{u,D}\right)^3 \cdot  {\cal T}_{t,D}(|a_1-a_2|)\cdot \sum_b {\cal T}_{t,D}(|b-a_1|) {\cal T}_{t,D}(|b-a_2|)
 \nonumber \\
\prec \; &  \left( {\cal J}^*_{u,D}\right)^3 \cdot \ell_t \cdot (W\ell_t\eta_t)^{-2} 
\cdot \left({\cal T}_{t,D}(|a_1-a_2|)\right)^{2}.
 \end{align}
Here the last line was bounded as in the proof of \eqref{mmxiaoxi}.  Putting  the bounds for (1a) and (1b) together, we obtain \eqref{res_deccalE_dif} and complete the proof of Lemma  \ref{lem_dec_calE}. 
\end{proof}

\noindent 
\begin{proof}[Proof of \eqref{Gt_bound_flow}] Combining  the $\cal L-\cal K$ estimate  \eqref{Eq:Gdecay_w} and the $\cal K$ bound in Lemma \ref{ML:Kbound}, we obtain that
\begin{equation}
\label{lk2safyas}
 \max_{a, b}{\cal L }_{u, (+,-), (a,b)}\prec \left(\eta_s/\eta_u\right)^4\cdot (W\ell_u\eta_u)^{-2}+(W\ell_u\eta_u)^{-1}\prec (W\ell_u\eta_u)^{-1}  ,\quad u\in [s,t]    
\end{equation}
where we have used the inductive assumption  \eqref{con_st_ind}. 
Then we can apply  the resolvent entry estimate in Lemma \ref{lem_GbEXP}.  The weak local law in \eqref{Gtmwc} implies the assumptions of  Lemma \ref{lem_GbEXP}. From \eqref{GiiGEX}, \eqref{GijGEX} and  \eqref{lk2safyas}, we have 
$$
\|G_t-m\|^2_{\max}\prec  \max_{a, b}{\cal L }_{u, (+,-), (a,b)}\prec (W\ell_u\eta_u)^{-1} .
$$
This completes the proof of \eqref{Gt_bound_flow}.   
\end{proof}

\subsection{Fast decay property}
In the second step of the proof for Theorem \ref{lem:main_ind}, we established the decay property of the $2\text{-}G$ loop in \eqref{Eq:Gdecay_w}. In this subsection, we extend this decay property to general loops $\mathcal{L}$ and $\mathcal{L} - \mathcal{K}$. This decay property proves to be highly useful for bounding the $\mathcal{E}$ terms in equation \eqref{int_K-LcalE}. We begin by defining the $\ell_u$-decay property.

\begin{definition}[The decay property]\label{Def_decay}
   Let  $\cal A$ be a tensor $\mathbb Z_L^n\to \mathbb R$.We say $\cal A$ has $(u, \tau, D)$ decay at the  time $u$ if  for some fixed small $\tau>0$ and large $D>0$, we have  
\begin{equation}\label{deccA}
    \max_i\|a_i-a_j\|\ge \ell_u W^{\tau} \implies  {\cal A}_{\textbf{a}}=O(W^{-D}),\quad \textbf{a}=(a_1,a_2\cdots, a_n) .
\end{equation}
\end{definition}
\begin{lemma}[The decay property of ${\cal L}_u$]\label{lem_decayLoop}  
Assume that \eqref{Gt_bound_flow} and \eqref{Eq:Gdecay_w} hold. Then  for 
any $n\ge 2$, any small $\tau$, large $D$ and $D'>0$,  the  ${\cal L}_u$ has $(u, \tau, D)$ decay  with probability $1-O(W^{-D'})$. More precisely, 
\begin{align}\label{res_decayLK}
\mathbb P\left( \max_{\boldsymbol{\sigma}}  \Big(\left|{\cal L}_{u,\boldsymbol{\sigma},\textbf{a}}\right|+\left|{(\cal L-K})_{u,\boldsymbol{\sigma},\textbf{a}}\right|\Big)\cdot{\bf 1}\left(\max_{  i, j } \|a_i-a_j\|\ge \ell_u W^{\tau}\right) \ge W^{-D}\right)\le W^{-D'} .
\end{align}
 \end{lemma}
 \begin{proof}[Proof of Lemma \ref{lem_decayLoop}] Combining  the decay properties  of ${\cal K}_{u,(+,-)}$ and $\cal L-\cal K$  \eqref{Eq:Gdecay_w}, we obtain that  
 ${\cal L}_{u,(+,-)}$ has $(u, \tau, D)$ decay property  for any fixed $(\tau, D)$  with high probability.
 Applying the $G_{ij}$ estimate in \eqref{GijGEX} and  the new ${\cal L}_u$ decay property, we have 
 $$
   \forall \tau, D, D', \quad \mathbb P\left( \max_{i\in {\cal I}_{a_1}}\max_{j\in {\cal I}_{a_2}}|G_{ij}|
 \cdot{\bf 1}\left( |a_1-a_2\|\ge \ell_u W^{\tau}\right) \ge W^{-D}\right)\le W^{-D'} .
 $$
By definition  ${\cal L}=\langle \prod\limits_{i=1}^n G_i E_{a_i}\rangle $  has the $(u, \tau, D)$ decay property.  On the other hand, by the decay of $\Theta^{(B)}_t$ and  the tree representation of $\cal K$ in Lemma \ref{eq_trcalK},  $\cal K$ also   has $(u, \tau, D)$ decay property. Therefore $({\cal L-\cal K})_u$ has $(u, \tau, D)$ decay property. 
 
 \end{proof}

The decay property enables us to use  Lemma \ref{lem:sum_decay} for $\|{\cal U}_{u,t}\|_{\max\to \max}$ on fast decay tensors. This provides a major  improvement over using Lemma \ref{lem:sum_Ndecay} to bound $\|{\cal U}_{u,t}\|_{\max\to \max}$.  We have the following lemma estimating   $\cal E$ terms.

 \begin{lemma}[Bounds on $\cal E$ terms] \label{lem_BcalE} 
Assume  that \eqref{Gt_bound_flow} and \eqref{Eq:Gdecay_w} hold.  
Define $\Xi^{({\cal L})}$ and  $\Xi^{({\cal L}-{\cal K})}$ as follows
\begin{align}\label{def:XIL-K}
  \Xi^{({\cal L} )}_{t, m} &:= \max_{\boldsymbol{\sigma}, \textbf{a}} \left|{\cal L} _{t, \boldsymbol{\sigma}, \textbf{a}} \right|\cdot \left(W\ell_t \eta_t\right)^{m-1} \cdot {\bf 1}(\boldsymbol{\sigma} \in \{+, -\}^m),  
\\\nonumber
     \Xi^{({\cal L}-{\cal K})}_{t, m} &:= \max_{\boldsymbol{\sigma}, \textbf{a}} \left|({\cal L}-{\cal K})_{t, \boldsymbol{\sigma}, \textbf{a}} \right|\cdot \left(W\ell_t \eta_t\right)^{m} \cdot {\bf 1}(\boldsymbol{\sigma} \in \{+, -\}^m).
\end{align}
Then 
\begin{align}\label{CalEbwXi}
\big[\mathcal{K} \sim(\mathcal{L}-\mathcal{K})\big]_{u, \boldsymbol{\sigma}}^{l_{\mathcal{K}}}
   \; \prec \; &
   \left( \max_{k<n }\Xi^{(\cal L-\cal K)}_{u,k}\right)\cdot (W\ell_u\eta_u)^{-n }\cdot \frac{1}{\eta_u}  , 
   \nonumber \\
    \mathcal{E}_{u, \boldsymbol{\sigma}}^{((\mathcal{L} - \mathcal{K}) \times (\mathcal{L} - \mathcal{K}))}
    \; \prec \; &
   \left( \max_{k:\;2\le k\le n }\Xi^{(\cal L-\cal K)}_{u,k}\cdot \Xi^{(\cal L-\cal K)}_{u,n-k+2}\cdot(W\ell_u\eta_u)^{-1}\right)\cdot (W\ell_u\eta_u)^{-n }\cdot \frac{1}{\eta_u}  , 
  \nonumber \\  \mathcal{E}_{u, \boldsymbol{\sigma}}^{(\widetilde{G})}
 \; \prec \; &
     \Xi^{(\cal L)}_{u,n+1} \cdot (W\ell_u\eta_u)^{-n }\cdot \frac{1}{\eta_u}   ,  
  \nonumber \\   
(\mathcal{E} \otimes \mathcal{E}) _{u, \boldsymbol{\sigma}}
\; \prec \; &  \; \Xi^{(\cal L)}_{u, {2n+2}}\cdot   (W\ell_u\eta_u)^{-2n } \cdot \frac1{\eta_u}. 
\end{align}
Furthermore, 
all these  $\cal E$ terms have the $(u, \tau, D)$ decay property for  $u\in [s,t]$.  
 \end{lemma}
 
 \begin{proof}[Proof of Lemma \ref{lem_BcalE}] The proof  follows directly  the definitions of these  terms  
 and the decay properties in Lemma \ref{lem_decayLoop}.  
     \begin{itemize}
    \item By definition \eqref{DefKsimLK},  $\mathcal{K} \sim(\mathcal{L}-\mathcal{K})$ can be written as the product of $\cal K$  and $\cal L-\cal K$ loops.   Bounding  $\cal K$  with \eqref{eq:bcal_k}, we have  
    \begin{align}
    \big[\mathcal{K} \sim(\mathcal{L}-\mathcal{K})\big]_{u, \boldsymbol{\sigma}}^{l_{\mathcal{K}}}
 \; \prec \; &
 W\cdot (W\ell_u\eta_u)^{-l_{\mathcal{K}}+1}\cdot \ell_u\cdot \Xi^{(\cal L-\cal K)}_{u,(n-l_{\mathcal{K}}+2)}\cdot (W\ell_u\eta_u)^{-(n-l_{\mathcal{K}}+2)}
    \end{align}
    where we used $l_{\mathcal{K}}\ge 3$ and the $\ell_u$ factor comes from the sum of the index $a$, which was restricted by the decay property of $\mathcal K$. 
    As a side remark, the other index  $b$ was restricted by    $S^{(B)}_{ab}$. 
    
    \item By  definition \eqref{def_ELKLK},  $ \mathcal{E}^{((\mathcal{L} - \mathcal{K}) \times (\mathcal{L} - \mathcal{K}))}$ can be written as the product of two loops, whose total length is $n+2$.
    \begin{align}
   \mathcal{E}^{((\mathcal{L} - \mathcal{K}) \times (\mathcal{L} - \mathcal{K}))}
 \; \prec \; &
 W\cdot \sum_{k=2}^{n} \Xi^{(\cal L-\cal K)}_{u, \,k}
 \cdot (W\ell_u\eta_u)^{-k}
 \cdot \ell_u\cdot \Xi^{(\cal L-\cal K)}_{u,(n-k+2)}
 \cdot (W\ell_u\eta_u)^{-(n-k+2)}
      \end{align}
    
    \item By definition \eqref{def_EwtG}, 
    $\mathcal{E}^{(\widetilde{G})}$  can be written as the product of  a  $(\cal L-\cal K)$ loop of length one and an  $(n+1)-G$  loop. With  \eqref{GavLGEX}, we have 
    \begin{align}
  \mathcal{E}^{(\widetilde{G})}
 \; \prec \; &
 W\cdot   \Xi^{\cal L }_{u, \,2}
 \cdot (W\ell_u\eta_u)^{-1}
 \cdot \ell_u\cdot \Xi^{\cal L }_{u,(n+1)}
 \cdot (W\ell_u\eta_u)^{-n} 
    \end{align} 
    Furthermore, with \eqref{Eq:Gdecay_w}, we can bound $\Xi^{\cal L }_{u, \,2}\prec 1$. 
 \item By definition  \eqref{def_diffakn_k}, 
    $(\mathcal{E} \otimes \mathcal{E})$ can be written in terms of  $(2n+2)$-$G$-loops.  Thus 
    \begin{align}
(\mathcal{E} \otimes \mathcal{E}) \prec  \Xi^{(\cal L)}_{u, {2n+2}}\cdot  W\cdot\ell_u\cdot (W\ell_u\eta_u)^{-2n-1 } .
    \end{align} 
\end{itemize}
Here all  $\ell_u$ factors  come from summing an index restricted to a range $\ell_u$. 
 \end{proof}
  
In the remaining of this subsection, we will use 
Lemma \ref{lem_BcalE} to  improve estimates on $\cal L-\cal K$.  

\begin{lemma}\label{lem:STOeq_NQ} 
Assume that $\boldsymbol{\sigma}$ is not an alternating sign vector, i.e., 
\begin{equation}\label{NALsigm}
    \exists k,\; s.t.\; \sigma_k=\sigma_{k+1}.
\end{equation} 
Suppose that the   assumptions of  Theorem \ref{lem:main_ind},  the local law \eqref{Gt_bound_flow} and the decay property for $\cal L-\cal K$ \eqref{Eq:Gdecay_w} hold. If 
 $$
 \max_{u\in [s,t]}\, \Xi^{(\cal L)}_{u, {2n+2}}\prec {  \Lambda}
 $$
 for some deterministic quantity $ \Lambda \ge 1$, 
then 
\begin{align}\label{am;asoiuw}
 \Xi^{({\cal L-\cal K})}_{u,n}
 \;\prec\;  &\; { \Lambda}^{1/2}+\max_{u\in [s,t]}
  \left(\; \max_{k<n }\Xi^{({\cal L-\cal K})}_{u,k}+\max_{k:\;2\le k\le n }\Xi^{({\cal L-\cal K})}_{u,k}\cdot \Xi^{({\cal L-\cal K})}_{u,n-k+2}\cdot(W\ell_u\eta_u)^{-1}+\Xi^{({\cal L})}_{u,n+1}\right).
\end{align}

\end{lemma}

\begin{proof}[Proof of Lemma \ref{lem:STOeq_NQ}]
 By  case 1 of \eqref{sum_res_2}, 
$$
\big\|{\cal U}_{u,t,\boldsymbol{\sigma}}\big\|_{\max\to \max}\Big|({\rm fast\; decay \;tensor satisfying \eqref{NALsigm}})\le C_n W^{C_n \tau} \cdot\left(\frac{\ell_u \eta_u}{\ell_t \eta_t}\right)^n.
$$
 By  assumptions of  
 Theorem \ref{lem:main_ind} on  $({\cal L-K})_{s}$, the $\cal E$ estimates in Lemma \ref{lem_BcalE}, and   the previous bound on  $\max\to \max$ norm of ${\cal U}_{u,t}$, 
 we have 
\begin{align}\label{sahwNQ}
  (W\ell_t\eta_t)^n\cdot      (\mathcal{L} - \mathcal{K})_{t, \boldsymbol{\sigma}, \textbf{a}} \;\prec\; \; &
  1+   \int_{s}^t  \left( \max_{k<n }\Xi^{({\cal L-\cal K})}_{u,k}\right) \cdot \frac{1}{\eta_u}   du 
    \nonumber \\
    &+ \int_{s}^t \left( \max_{k:\;2\le k\le n }\Xi^{({\cal L-\cal K})}_{u,k}\cdot \Xi^{({\cal L-\cal K})}_{u,n-k+2}\cdot(W\ell_u\eta_u)^{-1}\right)\cdot \frac{1}{\eta_u} du 
    \nonumber \\
    &+ \int_{s}^t \Xi^{(\cal L)}_{u,n+1}\cdot \frac1{\eta_u} du 
   +
   (W\ell_t\eta_t)^n\cdot\int_{s}^t \left(\mathcal{U}_{u, t, \boldsymbol{\sigma}} \circ    \mathcal{E}^{(M)}_{u, \boldsymbol{\sigma}}\right)_{\textbf{a}} du. 
\end{align}
By Lemma \ref{lem:DIfREP},  we know that 
\begin{equation} \label{sahwNQ2}
   \mathbb{E} \left [    \int_{s}^t \left(\mathcal{U}_{u, t, \boldsymbol{\sigma}} \circ \mathcal{E}^{(M)}_{u, \boldsymbol{\sigma}}\right)_{\textbf{a}} \right ]^{2p} 
   \le C_{n,p} \; 
   \mathbb{E}    \left(\int_{s}^t 
   \left(\left(
   \mathcal{U}_{u, t,  \boldsymbol{\sigma}}
   \otimes 
   \mathcal{U}_{u, t,  \overline{\boldsymbol{\sigma}}}
   \right) \;\circ  \;
   \left( \mathcal{E} \otimes  \mathcal{E} \right)
   _{u, \,\boldsymbol{\sigma}  }
   \right)_{{\textbf a}, {\textbf a}}du 
   \right)^{p}.
\end{equation} 
Using  \eqref{CalEbwXi} to bound $\mathcal{E} \otimes  \mathcal{E}$, we have 
\begin{align}\label{sahwNQ3}
   (W\ell_t\eta_t)^{2n}\cdot  \int_{s}^t  \left(\mathcal{U}_{u, t, (\boldsymbol{\sigma}, \overline{\boldsymbol{\sigma}})} \;\circ  \;
   \left( \mathcal{E} \otimes  \mathcal{E} \right)
   _{u, \,\boldsymbol{\sigma}  }
   \right)_{{\textbf a}, {\textbf a}}du
   \prec \int_{s}^t \Xi^{(\cal L)}_{u, {2n+2}}\cdot \frac{1}{\eta_u}du . 
\end{align}
This completes the proof of Lemma \ref{lem:STOeq_NQ}. 
 
\end{proof}

\subsection{Sum zero operator \texorpdfstring{${\cal Q}_t$}{Qt}}
In this subsection, we introduce another key tool, ${\cal Q}_t$,  for estimating the loop hierarchy. 
\begin{definition}[Definition of ${\cal Q}_t$]\label{Def:QtPt} Let ${\cal A}_\textbf{a}$ be a tensor with $\textbf{a}\in\mathbb Z_L^n$, $n\ge 2$.   Define ${\cal P}$  by
$$
\left( {\cal P} \circ {\cal A}\right)_{a_1}= \sum_{a_2,\, a_3\cdots a_n}  {\cal A}_{\textbf{a}} ,\quad \textbf{a}=(a_1,\cdots,a_n), 
$$
where the index $a_1$ was fixed. A tensor $A$ has a sum zero property if and only if $ {\cal P} \circ {\cal A}= 0$.   Define 
$$
 \left({\cal Q}_t\circ {\cal A}\right)_{\textbf{a}}={\cal A}_{\textbf{a}}-\left({\cal P} \circ {\cal A}\right)_{a_1}
 \cdot {  \vartheta}_{t,\,\textbf{a}},\quad 
 {  \vartheta}_{t, \textbf{a}}:=\big(1-t\big)^{n-1}\prod_{i=2}^n \left(\Theta^{(B)}_t\right)_{a_1a_i}.
$$
Since $\sum_b\left(\Theta^{(B)}_t\right)_{ab}=(1-t)^{-1}$, it is easy to check that $ {\cal P} \circ  \vartheta_{t,\textbf{a}}=1$ 
and  ${\cal P} \circ {\cal Q}_t =0$. 
\end{definition}
The reason for using $\vartheta_{t,\textbf{a}}$ in the definition of ${\cal Q}_t$ is to preserve both  the $\max$ norm and the rapid decay properties of ${\cal A}$, provided that ${\cal A}$ possesses such properties.

\begin{lemma}[Properties ${\cal Q}_t$]\label{lem_+Q} 
Let  $\cal A$ be a tensor $\mathbb Z_L^n\to \mathbb R$. Assume that  $\cal A$ has ($t, \tau, D$) decay property. 
Then  
\begin{align}\label{normQA}
        \max_{\textbf{a}} \left({\cal Q}_t\circ \cal A\right)_{\textbf{a}} \le W^{C_n \tau} \cdot \max_{\textbf{a}} {\cal A}_{\textbf{a}}+W^{-D+C_n}. 
\end{align}
Furthermore,   
$\left({\cal Q}_t\circ \cal A\right)_{\textbf{a}}$ has the ($ t, \tau, D$) decay property. 
\end{lemma}

We  now use  ${\cal Q}_t$ to improve our estimates on  ${\cal L} - {\cal K}$ integral representation \eqref{int_K-LcalE}.
Denote by $\dot \vartheta_{t,  \textbf{a}}=\frac{d}{dt}\vartheta_{t,  \textbf{a}}$. 
 From the hierarchy of \(\mathcal{L} - \mathcal{K}\) \eqref{eq_L-Keee}, we have 
\begin{align} \label{zjuii1}
    d{\cal Q}_t \circ (\mathcal{L} - \mathcal{K})_{t, \boldsymbol{\sigma}, \textbf{a}} 
    = & {\cal Q}_t \circ \Big[\mathcal{K} \sim (\mathcal{L} - \mathcal{K})\Big]^{l_\mathcal{K} = 2}_{t, \boldsymbol{\sigma}, \textbf{a}} 
    + \sum_{l_\mathcal{K} > 2} {\cal Q}_t \circ \Big[\mathcal{K} \sim (\mathcal{L} - \mathcal{K})\Big]^{l_\mathcal{K}}_{t, \boldsymbol{\sigma}, \textbf{a}} 
    \\\nonumber
    & + {\cal Q}_t \circ \mathcal{E}^{((\mathcal{L} - \mathcal{K}) \times (\mathcal{L} - \mathcal{K}))}_{t, \boldsymbol{\sigma}, \textbf{a}} 
    + {\cal Q}_t \circ \mathcal{E}^{(M)}_{t, \boldsymbol{\sigma}, \textbf{a}} 
    + {\cal Q}_t \circ \mathcal{E}^{(\widetilde{G})}_{t, \boldsymbol{\sigma}, \textbf{a}} 
    + \left[{\cal P} \circ \left(\mathcal{L} - \mathcal{K}\right)_{t, \boldsymbol{\sigma}} \cdot \dot \vartheta_t \right]_{\textbf{a}}.
\end{align} 
To  match the form appearing in Lemma \ref{Sol_CalL}, we write the first term on the right-hand side as 
\begin{align}\label{zjuii2}
{\cal Q}_t \circ \Big[\mathcal{K} \sim (\mathcal{L} - \mathcal{K})\Big]^{l_\mathcal{K} = 2}_{t, \boldsymbol{\sigma}, \textbf{a}}
& = {\cal Q}_t \circ \varTheta_{t, \boldsymbol{\sigma}} \circ (\mathcal{L} - \mathcal{K})_{t, \boldsymbol{\sigma}, \textbf{a}}
\\\nonumber
& = \varTheta_{t, \boldsymbol{\sigma}} \circ {\cal Q}_t \circ (\mathcal{L} - \mathcal{K})_{t, \boldsymbol{\sigma}, \textbf{a}}
+ \big[{\cal Q}_t , \varTheta_{t, \boldsymbol{\sigma}} \big] \circ (\mathcal{L} - \mathcal{K})_{t, \boldsymbol{\sigma}, \textbf{a}},
\end{align}
where  $\big[{\cal Q}_t , \varTheta_{t, \boldsymbol{\sigma}} \big] $ is the commutator
of ${\cal Q}_t $ and $\varTheta_{t, \boldsymbol{\sigma}} \big]$.
Since \({\cal P} \circ {\cal Q}_t = 0\), the left-hand side of \eqref{zjuii2} satisfies the sum-zero property. Furthermore, with the definition of \(\varTheta_{t, \boldsymbol{\sigma}} \) in Definition \ref{DefTHUST}, one can verify that if \({\cal A}\) has the sum-zero property,  so does \(\varTheta_{t, \boldsymbol{\sigma}}  \circ {\cal A}\), i.e., 
\[
{\cal P} \circ {\cal A} = 0 \implies {\cal P} \circ \varTheta_{t, \boldsymbol{\sigma}}  \circ {\cal A} = 0.
\]
It implies the first term in the right hand side of \eqref{zjuii2} has the sum-zero property. Since the left hand side of \eqref{zjuii2} also has this property, we obtain that the last term in \eqref{zjuii2} has sum zero property too, i.e., 
\begin{align}\label{pqthlk}
   {\cal P}\circ \big[{\cal Q}_t , \varTheta_{t,\boldsymbol{\sigma}} \big]\circ (\mathcal{L} - \mathcal{K}) _{t, \boldsymbol{\sigma}, \textbf{a}}=0. 
\end{align}
Then combining \eqref{zjuii1} and \eqref{zjuii2}, and applying Lemma \ref{Sol_CalL}, we obtain the following equation, paralleling to \eqref{int_K-LcalE}, 
 \begin{align}\label{int_K-L+Q}
  {\cal Q}_t\circ   (\mathcal{L} - \mathcal{K})_{t, \boldsymbol{\sigma}, \textbf{a}} \;=\; \; &
    \left(\mathcal{U}_{s, t, \boldsymbol{\sigma}} \circ  {\cal Q}_s\circ (\mathcal{L} - \mathcal{K})_{s, \boldsymbol{\sigma}}\right)_{\textbf{a}} \nonumber \\
    &+ \sum_{l_\mathcal{K} > 2} \int_{s}^t \left(\mathcal{U}_{u, t, \boldsymbol{\sigma}} \circ  {\cal Q}_u\circ \Big[\mathcal{K} \sim (\mathcal{L} - \mathcal{K})\Big]^{l_\mathcal{K}}_{u, \boldsymbol{\sigma}}\right)_{\textbf{a}} du \nonumber \\
    &+ \int_{s}^t \left(\mathcal{U}_{u, t, \boldsymbol{\sigma}} \circ  {\cal Q}_u\circ \mathcal{E}^{((\mathcal{L} - \mathcal{K}) \times (\mathcal{L} - \mathcal{K}))}_{u, \boldsymbol{\sigma}}\right)_{\textbf{a}} du \nonumber \\
    &+ \int_{s}^t \left(\mathcal{U}_{u, t, \boldsymbol{\sigma}} \circ  {\cal Q}_u\circ \mathcal{E}^{(\widetilde{G})}_{u, \boldsymbol{\sigma}}\right)_{\textbf{a}} du + \int_{s}^t \left(\mathcal{U}_{u, t, \boldsymbol{\sigma}} \circ  {\cal Q}_u\circ \mathcal{E}^{(M)}_{u, \boldsymbol{\sigma}}\right)_{\textbf{a}} du
    \nonumber \\
     &+ \int_{s}^t 
     \left(\mathcal{U}_{u, t, \boldsymbol{\sigma}} 
     \circ \left( \big[{\cal Q}_u , \varTheta_{u,\boldsymbol{\sigma}} \big]\circ (\mathcal{L} - \mathcal{K}) _{u, \boldsymbol{\sigma} }\right) \right)_{\textbf{a}} du
        \\\nonumber
    &+ \int_{s}^t 
     \left(\mathcal{U}_{u, t, \boldsymbol{\sigma}} 
     \circ \left[ {\cal P} \circ
      \left(\mathcal{L} - \mathcal{K}\right)_{u, \boldsymbol{\sigma}}
        \cdot \dot \vartheta_{u } \right]\right)_{\textbf{a}} du.
\end{align}
In the next lemma, we will utilize the fact that all terms on the right-hand side of \eqref{int_K-L+Q} exhibit both fast decay and sum-zero properties.

\begin{lemma}\label{lem:STOeq_Qt} 
Suppose that  the assumptions of 
 Theorem \ref{lem:main_ind},  the local law \eqref{Gt_bound_flow} and the decay property of $\cal L-\cal K$ \eqref{Eq:Gdecay_w} hold. 
If 
 $$
 \max_{u\in [s,t]}\, \Xi^{(\cal L)}_{u, {2n+2}}\prec {  \Lambda}
 $$
 for some deterministic quantity $ \Lambda \ge 1$, 
then 
\begin{align}\label{am;asoi222}
 \Xi^{({\cal L-\cal K})}_{u,n}
 \;\prec\;  &\; { \Lambda}^{1/2}+\max_{u\in [s,t]}
  \left(\; \max_{k<n }\Xi^{({\cal L-\cal K})}_{u,k}+\max_{k:\;2\le k\le n }\Xi^{({\cal L-\cal K})}_{u,k}\cdot \Xi^{({\cal L-\cal K})}_{u,n-k+2}\cdot(W\ell_u\eta_u)^{-1}+\Xi^{({\cal L})}_{u,n+1}\right).
\end{align}
Notice that  there is no prefactor depending on  $(\eta_s/\eta_t)$ on the right hand side of the last inequality.
\end{lemma}

\begin{proof}[Proof of Lemma \ref{lem:STOeq_Qt}]
In the previous subsection, we have proved \eqref{am;asoi222} for the non-alternating case, i.e., \eqref{NALsigm}. Hence we only need to focus on the alternating case 
$
\forall \; k \quad \sigma_k=\overline \sigma_{k+1} $.
By Lemma \ref{lem_+Q},  
$$
{\cal Q}_s\;\circ ({\cal  L-K})_s,\quad {\cal Q}\;\circ ({\cal K\sim L-K}),\quad {\cal Q}\;\circ{\cal E}^{\cal (L-K)\times (L-K)}, \quad {\cal Q}\;\circ {\cal E}^{(\widetilde G)} 
$$  
have
the sum zero and  fast decay properties. Furthermore, their $\max$ norms can be bounded with \eqref{normQA}, which amounts to  
$
\|{\cal Q}\|_{\max\to \max}\prec 1$.
By 
\eqref{sum_res_2},    we have the following bound for  sum zero fast decay tensors: 
\begin{align}\label{jsaudyfa}
    \big\|{\cal U}_{u,t,\boldsymbol{\sigma}}\big\|_{\max\to \max}\Big |({\rm fast\; decay \;and \; sum \; zero\;tensor})\le C_n W^{C_n \tau} \cdot\left(\frac{\ell_u \eta_u}{\ell_t \eta_t}\right)^n.
\end{align}
Following the same argument in the proof of Lemma \ref{lem:STOeq_NQ}, we have 
  \begin{align}\label{sahwmis}
  (W\ell_t\eta_t)^n\cdot  {\cal Q}_t\circ   (\mathcal{L} - \mathcal{K})_{t, \boldsymbol{\sigma}, \textbf{a}} \;\prec\; \; &
  \; 1+\max_{u\in [s,t]}
  \left(\; \max_{k<n }\Xi^{({\cal L-\cal K})}_{u,k}+\max_{2\le k\le n }\Xi^{({\cal L-\cal K})}_{u,k}\cdot \Xi^{({\cal L-\cal K})}_{u,n-k+2}\cdot(W\ell_u\eta_u)^{-1}+\Xi^{({\cal L})}_{u,n+1}\right)
  \nonumber \\
 \; + \; & \;(W\ell_t\eta_t)^n\cdot\int_{s}^t \left(\mathcal{U}_{u, t, \boldsymbol{\sigma}} \circ   Q_t\circ  \mathcal{E}^{(M)}_{u, \boldsymbol{\sigma}}\right)_{\textbf{a}} du 
   \nonumber \\ 
     \; + \; &\; (W\ell_t\eta_t)^n\int_{s}^t 
     \left(\mathcal{U}_{u, t, \boldsymbol{\sigma}} 
     \circ \left( \big[{\cal Q}_u , \varTheta_{u,\boldsymbol{\sigma}} \big]\circ (\mathcal{L} - \mathcal{K}) _{u, \boldsymbol{\sigma} }\right) \right)_{\textbf{a}} du
        \\\nonumber
 \; + \; &\;(W\ell_t\eta_t)^n\int_{s}^t 
     \left(\mathcal{U}_{u, t, \boldsymbol{\sigma}} 
     \circ \left[ {\cal P} \circ
      \left(\mathcal{L} - \mathcal{K}\right)_{u, \boldsymbol{\sigma}}
        \cdot \dot \vartheta_{u } \right]\right)_{\textbf{a}} du.
 \end{align}

Next, we estimate the last line in \eqref{sahwmis}. Since ${\cal P}  \circ \vartheta_{t,  \textbf{a}}=1$, we have ${\cal P}   \circ \dot  \vartheta_{t,  \textbf{a}}=0$ and 
$$
 {\cal P}  \circ \left[ {\cal P} \circ
      \left(\mathcal{L} - \mathcal{K}\right)_{u, \boldsymbol{\sigma}}
        \cdot \dot \vartheta_{u } \right]_{\textbf{a}}
        =\left({\cal P} \circ
      \left(\mathcal{L} - \mathcal{K}\right)_{u, \boldsymbol{\sigma}}\right)_{a_1} {\cal P}  \circ \dot \vartheta_{u,  \textbf{a}}=0.
$$
 Therefore, this term also has the sum zero property. From its  definition,  $\vartheta$ has a fast decay property. Then with \eqref{sum_res_2} (case 2), we have 
\begin{align}\label{jfasiuu}
 \int_{s}^t 
     \left(\mathcal{U}_{u, t, \boldsymbol{\sigma}} 
     \circ \left[ {\cal P} \circ
      \left(\mathcal{L} - \mathcal{K}\right)_{u, \boldsymbol{\sigma}}
        \cdot \dot \vartheta_{u } \right]\right)_{\textbf{a}} du  
         \;   \prec \;&\;\int_s^t \left(\frac{\ell_u\eta_u}{\ell_t\eta_t}\right)^{n}\cdot 
          \max_{a} \left|\left[{\cal P} \circ
      \left(\mathcal{L} - \mathcal{K}\right)_{u, \boldsymbol{\sigma}}\right]_a\right|
      \max_{\textbf{a}}\left| 
          \dot \vartheta_{u,\textbf{a} } \right|du .
\end{align}
We claim that for non-constant $\boldsymbol{\sigma}$, i.e., $\{\sigma_1, \cdots, \sigma_n\} = \{+, -\}$, the following holds:  
\begin{align}\label{jywiiwsoks}
    \left[{\cal P} \circ
    \left(\mathcal{L} - \mathcal{K}\right)_{u, \boldsymbol{\sigma}}\right]_a \prec (W\eta_u)^{-n} \ell_u^{-1} \cdot \Xi^{(\mathcal{L}-\mathcal{K})}_{u, n-1}
\end{align}
for any loop length $n$.   To see this, for a fixed non-constant $\boldsymbol{\sigma}$, there exists $1 < k \leq n$ such that $\sigma_k = \overline{\sigma}_{k+1}$ (i.e., a pair of opposite charges). By Ward's identity (Lemma \ref{lem_WI_K}), we can sum $\mathcal{L} - \mathcal{K}$ of rank $n$ over the index $a_k$ and express the resulting sum in terms of $\mathcal{L} - \mathcal{K}$ of rank $n-1$ with a multiplicative factor $(W \eta_u)^{-1}$. The summation over the remaining indices contributes an additional factor of $\ell_u^{n-2}$ due to the fast decay property.

 On the other hand, with the definition of $\vartheta$ and $\Theta^{(B)}$, we have 
$$
   \max_{\textbf{a}}\left| 
       \dot \vartheta_{u,\textbf{a} } \right|\prec \ell_u^{-n+1}\cdot \eta_u^{-1}.  
$$
 Inserting them back to \eqref{jfasiuu}, we obtain that
\begin{align}\label{kkuuwsaf}
  \int_{s}^t 
     \left(\mathcal{U}_{u, t, \boldsymbol{\sigma}} 
     \circ \left[ {\cal P} \circ
      \left(\mathcal{L} - \mathcal{K}\right)_{u, \boldsymbol{\sigma}}
        \cdot \dot \vartheta_{u } \right]\right)_{\textbf{a}} du  \prec (W\ell_t\eta_t)^{-n}\max_{u\in [s,t]} \Xi^{(\cal L-K)}_{u,\; n-1}. 
 \end{align}
Similarly, we can estimate the second line of \eqref{sahwmis}.  By definition of $\varTheta$ and Lemma \ref{lem_+Q},    $ {\varTheta}_u \circ {\cal A}$ and ${\cal Q}_u  \circ {\cal A}$
 are fast decay  if ${\cal A}_u$ is fast decay. Since $(\mathcal{L}-\mathcal{K})_u$ has fast decay property at time $u$,  so does $\left[\mathcal{Q}_u, \Theta_{u, \sigma}\right] \circ(\mathcal{L}-\mathcal{K})$.  On the other hand,  we just showed  that $
\left[\mathcal{Q}_u, \Theta_{u, \sigma}\right] \circ(\mathcal{L}-\mathcal{K})
$ also has sum zero property from \eqref{pqthlk}.  
Therefore,  with  \eqref{sum_res_2} (case 2), we have 
\begin{align}\label{jfasiuu222}
 \int_{s}^t 
   \left(\mathcal{U}_{u, t, \boldsymbol{\sigma}} 
     \circ \left( \big[{\cal Q}_u , \varTheta_{u,\boldsymbol{\sigma}} \big]\circ (\mathcal{L} - \mathcal{K}) _{u, \boldsymbol{\sigma} }\right) \right)_{\textbf{a}} du  
         \;   \prec \;&\;\int_s^t \left(\frac{\ell_u\eta_u}{\ell_t\eta_t}\right)^{n}\cdot 
          \max_{a} 
          \left|\left(\big[{\cal Q}_u , \varTheta_{u,\boldsymbol{\sigma}} \big]\circ (\mathcal{L} - \mathcal{K}) _{u, \boldsymbol{\sigma} }\right)_{\textbf{a}}\right|du.
\end{align}
By definition, we can bound 
\begin{align}
   \big[{\cal Q}_u , \varTheta_{u,\boldsymbol{\sigma}} \big]\circ (\mathcal{L} - \mathcal{K}) _{u, \boldsymbol{\sigma}, \textbf{a}}
   \;=\; & \; \left[ \left({\cal P} \circ \left(\varTheta_{u, \boldsymbol{\sigma}} \circ (\mathcal{L} - \mathcal{K})\right)\right) \cdot \vartheta_u \right]_{\textbf{a}} 
    - \varTheta_{u, \boldsymbol{\sigma}} \circ \left[ {\cal P} \circ \left( \mathcal{L} - \mathcal{K} \right) \cdot \vartheta_t \right]_{\textbf{a}}.
\\\nonumber
 \;\prec \; & \;  \|\vartheta_u\|_{\max}\|\cdot \eta_u^{-1}\max_a\left[ {\cal P} \circ \left(  \mathcal{L} - \mathcal{K} \right)_{u,\boldsymbol{\sigma}} \right]_{a} .
\end{align} 
Inserting it back to \eqref{jfasiuu222} and using \eqref{jywiiwsoks}, we obtain that 
\begin{align}\label{jfasiuu333}
 \int_{s}^t 
   \left(\mathcal{U}_{u, t, \boldsymbol{\sigma}} 
     \circ \left( \big[{\cal Q}_u , \varTheta_{u,\boldsymbol{\sigma}} \big]\circ (\mathcal{L} - \mathcal{K}) _{u, \boldsymbol{\sigma} }\right) \right)_{\textbf{a}} du  
         \;   \prec  (W\ell_t\eta_t)^{-n}\max_{u\in [s,t]} \Xi^{(\cal L-K)}_{u,\; n-1}.
\end{align}
Furthermore, we can write the left hand side  of \eqref{sahwmis} as 
\begin{align}\label{kolkisaf} 
{\cal Q}_t\circ   (\mathcal{L} - \mathcal{K})_{t, \boldsymbol{\sigma}, \textbf{a}}
 =(\mathcal{L} - \mathcal{K})_{t, \boldsymbol{\sigma}, \textbf{a}}-{\cal P} \circ   (\mathcal{L} - \mathcal{K})_{t, \boldsymbol{\sigma}, \textbf{a}}\cdot \vartheta_{\textbf{a}}
 =(\mathcal{L} - \mathcal{K})_{t, \boldsymbol{\sigma}, \textbf{a}}
 +O_{\prec} \left((W\ell_t\eta_t)^{-n}  \Xi^{(\cal L-K)}_{t,\; n-1}\right).
\end{align}
Inserting \eqref{kkuuwsaf}, \eqref{jfasiuu333} and \eqref{kolkisaf} into \eqref{sahwmis}, we obtain 
 \begin{align}\label{sahwmis2ws}
  (W\ell_t\eta_t)^n\cdot    (\mathcal{L} - \mathcal{K})_{t, \boldsymbol{\sigma}, \textbf{a}} \;\prec\; \; &
  \; 1+\max_{u\in [s,t]}
  \left(\; \max_{k<n }\Xi^{({\cal L-\cal K})}_{u,k}+\max_{k:\;2\le k\le n }\Xi^{({\cal L-\cal K})}_{u,k}\cdot \Xi^{({\cal L-\cal K})}_{u,n-k+2}\cdot(W\ell_u\eta_u)^{-1}+\Xi^{({\cal L})}_{u,n+1}\right)
  \nonumber \\
 \; + \; & \;(W\ell_t\eta_t)^n\cdot\int_{s}^t \left(\mathcal{U}_{u, t, \boldsymbol{\sigma}} \circ   {\cal Q}_t\circ  \mathcal{E}^{(M)}_{u, \boldsymbol{\sigma}}\right)_{\textbf{a}} du .
 \end{align}
For the martingale term, one can easily derive that 
\begin{align} \label{UEUEdif+Q}
  \left[\int \left(\mathcal{U}_{u, t, \boldsymbol{\sigma}} \circ {\cal Q}_u\circ \mathcal{E}^{(M)}_{u, \boldsymbol{\sigma}}\right)_{\textbf{a}} \right ]_t
\;=\; &  
  \int_s^t
\sum_{\alpha}\left|\sum_{k=1}^n \mathcal{U}_{u, t, \boldsymbol{\sigma}} \circ {\cal Q}_u\circ\mathcal{E}^{(M)}_{u, \boldsymbol{\sigma}}(\alpha, k)
 \right|^2 du
 \nonumber \\
 \;\le \; & C_n\;  \int_s^t
\sum_{\alpha}\sum_{k=1}^n\left| \mathcal{U}_{u, t, \boldsymbol{\sigma}} \circ {\cal Q}_u\circ\mathcal{E}^{(M)}_{u, \boldsymbol{\sigma}}(\alpha, k)
 \right|^2 du
\nonumber \\
 \;=  \; & C_n\;  \int_{s}^t
 \left( \left(
   \mathcal{U}_{u, t,  \boldsymbol{\sigma}}
   \otimes 
   \mathcal{U}_{u, t,  \overline{\boldsymbol{\sigma}}}
   \right)\circ 
\left( \mathcal{Q}_u\otimes \mathcal{Q}_u\right) \circ \left( \mathcal{E}\otimes  \mathcal{E} \right)_{u, \,\boldsymbol{\sigma}}\right)_{\textbf{a}, \textbf{a}}du, 
\end{align}
where 
\begin{align}
\left(\left( \mathcal{Q} _u\otimes \mathcal{Q}_u \right)\circ {\cal A}\right)_{\textbf{a}, \textbf{b}} 
 \;=\; & \; 
{\cal A }_{\textbf{a}, \textbf{b} }-
\delta_{a_1', a_1}\sum_{a_2'\cdots a_n'}{\cal A }_{ \textbf{a}' , \,\textbf{b}}\cdot \vartheta_{u, \textbf{a}}
-\delta_{b_1', b_1}
\sum_{b_2'\cdots b_n'}{\cal A }_{\textbf{a}, \textbf{b}'}\cdot \vartheta_{u, \textbf{b} }
\nonumber \\ \;+\; & \; 
\delta_{a_1', a_1}
\delta_{b_1', b_1}
\sum_{a'_2\cdots a'_n}\sum_{b_2'\cdots b_n'}{\cal A }_{\textbf{a}', \textbf{b}'}\cdot \vartheta_{u, \textbf{a}}\cdot \vartheta_{u, \textbf{b}}  , 
\end{align}
where 
$
\textbf{a}=(a_1,\cdots, a_n), \quad \textbf{a}'=(a'_1,\cdots, a'_n)$ and similarly for 
 $\textbf{b}$ and $\textbf{b}'$. By definition of   $\vartheta$, it is easy to check that 
$$
\sum_{a_2,\cdots,a_n}\sum_{b_1\cdots b_n}\left(\left( \mathcal{Q} _u\otimes \mathcal{Q}_u \right)\circ {\cal A}\right)_{\textbf{a}, \textbf{b}}=0.
$$
Using \eqref{jsaudyfa} on the fast decay tensor $\mathcal{E}\otimes  \mathcal{E} $, we have 
\begin{align}
  &  \left( \left(
   \mathcal{U}_{u, t,  \boldsymbol{\sigma}}
   \otimes 
   \mathcal{U}_{u, t,  \overline{\boldsymbol{\sigma}}}
   \right)\circ 
\left( \mathcal{Q}_u\otimes \mathcal{Q}_u\right) \circ \left( \mathcal{E}\otimes  \mathcal{E} \right)_{u, \,\boldsymbol{\sigma}}\right)_{\textbf{a}, \textbf{a}}
\;\prec\; \; \left(\frac{\ell_u \eta_u}{\ell_t \eta_t}\right)^{2n}\max_{\textbf{a},\textbf{a}'}
\left(  
\left( \mathcal{Q}_u\otimes \mathcal{Q}_u\right) \circ \left( \mathcal{E}\otimes  \mathcal{E} \right)_{u, \,\boldsymbol{\sigma}}\right)_{\textbf{a}, \textbf{a}'}
\nonumber \\\;\prec\;&\;
 \left(\frac{\ell_u \eta_u}{\ell_t \eta_t}\right)^{2n}\max_{\textbf{a},\textbf{a}'}
\left(     \left( \mathcal{E}\otimes  \mathcal{E} \right)_{u, \,\boldsymbol{\sigma}}\right)_{\textbf{a}, \textbf{a}'} \;\prec\;\;
 \left(W \ell_t \eta_t\right)^{-2n}\cdot \eta_u^{-1}\cdot \Xi^{\cal L}_{u,2n+2}, 
\end{align}
where we have used \eqref{CalEbwXi}. Inserting  it back to \eqref{UEUEdif+Q} and   using  BDG inequality and the  assumption $\Xi^{\cal L}_{u,2n+2}\prec \Lambda$,  we  obtain that  
$$
\left(W \ell_t \eta_t\right)^n \cdot \int_s^t\left(\mathcal{U}_{u, t, \boldsymbol{\sigma}} \circ {\cal Q}_u \circ \mathcal{E}_{u, \boldsymbol{\sigma}}^{(\mathrm{M})}\right)_{\mathbf{a}} d u\prec \Lambda ^{1/2}.
$$
This completes the proof of Lemma \ref{lem:STOeq_Qt}. 
 
\end{proof}
 
 \subsection{Proof of Theorem \ref{lem:main_ind}, step 3}
Recall  that in this step, in addition to the assumptions in Theorem \ref{lem:main_ind}, we also have the local law \eqref{Gt_bound_flow} and the decay property for $\cal L-\cal K$ \eqref{Eq:Gdecay_w}. With these inputs, we obtained \eqref{am;asoi222} 
in Lemma \ref{lem:STOeq_Qt}.  In this step,  we aim to prove the correct bound \eqref{Eq:LGxb} on $G$-loop. 
We  have proved in  Step 1 the following bounds 
\begin{align}\label{sef8w483r324}
  \Xi^{(\cal L)}_{u,n}=(\ell_u/\ell_s)^{n-1}, \quad 
 \Xi^{({\cal L-K})}_{u, n}\prec (W\ell_u\eta_u)\cdot (\ell_u/\ell_s)^{n-1} .
\end{align}
Using  the bound \eqref{eq:bcal_k} on $\cal K$, we can relate $\Xi^{(\cal L)}$ with $\Xi^{(\cal L-\cal K)}$ as follows:
\begin{align}\label{rela_XILXILK}
    \Xi^{(\cal L)}_{u,n}\prec 1+(W\ell_u\eta_u)^{-1}\cdot \Xi^{({\cal L-\cal K})}_{u, n},\quad \quad \Xi^{({\cal L-\cal K})}_{u, n}
\prec (W\ell_u\eta_u) \left(\Xi^{(\cal L)}_{u,n}+1\right).
\end{align}
Denote by 
\begin{equation}\label{adsyzz0s8d6}
      \Psi(n,k,s,u,t):=  (W\ell_s\eta_s)^{1/2}+(\ell_t/\ell_s)^{n-1}\times
    \begin{cases}
        (W\ell_u\eta_u)^{1-k/4} &  k=0\\
        (W\ell_s\eta_s)^{1-k/4} &  k\ge 1 
    \end{cases},\quad u\in [s,t] .
\end{equation}
We say that the estimate $ {\cal S}(n,k,s,u,t)$ holds if $\Xi^{({\cal L-K})}_{ u,n}\prec   \Psi(n,k,s,u,t)$.
We will prove that for any $n \ge 3$ and $k \ge 1$ 
\begin{align}\label{St3MnAsm} 
{\cal S}(m, l,s,u,t)  \text{ holds for }  \{(m, l) : l=k \text{ and }  m \le n-1, \text{ or } l=k-1,   m \le n+2 \}  
\implies 
{\cal S}(n,k,s,u,t)  \text{ holds}.
\end{align} 
Suppose the last statement  holds. We now prove  \eqref{Eq:LGxb}. 

Before delving into the detailed proof, we note that a heuristic argument will be presented  after the proof. This heuristic provides a high-level, intuitive explanation of the underlying ideas and complements the rigorous derivation offered here. Readers are encouraged to refer to these paragraphs, either before or after reading the detailed proof, for additional insights. 

\begin{proof}[Proof of \eqref{Eq:LGxb}]
With the bound \eqref{lRB1} for $\cal L$ and \eqref{eq:bcal_k_2} for $\cal K$, 
${\cal S}(m, l,s,u,t)$  holds for $l=0$ and any  $m\ge 1$. 
By  \eqref{Eq:Gdecay_w}, 
${\cal S}(m, l,s,u,t)$ holds for any $l$ and $m\le 2$. By \eqref{St3MnAsm}, ${\cal S}(m, l,s,u,t)$ holds for $(3, 1)$. 
Then we apply  \eqref{St3MnAsm} again and ${\cal S}(m, l,s,u,t)$ holds for $(4, 1)$. We can continue this procedure until 
${\cal S}(m, l,s,u,t)$ holds for  $(n, 1)$ for any $n$. We can now repeat this process and prove that ${\cal S}(m, l,s,u,t)$ holds for  $(n, 2)$ for any $n$. Eventually, the induction implies that ${\cal S}(n, k,s,u,t)$ for any $n, k$. 
By   condition \eqref{con_st_ind}, for any fixed $n$  there is a large enough $k$ so that 
$$
 (\ell_t/\ell_s)^{n-1}(W\ell_s\eta_s)^{1 -k/4} \ll  1.
$$
Thus the fact that $ {\cal S}(n,k,s,u,t)$ holds for such $n, k$ implies that 
$\Xi^{({\cal L-K})}_{ u,n}\prec (W\ell_s\eta_s)^{1/2}$. Together with   \eqref{rela_XILXILK}, we have proved   \eqref{Eq:LGxb}. 

Now we only need to prove \eqref{St3MnAsm}. Recall 
at the time $s$ that the following initial bound holds: 
\begin{equation}\label{inasfua}
    \Xi_{s,n }^{(\mathcal{L}-\mathcal{K})} \prec 1.
\end{equation}
We claim that 
 \begin{align} \label{xiu2n+2psi}
    \Xi^{({\cal L})}_{u, 2n+2}
 \prec \Psi(n,k,s,u,t)^2.
\end{align}
Using  this bound as an input, from  \eqref{am;asoi222} we obtain that 
\begin{align} \label{sadui_w0}
 \Xi_{t',\,n }^{(\mathcal{L}-\mathcal{K})}  \prec\;  &\; \max_{v\in [s,t']}
  \left(\Psi(n,k,s,v,t')+\; \max_{m<n }\Xi^{({\cal L-\cal K})}_{v,m}+\max_{m:\;2\le m\le n }\Xi^{({\cal L-\cal K})}_{v,m}\cdot \Xi^{({\cal L-\cal K})}_{v,n-m+2}\cdot(W\ell_v\eta_v)^{-1}+\Xi^{({\cal L})}_{v,n+1}\right)
 \end{align}
 for any $s \le t' \le t$.  We aim to prove the conclusion of  \eqref{St3MnAsm} for a fixed $(n, k)$.  Using \eqref{rela_XILXILK}, the last term $ \Xi^{({\cal L})}_{v,n+1}$ can be bounded by $ \Xi^{{(\cal L-\cal K})}_{v,n+1}$ with a  small prefactor.  
Inserting  this  bound and the assumption \eqref{St3MnAsm} into the last display,  we have 
\begin{align}\label{sadui-w}
    \Xi_{t',\,n }^{(\mathcal{L}-\mathcal{K})} \; \prec\; \max_{v\in[s,t']}\Psi(n,k ,s,v,t').
\end{align}
Here we have used $k\ge 1$.  Setting $t'=u$ and using the monotonicity of $\Psi$ in the $t$ variable, we have proved  the conclusion of  \eqref{St3MnAsm}. 

We now  prove \eqref{xiu2n+2psi} for $\Xi^{(\cal L)}_{u, {2n+2}}$ under the assumption in \eqref{St3MnAsm}.  First, the long loop can be bounded with  short chains.  Let 
\[
\boldsymbol{\sigma} = (\sigma_1, \sigma_2, \ldots, \sigma_n), \quad \textbf{a} = (a_1, a_2, \ldots, a_{n-1}), \quad \sigma_k \in \{+, -\}, \quad a_k \in \mathbb{Z}_L.
\]
We temporally define a $G$-chain as
\begin{equation}\label{defC=GEG}
    {\cal C}_{t, \boldsymbol{\sigma}, \textbf{a}} = G_t(\sigma_1) E_{a_1} G_t(\sigma_2) E_{a_2} \cdots E_{a_{n-1}} G_t(\sigma_n).
\end{equation}
 Given an arbitrary rank $(2n+2)$ loop  ${\cal L}^{(2n+2)}_{u,\boldsymbol{\sigma},\ba}$ with 
$$\boldsymbol{\sigma}=(\sigma, \sigma_1,\ldots, \sigma_n, \sigma',\sigma_1',\ldots, \sigma_n'),\quad \ba=(b, a_1 ,\ldots, a_n , b'  , a_1' ,\ldots, a_n' ),$$ 
we can write it is 
\begin{align*}
{\cal L}^{(2n+2)}_{u,\boldsymbol{\sigma},\ba}=W^{-2d}\sum_{i\in{\cal I}_b,\,j\in{\cal I}_{b'}} \left({\cal C}^{(n+1)}_{u,\boldsymbol{\sigma}_1,\ba_1}\right)_{ij} \left({\cal C}^{(n+1)}_{u,\boldsymbol{\sigma}_2,\ba_2}\right)_{ji} ,  
\end{align*}
where $\boldsymbol{\sigma}_1=(\sigma, \sigma_1,\ldots, \sigma_n, \sigma')$, $\boldsymbol{\sigma}_2=(\sigma_1',\ldots, \sigma_n',\sigma)$, $\ba_1=(a_1,\ldots, a_n )$, and $\ba_2=( a_1' ,\ldots, a_n' )$. Applying the Cauchy-Schwarz inequality, we obtain that 
\begin{equation}\label{suauwiioo1}
\max_{\ba }\max_{\boldsymbol{\sigma}\in \{+,-\}^{2n+2}}\left|{\cal L}^{(2n+2)}_{u,\boldsymbol{\sigma},\ba}\right|
\le 
\max_{\ba_1}\max_{\boldsymbol{\sigma}_1 \in\{+,-\}^{n+1}} W^{-2 }\sum_{i\in{\cal I}_b,\,j\in{\cal I}_{b'}}
\left|\left({\cal C}^{(n+1)} _{u,\boldsymbol{\sigma}_1,\ba_1}\right)_{ij} 
\right|^2 .
\end{equation}
We now split the $(n+1)$ chain further into an $l_1$-$G$ chain and an $l_2$-$G$ chain, where $l_1=\lfloor(n+1)/{2}\rfloor$ and $l_2=n+1-l_1$. 
Using the Cauchy-Schwarz inequality again, we get that  \begin{align}\label{suauwiioo2}
    W^{-2 }\sum_{i\in{\cal I}_b,\,j\in{\cal I}_{b'}}
\left|\left({\cal C}^{(n+1)} _{u,\boldsymbol{\sigma}_1,\ba_1}\right)_{ij} 
\right|^2 &\le W^{-2 }\sum_{i\in{\cal I}_b,\,j\in{\cal I}_{b'}} \left({\cal C}^{(2l_1)}_{u,\boldsymbol{\sigma}_1',\ba_1'}\right)_{ii} \left({\cal C}^{(2l_2)}_{u,\boldsymbol{\sigma}_2',\ba_2'}\right)_{jj}  \;=\; {\cal L}^{(2l_1)}_{u,\boldsymbol{\sigma}_1',\wt\ba_1}\cdot {\cal L}^{(2l_2)}_{u,\boldsymbol{\sigma}_2',\wt\ba_2},
\end{align}
where $\boldsymbol{\sigma}_1'$, $\ba_1'$, $\boldsymbol{\sigma}_2'$, $\ba_2'$ are defined by  
\begin{align*}
\boldsymbol{\sigma}_1'=(\sigma_1,\ldots,\sigma_{l_1},-\sigma_{l_1},\ldots,-\sigma_1),\quad  &\ba_1'=( a_1 ,\ldots,  a_{l_1-1} , a_{l_1} , a_{l_1-1} ,\ldots,  a_1 ),\\
\boldsymbol{\sigma}_2'=(-\sigma_{n+1},\ldots,-\sigma_{l_1+1},\sigma_{l_1+1},\ldots,\sigma_{n+1}),\quad &\ba_2'=( a_{n} ,\ldots,  a_{l_1+1} , a_{l_1} , a_{l_1+1} ,\ldots,  a_{n} ),
\end{align*}
$\wt\ba_1$ is obtained by adding $ b $ to $\ba_1'$, and $\wt\ba_2$ is obtained by adding $ b $ to $\ba_2'$. Plugging \eqref{suauwiioo2} into \eqref{suauwiioo1} and using that $l_1+l_2=n+1$, we obtain
\begin{align} \label{asu-2w}
   \max_{\ba }\max_{\boldsymbol{\sigma}\in \{+,-\}^{2n+2}}\left|{\cal L}^{(2n+2)}_{u,\boldsymbol{\sigma},\ba}\right|\le 
    \left(\max_{\ba }\max_{\boldsymbol{\sigma}\in \{+,-\}^{2l_1}}\left|{\cal L}^{(2l_1)}_{u,\boldsymbol{\sigma},\ba}\right|\right)\cdot \left(\max_{\ba }\max_{\boldsymbol{\sigma}\in \{+,-\}^{2l_2}}\left|{\cal L}^{(2l_2)}_{u,\boldsymbol{\sigma},\ba}\right|\right), 
\end{align}
i.e.
\begin{align} \label{asuwmpwpisu}
\Xi^{({\cal L})}_{u, 2n+2}\le \Xi^{({\cal L})}_{u, 2l_1}\cdot \Xi^{({\cal L})}_{u, 2l_2}\cdot (W\ell_u\eta_u).
\end{align}  
Since   $\max(2l_1, 2l_2)\le n+2$, with the assumption of \eqref{St3MnAsm} and \eqref{asuwmpwpisu}, we can bound $\Xi^{({\cal L})}_{u, 2n+2}$ as follows
\begin{align}
 \Xi^{({\cal L})}_{u, 2n+2}
\; \prec \; &(W\ell_u\eta_u)\cdot \prod_{m=1}^2\left(1+(W\ell_u\eta_u)^{-1}\cdot \Psi(2l_1, k-1, s,u,t)\right) . 
 \end{align} 
 Then with $l_1+l_2=n+1$, and $|l_1-l_2|\le 1$,  we obtain  that
 \begin{align}
\Xi^{({\cal L})}_{u, 2n+2}\; \prec \; &\left((W\ell_s\eta_s)^{1/2}+(\ell_t/\ell_s)^{n-1}(W\ell_s\eta_s)^{1-k/4}\right)^2
\; = \; &\Psi(n,k,s,u,t)^2.
 \end{align} 
Hence we have proved  \eqref{xiu2n+2psi}. This completes the proof of \eqref{Eq:LGxb}, i.e., the Step 3 of  proving Theorem \ref{lem:main_ind}.

\end{proof}

 Roughly speaking,  an  $n-{\cal L}_u$ loop is of order $A_u^{-n+1}$ with $A_u \sim  W\ell_u\eta_u$. 
 Assuming $n$ is odd for simplicity,  from \eqref{asuwmpwpisu} we  have that
 \begin{equation} 
  \Xi^{({\cal L})}_{u,2n+2} 
\le \big ( \Xi^{({\cal L})}_{u,n+1} \big )^2 A_u 
\end{equation}
Hence    \eqref{am;asoi222} takes the form 
\begin{align} 
 \Xi_{t',\,n }^{(\mathcal{L}-\mathcal{K})}  \prec\;  \; \max_{v\in [s,t']}
  \left(   \Xi^{({\cal L})}_{v,n+1}  A_v^{1/2} +\; \max_{m<n }\Xi^{({\cal L-\cal K})}_{v,m}+\max_{m:\;2\le m\le n }\Xi^{({\cal L-\cal K})}_{v,m}\cdot \Xi^{({\cal L-\cal K})}_{v,n-m+2}A_v^{-1}+\Xi^{({\cal L})}_{v,n+1}\right)
 \end{align}
 The third term on the right hand side is of lower order while the second term is bounded by induction on $n$. We are left with the first and the last terms which involving  $n+1$ (and $n+2$ in case $n$ is even) $\cal L$ loops. By \eqref{rela_XILXILK}, 
 \begin{align} 
    \Xi^{(\cal L)}_{v,n+1}\prec 1+A_v^{-1} \Xi^{({\cal L-\cal K})}_{v, n+1}.
\end{align}
Therefore, up to  terms which are either negligible or can  be bounded by induction, we have 
\begin{align} 
 \Xi_{t',\,n }^{(\mathcal{L}-\mathcal{K})}  \prec\;  \; \max_{v\in [s,t']}
  \left(    A_v^{1/2} +A_v^{-1/2} \Xi^{({\cal L-\cal K})}_{v, n+1} +A_v^{-1} \Xi^{({\cal L-\cal K})}_{v, n+1}\right)
\end{align}
Although we cannot bound  $ \Xi_{n }^{(\mathcal{L}-\mathcal{K})}$ by $ \Xi_{n+1 }^{(\mathcal{L}-\mathcal{K})}$ by induction, 
we can derive rough bounds on  $ \Xi_{n }^{(\mathcal{L}-\mathcal{K})}$ for all $n$ and then bootstrap the argument. This is basically our procedure to arrive at $ \Xi_{t,\,n }^{(\mathcal{L}-\mathcal{K})}  \prec \;    A_t^{1/2}$.

\subsection{Proof of  Theorem \ref{lem:main_ind}, step 4 and 5.}

\begin{proof}[Step 4: Proof of  \eqref{Eq:L-KGt-flow} ] 
 Recall the key inequalities \eqref{xiu2n+2psi} and  \eqref{sadui_w0} in the proof of step 3. 
We can now use \eqref{Eq:LGxb} to bound the $\Xi^{(\cal L)}_{2n+2}$ in \eqref{xiu2n+2psi} and $\Xi^{(\cal L)}_{n+1}$ in \eqref{sadui_w0}.    Then we obtain 
 \begin{align} \label{saww02}
 \Xi_{t,\,n }^{(\mathcal{L}-\mathcal{K})}  \prec\;  &\; 1+\max_{u\in [s,t]}
  \left(\; \max_{k<n }\Xi^{({\cal L-\cal K})}_{u,k}+\max_{k:\;2\le k\le n }\Xi^{({\cal L-\cal K})}_{u,k}\cdot \Xi^{({\cal L-\cal K})}_{u,n-k+2}\cdot(W\ell_u\eta_u)^{-1} \right).
\end{align}
 By   \eqref{Eq:Gdecay_w} and \eqref{GavLGEX}  for $(\cal L-\cal K)$-loops or length $1$ and $2$ and 
 the  condition \eqref{con_st_ind}, we have 
 $$
  \Xi_{t,\,1 }^{(\mathcal{L}-\mathcal{K})}\prec 1, \quad 
   \Xi_{t,\,2 }^{(\mathcal{L}-\mathcal{K})}\prec (W\ell_u\eta_u)^{1/4}.
 $$
Using  \eqref{saww02} and induction on $n$, one can easily prove   
 $ \Xi_{t,\,n }^{(\mathcal{L}-\mathcal{K})}\prec 1
 $ for any fixed $n$. This completes the proof of  \eqref{Eq:L-KGt-flow}. 
 Notice that the prefactor $(\eta_s/\eta_t)^2$  in Step 2 was eliminated in Lemma \ref{lem:STOeq_Qt}  partly by using the sum zero property. 
 It is important that this factor is eliminated so it will not  accumulate from    the time splitting argument  in the proof of Theorem \ref{lem:main_ind}.
\end{proof}

\begin{proof}[Step 5: Proof of \eqref{Eq:Gdecay_flow}]   
Our goal is to prove ${\cal J}^*_{t,D} \prec  1$. To this end,  note that 
we have  proved  in Step 4 that 
$$
\left({\cal L}-{\cal K}\right)_{t,\, \boldsymbol{\sigma},  \,\textbf{a}}\prec (W\ell_t\eta_t)^{-2} 
$$
for the ${\cal L}-{\cal K}$-loop of length $2$.  It implies that 
$$
 \left({\cal L}-{\cal K}\right)_{t,\, \boldsymbol{\sigma},  \,\textbf{a}}\Big/{\cal T}_{t, D}(|a_1-a_2|) \prec 1, \quad \text{ for } |a_1-a_2|=O(\ell_t^*).
$$
It remains to consider  the case $|a_1-a_2| > 6 \ell_t^*$. But \eqref{Eq:Gdecay_flow} is a consequence of \eqref{53} in this case. 
This completes the proof of Step 5. 

\end{proof}

 \subsection{Proof of Theorem \ref{lem:main_ind}, step 6}\label{sec:step6E}
In this step, we have the estimates in the assumption \eqref{Eq:Gtlp_exp+IND} and all  results from step 4 and 5, i.e., \eqref{Eq:L-KGt-flow} and \eqref{Eq:Gdecay_flow}. Our goal is to prove \eqref{Eq:Gtlp_exp_flow}. 
We first  prove the following bound on   $\mathbb E{(\cal L-\cal K)}$ which improves the $1$ loop bound on  ${(\cal L-\cal K)}$.
\begin{lemma}\label{lem_ELK_n=1}
  \begin{align}\label{res_ELK_n=1}
  \max_a \left|\mathbb E({\cal L-\cal K})_{u, \;(+),\; (a)} \right|=  \max_a  \left|\mathbb E\langle (G_u-m) E_a\rangle\right|
  \prec (W\ell_u\eta_u)^{-2}  .
  \end{align} 
\end{lemma}
\begin{proof}[Proof of Lemma \ref{lem_ELK_n=1}]
We will denote $G_u$ by $G$ in this proof.  
Recall that $G - m = -m(H + m)G$. Using Gaussian integration by parts, we have  
\begin{align}
    \mathbb{E} \langle (G - m) E_a \rangle &= m \mathbb{E} \langle (-H - m) G E_a \rangle \nonumber \\
    &= u m \sum_b \mathbb{E} \langle (G - m) E_b \rangle W \cdot S^{(B)}_{bb'} \langle E_{b'} G E_a \rangle \nonumber \\
    &= u m \sum_b \mathbb{E} \langle (G - m) E_b \rangle S^{(B)}_{ba} \langle G E_a \rangle.
\end{align}
Writing $\langle G E_a \rangle = m + \langle (G - m) E_a \rangle$, we obtain  
\begin{align}
    \mathbb{E} \langle (G - m) E_a \rangle 
    &= \sum_{a'} \left((1 - um^2 S^{(B)})^{-1}\right)_{aa'} \mathbb{E} \left( um \sum_b \langle (G - m) E_b \rangle S^{(B)}_{ba'} \langle (G - m) E_{a'} \rangle \right) \nonumber \\
    &= \sum_{a'} \left(\Theta^{(B)}_{um^2}\right)_{aa'} \cdot \mathbb{E} \left( um \sum_b \langle (G - m) E_b \rangle S^{(B)}_{ba'} \langle (G - m) E_{a'} \rangle \right).
\end{align}
From \eqref{Eq:L-KGt-flow}, we know that $\langle (G - m) E_{a'} \rangle \prec (W \ell_u \eta_u)^{-1}$. Substituting this into the above estimate, we obtain \eqref{res_ELK_n=1} and  have completed the proof of Lemma \ref{lem_ELK_n=1}.
\end{proof}

\begin{proof}[Proof of \eqref{Eq:Gtlp_exp_flow}]
Taking expectation of both side of  \eqref{int_K-LcalE} for the case $n=2$, we have  
\begin{align}\label{Eexpint_K-L}
 \mathbb E\;    (\mathcal{L} - \mathcal{K})_{t, \boldsymbol{\sigma}, \textbf{a}} \;=
 \; \; &
    \left(\mathcal{U}_{s, t, \boldsymbol{\sigma}} \circ  \mathbb E\, (\mathcal{L} - \mathcal{K})_{s, \boldsymbol{\sigma}}\right)_{\textbf{a}} \\ 
    &+ \;  \int_{s}^t \left(\mathcal{U}_{u, t, \boldsymbol{\sigma}}  \circ \mathbb E\,\mathcal{E}^{((\mathcal{L} - \mathcal{K}) \times (\mathcal{L} - \mathcal{K}))}_{u, \boldsymbol{\sigma}}\right)_{\textbf{a}} du \\
    &+   \int_{s}^t \left(\mathcal{U}_{u, t, \boldsymbol{\sigma}} \circ  \mathbb E\, \mathcal{E}^{(\widetilde{G})}_{u, \boldsymbol{\sigma}}\right)_{\textbf{a}} du.
\end{align} 
By assumption, we have 
\begin{align}\label{jdasfuao01}
  \mathbb E\;  (\mathcal{L} - \mathcal{K})_{s, \boldsymbol{\sigma},\textbf{a}}\prec (W\ell_s\eta_s)^{-3}.
\end{align}
Similarly ${\cal E}^{( (L-K)\times(L-K) )}_{u, \boldsymbol{\sigma}}$ can be bounded as the product of two ${\cal (L-K)}$, and with the fast decay property, we obtain  
\begin{align}\label{jdasfuao02}
 {\cal E}^{( (L-K)\times(L-K) )}_{u, \boldsymbol{\sigma}}\prec W\ell_u (W\ell_u\eta_u)^{-4}\prec (\eta_u)^{-1} (W\ell_u\eta_u)^{-3}
\end{align}
Similarly to the arguments used in proving \eqref{jiizziy}, we can bound $\mathbb E \mathcal{E}^{(\widetilde{G})}$ as   the product of a 
$1$-${\cal (L-K)}$ loop and a $3$-$\cal L$ loop. Writing  ${\cal L}={\cal K}+({\cal L} -{\cal K})$, we  obtain  
\begin{align}
    \mathbb E \mathcal{E}^{(\widetilde{G})}_{u, \boldsymbol{\sigma},\textbf{a}}
   \; \prec \; & \;
W\cdot\ell_u\cdot  \max_a \max_{\textbf{a}}\max_{ \boldsymbol{\sigma}\in \{+,-\}^3} \Big|\mathbb E
\Big [ \langle (G-M) E_a\rangle \cdot {\cal L}_{u,\textbf{a}, \boldsymbol{\sigma} } \Big ] \Big|
 \nonumber \\
  \; \prec \; & \;
W\cdot\ell_u\cdot \max_a  \big|\mathbb E \langle (G-M) E_a\rangle \big|
 \cdot 
 \max_{\textbf{a}}\max_{ \boldsymbol{\sigma}\in \{+,-\}^3} {\cal K}_{u,\textbf{a}, \boldsymbol{\sigma} }
\nonumber \\
 + \; \;& W\cdot\ell_u\cdot \max_a \big|{({\cal L-\cal K})}_{u,(+), (a),   }\rangle \big|
 \cdot 
 \max_{\textbf{a}} \max_{ \boldsymbol{\sigma}\in \{+,-\}^3} {({\cal L-\cal K})}_{u,\textbf{a}, \boldsymbol{\sigma} }.
 \end{align}
Using  \eqref{res_ELK_n=1} for $\mathbb E \langle (G-M) E_a\rangle$, \eqref{eq:bcal_k} for $\cal K$, 
and \eqref{Eq:L-KGt-flow} for $\cal L-\cal K$, we have   
 \begin{align} \label{jdasfuao03}
  \mathbb E \mathcal{E}^{(\widetilde{G})}_{u, \boldsymbol{\sigma},\textbf{a}} \prec \; \;& (\eta_u)^{-1} (W\ell_u\eta_u)^{-3}. 
 \end{align}
One can easily see that all tensors discussed above  have  fast decay. Applying \eqref{sum_res_1} to bound  $\max\to \max$ norm of ${\cal U}_{u,t}$ and using   \eqref{jdasfuao01}, \eqref{jdasfuao02} and \eqref{jdasfuao03}, we can estimate the right hand side  of  \eqref{Eexpint_K-L} by 
 \begin{align}
   (W\ell_t\eta_t)^{3}\cdot   \mathbb E\;    (\mathcal{L} - \mathcal{K})_{t, \boldsymbol{\sigma}, \textbf{a}} \;\prec 
 (\ell_t/\ell_s)\frac{\ell_t\eta_t}{\ell_s\eta_s}+
  \int_{s}^t(\ell_t/\ell_u)\frac{\ell_t\eta_t}{\ell_u\eta_u}\cdot \eta_u^{-1}
 du 
 \prec 1, 
\end{align}
where  we have used 
 $$
u\le t\implies\frac{\ell^2_t\eta_t}{\ell^2_u\eta_u}\le 1.
 $$
This  completes the proof of \eqref{Eq:Gtlp_exp_flow}. 
\end{proof}

\section{Continuity Estimates on Loops}\label{sec:Inh_LE}
In this section, we  prove  Lemma \ref{lem_ConArg}. Recall  $t_1\le t_2\le 1, \; G_{t_1}=(H_{t_1}-z_{t_1})^{-1}$ and  $G_{t_2}=(H_{t_2}-z_{t_2})^{-1} $.
In this section, we will use the notation 
$$
\widetilde{ G}_{t_1}: = 
\left(  H_{t_2}-
 \widetilde z_{t_1}\right)^{-1},\quad 
 \widetilde z_{t_1}:=\left(\frac{t_2}{t_1}\right)^{1/2}z_{t_1}
$$
and denote by   \(\widetilde{\cal L}_{t_1}\)  the loops   with resolvent  \(\widetilde{G}_{t_1}\).
By scaling, we have 
\begin{equation}\label{calLcalL}
  \widetilde{\cal L}_{t_1}\sim \left(\frac{t_1}{t_2}\right)^{n/2} {\cal L}_{t_1}, \quad  
  \widetilde{ G}_{t_1}\sim \left(\frac{t_1}{t_2}\right)^{1/2}\cdot { G}_{t_1} \quad \hbox{in distribution}.  
\end{equation}
One can check that $|z_{t_2}-\widetilde{z}_{t_1}|\le C(1-t_1)$ for  $t_1>c>0$.   From now on,  we drop the $t$ subscript 
and use the following notations 
\begin{align}\label{GGLLzz}
G_{t_2} \longleftrightarrow  G,\quad \widetilde{G}_{t_1} \longleftrightarrow   \widetilde{G},\quad\quad
{\cal L}_{t_2} \longleftrightarrow  {\cal L},\quad \widetilde{{\cal L}}_{t_1} \longleftrightarrow   \widetilde{{\cal L}},
\quad\quad
{z}_{t_2} \longleftrightarrow  {z},\quad \widetilde{{z}}_{t_1} \longleftrightarrow   \widetilde{{z}}.
\end{align}
With these notations and the resolvent formula, 
\begin{equation}\label{GGZGG}
  G=\widetilde{G}+(z-\widetilde{z})\cdot G\cdot \widetilde{G}  .
\end{equation}
We first recall a basic linear algebra fact. 

\begin{lemma}\label{lem_L2Loc}
Let $v$ and $w_i$ ($1 \le i \le m$) be vectors in a Hilbert space $H$ and  $A_{ij} = (w_i, w_j)$.
Then for any $p \ge 1$,
$$
\sum_i |(v, w_i)|^2 \le \|v\|_2^2 \cdot (\operatorname{tr} A^p)^{1/p}.
$$
\end{lemma}

\begin{proof}
Define a linear operator $T: H \to \mathbb{C}^m$ by $w \to \sum_i (w, w_i) e_i$. Then
$$
\frac{\sum_i |(v, w_i)|^2}{\|v\|_2} = \frac{\|Tv\|^2}{\|v\|^2} = \|T^* T\|_{l_2 \to l_2} = \|TT^*\|_{l_2 \to l_2} = \|A\|_{l_2 \to l_2} \le (\operatorname{tr} A^p)^{1/p}.
$$
\end{proof}

\begin{proof}[Proof of Lemma \ref{lem_ConArg}]
Using the Cauchy-Schwarz inequality, we know that the $G$ loop with length $2m+1$ ($m\ge 1$) can be bounded by the $G$ loop with lengths $2m$ and $2m+2$, i.e.,
\begin{align}\label{adummzps}
    \max_{\textbf{a}, \boldsymbol{\sigma}} \left|\mathcal{L}^{(\text{length} = 2m+1)}_{t, \boldsymbol{\sigma}, \textbf{a}}\right|^2 \le \left|\max_{\textbf{a}, \boldsymbol{\sigma}} \mathcal{L}^{(\text{length} = 2m)}_{t, \boldsymbol{\sigma}, \textbf{a}} \cdot \max_{\textbf{a}, \boldsymbol{\sigma}} \mathcal{L}^{(\text{length} = 2m+2)}_{t, \boldsymbol{\sigma}, \textbf{a}}\right|.  
  \end{align}
Therefore, we will only prove \eqref{res_lo_bo_eta} for even loop lengths. Using the Cauchy-Schwarz inequality again, we can assume that the loop is symmetric (recall $G$ chain $\cal C$ defined in \eqref{defC=GEG}), i.e., 
\begin{align}\label{sym_calL}
        \mathcal{L}_{t, \boldsymbol{\sigma}', \textbf{a}'} = & \langle E_{a_0} \cdot \mathcal{C}_{t, \boldsymbol{\sigma}, \textbf{a}} \cdot E_{a_m} \cdot \mathcal{C}^\dagger_{t, \boldsymbol{\sigma}, \textbf{a}} \rangle, \quad \boldsymbol{\sigma} = (z_1, z_2, \ldots, z_m), \quad \textbf{a} = (a_1, a_2, \ldots, a_{m-1}) \nonumber \\
 \boldsymbol{\sigma}' = & (z_1, z_2, \ldots, z_m, \bar{z}_m, \bar{z}_{m-1}, \ldots, \bar{z}_1), \quad \textbf{a}' = ( a_1, a_2, \ldots, a_{m-1}, a_m, a_{m-1}, \ldots, a_1, a_0).
\end{align}
We only need to prove that for any fixed $m\ge 1$, the following holds 
\begin{equation}\label{eq_main_ConAr}
        {\bf 1}_{\Omega}\cdot  \mathcal{L}_{t_2, \boldsymbol{\sigma}', \textbf{a}'}
        \prec
        \left( W \cdot \ell_{t_1}\cdot \eta_{t_2} \right)^{-2m+1},
\end{equation}
under the inductive assumption that this bound holds for $  m' \le m-1$  for all $\textbf{a}$ and $\boldsymbol{\sigma}$.
Recall  the identity 
\begin{equation}\label{abform1}
        \prod_{k=1}^m (a_k + b_k) = \prod_{k=1}^m a_k + \sum_{l=1}^m \left(\prod_{j=1}^{l-1} (a_j + b_j)\right) b_l \left(\prod_{j=l+1}^m a_j\right). 
\end{equation}
With the notations $ G_i = G_{t_2}(\sigma_i), \quad \widetilde{G}_i = {\widetilde G}_{t_1}(\sigma_i)$, 
we apply this identity with $  a_i+b_i\to G_i,\quad a\to \widetilde{G}_i$.
Together with  the resolvent identity \eqref{GGZGG}, we have
\begin{align}
        G_1 E_{a_1} \cdots G_{m-1} E_{a_{m-1}} G_m = & \; \widetilde{G}_1 E_{a_1} \cdots \widetilde{G}_{m-1} E_{a_{m-1}} \widetilde{G}_m 
        \nonumber \\
        & + (z - \widetilde{z}) \cdot \sum_{l=1}^m G_1 E_{a_1} \cdots G_{l-1} E_{a_{l-1}} \cdot \left( G_l \cdot \widetilde{G}_l\right)
        E_{a_l} \widetilde{G}_{l+1} \cdots \widetilde{G}_m.
\end{align}
In the last term, there are \( l \) instances of \( G \) and \( m - l + 1 \) instances of \( \widetilde{G} \)
and there is no matrix $E$ between $G_l$ and $\widetilde{G}_l$. 
It implies that for any $i$ and $j$,
    \begin{align}\label{GwGtC}
        \left| \left( G_1 E_{a_1} \cdots G_{m-1} E_{a_{m-1}} G_m \right)_{ij} \right|^2 \le C_m & \;
        \left| \left( \widetilde{G}_1 E_{a_1} \cdots \widetilde{G}_{m-1} E_{a_{m-1}} \widetilde{G}_m \right)_{ij} \right|^2 \nonumber \\
        & + C_m \left|z_{t_2} - \widetilde{z}_{t_1}\right|^2 \cdot \sum_{l=1}^m \left| \left( G_1 E_{a_1} \cdots G_l \cdot \widetilde{G}_l E_{a_l} \cdots \widetilde{G}_m \right)_{ij} \right|^2.
    \end{align}
For any  $l, j$, denote by $ v^{(l)} \in \mathbb C^{N}, \; w^{(l)}_j\in \mathbb C^{N}$  the vectors with components given by  
$$
    v^{(l)}(k) = \left( G_1 E_{a_1} \cdots E_{a_{l-1}} G_l \right)_{ik}, \quad 
   \overline{  w^{(l)}_j(k)} = \left( \widetilde{G}_l E_{a_l} \cdots E_{a_{m-1}} \widetilde{G}_m \right)_{kj}.
$$
Here $i$ is treated as a given  index and it will not play any active role in the following proof. Denote by  $A^{(l)}$   the matrix with  elements given by 
$$
    A^{(l)}_{jj'} = \left( w^{(l)}_j, w^{(l)}_{j'} \right), \quad j, j' \in \mathcal{I}_{a_m}.
$$
Applying Lemma \ref{lem_L2Loc} to the last term in \eqref{GwGtC}, we have 
\begin{align}\label{jziwmas}
      \sum_{j \in \mathcal{I}_{a_m}} \left| \left( G_1 E_{a_1} \cdots G_l \cdot \widetilde{G}_l E_{a_l} \cdots \widetilde{G}_m \right)_{ij} \right|^2
    = \sum_{j \in \mathcal{I}_{a_m}} \left| \left( v^{(l)}, w^{(l)}_j \right)\right|^2
    \le \|v^{(l)}\|_2^2 \left( \operatorname{tr} \left( A^{(l)} \right)^p \right)^{1/p}.  
\end{align}
 
With  the identity $\widetilde{G} \cdot \widetilde{G}^\dagger = \frac{\widetilde{G} - \widetilde{G}^\dagger}{2 \operatorname{Im} \widetilde{z}}$ 
$  A^{(l)}_{jj'}$ can be written as $\widetilde{\cal C}$ chain of length $2m-2l+1$, namely,   
$$
     A^{(l)}_{jj'}
     =  \frac{1}{2 \operatorname{Im}\widetilde{z}}
   \left[ \left(
    \widetilde{G}_m  E_{a_{m-1}}  \cdots E_{a_l}\widetilde{G}_l \overline{ E_{a_l} \cdots E_{a_{m-1}} \widetilde{G}_m}
  \right)_{j' j}   
  -
    \left(
    \widetilde{G}_m  E_{a_{m-1}}  \cdots E_{a_l}\overline{\widetilde{G}_l E_{a_l} \cdots E_{a_{m-1}} \widetilde{G}_m} \right)_{j' j} 
    \right].
$$
Hence we can write $\left( \frac{2 \operatorname{Im} \widetilde{z}}{W} \right)^p \cdot \operatorname{tr} \left( A^{(l)} \right)^p$
as a sum of $2^p$  $\widetilde{\cal L}_{t_1}-$loops of Using  $2p(m-l) + p$. With \eqref{calLcalL}, we can bound 
 $\widetilde{\cal L}_{t_1}$ with $ {\cal L}_{t_1}$.  Using  the inductive assumption on $ {\cal L}_{t_1}$, we have 
    $$
    \left( \operatorname{tr} \left( A^{(l)} \right)^p \right)^{1/p} 
    \prec \frac{W}{\operatorname{Im} \widetilde{z}} 
    \left( W \ell_{t_1} \eta_{t_1} \right)^{-2(m-l)-1+1/p}, \quad \im \widetilde{z}_{t_1}\sim \im z_{t_1}. 
    $$
Inserting  this bound  into \eqref{jziwmas} and \eqref{GwGtC} and averaging  over $i \in \mathcal{I}_{a_0}$ and $j \in \mathcal{I}_{a_m}$,  we can bound \eqref{sym_calL} by 
\begin{align}\label{LLVWL}
        \mathcal{L}_{t_2, \boldsymbol{\sigma}', \textbf{a}'} 
        & \; = \left\langle E_{a_0} \cdot \mathcal{C}_{t, \boldsymbol{\sigma}, \textbf{a}} \cdot E_{a_m} \cdot \mathcal{C}^\dagger_{t, \boldsymbol{\sigma}, \textbf{a}} \right\rangle 
     \nonumber    \\
         & \;  \prec \; \widetilde{\mathcal{L}}_{t_1 , \;  \boldsymbol{\sigma}' , \textbf{a}'} 
         +\frac{|z_{t_2}-\widetilde{z}_{t_1}|^2}{W\operatorname{Im} \widetilde{z}_{t_1}}
         \sum_{l=1}^m 
         \left(\sum_{i \in \mathcal{I}_{a_0}}  \|v^{(l)}\|_2^2\right)\cdot 
         \left( W \ell_{t_1} \eta_{t_1} \right)^{-2(m-l)-1}.
\end{align} 
Here we have taken $p$ to be large enough so that  $1/p$ can be absorbed dropped in  the $\prec$ notation.  
By definition, one can easily check that (here  we use the assumption that   $t_1>c>0$ for some $c$)
$$
    |z_{t_2}-\widetilde{z}_{t_1}|^2\le C |t_2-t_1|^2\le C\left(\eta _{t_1}\right)^2,\quad \frac{|z_{t_2}-\widetilde{z}_{t_1}|^2}{\operatorname{Im} \widetilde{z}_{t_1}}
    \le C \eta _{t_1} .
$$
Using Ward's  identity $G \cdot G^\dagger = \frac{G - G^\dagger}{2 \operatorname{Im} z}$ again, we have 
\begin{align} 
 &   \frac{1}{W} \sum_{i \in \mathcal{I}_{a_0}} \|v^{(l)}\|_2^2  = \frac{1}{2 \operatorname{Im} z} \left( \mathcal{L}_{t, \boldsymbol{\sigma}_l^{(1)}, \textbf{a}_l} - \mathcal{L}_{t, \boldsymbol{\sigma}_l^{(2)}, \textbf{a}_l} \right), \quad 
     \textbf{a}_l = (a_0, a_1, \ldots, a_{l-1}, a_{l-1}, \ldots, a_1), \nonumber \\
& \boldsymbol{\sigma}_l^{(1)}  = (\sigma_1, \sigma_2, \ldots, \sigma_{l-1}, \sigma_l, \bar{\sigma}_{l-1}, \ldots, \bar{\sigma}_2, \bar{\sigma}_1),\quad 
\boldsymbol{\sigma}_l^{(2)}  = (\sigma_1, \sigma_2, \ldots, \sigma_{l-1}, \bar{\sigma}_l, \bar{\sigma}_{l-1}, \ldots, \bar{\sigma}_2, \bar{\sigma}_1).    
\end{align} 
Here the  $\cal L$ loops are length $2l-1$. 
We can bound these loops by  induction so that 
$$
  {\bf 1}_{\Omega}\cdot   \frac{1}{W} \sum_{i \in \mathcal{I}_{a_0}} \|v^{(l)}\|_2^2 \prec \frac{1}{\operatorname{Im} z} \left( W \ell_{t_1} \eta_{t_2} \right)^{-2l+2}, \quad l < m,\quad z=z_{t_2}. 
$$
Inserting  this bound  to \eqref{LLVWL},   we have proved that
    \begin{equation}\label{LWLIMZ}
      {\bf 1}_{\Omega}\cdot    \mathcal{L}_{t_2, \boldsymbol{\sigma}', \textbf{a}'} \prec \left( W \ell_{t_1} \eta_{t_2} \right)^{-2m+1} + \frac{\eta_{t_1}}{\eta_{t_2}}
        \cdot
        \left(
        \mathcal{L}_{t, \boldsymbol{\sigma}_m^{(1)}, \textbf{a}_m} - \mathcal{L}_{t, \boldsymbol{\sigma}_m^{(2)}, \textbf{a}_m} \right) \left( W \ell_{t_1}\eta_{t_1} \right)^{-1}.
\end{equation}
Here $\mathcal{L}_{t_2, \boldsymbol{\sigma}_m^{(1)}, \textbf{a}_m}$ and $\mathcal{L}_{t_2, \boldsymbol{\sigma}_m^{(2)}, \textbf{a}_m}$ are   $G$ loops with the length $2m-1$ and   $\boldsymbol{\sigma}_m^{(1)}, \; \boldsymbol{\sigma}_m^{(2)}\in \{+,-\}^{2m-1}$.

For  $m=1$, by definition \eqref{usuayzoo}, 
\[
      {\bf 1}_{\Omega}\cdot  \mathcal{L}_{t_2, \boldsymbol{\sigma}_m^{(1)}, \textbf{a}_m},\; \;
         {\bf 1}_{\Omega}\cdot \mathcal{L}_{t_2, \boldsymbol{\sigma}_m^{(2)}, \textbf{a}_m}
    \prec 1, \quad m=1,\quad \boldsymbol{\sigma}_m^{(1)}, \;\boldsymbol{\sigma}_m^{(2)}\in\{+,-\}.
\]
Inserting  these bounds  to \eqref{LWLIMZ}, we obtain \eqref{eq_main_ConAr} in the case $m=1$.  
For $m>1$,  applying \eqref{adummzps} again,  we can bound $2m-1$-$G$-loop by a $G-$loop of length $2m-2$ and a $G-$loop of length $2m$. Furthermore, we can use assumption \eqref{eq_main_ConAr} to  bound this  $G$ loop with length $2m-2$. Then we obtain that 
$$
   {\bf 1}_{\Omega}\cdot  \mathcal{L}_{t_2, \boldsymbol{\sigma}_m^{(1,2)}, \textbf{a}_m} \prec \left[ \left( W \ell_{t_1} \eta_{t_2} \right)^{-2m+3} \max_{\boldsymbol{\sigma} , \textbf{a} } \left| \mathcal{L}^{\text{length=2m}}_{t_2, \boldsymbol{\sigma} , \textbf{a} } \right| \right]^{1/2}.
    $$
Inserting  this bound to \eqref{LWLIMZ}, we obtain 
$$
      {\bf 1}_{\Omega}\cdot\max_{\boldsymbol{\sigma}' , \textbf{a}' } \left| \mathcal{L}^{\text{length=2m}}_{t_2, \boldsymbol{\sigma}' , \textbf{a}' } \right|
    \prec \left( W \ell_{t_1} \eta_{t_2} \right)^{-2m+1 } 
    +\left( W \ell_{t_1} \eta_{t_2} \right)^{-m+1/2}
     \max_{\boldsymbol{\sigma}' , \textbf{a}' } \left| \mathcal{L}^{\text{length=2m}}_{t_2, \boldsymbol{\sigma}' , \textbf{a}' } \right|^{1/2}. 
$$
This  implies  \eqref{eq_main_ConAr} and we have thus proved  Lemma \ref{lem_ConArg}.
\end{proof}

\section{Auxiliary proofs and technical estimates}

\subsection{Evolution kernel estimates}\label{ks}

We first prove a simple $L_{\max\to \max}$ bound  for $\cal U$ \ref{def_Ustz} in the following lemma. 

\begin{lemma}[$\|{\cal U}\|_{\max\to \max}$ estimate]\label{lem:sum_Ndecay}
Let $\cal A$ be a tensor $\mathbb Z_L^n\to \mathbb C$ and $n\ge 2$.  
Then 
\begin{align}\label{sum_res_Ndecay}
   \| {\cal U}_{s,t,\boldsymbol{\sigma}}\;\circ {\cal A}\|_{\max } \prec \|{\cal A}\|_{\max}
    \cdot \left(  \eta_s/ \eta_t\right)^{n } ,\quad s<t. 
\end{align}
\end{lemma}
\begin{proof} 
Recall that ${\cal U}_{s, t}$ was defined in \ref{def_Ustz}.  We  write  
$$
\left({\cal U}_{s, t}\circ {\cal A}^{\cal T}\right)_{\textbf{a}} =\sum_{\textbf{b}} \prod_{i=1}^2 (\Theta_{t  }\Theta^{-1}_{s })_{a_ib_i}
{\cal A}_{\textbf{b}}.
$$
With $ \Theta_t\cdot \Theta_{s}^{-1}=1-(t-s)\cdot S^{(B)}\cdot \Theta_t $,  we have 
\begin{align}\nonumber
  \prod_{i=1}^2  (\Theta_{t  }\Theta^{-1}_{s })_{a_ib_i}
  = &  \prod_{i=1}^2 \delta_{a_ib_i}
  -(t-s) \delta_{a_1b_1}\cdot \left(S^{(B)}\cdot \Theta_t\right)_{a_2b_2}
 -(t-s) \delta_{a_2b_2}\cdot \left(S^{(B)}\cdot \Theta_t\right)_{a_1b_1}
 \\\label{ajzuis}
 + &(t-s)^2   \left(S^{(B)}\cdot \Theta_t\right)_{a_1b_1}\left(S^{(B)}\cdot \Theta_t\right)_{a_2b_2}  
 := {\cal V} ^{(1)}_{\textbf{a}, \textbf{b}}+ {\cal V} ^{(2)}_{\textbf{a}, \textbf{b}}+ {\cal V} ^{(3)}_{\textbf{a}, \textbf{b}}+ {\cal V} ^{(4)}_{\textbf{a}, \textbf{b}}.
\end{align}
Lemma \ref{lem:sum_Ndecay} follows from this decomposition and the simple fact 
$$
\sum_{b}\left|\Theta^{(B)}_{t}\right|_{ab}=O\left(1/\eta_t\right). 
$$
\end{proof}

If $\cal A $ decays on the scale $\ell_s$, then ${\cal U}_{s,t, (+,-)}\circ \cal A$ decays on the scale $\ell_t$.

\begin{lemma}[Tail Estimates]\label{TailtoTail}
Recall  
 $${\cal T}_{t}(\ell) := \frac{1}{(W\ell_t\eta_t)^2} \exp \left(- \left| \ell /\ell_t\right|^{1/2} \right).
$$
For $\boldsymbol{\sigma}=(+,-)$, assume that for ${\cal A}_{\textbf{a}}$, $\textbf{a}\in \mathbb Z^2_L$ and some large $D>0$, we have 
$$
{\cal A}_{\textbf{a}}\le {\cal T}_{s }(a_1-a_2)+W^{-D}. 
$$
Recalling  $\cal U$ defined in Definition  \ref{def_Ustz}, we have 
\begin{align}
\label{neiwuj} 
\left({\cal U}_{s,t,\boldsymbol{\sigma}} \circ 
{\cal A}\right)_{\textbf{a}} & \prec
 {\cal T}_{t}( a_1-a_2)+W^{-D} \cdot(\eta_s/\eta_t)^2, \quad {\rm if}\quad  |a_1-a_2|\ge \ell_t^* :=(\log W)^{3/2}\ell_t . 
\end{align}
 \end{lemma}
We first give a heuristic argument for the proof. Notice that the key term in ${\cal U}_{s,t,\boldsymbol{\sigma}}$ is 
\[
(t-s)^2    \Theta_{t, a_1b_1} \Theta_{t, a_2b_2}
\sim  (\eta_s/\eta_t)^2 \omega_{t}( a_1-b_1)\omega_{t}( a_2-b_2);  \quad 
\omega_{t}(a-b) =  \eta_t  \Theta_{t, ab}.
\]
With this definition, we have the normalization $ \sum_b \omega_{t} (b) = O(1)$.
The corresponding term in $\left({\cal U}_{s,t,\boldsymbol{\sigma}} \circ 
{\cal A}\right)_{\textbf{a}} $ is bounded by 
\begin{align}
(t-s)^2   \sum_{b_1, b_2}  \Theta_{t, a_1b_1} \Theta_{t, a_2b_2} {\cal T}_{s }(b_1-b_2)
\sim (\eta_s/\eta_t)^2  \sum_{b_1, b_2} \omega_{t}( a_1-b_1) \omega_{t}( a_2-b_2) {\cal T}_{s }(b_1-b_2).
\end{align}
Using $\omega_{t} \ast  \omega_{t} \sim  \omega_{t}$ and denoting $a_1-a_2 = x$, the last line is bounded by 
\begin{align}\label{72} 
& (\eta_s/\eta_t)^2  (W\ell_s\eta_s)^{-2}  \ell_t^{-1} \sum_{b}  \exp \left(- | x-b |/\ell_t - \sqrt { | b |/\ell_s  } \right) \nonumber \\
& \le   (W\ell_t\eta_t)^{-2}  \big ( \ell_t \ell_s^{-2} ) \sum_{b}  \exp \left(- | x-b |/\ell_t - \sqrt { | b |/\ell_s  } \right) .
\end{align}
We can also assume that $b \le \ell_s (\log W)^{2+1/4}$ . Otherwise, the last line  is exponentially  small.   By assumption,  $x \ge (\log W)^{3/2}\ell_t$.  If  $(\log W) \ell_s \le \ell_t$,   then  $ |x-b| \ge (\log W)^{3/2}\ell_t/2$  and the last line is exponentially small.  Hence we can assume that  $(\log W) \ell_s \ge \ell_t$. In this case,  we use the trivial bound  $-\sqrt { | b |/\ell_s  } \le -\sqrt { | b |/\ell_t  }$   and perform the $b$ summation to bound   \eqref{72}  by 
\begin{align}
(\log W)^C  (W\ell_t\eta_t)^{-2}    \exp \left( - \sqrt {  x/\ell_t } \right). 
\end{align}

\begin{proof}[Proof of Lemma \ref{TailtoTail}] We split $\cal A$ into two parts:
$$
{\cal A}={\cal A}^{\cal T} + {\cal A}^{D},\quad {\cal A}^{\cal T}_{\textbf{a}}\le {\cal T}_{s }(a_1-a_2),\quad {\cal A}^{D}\le W^{-D}.
$$ 
Since $\cal U$ is linear, we only need to prove that 
 \begin{align}\label{AAbound}
&\left({\cal U}_{s,t,\boldsymbol{\sigma}} \circ 
{\cal A}^{\cal T}\right)_{\textbf{a}}  \prec
 \; {\cal T}_{t}( a_1-a_2)+W^{-D}, 
\quad 
& &\left({\cal U}_{s,t,\boldsymbol{\sigma}} \circ  
{\cal A}^{D}\right)_{\textbf{a}} \prec 
W^{-D} \cdot(\eta_s/\eta_t)^2.
\end{align}
One can easily use the  $\cal U_{\max\to \max}$ bound in Lemma \ref{lem:sum_Ndecay} to prove
the estimate on ${\cal A}^D$. It remains  to prove the ${\cal A}^{\cal T}$ part in \eqref{AAbound}. 
Recall \eqref{ajzuis} and the bound 
\begin{align}\label{asfawwws}
  \|\Theta^{(B)}_t\|_{\max}\prec (\ell_t\eta_t)^{-1} . 
\end{align}
 Due to  the decay of  $\Theta_t$ and  ${\cal A}^{\cal T}$, we can restrict $\textbf{b}$ in $\sum_{\textbf{b}}$ to 
\begin{equation}\label{UATSI}
\left({\cal U}_{s, t}\circ {\cal A}^{\cal T}\right)_{\textbf{a}} =\sum_{b_1,\,b_2}^*\;\prod_{i=1}^2 (\Theta_{t  }\Theta^{-1}_{s })_{a_ib_i}
{\cal A}_{\textbf{b}}+W^{-D},
\end{equation}
where 
\begin{align}\label{condbbb}
\sum_{b_1,\,b_2}^*=\sum_{b_1,\,b_2}{\bf 1}\left(  |a_i-b_i|\le \frac14  \ell^*_t\right)\cdot {\bf 1}\left(|b_1-b_2|\le \log^3 W \cdot \ell_s\right).
\end{align}
By assumption $|a_1-a_2|\ge \ell^*_t=(\log W)^{3/2}\cdot  \ell_t$, the above summation is nontrivial only if 
$$
\ell_s\ge \ell_t \cdot (\log  W)^{-2}.
$$
In another words,  if $|a_1-a_2|\ge \ell_t^*$, then for any $D>0$,
$$
\ell_s\le \ell_t \cdot (\log  W)^{-2}\implies \left(\mathcal{U}_{s, t, \sigma} \circ \mathcal{A}\right)_a\le W^{-D}.
$$
From now on, we  assume that 
$$\ell_t\prec \ell_s,\quad\sum_{b_1,\,b_2}^* 1 \prec \ell_t\ell_s, \quad \sum_{b_1,\,b_2}^* \delta_{a_1b_1}\prec \ell_s.
$$ 
Together with \eqref{ajzuis}, \eqref{asfawwws},  ${\cal A}^{\cal T}_{b_1b_2}\prec {\cal T}_{s}(b_1-b_2)$, and $t-s=O(\eta_s)$),  we  have 
\begin{samepage}
\begin{align}\label{jssfgyz2}
    \left( \sum_{\textbf{b}}^* {\cal V}^{(1,2,3,4)}\circ
{\cal A}^{\cal T}_{\textbf{b}}\right)\Bigg/{\cal T}_{t }(a_1-a_2)
 \prec   &\max_{\textbf{b}}^*\;  \exp\left(  \ell_t^{-1/2}\left( |a_1-a_2|^{1/2}-|b_1-b_2|^{1/2} \right)\right),  
\end{align}
\end{samepage}
where $\max_{\textbf{b}}^*$ satisfies the condition in \eqref{condbbb}. Under this condition, we have 
$$
 |a_1-a_2|-|b_1-b_2|\le \sum_i|a_i-b_i|\le \frac12 (\log W)^{3/2}\ell_t. 
$$
Note that there is no absolute values on the left hand side. 
It is easy to prove that 
\begin{equation}\label{maxbaabb}
     \max_{\textbf{b}}^*\left( |a_1-a_2|^{1/2}-|b_1-b_2|^{1/2}\right) \le C\log^{3/4 } W\cdot \ell^{1/2}_t.
\end{equation}
Together with \eqref{jssfgyz2} and \eqref{UATSI},  
we have proved  Lemma \ref{TailtoTail}.  
 \end{proof}
 
 \bigskip
 
\begin{lemma}[${\cal U}_{s,t}$ on fast decay tensor]\label{lem:sum_decay}
Let $\cal A$ be a tensor $\mathbb Z_L^n\to \mathbb R$, $n\ge 2$. We say $\cal A$ is $(\tau, D)$ decay at time $s$ if  for some fixed small $\tau>0$ and large $D>0$,   
\begin{equation}\label{deccA0}
    \max_i\|a_i-a_j\|\ge \ell_s W^{\tau} \implies  {\cal A}_{\textbf{a}}=O(W^{-D}),\quad \textbf{a}=(a_1,a_2\cdots, a_n) .
\end{equation}
Then we have the following $\max\to \max$ norm for ${\cal U}_{s,t,\boldsymbol{\sigma}}$ for  $n\ge 2$: 
\begin{align}\label{sum_res_1}
    \left({\cal U}_{s,t,\boldsymbol{\sigma}} \circ {\cal A}\right)_\textbf{a} \le C_n W^{C_n\tau}\cdot   \|{\cal A}\|_{\max}
    \cdot \left(\frac{\ell_t}{\ell_s}\right) 
    \cdot \left(\frac{\ell_s\eta_s}{\ell_t\eta_t}\right)^{n} +W^{-D+C_n}.
\end{align}
Suppose either one of the following two assumptions hold: 
Case 1: For some $k:$ $1\le k\le n$,  
$$ 
\sigma_k=\sigma_{k-1}, \quad \boldsymbol{\sigma}=(\sigma_1,\cdots, \sigma_n).
$$
Case 2: ${\cal A}_{\textbf{b}}$ has the sum zero property in the sense that 
\begin{align}\label{sumAzero}
 \sum_{a_2,\,\cdots,\, a_n}{\cal A}_{\textbf{a}}=0, \quad \forall a_1 \; .
\end{align}
Then we have the following  stronger bound 
\begin{align}\label{sum_res_2}
    \left({\cal U}_{s,t,\boldsymbol{\sigma}} \circ {\cal A}\right)_\textbf{a} \le W^{C_n\tau}\cdot   \|{\cal A}\|_{\max} 
    \cdot \left(\frac{\ell_s\eta_s}{\ell_t\eta_t}\right)^{n } +W^{-D+C_n}. 
\end{align} 
\end{lemma} 

Lemma \ref{lem:sum_decay} can be understood as follows. The evolution kernel is approximately given by 
\[
(\eta_s/\eta_t)^n  {\otimes^n}   \omega_{t}
\]
where $\omega_{t}$ is an $L_1$ normalized convolution kernel of width $\ell_t$. 
So operator norm of this kernel in $L_\infty$ is $(\eta_s/\eta_t)^n $.  The decay length of 
$A$ is $ \ell_s < \ell_t$. Hence we gain a factor $\ell_s/\ell_t$ for each summation restricted by $A$. Since there are $n-1$ summations restricted by $A$, we gain $ (\ell_s/\ell_t)^{n-1}$
and this explains \eqref{sum_res_1}. For the sum zero case, we gain an extra 
$(\ell_s/\ell_t)$ as we can sum by parts once  and the ratio of smoothness between $\omega_t$ and $A$ is $(\ell_s/\ell_t)$. For the case that  $\sigma_k=\sigma_{k-1}$, the $\Theta$ operator has become significantly smaller and we also gain  an extra factor. The details will be given in the proof. 

\begin{proof}[Proof of Lemma \ref{lem:sum_decay}]
We first prove \eqref{sum_res_1}.  By definition of ${\cal U}_{s, t}$ in \eqref{def_Ustz} and \eqref{def_Ustz_2}, we have 
\begin{align}
      \left(\mathcal{U}_{s, t, \boldsymbol{\sigma}} \circ \mathcal{A}\right)_a=\sum_{b_1, \ldots, b_n} \prod_{i=1}^n\psi_i  \cdot \mathcal{A}_{\boldsymbol{b}}, \quad \boldsymbol{b}=\left(b_1, \ldots, b_n\right),
\end{align}
\begin{align}\label{def_psixi}
     \psi_i=\delta_{a_ib_i}+\Xi_i,  \quad \quad 
\Xi_i:=-(s-t)\xi_i \cdot \left(S^{(B)}\cdot \Theta^{(B)}_{t \xi_i}\right)_{a_ib_i},\quad\quad \xi_i=m(\sigma_i)m(\sigma_{i+1}).
 \end{align} 
Recall the following identity for all $\textbf{x}$ and $\textbf{y}$:
$$
\prod_i(x_i)-\prod_i (x_i-y_i) 
=-\sum_{\emptyset\ne A\subset[[1,n]]}\left(\prod_{j\in A^c}x_j \right)
\left(\prod_{j\in A }(-y_i)\right).
$$
Choosing   $x_i=\psi_i$ and $y_i=\delta_{a_ib_i}$, we have  $x_i-y_i=\Xi_i$ and  
 \begin{align}
    \prod_{i=1}^n  \psi_i 
    -\prod_{i=1}^n \ \Xi_i = \sum_{\emptyset\ne A\subset[[1,n]]}\left(\prod_{j\in A^c}\psi_j \right)
(-1)^{|A|+1}\left(\prod_{j\in A }\delta_{a_jb_j} \right).
\end{align}
Assume that we have proved Lemma \ref{lem:sum_decay} for any $k<n$.  Then by  inductive assumption, 
\begin{align}\label{iukwjn}
\sum_{\textbf{b}} \left(\prod_{i=1}^n  \psi_i- \prod_{i=1}^n  \Xi_i\right) \cdot 
{\cal A}_{\textbf{b}} 
 \; \le \; & C_n W^{C_n\tau}\|{\cal A}\|_{\max}\cdot \left(\frac{\ell_t}{\ell_s}\right) 
    \cdot \left(\frac{\ell_s\eta_s}{\ell_t\eta_t}\right)^{n-1 } +W^{-D+C_n}
    \\\nonumber
  \; \le \; &  C_n W^{C_n\tau}\|{\cal A}\|_{\max}\left(\frac{\ell_s\eta_s}{\ell_t\eta_t}\right)^{n  }    +W^{-D+C_n} ,  
\end{align}
where we have used $\ell_t/\ell_s\le \left( \ell_s\eta_s\right)/\left(\ell_t\eta_t\right)$. 

For \eqref{sum_res_1} we only need to bound
\begin{align}
    \sum_{\textbf{b}}  \left(\prod_{i=1}^n  \Xi_i \right) \cdot 
{\cal A}_{\textbf{b}} \le C_n W^{C_n \tau} \cdot\|\mathcal{A}\|_{\max } \cdot\left(\frac{\ell_t}{\ell_s}\right) \cdot\left(\frac{\ell_s \eta_s}{\ell_t \eta_t}\right)^n+W^{-D+C_n}.
\end{align}
By definition of $\Theta^{(B)}$, we have 
$$
\Xi_i=O\left( \eta_s \,\ell_t^{-1}\eta_t^{-1}\right).
$$
By  $\ell_s$-decay property of $\cal A$ in \eqref{deccA0}, 
\begin{align}\label{llswuwg}
  \sum_{\textbf{b}}   \prod_{i=1}^n  \Xi_i  \cdot 
{\cal A}_{\textbf{b}}
\;\le\; & C W^{\tau}\cdot \frac{\ell_s \eta_s}{\ell_t \eta_t} \max_{b_n}
\cdot \left|\sum_{b_1\cdots b_{n-1}} \prod_{i=1}^{n-1}  \Xi_i  \cdot 
{\cal A}_{\textbf{b}}\right|+W^{-D+C}
\\\nonumber
\;\le\; & C_n W^{C_n\tau} \left(\frac{\ell_s \eta_s}{\ell_t \eta_t} \right)^{n-1} \max_{b_2,\cdots, b_n}  \left|\sum_{b_1 }\Xi_1 {\cal A}_{\textbf{b}}\right|+W^{-D+C_n}. 
\end{align}
By the decay of $\Theta^{(B)}_t$, we have 
\begin{align}\label{tyzcsq}
  \left|\sum_{b_1 }\Xi_1 {\cal A}_{\textbf{b}}\right|\le CW^\tau (\eta_s/\eta_t)\|{\cal A}\|_{\max}+W^{-D}  .
\end{align}
 Together with \eqref{llswuwg} and \eqref{iukwjn}, we have proved  \eqref{sum_res_1}. 

Next we prove Case 1 of \eqref{sum_res_2}. Without loss of generality, we assume that $\sigma_1=\sigma_2=+$. In this case, $\xi_1=m^2$, and thus  $\sum_{b_1}|\Xi_1|=O(1)$. this  implies that 
$$
\left|\sum_{b_1 }\Xi_1 {\cal A}_{\textbf{b}}\right|\le  C\|{\cal A}\|_{\max} .
$$
Together with \eqref{llswuwg} and \eqref{iukwjn}, we have proved  \eqref{sum_res_2} in this case.

We finally  prove Case 2 of \eqref{sum_res_2}, i.e., $\cal A$ has the sum zero property \eqref{sumAzero}. 
Since we have proved \eqref{sum_res_2} for the case 1,  we can now assume  that $\sigma_i\ne \sigma_{i+1}$ for any $1\le i\le n$. It implies $\xi_i=|m|^2=1$ (in \eqref{def_psixi}).    For each $\Xi_i$ with $i\ge 2$,  we write it as
$\Xi_i=\Xi^0_i+\Xi^*_i$ where
$$
\frac{\Xi^0_i}{-(s-t)\xi_i }:= \left(S^{(B)}\cdot \Theta^{(B)}_{t \xi_i}\right)_{a_ib_1},\quad
\frac{\Xi^*_i}{-(s-t)\xi_i }:= \left(\left(S^{(B)}\cdot \Theta^{(B)}_{t \xi_i}\right)_{a_ib_i}-\left(S^{(B)}\cdot \Theta^{(B)}_{t \xi_i}\right)_{a_ib_1}\right).
$$
Note that the subscript of $\Xi^0_i$ is $a_ib_1$, but the one for $\Xi_i$ is $a_ib_1$.  Thus 
$$\sum_{\textbf{b}}   \left(\prod_{i=1}^n  \Xi_i\right)  \cdot 
{\cal A}_{\textbf{b}}=\sum_{\textbf{b}} 
\left(\prod_{i=1}^n  \left(\Xi^0_i+  \Xi^*_i \right)\right)\cdot 
{\cal A}_{\textbf{b}}.
$$
After expanding the $\prod_i$, the leading term disappears due to the sum zero property (and the fact that $\Xi_i^0$ is independent of  $b_2,\cdots, b_n$), namely,  
$$
\sum_{\textbf{b}}
\left(\prod_{i=1}^n \Xi^0_i\right)\cdot {\cal A}_{\textbf{b}}=0.
$$
For the other terms, we bound them by 
$$
\left|\Xi_i^0\right|\le C\frac{\eta_s}{\ell_t\eta_t}, \quad \quad   \left|\Xi_i^*\right|\le C\frac{\eta_s}{\ell _t\,\eta_t}\cdot \frac{|b_i-b_1|}{\ell_t}.
$$
By the decay  property of $\cal A$  \eqref{deccA0},  the main contribution tothe last equation comes from $|b_i-b_1|\le W^\tau \ell_s$. Therefore, we obtain an improved bound for  \eqref{llswuwg}, i.e., 
\begin{align}\label{llswuwg2}
  \sum_{\textbf{b}}   \prod_{i=1}^n  \Xi_i  \cdot 
{\cal A}_{\textbf{b}}
\;\le\; & C_n W^{C_n\tau} \left(\frac{\ell_s \eta_s}{\ell_t \eta_t} \right)^{n-1} \cdot \frac{\ell_s}{\ell_t}\cdot  \max_{b_2,\cdots, b_n}  \left|\sum_{b_1 }\Xi_1 {\cal A}_{\textbf{b}}\right|+W^{-D+C_n}. 
\end{align}
Together with \eqref{tyzcsq}, we have proved Case 2 of   \eqref{sum_res_2}. This completes the proof of Lemma \ref{lem:sum_decay}. 
 
\end{proof}

\subsection{Auxiliary proofs for universality}\label{Apu}

In this subsection we prove \eqref{WeakQUE} and \eqref{Weakloclaw} for $H_t$ defined in \eqref{defUnivHt}. This notation is particular for universality proof and different from $H_t$ defined in \eqref{def_mainHT} for the main proof) For the matrix $H_t$, the case $t=0$ corresponds to our primary result for the band matrix resolvent. The cases for $0 < t < t_U$ can be handled similarly to the case $t = t_U$. Thus, for simplicity, we only provide the proof for the case $t = t_U$, i.e., the estimate for $H_{t_U}$. We define the random matrix \begin{equation}\label{defwtHHtu}
     \widetilde{H} := H_{t_U}
\end{equation}
 
 By definition, \( \widetilde{ H} \) has entry variances given by  
\[
\mathbb{E} \big| \widetilde{H}_{xy} \big|^2 =\mathbb{E} \big|  \left(H_{t_U}\right)_{xy} \big|^2= \widetilde{S}_{xy}  = (1-\zeta_U) S_{xy} + \frac{\zeta_U}{N}, \quad \zeta_U \sim  N^{-1+\tau_U}, \quad \zeta_U=1-e^{-t_U}
\]

In our main proof, we use the flow 
 $$
  z^{(E)}_t=E+(1-t) m^{(E)}  ,\quad 0\le t\le 1.   
$$ 
Then for \( z \in \mathbb{C} \), we choose \( t_0 \) and \( E \) such that  
\[
z = t_0^{-1/2} z_{t_0}^{(E)},
\]
and study \( G(z) \) via \( t_0^{1/2} \cdot G^{(E)}_{t_0} \) as in \eqref{GtEGz}. Here for fixed $z$,  we keep the same stochastic flow for \( z^{(E)}_t \), but the  flow for \( \widetilde{H}_t \) is defined by \( \widetilde{H}_0 = 0 \) and  
\[
d\widetilde{H}_{t,ij} =
\begin{cases}
    \sqrt{S_{ij}} \cdot d{\cal B}_{t,ij} & 0\le t \le t_1 := (1-\zeta_U)t_0, \\[6pt]
      N^{-1/2} \cdot d{\cal B}_{t,ij} &   t_1 \le   t\le t_0.
\end{cases} 
\]
It is easy to check that for any fixed $z$ satisfying $z = t_0^{-1/2} z_{t_0}^{(E)}$, 
\begin{align}\label{whheiyslwuw}
(\widetilde{H}-z)^{-1}\; \overset{d}{\sim}\;  t_0^{1/2}(\widetilde{H}_{t_0}-z_{t_0})^{-1} 
\end{align}
Define $\widetilde{G}_t$ and $\widetilde{{\cal L}}_t$ with $\widetilde{H}_t$ as in Definitions \ref{def_flow} and  \ref{Def:G_loop}. By definition for $t<t_1$, 
$$
 \widetilde{H}_t\; \overset{d}{\sim}\; H_t,\quad
\widetilde{{\cal L}}_t\;\overset{d}{\sim}\; {{\cal L}}_t, \quad t\le t_1.
 $$
 Here  $H_t$ and ${{\cal L}}_t$ are defined in \eqref{def_mainHT} and \eqref{Eq:defGLoop} respectively.  We aim to estimate $\widetilde{{\cal L}}_t$ with $\widetilde{{\cal K}}_t$ for $t\in [t_1,t_0]$. Recall  $\widetilde{{\cal K}}$ is the solution to the primitive equation in Definition \ref{Def_Ktza}. For  $t\in  [t_1,t_0]$ the flow equation    becomes 
$$\frac{d}{d t} \mathcal{\widetilde{K}}_{t, \sigma, a}=W \cdot \sum_{1 \leq k<l \leq n} \sum_{a, b}\left(\mathcal{G}_{k, l}^{(a), L} \circ \mathcal{\widetilde{K}}_{t, \sigma, a}\right) \left(S_{GUE}^{(B)}\right)_{a b}\left(\mathcal{G}_{k, l}^{(b), R} \circ \mathcal{\widetilde{K}}_{t, \sigma, a}\right),\quad t\in [t_1,t_0]
$$
with 
$$ \left(S^{(B)}_{GUE}\right)_{ab}=1/L, \quad \mathcal{\widetilde{K}}_{t, \sigma, a}=\mathcal{ {K}}_{t, \sigma, a},\quad t\in [0,t_1].
$$ 
To prove  \eqref{WeakQUE} and \eqref{Weakloclaw}, we only need to focus on the following case: for small enough $\tau_U$ with 
$$
z=E+\eta i, \quad \quad  N^{-1+2\tau_U }\le \eta\le N^{-1+\mathfrak c/3}, 
$$
and $t\in [t_1,t_0]$, we have 
\begin{align} \label{WLKNE1}
\max_{\boldsymbol{\sigma}, \textbf{a}}\left|{\cal \widetilde{L}}_{t , \boldsymbol{\sigma}, \textbf{a}} \right|&\;\prec\; (N\eta_{t })^{-1}, 
\\ \label{WLKNE2}
\max_{\boldsymbol{\sigma}, \textbf{a}}\left|{\cal \widetilde{L}}_{t , \boldsymbol{\sigma}, \textbf{a}}-{\cal \widetilde{K}}_{t , \boldsymbol{\sigma}, \textbf{a}}\right|&\;\prec\; (N\eta_{t })^{-2}, 
\\ \label{WLKNE3}
\max_{\boldsymbol{\sigma}\in\{+,-\}^2}\max_{\textbf{a}}\left|\mathbb E {\cal \widetilde{L}}_{t , \boldsymbol{\sigma}, \textbf{a}}-\mathbb E{\cal \widetilde{K}}_{t , \boldsymbol{\sigma}, \textbf{a}}\right|&\;\prec\; (N\eta_{t })^{-3}.
\end{align} 
In this setting ($\im z= \eta$) 
\begin{equation}\label{seywoim29}
     1-t_0\sim 1-t_1\sim \eta,\quad \quad   t_0-t_1\sim N^{-1+ \tau_U},\quad\quad N^{-1+2\tau_U }\le \eta\le N^{-1+\mathfrak c/3}.
\end{equation}
 Notice that $t_0-t_1 \ll \eta$. By Lemma \ref{ML:Kbound} and \ref{ML:GLoop}, for $t\in [0,t_1]$, 
\begin{align}
 \max_{\boldsymbol{\sigma}, \textbf{a}}\left|{\cal \widetilde{L}}_{t, \boldsymbol{\sigma}, \textbf{a}}-{\cal \widetilde{K}}_{t, \boldsymbol{\sigma}, \textbf{a}}\right|\prec (W\ell_t\eta_t)^{-n},\quad \eta_t:=\im z_t,  
\\\nonumber
\max_{\boldsymbol{\sigma}, \textbf{a}}\left|{\cal \widetilde{K}}_{t, \boldsymbol{\sigma}, \textbf{a}} \right| \prec (W\ell_t\eta_t)^{-n+1},\quad \quad \max_{\boldsymbol{\sigma}, \textbf{a}}\left|{\cal \widetilde{L}}_{t, \boldsymbol{\sigma}, \textbf{a}} \right|\prec (W\ell_t\eta_t)^{-n+1}. 
\end{align} 
With assumption \eqref{seywoim29},  we have $\ell_{t_1}=L$.  
Then for $t=t_1$,
\begin{align}\label{uwi2oo1}
\quad & \max_{\boldsymbol{\sigma}, \textbf{a}}\left|{\cal \widetilde{L}}_{t_1, \boldsymbol{\sigma}, \textbf{a}}-{\cal \widetilde{K}}_{t_1, \boldsymbol{\sigma}, \textbf{a}}\right|\prec (N\eta_{t_1})^{-n}, 
\quad \max_{\boldsymbol{\sigma}, \textbf{a}}\left(\left|{\cal \widetilde{K}}_{t_1, \boldsymbol{\sigma}, \textbf{a}} \right| +\left|{\cal \widetilde{L}}_{t_1, \boldsymbol{\sigma}, \textbf{a}} \right|\right)\prec (N\eta_{t_1})^{-n+1}.  \end{align}

For $  t_1\le t\le t_0$, we will have a new  loop hierarchy and primitive equation. The new version of these equations are the same as the standard versions in  \eqref{eq:mainStoflow} and \eqref{pro_dyncalK} except that  $S^{(B)}$ and $S$ are replaced by $S^{(B)}_{GUE}$ and $S_{GUE}$ respectively, i.e., 
$$
S^{(B)}\to S^{(B)}_{GUE},\quad \quad \left(S^{(B)}_{GUE}\right)_{ab}=1/L,\quad  \quad S \to S _{GUE},\quad \quad \left(S _{GUE}\right)_{xy}=1/N.
$$ 
Thus we can bound  $\mathcal{\widetilde{K}}$ by 
\begin{align}
\left|\frac{d}{d t} \mathcal{\widetilde{K}}_{t, \sigma, a}\right|&\;\le\;  C_nN\max_{ab}\max_{kl}   \left|\mathcal{G}_{k, l}^{(a), L} \circ \mathcal{\widetilde{K}}_{t, \sigma, a}\right|\cdot   \left|\mathcal{G}_{k, l}^{(b), R} \circ \mathcal{\widetilde{K}}_{t, \sigma, a}\right|,\quad t\in [t_1,t_0], 
\\ &
\;\le\; C_n\cdot   N \sum_{2\le k\le n}\left(\max_{\boldsymbol{\sigma}, \textbf{a}}
\left|\mathcal{\widetilde{K}}^{(k)}_{t, \boldsymbol{\sigma}, \textbf{a}}\right|\right)\cdot \left(\max_{\boldsymbol{\sigma}, \textbf{a}}
\left|\mathcal{\widetilde{K}}^{(n-k+2)}_{t, \boldsymbol{\sigma}, \textbf{a}}\right|\right). 
\end{align}
Then for $t\in [t_1,t_0]$, 
\begin{align}
|\mathcal{\widetilde{K}}_{t, \sigma, a}|
\le |\mathcal{\widetilde{K}}_{t_1, \sigma, a}|+C_n N\cdot(t -t_1)
\cdot \sum_{2\le k\le n}\left(\max_{\boldsymbol{\sigma}, \textbf{a}}
\left|\mathcal{\widetilde{K}}^{(k)}_{t, \boldsymbol{\sigma}, \textbf{a}}\right|\right)\cdot \left(\max_{\boldsymbol{\sigma}, \textbf{a}}
\left|\mathcal{\widetilde{K}}^{(n-k+2)}_{t, \boldsymbol{\sigma}, \textbf{a}}\right|\right)
\end{align}
By \eqref{seywoim29},  $t-t_1\le t_0-t_1$ and  
$$
N(t-t_1) \le  N\eta_{t }\cdot N^{-\tau_U} , \quad t\in [t_1,t_0].
$$  
By a continuity argument and the initial estimate  \eqref{uwi2oo1}, we have 
\begin{equation}\label{45} 
 {\cal \widetilde{K}}_{t , \boldsymbol{\sigma}, \textbf{a}}  \prec (N\eta_{t })^{-n+1}, \quad t\in [t_1,t_0]. 
\end{equation}

To bouond $\cal L$, we apply the  main flow estimates from  Lemmas \ref{Sol_CalL} and \ref{lem:DIfREP} with  $s=t_1$ and $t\in [t_1,t_0]$. Together with Lemma \ref{lem:sum_Ndecay} which implies that 
$$
\|{\cal U}_{t_1,t,\boldsymbol{\sigma}}\|_{\max\to\max}\prec (\eta_{t_1}/\eta_{t})^n \prec 1, 
$$
we have 
\begin{align} \label{WLK1}
    (\mathcal{\widetilde{L}} - \mathcal{\widetilde{K}})_{t, \boldsymbol{\sigma}, \textbf{a}} \;\prec \; \; &
    \max_{\boldsymbol{\sigma}, \,\textbf{a}} \left|(\mathcal{\widetilde{L}} - \mathcal{\widetilde{K}})_{t_1, \boldsymbol{\sigma}, \textbf{a}} \right|
    +   \int_{t_1}^t \left(\mathcal{U}_{u, t, \boldsymbol{\sigma}} \circ \mathcal{\widetilde{E}}^{(M)}_{u, \boldsymbol{\sigma}}\right)_{\textbf{a}}  du \nonumber \\
    &+  \int_{t_1}^t \max_{\boldsymbol{\sigma}, \,\textbf{a}}   \left(\sum_{l_\mathcal{K} > 2} \left|\Big[\mathcal{\widetilde{K}} \sim (\mathcal{\widetilde{L}} - \mathcal{\widetilde{K}})\Big]^{l_\mathcal{K}}_{u, \boldsymbol{\sigma},\textbf{a}} \right|+ 
    \left|\mathcal{\widetilde{E}}^{((\mathcal{L} - \mathcal{K}) \times (\mathcal{L} - \mathcal{K}))}_{u, \boldsymbol{\sigma}, \textbf{a}}\right|
+\left|\mathcal{\widetilde{E}}^{(\widetilde{G})}_{u, \boldsymbol{\sigma}, \textbf{a}}\right|\right)du  \end{align}
and 
 \begin{equation} \label{WLK1.5}
   \mathbb{E} \left [    \int_{t_1}^\tau\left(\mathcal{U}_{u, t, \boldsymbol{\sigma}} \circ \mathcal{\widetilde{E}}^{(M)}_{u, \boldsymbol{\sigma}}\right)_{\textbf{a}} du \right ]^{2p} 
   \le C_{n,p} \; 
   \mathbb{E}  \left(
   \int_{t_1}^\tau 
   \max_{\boldsymbol{\sigma}, \,\textbf{a}}   
   \left( \mathcal{\widetilde{E}} \otimes  \mathcal{\widetilde{E}} \right)
   _{u, \,\boldsymbol{\sigma},{\textbf a}, {\textbf a}}du 
   \right)^{p}. 
  \end{equation} 
We  have the following estimates
   \begin{itemize}
       \item Similar to \eqref{def_ELKLK} and \eqref{DefKsimLK},  with $S^{(B)}$ replaced by $S^{(B)}_{GUE}$, we have 
      \begin{equation}
       \label{WLK2}\sum_{l_\mathcal{K} > 2} \left|\Big[\mathcal{\widetilde{K}} \sim (\mathcal{\widetilde{L}} - \mathcal{\widetilde{K}})\Big]^{l_\mathcal{K}}_{u, \boldsymbol{\sigma},\textbf{a}} \right|
       \prec N\cdot
       \sum_{2<k\le n} (N\eta_u)^{k-1}\cdot \max_{\boldsymbol{\sigma}, \,\textbf{a}}   \left|(\mathcal{\widetilde{L}}-\mathcal{\widetilde{K}})^{(n-k+2)}_{t, \boldsymbol{\sigma}, \textbf{a}}\right|
       \end{equation}
        and  \begin{equation}\label{WLK3}\mathcal{\widetilde{E}}^{((\mathcal{L} - \mathcal{K}) \times (\mathcal{L} - \mathcal{K}))}_{u, \boldsymbol{\sigma}, \textbf{a}}\prec N\cdot \sum_{2\le k\le n} \max_{\boldsymbol{\sigma}, \,\textbf{a}}   \left|(\mathcal{\widetilde{L}}-\mathcal{\widetilde{K}})^{(k)}_{t, \boldsymbol{\sigma}, \textbf{a}}\right|\cdot \max_{\boldsymbol{\sigma}, \,\textbf{a}}   \left|(\mathcal{\widetilde{L}}-\mathcal{\widetilde{K}})^{(n-k+2)}_{t, \boldsymbol{\sigma}, \textbf{a}}\right|
        \end{equation}
      where superscript $(k)$ represents the rank, i.e., ${\cal L}^{(n)}$ is an $n$-loop.        \bigskip

 \item For $\mathcal{\widetilde{E}}^{ (\widetilde{G})}$, similarly one can prove \eqref{GavLGEX} for $\widetilde{G}_t$. Then
 \begin{align}\label{WLK4}
       \mathcal{\widetilde{E}}^{ (\widetilde{G})}_{u, \boldsymbol{\sigma}, \textbf{a}}
       \;\prec\; & N\cdot  \max_{\boldsymbol{\sigma}, \,\textbf{a}}   \left| \mathcal{\widetilde{L}} ^{(2)}_{t, \boldsymbol{\sigma}, \textbf{a}}\right|\cdot \max_{\boldsymbol{\sigma}, \,\textbf{a}}  \left|\mathcal{\widetilde{L}} ^{(n+1)} _{t, \boldsymbol{\sigma}, \textbf{a}}\right|
        \end{align}   
 By Cauchy Schwartz's inequality we can bound $n+1$ loop by a product of 
$2$ loop and a $2n$ loop. On the other hand, with \eqref{asu-2w},  a $2n$ loop can be bounded by the square of an $n$ loop. Inserting all these bounds into the previous inequality, we have    
        \begin{align}\label{WLK5} 
        \mathcal{\widetilde{E}}^{ (\widetilde{G})}_{u, \boldsymbol{\sigma}, \textbf{a}}
       \;\prec\; &N\cdot  \max_{\boldsymbol{\sigma}, \,\textbf{a}}   \left| \mathcal{\widetilde{L}} ^{(2)}_{t, \boldsymbol{\sigma}, \textbf{a}}\right|^{3/2}\cdot \max_{\boldsymbol{\sigma}, \,\textbf{a}}   \left|\mathcal{\widetilde{L}} ^{(n )} _{t, \boldsymbol{\sigma}, \textbf{a}}\right|
      \end{align}
      \item For $\mathcal{\widetilde{E}} \otimes  \mathcal{\widetilde{E}}$ term, recall the loops in Definition \ref{def:CALE}. The $S^{(B)}$ in $(\mathcal{\widetilde{E}} \otimes  \mathcal{\widetilde{E}})^{(k)}$ is replaced by $S^{(B)}_{GUE}$ in  $(\mathcal{\widetilde{E}} \otimes  \mathcal{\widetilde{E}})^{(k)}$. Since $S_{GUE}^{(B)}$ is a constant matrix, we can use Ward's identity to sum up the $b$ and $b'$ in the definition of $(\mathcal{\widetilde{E}} \otimes  \mathcal{\widetilde{E}})^{(k)}$ and obtain that  
     \begin{equation}\label{WLK6}
      \left(\mathcal{\widetilde{E}} \otimes  \mathcal{\widetilde{E}} \right)
   _{u, \,\boldsymbol{\sigma},{\textbf a}, {\textbf a}}
   \prec N^{-1}\cdot \eta_u^{-2} \cdot \max_{\boldsymbol{\sigma}, \,\textbf{a}}   \left|\mathcal{\widetilde{L}} ^{(2n )} _{t, \boldsymbol{\sigma}, \textbf{a}}\right|
 \end{equation}
   Then use \eqref{asu-2w}, we can bound $2n $ loop with $n$ loops
   \begin{equation}\label{WLK7}
    \left(\mathcal{\widetilde{E}} \otimes  \mathcal{\widetilde{E}} \right)
   _{u, \,\boldsymbol{\sigma},{\textbf a}, {\textbf a}}\prec  N^{-1}\cdot \eta_u^{-2} \cdot \max_{\boldsymbol{\sigma}, \,\textbf{a}}   \left|\mathcal{\widetilde{L}} ^{( n )} _{t, \boldsymbol{\sigma}, \textbf{a}}\right|^2
\end{equation}
   \end{itemize}
Combining \eqref{uwi2oo1}, \eqref{WLK1}, \eqref{WLK1.5}, \eqref{WLK2}, \eqref{WLK3}, \eqref{WLK5} and \eqref{WLK7} we obtain that 
     \begin{align} \nonumber
    (\mathcal{\widetilde{L}} - \mathcal{\widetilde{K}})_{t, \boldsymbol{\sigma}, \textbf{a}} \;\prec \; \; &
     \left(N|t-t_1|\right)\cdot \sup_{u\in [t_1,t]}\;\sum_{2\le k\le n} \left((N\eta_u)^{k-1}+\max_{\boldsymbol{\sigma}, \,\textbf{a}}   \left|(\mathcal{\widetilde{L}}-\mathcal{\widetilde{K}})^{(k)}_{u, \boldsymbol{\sigma}, \textbf{a}}\right|\right)\cdot \max_{\boldsymbol{\sigma}, \,\textbf{a}}   \left|(\mathcal{\widetilde{L}}-\mathcal{\widetilde{K}})^{(n-k+2)}_{u, \boldsymbol{\sigma}, \textbf{a}}\right| 
    \\\label{WLK8}
    &\;+\;(N\eta_t)^{-n}
    +   \left(N|t-t_1|\right)\cdot \sup_{u\in [t_1,t]} \max_{\boldsymbol{\sigma}, \,\textbf{a}}   \left| \mathcal{\widetilde{L}} ^{(2)}_{u, \boldsymbol{\sigma}, \textbf{a}}\right|^{3/2}\cdot \max_{\boldsymbol{\sigma}, \,\textbf{a}}   \left|\mathcal{\widetilde{L}} ^{(n )} _{u, \boldsymbol{\sigma}, \textbf{a}}\right|
    \\\nonumber
    &\;+\; |t-t_1|^{1/2}\cdot \sup_{u\in [t_1,t]}\left(  N^{-1}\eta_u^{-2}\right)^{1/2} \cdot \max_{\boldsymbol{\sigma}, \,\textbf{a}}   \left|\mathcal{\widetilde{L}} ^{( n )} _{u, \boldsymbol{\sigma}, \textbf{a}}\right| 
    \end{align} 
By writing $\mathcal{\widetilde{L}} 
= (\mathcal{\widetilde{L}} - \mathcal{\widetilde{K}}) + \mathcal{\widetilde{K}} $ and using the bound \eqref{45} on $\mathcal{\widetilde{K}}$, 
the last inequality becomes an inequality on $\mathcal{\widetilde{L}} $.     
By assumption  \eqref{seywoim29}, we have 
    $$
    \eta_u\sim 1-u \ge N^{1-2\tau_U}\gg N^{1- \tau_U}\sim |t_0-t_1|\ge |t-t_1| .
    $$
With  a continuity argument and an  induction on $n\ge 2$, this inequality thus implies the estimate  \eqref{WLKNE1} on $\mathcal{\widetilde{L}}$. 
 
Combining \eqref{uwi2oo1}, \eqref{WLK1}, \eqref{WLK1.5}, \eqref{WLK2}, \eqref{WLK3}, \eqref{WLK4} and \eqref{WLK6},   we obtain that 
     \begin{align} \nonumber
    (\mathcal{\widetilde{L}} - \mathcal{\widetilde{K}})_{t, \boldsymbol{\sigma}, \textbf{a}} \;\prec \; \; &
     \left(N|t-t_1|\right)\cdot \sup_{u\in [t_1,t]}\;\sum_{2\le k\le n} \left((N\eta_u)^{k-1}+\max_{\boldsymbol{\sigma}, \,\textbf{a}}   \left|(\mathcal{\widetilde{L}}-\mathcal{\widetilde{K}})^{(k)}_{u, \boldsymbol{\sigma}, \textbf{a}}\right|\right)\cdot \max_{\boldsymbol{\sigma}, \,\textbf{a}}   \left|(\mathcal{\widetilde{L}}-\mathcal{\widetilde{K}})^{(n-k+2)}_{u, \boldsymbol{\sigma}, \textbf{a}}\right| 
    \\\label{WLK8}
    &\;+\;(N\eta_t)^{-n}
    +   \left(N|t-t_1|\right)\cdot \sup_{u\in [t_1,t]} \max_{\boldsymbol{\sigma}, \,\textbf{a}}   \left| \mathcal{\widetilde{L}} ^{(2)}_{u, \boldsymbol{\sigma}, \textbf{a}}\right| \cdot \max_{\boldsymbol{\sigma}, \,\textbf{a}}   \left|\mathcal{\widetilde{L}} ^{(n+1 )} _{u, \boldsymbol{\sigma}, \textbf{a}}\right|
    \\\nonumber
    &\;+\; |t-t_1|^{1/2}\cdot \sup_{u\in [t_1,t]}\left(  N^{-1}\eta_u^{-2}\right)^{1/2} \cdot \max_{\boldsymbol{\sigma}, \,\textbf{a}}   \left|\mathcal{\widetilde{L}} ^{( 2n )} _{u, \boldsymbol{\sigma}, \textbf{a}}\right| ^{1/2}.
    \end{align}  
 Inserting  the $\cal \widetilde{L}$ bound \eqref{WLKNE1} into the above $\cal \widetilde{L}- \cal \widetilde{K}$ estimate, we have proved that  \eqref{WLKNE2} holds. 
Finally  \eqref{WLKNE3} can be proved by following  the proof of \eqref{Eq:Gtlp_exp_flow} in subsection \ref{sec:step6E}. Since    
  $$
  \ell_u\sim L, \quad \eta_u\sim \eta_{t_0}, \quad t\in [t_1,t_0]$$
in the current setting
 we can prove \eqref{WLKNE3} similarly to  \eqref{Eq:Gtlp_exp_flow} in subsection \ref{sec:step6E}  with very minor changes.

We now prove  \eqref{Weakloclaw} and  \eqref{WeakQUE} for $\widetilde{H}$ (defined in \eqref{defwtHHtu}) by using  \eqref{WLKNE1}, \eqref{WLKNE2}, \eqref{WLKNE3} and the identity \eqref{whheiyslwuw}. First, \eqref{Weakloclaw} clearly follows from  \eqref{whheiyslwuw} and \eqref{WLKNE2}. To prove \eqref{WeakQUE},  by  \eqref{WLKNE3}  we have  the following estimates similar to  \eqref{Meq:QdS1} and \eqref{Meq:QdS2}, namely, with the notations 
$$\widetilde{G}=(\widetilde{H}-z)^{-1},\quad  \im z=\eta=N^{-1+\mathfrak c/ 3}, 
$$ 
and 
 \[
  \widetilde{S}^{(B)}_{xy}  = (1-\zeta_U) S^{(B)}_{xy} + \frac{\zeta_U}{L}, \quad \zeta_U \sim  N^{-1+\tau_U}, \quad \zeta_U=1-e^{-t_U}, 
\]
 we have 
  \begin{align} 
 \max_{a,b }   \left|\mathbb E\tr \widetilde{G} E_a \widetilde{G}^\dagger E_b- W^{-1}\left(\frac{|m|^2}{1-|m|^2 \widetilde{S}^{(B)}} \right)_{ab}\right|\le  \left(N\eta \right)^{-3} \cdot W^\tau,   \\
 \nonumber
 \max_{a,b }   \left|\mathbb E\tr \widetilde{G} E_a \widetilde{G}  E_b- W^{-1}\left(\frac{m^2}{1-m^2 \widetilde{S}^{(B)}} \right)_{ab}\right|\le \left(N\eta \right)^{-3} \cdot W^\tau   .
\end{align}
We have thus proved the  conclusion of Theorem \ref{MR:QDiff} except that $S^{(B)}$ is replaced by $\widetilde{S}^{(B)}$.  We claimed that  the proof for \eqref{Meq:QUE} 
from \eqref{ssfa2} to \eqref{que2} can be used 
to derive the   \eqref{WeakQUE} for $\widetilde{H}$ with only very minor modifications.
By definition of $\widetilde{S}^{(B)}$, we can remove the zero mode in $\widetilde{S}^{(B)}$ and write
  $$\left| \left(\frac{1}{1-\xi\cdot \widetilde{S}^{(B)}}\right)_{ab}-\left(\frac{1}{1-\xi\cdot \widetilde{S}^{(B)}}\right)_{a'b'}\right|
 =
 \left| \left(\frac{1}{1-\xi\cdot  (1-\zeta_U){S}^{(B)}}\right)_{ab}-\left(\frac{1}{1-\xi\cdot  (1-\zeta_U){S}^{(B)}}\right)_{a'b'}\right|.
 $$
In both  $\xi=m^2$ or $|m^2|$ we have 
$$
|1-\xi |\sim |1-\xi (1-\zeta_U)|.
$$
We can now follow the argument from \eqref{ssfa2} to \eqref{que2} to prove 
\eqref{WeakQUE} for $\widetilde{H}$ \eqref{WLKNE3}. This completes the proof of \eqref{WeakQUE}.

\appendix
 \section{\texorpdfstring{$G$}{G}-chains}
 Besides the loop of \( G \) entries defined in Definition~\ref{Def:G_loop}, we can also bound certain \textit{chains} of \( G \) entries. Although the following  results were not used in the proof of our main theorem, they might  be useful in the future applications of the methods discussed  in this paper. 

\begin{definition}[$G$-chains]\label{def_Gchain}
Let 
\[
\boldsymbol{\sigma} = (\sigma_1, \sigma_2, \ldots, \sigma_n), \quad \mathbf{a} = (a_1, a_2, \ldots, a_{n-1}), \quad \sigma_k \in \{+, -\}, \quad a_k \in \mathbb{Z}_L.
\]
Following the structure in~\eqref{defC=GEG}, we define a \emph{$G$-chain} as
\begin{equation}
    \mathcal{C}_{t, \boldsymbol{\sigma}, \mathbf{a}} = G_t(\sigma_1) E_{a_1} G_t(\sigma_2) E_{a_2} \cdots E_{a_{n-1}} G_t(\sigma_n).
\end{equation}
By definition, $\mathcal{C}_{t, \boldsymbol{\sigma}, \mathbf{a}}$ is an \( N \times N \) matrix. We classify its entries as
\[
\left( \mathcal{C}_{t, \boldsymbol{\sigma}, \mathbf{a}} \right)_{ij} \sim 
\begin{cases}
    \text{off-diagonal terms of } \mathcal{C}_{t, \boldsymbol{\sigma}, \mathbf{a}}, & \text{if } i \neq j, \\ 
    \text{diagonal terms of } \mathcal{C}_{t, \boldsymbol{\sigma}, \mathbf{a}}, & \text{if } i = j.
\end{cases}
\]
\end{definition}

Before presenting the lemmas for estimating $G$-chains, we provide a heuristic argument for the following bounds on  $G$-chains:
\begin{equation}\label{res:heur_Chain}
    \left( \mathcal{C}_{t, \boldsymbol{\sigma}, \mathbf{a}} \right)_{ii} \prec (W\ell_t\eta_t)^{-n+1}, \quad 
    \left( \mathcal{C}_{t, \boldsymbol{\sigma}, \mathbf{a}} \right)_{ij} \prec (W\ell_t\eta_t)^{-n+1/2}.
\end{equation}
A $G$-chain can be converted into a $G$-loop by multiplication with an $E_a$ operator followed by taking the trace:
\[
\mathcal{L}_{t, \boldsymbol{\sigma}, \mathbf{a}'} = \langle \mathcal{C}_{t, \boldsymbol{\sigma}, \mathbf{a}} E_a \rangle, \quad \mathbf{a}' = (a_1, a_2, \ldots, a_{n-1}, a).
\]
Assuming that \( \left(\mathcal{C}_{t, \boldsymbol{\sigma}, \mathbf{a}}\right)_{ii} \) is of comparable size for all \( i \in \mathcal{I}_a \), we obtain from~\eqref{Eq:LGxb}
\[
\left(\mathcal{C}_{t, \boldsymbol{\sigma}, \mathbf{a}}\right)_{ii} \sim \mathcal{L}_{t, \boldsymbol{\sigma}, \mathbf{a}'} \sim (W\ell_t\eta_t)^{-n+1}.
\]
Similarly, a different $G$-loop can be constructed by multiplying a $G$-chain with two operators $E_a$ and $E_b$ and taking the trace:
\[
\mathcal{L}_{t, \boldsymbol{\sigma}'', \mathbf{a}''} = \langle E_a \cdot \mathcal{C}_{t, \boldsymbol{\sigma}, \mathbf{a}} \cdot E_b \cdot \mathcal{C}^\dagger_{t, \boldsymbol{\sigma}, \mathbf{a}} \rangle,
\]
where
\[
\boldsymbol{\sigma}'' = (\sigma_1, \sigma_2, \ldots, \sigma_n, \bar{\sigma}_n, \bar{\sigma}_{n-1}, \ldots, \bar{\sigma}_1), \quad 
\mathbf{a}'' = (a_1, a_2, \ldots, a_{n-1}, b, a_{n-1}, \ldots, a_1, a).
\]
Assuming that \( \left(\mathcal{C}_{t, \boldsymbol{\sigma}, \mathbf{a}}\right)_{ij} \) has a typical size across all \( i \in \mathcal{I}_a \), \( j \in \mathcal{I}_b \), we deduce from~\eqref{Eq:LGxb} that
\[
\left(\mathcal{C}_{t, \boldsymbol{\sigma}, \mathbf{a}}\right)_{ij} \sim \left( \mathcal{L}_{t, \boldsymbol{\sigma}'', \mathbf{a}''} \right)^{1/2} \sim (W\ell_t\eta_t)^{-n+1/2}.
\]

\begin{lemma}[$n$-chain estimate]\label{lem_GbEXP_n2}  
Under the assumptions of Lemma~\ref{ML:GLoop}, for any fixed \( n \in \mathbb{N} \) and \( \boldsymbol{\sigma} \in \{+, -\}^n \), we have
\begin{align}\label{res_XICdi}
    \max_{\mathbf{a}} \max_i \left| \left( \mathcal{C}_{t, \boldsymbol{\sigma}, \mathbf{a}} \right)_{ii} \right| 
    &\prec \left(W \ell_t \eta_t\right)^{-n+1}, \\\label{res_XICoff}
    \max_{\mathbf{a}} \max_{i \ne j} \left| \left( \mathcal{C}_{t, \boldsymbol{\sigma}, \mathbf{a}} \right)_{ij} \right| 
    &\prec \left(W \ell_t \eta_t\right)^{-n + 1/2}.
\end{align}
\end{lemma}

The proof relies on a standard decomposition widely used for Wigner matrices (see, e.g.,~\cite{erdHos2012rigidity}) and random band matrices (e.g.,~\cite{Average_fluc}).

\begin{proof}[Proof of Lemma~\ref{lem_GbEXP_n2}]
We begin by introducing the following shorthand:
\begin{align}
\Xi_n^{(\mathrm{d})} := \Xi_{t,n}^{(\mathcal{C}, \mathrm{diag})} 
&:= \max_{\boldsymbol{\sigma}, \mathbf{a}} \max_i \left| \left( \mathcal{C}_{t, \boldsymbol{\sigma}, \mathbf{a}} \right)_{ii} \right| 
\cdot (W \ell_t \eta_t)^{n-1} \cdot \mathbf{1}\left( \boldsymbol{\sigma} \in \{+, -\}^n \right), \\
\Xi_n^{(\mathrm{o})} := \Xi_{t,n}^{(\mathcal{C}, \mathrm{off})} 
&:= \max_{\boldsymbol{\sigma}, \mathbf{a}} \max_{i \ne j} \left| \left( \mathcal{C}_{t, \boldsymbol{\sigma}, \mathbf{a}} \right)_{ij} \right| 
\cdot (W \ell_t \eta_t)^{n - 1/2} \cdot \mathbf{1}\left( \boldsymbol{\sigma} \in \{+, -\}^n \right).
\end{align}
For the base case \( n = 1 \), Theorem~\ref{MR:locSC} implies
\[
\Xi_1^{(\mathrm{d})} + \Xi_1^{(\mathrm{o})} \prec 1.
\]
Applying the Cauchy–Schwarz inequality to the diagonal part of a 2-chain yields
\[
\Xi_2^{(\mathrm{d})} \prec \left( \Xi_1^{(\mathrm{d})} \right)^2 + \left( \Xi_1^{(\mathrm{o})} \right)^2 \prec 1.
\]

We now prove the inductive step. Suppose that for all \( n \ge 2 \), we have the bound:
\begin{align}\label{aspchain}
\left( \max_{k \le n - 1} \Xi_k^{(\mathrm{o})} + \max_{k \le 2n - 2} \Xi_k^{(\mathrm{d})} \right) \prec 1 
\quad \Longrightarrow \quad 
\left( \max_{k \le n} \Xi_k^{(\mathrm{o})} + \max_{k \le 2n} \Xi_k^{(\mathrm{d})} \right) \prec 1.
\end{align}
This inductive bound immediately implies the estimates~\eqref{res_XICdi} and~\eqref{res_XICoff}.
To proceed, define the $n$-chain
\begin{equation}\label{defSPChain}
\mathcal{C}_n := G_1 E_{a_1} G_2 \cdots G_{n-1} E_{a_{n-1}} G_n.
\end{equation}
Let \( \mathcal{C}_n^{(i)} \) denote the $n$-chain where all \( G_k \) with \( k \ge 2 \) are replaced by \( G_k^{(i)} \) (defined in~\eqref{def_G(i)}):
\begin{equation}\label{defCn_i}
\mathcal{C}_n^{(i)} := G_1 E_{a_1} G^{(i)}_2 \cdots G^{(i)}_{n-1} E_{a_{n-1}} G^{(i)}_n.
\end{equation}
Similarly, let \( \mathcal{C}_n^{(ii)} \) be the $n$-chain with all \( G_k \) replaced by \( G_k^{(i)} \), including the first:
\begin{equation}\label{defCn_ii}
\mathcal{C}_n^{(ii)} := G^{(i)}_1 E_{a_1} G^{(i)}_2 \cdots G^{(i)}_{n-1} E_{a_{n-1}} G^{(i)}_n.
\end{equation}

Under the assumption of~\eqref{aspchain}, we claim the following perturbation estimates hold for \( n \ge 2 \):
\begin{equation}\label{tancca}
i \ne j, \qquad 
\left( \mathcal{C}_n \right)_{ij}
- \left( \mathcal{C}^{(i)}_n \right)_{ij}
\prec (W \ell \eta)^{-n + 1/2},
\end{equation}
and
\begin{equation}\label{jsjuzggfnw}
\left\langle \mathcal{C}_n^{(ii)} E_a \mathcal{C}_n^{(ii)\;\dagger} E_b \right\rangle
- \left\langle \mathcal{C}_n E_a \mathcal{C}_n^\dagger E_b \right\rangle
\prec \left( \Xi_{2n}^{(\mathrm{d})} + 1 \right) \cdot (W \ell \eta)^{-2n + 1/2}.
\end{equation}
Moreover, if in addition \( \Xi_{2n-1}^{(\mathrm{d})}+\Xi_{2n}^{(\mathrm{d})} \prec 1 \), then we also have
\begin{equation}\label{jsjuzggf}
i \ne j, \qquad
\left( \mathcal{C}_n^{(ii)} E_a \mathcal{C}_n^{(ii)\;\dagger} \right)_{jj} 
- \left( \mathcal{C}_n E_a \mathcal{C}_n^\dagger \right)_{jj} 
\prec \left( \left( \Xi_n^{(\mathrm{o})} \right)^2 + 1 \right) \cdot (W \ell \eta)^{-2n + 1/2}.
\end{equation}
We postpone the proofs of these perturbative estimates to the end of this section.

We now  prove
\begin{equation}\label{juwiuo2pl}
    \Xi^{(\mathrm{d})}_{2n-1} + \Xi^{(\mathrm{d})}_{2n} \prec 1,
\end{equation}
using~\eqref{tancca} and~\eqref{jsjuzggfnw}. By the Cauchy–Schwarz inequality and the assumption in~\eqref{aspchain}, it suffices to control the case \( 2n \). Without loss of generality, we may assume that the $2n$-chain \( \mathcal{C} \) is symmetric and of the form
\[
\left( \mathcal{C}_n E_a \mathcal{C}_n^\dagger \right)_{ii}.
\]
We decompose this expression as
\[
\left( \mathcal{C}_n E_a \mathcal{C}_n^\dagger \right)_{ii}
= W^{-1} \mathcal{C}_{n; ii} \cdot \mathcal{C}^\dagger_{n; ii}
+ \sum_{j \ne i} \mathcal{C}_{n; ij} \cdot E_a(j) \cdot \mathcal{C}^\dagger_{n; ji},
\]
and similarly for \( \mathcal{C}_n^{(i)} E_a \left(\mathcal{C}_n^{(i)}\right)^\dagger \). Then,
\begin{align*}
\left( \mathcal{C}_n E_a \mathcal{C}_n^\dagger \right)_{ii}
- \left( \mathcal{C}_n^{(i)} E_a \left( \mathcal{C}_n^{(i)} \right)^\dagger \right)_{ii}
\;\prec\;& (W \ell \eta)^{-2n + 1} \left( \Xi_n^{(\mathrm{d})} \right)^2
+ \max_{j \ne i} \left| \left( \mathcal{C}_n \right)_{ij} - \left( \mathcal{C}_n^{(i)} \right)_{ij} \right|^2\\
&+ \left( \sum_{j \ne i} \left| \mathcal{C}_{n; ij} \right|^2 E_a(j) \right)^{1/2} \cdot \max_{j \ne i} \left| \left( \mathcal{C}_n \right)_{ij} - \left( \mathcal{C}_n^{(i)} \right)_{ij} \right|.
\end{align*}
Applying~\eqref{tancca} and the assumption in~\eqref{aspchain}, we obtain
\begin{equation}\label{j,mssl}
\left( \mathcal{C}_n E_a \mathcal{C}_n^\dagger \right)_{ii}
- \left( \mathcal{C}_n^{(i)} E_a \left( \mathcal{C}_n^{(i)} \right)^\dagger \right)_{ii}
\prec (W \ell \eta)^{-2n + 1} \cdot \left( \Xi_{2n}^{(\mathrm{d})} + 1 \right)^{1/2}.
\end{equation}

To conclude~\eqref{juwiuo2pl}, it remains to estimate \( \mathcal{C}_n^{(i)} E_a \left( \mathcal{C}_n^{(i)} \right)^\dagger \). Recall the definition of \( \mathcal{C}_n^{(ii)} \) in~\eqref{defCn_ii}. Using~\eqref{GijSCP2}, we write:
\[
\left( \mathcal{C}_n^{(i)} E_a \left( \mathcal{C}_n^{(i)} \right)^\dagger \right)_{ii}
= |G_{ii}|^2 \cdot \left( H \cdot \mathcal{C}_n^{(ii)} E_a \left( \mathcal{C}_n^{(ii)} \right)^\dagger \cdot H \right)_{ii}.
\]
Since \( \mathcal{C}_n^{(ii)} \) is independent of \( \{ H_{ik} \}_{k=1}^N \), we apply Lemma~3.3 from~\cite{erdHos2012rigidity} to obtain
\begin{equation}\label{rejdssaw2s}
\left( \mathcal{C}_n^{(i)} E_a \left( \mathcal{C}_n^{(i)} \right)^\dagger \right)_{ii}
= |G_{ii}|^2 \sum_k S_{ik} \left( \mathcal{C}_n^{(ii)} E_a \left( \mathcal{C}_n^{(ii)} \right)^\dagger \right)_{kk}
+ O_\prec \left( \sum_{kl} S_{ik} \left| \left( \mathcal{C}_n^{(ii)} E_a \left( \mathcal{C}_n^{(ii)} \right)^\dagger \right)_{kl} \right|^2 S_{li} \right)^{1/2}.
\end{equation}
The first term on the right can be bounded as   $2n$-chain loops with $G^{(i)}$ entries, while the second can be interpreted as   $4n$-chain loops with $G^{(i)}$ entries. Similar to \eqref{asuwmpwpisu}, we can bound $4n-G^{(i)}$ loop with two $2n-G^{(i)}$ loops as follows
\[
\Xi_{4n}^{\mathcal{L},(i)} \prec \left( \Xi_{2n}^{\mathcal{L},(i)} \right)^2 \cdot (W \ell \eta).
\]
Using \eqref{jsjuzggfnw}, we have 
\[ 
\Xi_{2n}^{\mathcal{L},(i)} \prec \left( \Xi_{2n}^{(\mathrm{d})} + (W \ell \eta)^{1/2} \right) \cdot (W \ell \eta)^{-2n + 1/2}.
\]
Hence,
\[
\left( \mathcal{C}_n^{(i)} E_a \left( \mathcal{C}_n^{(i)} \right)^\dagger \right)_{ii}
\prec \left( \Xi_{2n}^{(\mathrm{d})} + (W \ell \eta)^{1/2} \right) \cdot (W \ell \eta)^{-2n + 1/2}.
\]
Combining this with~\eqref{j,mssl}, we conclude
\[
\left( \mathcal{C}_n E_a \mathcal{C}_n^\dagger \right)_{ii}
\prec (W \ell \eta)^{-2n + 1} \cdot \left( \left( \Xi_{2n}^{(\mathrm{d})} \right)^{1/2} + \Xi_{2n}^{(\mathrm{d})} \cdot (W \ell \eta)^{-1/2} + 1 \right),
\]
which establishes~\eqref{juwiuo2pl}.

\bigskip

We now  prove
\begin{equation}\label{juwiu123l}
    \Xi^{(\mathrm{o})}_n \prec 1,
\end{equation}
using~\eqref{tancca} and~\eqref{jsjuzggf}. By~\eqref{GijSCP2}, we express \( \left( \mathcal{C}_n^{(i)} \right)_{ij} \) in terms of \( \mathcal{C}_n^{(ii)} \) as
\begin{equation}\label{ajplaos}
\left( \mathcal{C}_n^{(i)} \right)_{ij} = (G_1)_{ii} \cdot \left( H \cdot \mathcal{C}_n^{(ii)} \right)_{ij}.
\end{equation}
Since \( \mathcal{C}_n^{(ii)} \) is independent of \( \{ H_{ik} \}_{k=1}^N \), we again apply Lemma~3.3 of~\cite{erdHos2012rigidity}:
\begin{align}\label{lkjdsaos}
\left( H \cdot \mathcal{C}_n^{(ii)} \right)_{ij}
&\prec \left( \sum_k S_{ik} \left| \left( \mathcal{C}_n^{(ii)} \right)_{kj} \right|^2 \right)^{1/2}
\prec \sum_a \mathbf{1}(|a - [i]| \le 1) \cdot \left( \mathcal{C}_n^{(ii)} E_a \mathcal{C}_n^{(ii)\dagger} \right)_{jj}^{1/2}.
\end{align}
Combining this with~\eqref{ajplaos}, we find that for \( i \ne j \),
\begin{equation}\label{bCimmax}
\left( \mathcal{C}_n^{(i)} \right)_{ij} \prec \max_a \left( \mathcal{C}_n^{(ii)} E_a \mathcal{C}_n^{(ii)\dagger} \right)_{jj}^{1/2}.
\end{equation}
Using this estimate along with~\eqref{jsjuzggf} and~\eqref{tancca}, we obtain
\[
\left( \mathcal{C}_n \right)_{ij} \prec (W \ell \eta)^{-n + 1/2} \left( 1 + \Xi_n^{(\mathrm{o})} \cdot (W \ell \eta)^{-1/4} \right),
\]
which yields the desired bound~\eqref{juwiu123l}.

\medskip
\noindent
\textbf{Proof of~\eqref{tancca}.} Using~\eqref{GijkSCP} to express \( G^{(i)} \) in~\eqref{defCn_i} in terms of \( G \), and bounding \( 1/G_{ii} = O(1) \) via Theorem~\ref{MR:locSC}, we can estimate the difference \( \left( \mathcal{C}_n \right)_{ij} - \left( \mathcal{C}^{(i)}_n \right)_{ij} \) in terms of products of diagonal and off-diagonal chain entries. For example, for \( n = 2 \), we have:
\begin{align}
\left( \mathcal{C}^{(i)}_2 \right)_{ij} - \left( \mathcal{C}_2 \right)_{ij}
&= -\frac{\left( G_1 E_{a_1} G_2 \right)_{ii} \cdot (G_2)_{ij}}{(G_2)_{ii}} 
= O\left( \Xi^{(\mathrm{d})}_2 \cdot \Xi^{(\mathrm{o})}_1 \cdot (W \ell \eta)^{-3/2} \right).
\end{align}

For \( n = 3 \), we similarly write:
\begin{align}
\left( \mathcal{C}^{(i)}_3 \right)_{ij} - \left( \mathcal{C}_3 \right)_{ij}
&= - \frac{\left( G_1 E_{a_1} G_2 \right)_{ii} \cdot \left( G_2 E_{a_2} G_3 \right)_{ij}}{(G_2)_{ii}} 
- \frac{\left( G_1 E_{a_1} G_2 E_{a_2} G_3 \right)_{ii} \cdot (G_3)_{ij}}{(G_3)_{ii}} \\
&\quad + \frac{\left( G_1 E_{a_1} G_2 \right)_{ii} \cdot \left( G_2 E_{a_2} G_3 \right)_{ii} \cdot (G_3)_{ij}}{(G_2)_{ii} (G_3)_{ii}} \notag \\
&= O\left( \Xi^{(\mathrm{d})}_2 \Xi^{(\mathrm{o})}_2 + \Xi^{(\mathrm{d})}_3 \Xi^{(\mathrm{o})}_1 
+ \Xi^{(\mathrm{d})}_2 \Xi^{(\mathrm{d})}_2 \Xi^{(\mathrm{o})}_1 \right) \cdot (W \ell \eta)^{-5/2}.
\end{align}

In general, we obtain the bound:
\begin{align} \label{jasuznns}
\left( \mathcal{C}_n \right)_{ij} - \left( \mathcal{C}^{(i)}_n \right)_{ij}
\prec \sum_{k \ge 1} \sum_{\{n_i\}_{i=1}^k} \left( \prod_{i=1}^k \Xi^{(\mathrm{d})}_{n_i} \right) \Xi^{(\mathrm{o})}_\ell \cdot (W \ell \eta)^{-n + 1/2} \cdot \mathbf{1}\left( \ell + \sum_{i=1}^k n_i = n + k \right),
\end{align}
where \( 2 \le n_i \le n \le 2n - 2 \), and \( 1 \le \ell \le n - 1 \). Applying the induction hypothesis~\eqref{aspchain} to all shorter chains involved, we get:
\[
\Xi^{(\mathrm{d})}_{n_i} + \Xi^{(\mathrm{o})}_\ell \prec 1,
\]
which implies~\eqref{tancca}.

\bigskip
\noindent
\textbf{Proof of~\eqref{jsjuzggfnw}.} Observe that \( \left( \mathcal{C}^{(ii)}_n E_a \mathcal{C}^{(ii)\dagger}_n \right)_{jj} \) is a diagonal entry of a $2n$-$G^{(i)}$ chain, while \( \left( \mathcal{C}_n E_a \mathcal{C}^\dagger_n \right)_{jj} \) is the same object but with \( G^{(i)} \) replaced by \( G \). Applying~\eqref{GijkSCP} to express all \( G^{(i)} \) in terms of \( G \), and again bounding \( 1/G_{ii} = O(1) \), we estimate the difference:
\[
\left\langle \mathcal{C}^{(ii)}_n E_a \mathcal{C}^{(ii)\dagger}_n E_b \right\rangle - \left\langle \mathcal{C}_n E_a \mathcal{C}^\dagger_n E_b \right\rangle.
\]
This difference involves products of diagonal chains, and similar to~\eqref{jasuznns}, we obtain:
\begin{align}\label{asdjzzu,l}
\left\langle \mathcal{C}^{(ii)}_n E_a \mathcal{C}^{(ii)\dagger}_n E_b \right\rangle 
- \left\langle \mathcal{C}_n E_a \mathcal{C}^\dagger_n E_b \right\rangle 
\prec \sum_{k \ge 1} \sum_{\{n_i\}_{i=1}^k} 
\left( \prod_{i=1}^k \Xi^{(\mathrm{d})}_{n_i} \right) \cdot (W \ell \eta)^{-2n} \cdot \mathbf{1}\left( \sum_{i=1}^k n_i = 2n + k \right),
\end{align}
with \( 2 \le n_i \le 2n + 1 \).
By the inductive assumption~\eqref{aspchain}, we can bound:
\[
\left( \prod_{i=1}^k \Xi^{(\mathrm{d})}_{n_i} \right) \cdot \mathbf{1}\left( \sum_{i=1}^k n_i = 2n + k \right)
\; \;\prec \;\;1 + \mathbf{1}_{n=2} \left( \Xi^{(\mathrm{d})}_{2n-1} \right)^2 + \Xi^{(\mathrm{d})}_{2n} + \Xi^{(\mathrm{d})}_{2n+1}.
\]
Applying the Cauchy–Schwarz inequality yields:
\[
\Xi^{(\mathrm{d})}_{2n - 1} \prec \left( \Xi^{(\mathrm{d})}_{2n - 2} \cdot \Xi^{(\mathrm{d})}_{2n} \right)^{1/2} \prec \left( \Xi^{(\mathrm{d})}_{2n} \right)^{1/2}.
\]
Furthermore, splitting a \( 2n + 1 \)-chain into two \( n \)-chains and one 1-chain, and using a triple-product Cauchy–Schwarz inequality (i.e., \( \int f(x) g(y) h(x,y) \le \|f\|_2 \|g\|_2 \|h\|_2 \)), we get:
\[
\Xi^{(\mathrm{d})}_{2n + 1} \prec \Xi^{(\mathrm{d})}_{2n} \cdot (W \ell \eta)^{1/2}.
\]
Putting all of these bounds into~\eqref{asdjzzu,l}, we conclude the proof of~\eqref{jsjuzggfnw}.

\bigskip
\bigskip
\noindent
\textbf{Proof of~\eqref{jsjuzggf}.} Similar to the expansion in~\eqref{asdjzzu,l}, we estimate the difference:
\begin{align}\label{asdjzzu,l2}
&\left( \mathcal{C}^{(ii)}_n E_a \mathcal{C}^{(ii)\, \dagger}_n \right)_{jj} 
- \left( \mathcal{C}_n E_a \mathcal{C}^\dagger_n \right)_{jj}
\\\nonumber
  \;\prec\; &  \sum_{k \ge 0} \sum_{\{n_i\}_{i=1}^k} \sum_{l_1,\, l_2}
\left( \prod_{i=1}^k \Xi^{(\mathrm{d})}_{n_i} \right)
\left( \prod_{j=1}^2 \Xi^{(\mathrm{o})}_{l_j} \right) \cdot (W \ell \eta)^{-2n}
\cdot \mathbf{1} \left( l_1 + l_2 + \sum_{i=1}^k n_i = 2n + k + 1 \right),
\end{align}
where the index ranges satisfy:
\[
2 \le n_i \le 2n, \qquad 1 \le l_1 \le l_2 \le 2n.
\]
Under the assumptions on \( \Xi^{(\mathrm{d})}_{n_i} \), it suffices to bound the product of off-diagonal terms. That is, we only need to show that:
\begin{align}\label{suyww2}
l_1 + l_2 \le 2n + 1 \quad \Longrightarrow \quad 
\Xi^{(\mathrm{o})}_{l_1} \cdot \Xi^{(\mathrm{o})}_{l_2} 
\prec \left( \left( \Xi^{(\mathrm{o})}_n \right)^2 + 1 \right) \cdot (W \ell \eta)^{1/2}.
\end{align}

\medskip

By the inductive assumption, we have 
\[
\Xi^{(\mathrm{o})}_{l_j} \prec
\begin{cases}
1, & l_j \le n - 1, \\[6pt]
\Xi^{(\mathrm{o})}_n, & l_j = n, \\[6pt]
(W \ell \eta)^{1/2}, & n < l_j \le 2n,
\end{cases}
\]
where for \( l_j > n \), the bound follows from applying the Cauchy–Schwarz inequality.
These bounds clearly imply~\eqref{suyww2}, and hence we obtain~\eqref{jsjuzggf}. This completes the proof of Lemma~\ref{lem_GbEXP_n2}.

\end{proof}

\printbibliography 

\end{document}